\documentclass[11pt]{amsproc}
\pdfoutput=1
\usepackage{amsmath,amssymb,amsthm,eucal, eufrak,amscd,tikz,mathtools}
\usepackage{xcolor}
\usepackage{enumitem}
\usepackage[shortalphabetic]{amsrefs}
\usepackage{xypic}
\usepackage{graphicx}
\definecolor{allrefcolors}{rgb}{0,0.2,0.5}
\usepackage[linktocpage=true,colorlinks=true,allcolors=allrefcolors,bookmarksopen,bookmarksdepth=3]{hyperref}
\usepackage[margin=1.25in]{geometry}

\newcommand{\cev}[1]{\reflectbox{\ensuremath{\vec{\reflectbox{\ensuremath{#1}}}}}}

\newcommand{\D}{\mathbf{D}}
\newcommand{\PSS}{\operatorname{PSS}}
\newcommand{\PSSlog}{\operatorname{PSS}_{log}}

\newcommand{\DIo}{\mathring{D}_I}

\newcommand{\SIo}{\mathring{S}_I}
\newcommand{\Etop}{E_{\operatorname{top}}}
\newcommand{\Egeo}{E_{\operatorname{geo}}}
\newcommand{\Evzo}{\operatorname{Ev}^{\v}_{z_{0}}}
\newcommand {\Evo}{\operatorname{Ev}^{\v}_{0}}
 
\newcommand{\GWv}{GW_{\v_I}(\alpha)}
\newcommand{\GWvc}{GW_{\v_I}(\underline{\alpha})}

\newcommand{\vi}{v_i}

\def\hlm{h^{\ell}}
\def\Hlm{H^{\ell}}
\def\Hm{H}

\def\Glm{G^{\ell}}
\def\Gm{G}
\def\rlm{R^{\ell}}
\def\rm{R}

\newcommand{\HFXM}{HF^*(X \subset M; \Hlm)} 
\newcommand{\HFw}{HF^*(X \subset M; \Hlm)_{w}}
\newcommand{\FwHF}{F_wHF^*(X \subset M; \Hlm)}

\newcommand{\FpHF}{F_{w_{p}}HF^*(X \subset M; H_p^\ell)}

\newcommand{\FwCF}{F_wCF^*(X \subset M; \Hlm)}
 
\newcommand{\kappail}{\kappa_{i,\ell}}
\newcommand{\hatkappa}{\widehat{\kappa}_{i,\ell}}
\newcommand{\hatX}{\widehat{\Sigma}_{\epsilon_{1}}}
\newcommand{\hatXl}{\widehat{\Sigma}_{\epsilon_{\ell}}}
\newcommand{\hatXlp}{\widehat{\Sigma}_{\epsilon_{\ell}^p}}
\newcommand{\Uiloc}{U_{i,\ell}^{loc}}
\newcommand{\UIloc}{U_{I,\ell}^{loc}}
\newcommand{\out}{\operatorname{out}}
\newcommand{\chifI}{\mathcal{X}(\mathring{S}_I,f_I)}
\newcommand{\critfI}{\mathcal{X}(\mathring{S}_I,f_I)} 

\def\ra{\rightarrow}

\def\Z{{\mathbb Z}}
\def\R{{\mathbb R}}
\def\C{{\mathbb C}}
\def\P{{\mathbb P}}
\def\K{{\mathbf{k}}}

\newcommand{\QH}{H_{log}}

\newcommand{\kappaipert}{\kappa_{i,p}}

\def\mc#1{\mathcal{#1}}

\def\k{\kappa}

\def\ainf{A_\infty}

\def\z2{\Z / 2\Z}

\def\z{\mc{Z}}

\def\v{\mathbf{v}}

\newtheorem{lem}{Lemma}[section]

\newtheorem{thm}[lem]{Theorem}
\newtheorem{cor}[lem]{Corollary}

\newtheorem{defn}[lem]{Definition}

\newtheorem{ques}[lem]{Question}
\newtheorem{rem}[lem]{Remark}

\theoremstyle{remark}

\newtheorem{example}{Example}[section]
\numberwithin{equation}{section}
\begin{document}
\begin{abstract}

    We construct a multiplicative spectral sequence converging to the
    symplectic cohomology ring of any affine variety $X$, with first page built
    out of topological invariants associated to strata of any fixed normal
    crossings compactification $(M,\D)$ of $X$. We exhibit a broad class of
    pairs $(M,\D)$ (characterized by the absence of relative holomorphic
    spheres or vanishing of certain relative GW invariants) for which the
    spectral sequence degenerates, and a broad subclass of pairs (similarly
    characterized) for which the ring structure on symplectic cohomology can
    also be described topologically. Sample applications include: (a) a
    complete topological description of the symplectic cohomology ring of the
    complement, in any projective $M$, of the union of sufficiently many
    generic ample divisors whose homology classes span a rank one subspace,
    (b) complete additive and partial multiplicative computations of degree
    zero symplectic cohomology rings of many log Calabi-Yau varieties, and 
    (c) a proof in many cases that symplectic cohomology is finitely generated
    as a ring. A key technical ingredient in our results is a logarithmic
    version of the PSS morphism, introduced in our earlier work \cite{GP1}.

\end{abstract}
\title{Symplectic cohomology rings of affine varieties in the topological limit}
\author{Sheel Ganatra and Daniel Pomerleano}
\thanks{S.~G.~  was partially supported by the National Science Foundation through a postdoctoral fellowship --- grant number DMS-1204393 --- and agreement number DMS-1128155. Any opinions, findings and conclusions or recommendations expressed in this material are those of the author(s) and do not necessarily reflect the views of the National Science Foundation.\\
    \mbox{ }\mbox{ }\mbox{ }\mbox{ }\mbox{ }D.~P.~was supported by EPSRC, Imperial College, University of Cambridge, and UMass Boston.}
\maketitle

\section{Introduction} \label{subsection:intro} 

Symplectic cohomology, a version of Hamiltonian Floer homology for exact
convex symplectic manifolds (such as affine varieties and more generally
Stein manifolds), has attracted widespread attention as a powerful
invariant of symplectic manifolds. 
A landmark result of Viterbo \cite{Viterbo:1996kx, Salamon:2006ys, AbSch1}
gives a topological description of symplectic cohomology in the case of
cotangent bundles and it is known that the invariant vanishes under particular
hypotheses on the underlying Stein topology (such as being ``subcritical'' or
``flexible'', see \cite{Cieliebak, BEE, MurphySiegel}). 
Outside of these cases, there are relatively few explicit computations, which
tend to be difficult and rely, in one way or another, on
the enumeration of pseudo-holomorphic curves (see \S \ref{subsection:
othermthds} for a comparison of other computational approaches with the
methods developed here).

The present paper gives a new technique for computing symplectic cohomology
rings of affine varieties via reduction to algebraic topology and algebraic
geometry. To set notation, let $X$ be a smooth complex affine variety, and let
us fix any smooth projective compactification $M$ of $X$ by a simple normal
crossings divisor $\D = D_1 \cup \cdots \cup D_k$ supporting an ample line
bundle, which is guaranteed to us by Hironaka's theorem. One can encode much of
the combinatorics and algebraic topology of the pair $(M, \D)$ in a ring called
the {\em log(arithmic) cohomology} of $(M,\D)$, which we denote $H^*_{log}(M,\D)$ 
(the underlying abelian group was introduced in our earlier work \cite{GP1}) .
See \S \ref{subsec:logcoh} for a precise definition; roughly, $H^*_{log}(M,\D)$
is the direct sum of the cohomology $H^*(X)$ and classes of the form
$\alpha t^{\v}$, where $\v \in (\mathbb{Z}^{\geq 0})^{k}$ is a multiplicity
vector, $\alpha$ is a cohomology class in the normal torus bundle $\SIo$ to
the (open part of the) stratum $D_I = D_{I_1} \cap \cdots \cap D_{I_{|I|}}$ of
$(M,\D)$; the subset $I \subset \{1, \ldots, k\}$ is required to consist
exactly of the indices $i$ for which $\v_i \neq 0$. The product is given by adding
multiplicity vectors and restricting the cohomology classes to a common stratum
where they can be multiplied.

Our first main result relates the logarithmic cohomology of the pair
$H^*_{log}(M,\D)$ to the symplectic cohomology of the complement $SH^*(X)$, by
way of a spectral sequence:
\begin{thm}\label{thm:main} (Theorem \ref{thm:spectral}) 
There is a multiplicative spectral sequence converging to the symplectic cohomology ring
\begin{align} 
    \label{eq:ss} \lbrace E_r^{p,q},d_r \rbrace \Rightarrow SH^*(X). 
\end{align}
whose first page is isomorphic as rings to the logarithmic cohomology of $(M,\D)$:
\begin{equation}\label{eq:sspage1}
 H^*_{log}(M,\D)  \stackrel{\cong}{\rightarrow} \bigoplus_{p,q} E_1^{p,q}.
\end{equation}
\end{thm} 
For a precise relationship between the spectral sequence gradings $p,q$ and the
(bi)-grading on $H^*_{log}(M,\D)$, see Theorem \ref{thm:spectral} below. We
content ourselves here by observing that while $SH^*(X)$ is typically
$\Z_2$-graded, it can be made $\Z$-graded when $c_1(X) = 0$, and can additionally
be equipped with a grading by $H_1(X)$ measuring homology classes of generating
orbits; all of these choices can be realized on the level of log cohomology and
the identification \eqref{eq:sspage1}. We have opted to describe our proof in
the $\Z$-graded setting for notational simplicity when defining moduli spaces
and operations, but our methods are grading-independent and apply immediately to
the above (as well as other) graded situations with minor modifications to
definitions; see the discussion below Theorem \ref{thm:spectral}.

There are easy examples where the spectral sequence \eqref{eq:ss} fails to degenerate at
the $E_1$ page (for example when $X=\mathbb{C}$, $SH^*(X)$ vanishes). On the other hand, one of the main themes of the present paper is that there are many cases in
which \eqref{eq:ss} {\em does} degenerate at $E_1$.
In particular, the multiplicative structure in Theorem \ref{thm:main} gives a powerful 
computational tool for proving degeneration (or more generally, analyzing differentials). 
Note that the $E_1$ page $\QH^*(M,\D)$ is generated as a
$\K$-algebra by classes of the form $y \in H^*(X)$ or $\alpha t^{\v_I}$, where
for any subset $I \subset \{1, \ldots, k\}$, 
\[
    (\v_I)_i := \begin{cases} 1 & i \in I \\ 0 & \textrm{otherwise}. \end{cases} 
\] 
i.e., $\v_I = \sum_{i \in I} \mathbf{e}_i$ is the primitive multiplicity 1
vector supported on elements of $I$.  
The multiplicatively of \eqref{eq:ss} then implies that the spectral sequence \eqref{eq:ss}
degenerates if $d_r(\alpha t^{\v_I})= d_r(y) = 0$ for all $\alpha t^{\v_I}, y
\in H^*(X), r\geq 1$. Here is an easy corollary of this observation, which
follows from analyzing homology classes of possible differentials on such
generators:
\begin{cor}[Corollary \ref{thm: easycor}]
    \label{cor:simplecor} Suppose that there is a divisor $H$ such that for
    each smooth component $D_i \subset \D$, $\mathcal{O}(D_i)\cong
    \mathcal{O}(n_i H)$ for $n_i \in \mathbb{Z}^{>0}$.  If any of the $n_i >
    1$, then the spectral sequence \eqref{eq:ss} degenerates at the $E_1$ page.   
\end{cor}

In addition to being a useful computational tool, Theorem \ref{thm:main}  is also very useful for proving general qualitative results, for example:

\begin{thm} \label{thm:introfinite}
    (Theorem \ref{thm: finitely}) 
    Assume that $\D= D_1 \cup \cdots \cup D_k $ is an ample divisor and all of the strata $D_I$ are connected. Then $SH^*(X)$ is
    finitely generated as a graded algebra over $\mathbf{k}$.  
\end{thm}

A few comments on the proof of Theorem \ref{thm:main} are in order. The spectral sequence in Theorem \ref{thm:main} is induced by a version of
the action filtration on (the cochain complex computing) $SH^*(X$), specifically a 
nice integer-valued version of this filtration arising from the
compactification $(M,\D)$ (constructed in \S \ref{section:SHtor}). The product
operation on the symplectic cohomology cochain complex can be made to respect
this filtration, giving us the multiplicative structure of the spectral sequence.
Additively, the identification of the first page \eqref{eq:sspage1} can be
thought of as a consequence of a model of Reeb (or Hamiltonian) flow near $\D$
for which the orbit sets come in Morse-Bott families associated to normal
bundles to strata $\SIo$.\footnote{When $\D = D$ is a smooth divisor, this is standard and in the literature, compare \cite[eq. (3.2)]{Seidel:2010fk}.} 
The construction of this model and the explicit description of its periodic
orbits uses our earlier work \cite{GP1} and, like that work, relies extensively
on the study of symplectic structures and Liouville
geometry near $\D$ as developed by McLean \cite{McLean:2012ab, McLean2}.
The families of orbit sets produced are manifolds with corners, but nevertheless we
can apply a careful version of Morse-Bott analysis to them to
produce \eqref{eq:sspage1} additively (compare \cite[\S
1.1]{McLeanMultiplicity} for a spectral sequence in a related situation).
However, such analysis does not make it easy to see that the multiplicative structure on the
first page is compatible with log cohomology via \eqref{eq:sspage1},
and a new argument is needed.

The basic idea, coming from our earlier work \cite{GP1} and inspired by
Piunikhin-Salamon-Schwarz's \cite{Piunikhin:1996aa} classic construction, is
the introduction of log PSS moduli spaces. Roughly speaking, these moduli
spaces count maps from a punctured sphere to $M$ which are holomorphic near a
marked point $z_0$ and solve Floer's equation along a cylindrical end around
the puncture. At the marked point $z_0$, we place tangency and jet constraints
on the intersection of $z_0$ with strata of $\D$, and away from $z_0$ we
require the map to land in $X$, allowing us to interpolate between log
cohomology classes for $(M,\D)$ and Hamiltonian Floer cochains in $X$. In the
presence of sphere bubbling, counting such log PSS solutions does not
necessarily produce a well-defined cochain map from log cochains to Floer
cohomology.  However, energy considerations show that such bubbling does not
arise for \emph{low energy} solutions. Thus, by counting these low energy
solutions and suitably quotienting the Floer complexes, we can define the map
\eqref{eq:sspage1}, which we call the {\em low energy log PSS map}.

A considerable amount of work is then required to prove that this map is an
isomorphism. 
However, a benefit of this approach
is that one may use a standard TQFT argument to prove that the map
\eqref{eq:sspage1} intertwines product structures. In fact, this argument
closely follows Piunikhin-Salamon-Schwarz's original argument
\cite{Piunikhin:1996aa} that the
 PSS map intertwines product structures, adapted in a non-trivial way to our (``relative $\D$'') setup.

\subsection{Computations in the absence of pointed relative holomorphic spheres} \label{sec:intropointedspheres}
We next describe how, in the absence of certain relative holomorphic spheres, we can use similar methods to the proof of Theorem \ref{thm:main} to directly define an isomorphism from log cohomology of $(M,\D)$ to symplectic cohomology that splits the spectral sequence \eqref{eq:ss}. More precisely, we obtain geometric and often checkable criteria (i) under which the spectral sequence \eqref{eq:ss} degenerates and (ii) under which we can further topologically compute (in terms of $\QH^*(M,\D)$)
the ring structure on $SH^*$.\footnote{We remind the reader that for any (convergent)
spectral sequence associated to a filtered dg or
$A_\infty$ algebra $F^{p}C^{\bullet}$, the algebra structure on the final
$E_\infty$ page, the associated graded algebra $gr_F(H^*(C^{\bullet}))$, need not be isomorphic as rings to $H^*(C^{\bullet})$ even if it is additively isomorphic: in general $H^*(C^{\bullet})$ could be a non-trivial deformation of
$gr_F(H^*(C^{\bullet}))$.} 
To state these criteria, let $J$ be an almost complex
structure on $M$, and let $m$ be a non-negative integer. 
An {\em $m$-pointed relative $J$-holomorphic sphere} in $(M,\D)$ is a non-constant
$J$-holomorphic map $u: \mathbb{C}P^1 \to M$ whose image lies generically in
some open stratum ($X$ or $\mathring{D}_I$) and which intersects a deeper stratum ($\D$
or $\cup_{j \notin I} D_I \cap D_j$) in exactly $m$ distinct points. 
We work with a class of almost complex structures $\mathcal{J}(M,\D)$ (see
Definition \ref{defn:complexint}) that tame the symplectic form $\omega$,
preserve $\D$ and satisfy a certain integrability condition in the normal
directions to $\D$; 
for such an almost complex structure $J \in \mathcal{J}(M,\D)$, every
$J$-holomorphic sphere in $M$ is $m$-pointed for some $m$.
Our criterion for
degeneration is:
\begin{thm} \label{thm: toppair}(Theorem \ref{thm: toppair2}) 
    Suppose that $(M, \D)$ has no 0 or 1-pointed relative $J_0$-holomorphic
    spheres, for some $J_0 \in \mathcal{J}(M,\D)$.  Then the spectral sequence
    \eqref{eq:ss} degenerates.  More precisely, there is a canonical {\em splitting}
    of the spectral sequence, i.e., a filtered isomorphism of
    $\mathbf{k}$-modules 
     \begin{align} 
         \label{eq:PSSintro} \operatorname{PSS}_{log}: H^*_{log}(M,\D) \stackrel{\cong}{\to} SH^*(X). 
     \end{align} 
     such that the induced map $\oplus_{p,q} E^1_{p,q} \cong H^*_{log}(M, \D)
     \to \oplus_p F^p SH^*(X) / F^{p-1} SH^*(X) \cong \oplus_{p,q} E^{\infty}_{p,q}$ is the isomorphism specifying
     collapse
     of the spectral sequence.
\end{thm}
The map \eqref{eq:PSSintro}, which we call the {\em log PSS map}, was
introduced in our earlier work \cite{GP1} and is constructed in a similar way
to \eqref{eq:sspage1}, except we no longer
restrict to low energy log PSS solutions. The main work is showing that such counts
now define a cochain map in the absence of 0 and 1-pointed relative spheres
(note there may be other spheres in $M$, but the point is to show they do not arise in
the compactification of log PSS moduli spaces). From there, given that the
``associated graded'' of the map \eqref{eq:PSSintro} is easily seen to be
\eqref{eq:sspage1}, Theorem \ref{thm: toppair} becomes a straightforward
consequence of Theorem \ref{thm:main}.

Under the hypothesis in which Theorem \ref{thm: toppair} applies, the multiplicative structure of the spectral sequence \eqref{eq:ss} implies that the Log PSS morphism \eqref{eq:PSSintro} induces a ring isomorphism between the log cohomology
$H^*_{log}(M,\D)$ and the associated graded\footnote{with respect to the (homological shadow of) action filtration
inducing the spectral sequence \eqref{eq:ss}} symplectic cohomology ring $gr_F
SH^*(X)$. Our next result gives a criterion under which this can be promoted to
a ring isomorphism 

with symplectic cohomology $SH^*(X)$:
\begin{thm} \label{thm: toppairring}
    Suppose that $(M, \D)$ has no 0, 1, or 2-pointed relative $J_0$-holomorphic
    spheres, for some $J_0 \in \mathcal{J}(M,\D)$.
    Then, the additive isomorphism \eqref{eq:PSSintro} given in
    Theorem \ref{thm: toppair} is a ring isomorphism.
\end{thm}
Once more, this follows from running the same TQFT argument which
showed the isomorphism \eqref{eq:sspage1} in Theorem \ref{thm:main} was a ring
map, applied now to the ``global'' (rather than just low energy) log PSS moduli
spaces that arise in the construction of \eqref{eq:PSSintro}. 
The only possible obstruction to our
TQFT argument applying is the failure of the relevant interpolating moduli
spaces to be compact, and we show the only possible problems are the bubbling of
0, 1 or 2-pointed relative spheres. With such spheres excluded, the proof then
straightforwardly reduces to earlier methods.

Throughout this paper, we will call $(M,\D)$ a {\em topological
pair}\footnote{The terminology ``topological pair'' was introduced in our
    earlier work \cite{GP1}, where its usage is slightly less general than what
is used here.}, respectively a {\em multiplicatively topological pair} if, for
some almost complex structure $J_0 \in \mathcal{J}(M,\D)$, $(M,\D)$
has no $m$-pointed relative $J_0$-holomorphic spheres for $m \leq 1$,
respectively $m \leq 2$, i.e., if $(M,\D)$ satisfy the hypotheses of Theorems \ref{thm: toppair}, respectively \ref{thm: toppairring}.\footnote{An alternative naming convention would be to say a pair $(M,\D)$ is {\em $r$-topological} if (for some $J_0$) it contains  no $m$-pointed relative $J_0$-holomorhic spheres for all $m \leq r$. We expect such generalized notions to be useful in the study of higher-arity operations on symplectic cohomology, such as $\ainf$ or $L_{\infty}$ structures.}

Theorems \ref{thm: toppair} and \ref{thm: toppairring} allow us to deduce many
complete topological computations of symplectic cohomology groups and rings.
An incomplete list of examples of topological and multiplicatively topological pairs
(for which the relevant Theorems apply) are provided in Examples
\ref{example:toppairs} and \ref{example:multoppairs}. As a first example, we
note that for any smooth projective $M$, the pair $(M, \D)$ will be
topological (respectively multiplicatively topological) if $\D$ consists of at
least $\dim_{\C} M + 1$ (respectively $2 \dim_{\C} + 1)$ generic ample divisors
which are powers of the same line bundle.

\subsection{Computations in the presence of pointed spheres}

We expect that, for general pairs $(M,\D)$, by counting the relative spheres
that arise in the compactification of log PSS moduli spaces (using a cochain
level version of log Gromov-Witten theory), one could write down a corrected
differential on the log cohomology cochain complex, along with a corrected
product, for which counts of (compactified) log PSS moduli spaces induce both a
cochain map and a ring map (up to chain homotopy); note that if such a cochain map were
constructed, Theorem \ref{thm:main} immediately implies that
it would be a quasi-isomorphism. In the case that
$\mathbf{D} = D$ is a smooth divisor, under suitable monotonicity assumptions
Diogo and Lisi \cite{diogothesis, diogolisi} have given a related additive
cochain level model for symplectic cohomology in terms of (relative)
GW-invariants, and there is work towards a model for the product in this setting \cite{diogothesis}.

At the level of the spectral sequence \eqref{eq:ss}, the existence of such a
model would imply that the differentials $d_r$ on the $E_r$ pages of the
spectral sequence could be calculated in terms of \emph{cohomological} log GW
invariants. When all of these invariants vanish, the spectral sequence would
degenerate.  There would be further (analogous) vanishing criteria under which
the ring structure could also be computed in terms of log cohomology.

The foundations of symplectic log Gromov-Witten theory are still under active
development (see \cite{Ionel:2011fk}, \cite{MR3383807}, \cite{Tehrani} for different approaches) and constructing such a
cochain level model goes beyond the scope of the present
article.
Nevertheless, motivated by these ideas, we
show that in a broad class of examples, additive and multiplicative
computations of symplectic cohomology can be reduced to only counting (rather, 
exhibiting the vanishing of counts of) 
the \emph{simplest kinds of relative curves}
(those which intersect each component of $\D$ at most once). We focus on two
distinct situations: (a) computing the ring structure topologically in the
absence of 0 or 1-pointed spheres but presence of 2-pointed spheres (meaning
when we already understand $H^*_{log}(M, \D) \cong SH^*(X)$ additively), and
(b) computing $SH^*(X)$ additively in terms of $H^*_{log}(M,\D)$ in the
possible presence of 0 or 1-pointed spheres.

\subsubsection{Trivializing deformations of rings in the presence of 2-pointed relative spheres}

For topological pairs, there is an additive (but not necessarily
multiplicative) isomorphism $SH^*(X) \cong H^*_{log}(M,\D)$; moreoever the
product structure on $SH^*(X)$ induces a deformation of the product structure
on $H^*_{log}(M,\D)$ which is in general non-trivial. For a number of such
pairs $(M,\D)$ (under topological hypotheses and hypotheses on the vanishing of
certain two-point Gromov-Witten counts), we have an a posteriori argument
establishing a ring isomorphism $SH^*(X) \cong H^*_{log}(M,\D)$, via showing
that the deformation of the product structure on $H^*_{log}(M,\D)$ is
trivial(izable); see Theorem \ref{eq:topproduct} for such a
criterion.

\begin{example}[Symplectic cohomology rings of $\P^n$ minus generic hyperplanes]
    Let $M = \P^n$ and $\mathbf{D}_k$ be a union of $k$ generic planes, and
    $X_k = M \backslash \mathbf{D}_k$ the complement; note that $X_{n+2}$ is
    Mikhalkin's generalized pair of pants \cite{Mikhalkin}.  Our
    results lead to a complete computation of $SH^*(X_k)$ as a ring for all
    $k$, extending well-understood computations when $k \leq n+1$:  
    \begin{itemize}
        \item when $k\leq
    n$ then $X_k = M \backslash \mathbf{D}_k =  (\C^*)^{k-1} \times
    \C^{n-k+1}$ with $n-k+1 > 0$. Since $SH^*(\C) = 0$, the K\"{u}nneth formula
    \cite{MR2208949} implies that $SH^*(X_k) = 0$.
    \item When $k=n+1$, $X_{n+1} \cong (\C^*)^{n} \cong  T^*(T^n)$, so
        Viterbo's formula implies $SH^*(X_{n+1}) \cong H_{n-*}(\mathcal{L}
        T^n)$. 

    \item 
    For $k > n+1$, $(M, \mathbf{D}_k)$ is a topological pair,
    so Theorem \ref{thm: toppair} gives an additive isomorphism 
    $SH^*(X_k)\cong H^*_{log}(M,\D_k)$. In fact, this is a ring isomorphism by
    the following argument:
    $(M,\mathbf{D}_k)$ is multiplicatively topological when $k \geq 2n+1$ (see Example \ref{example:multoppairs}) 
    so Theorem \ref{thm: toppairring} says that $H^*_{log}(M,\D) \to SH^*(X_k)$ is a ring isomorphism in that range.
    For the remaining intermediate range $n+1 < k < 2n+1$, Theorem
    \ref{eq:topproduct} applies to argue that, while in principle there could
    have been a deformation of the product structure on $H^*_{log}(M,
    \mathbf{D}_k)$ contributing to $SH^*(X_k)$, the deformation was in fact
    trivial.
    \end{itemize}
\end{example}

\subsubsection{Degeneration arguments in the presence of spheres}

We show that that for many pairs $(M,\D)$, the multiplicative structure on the
spectral sequence allows us to reduce (for the purposes of degeneration
arguments) to counting moduli spaces of
relative spheres which intersect each component of $\D$ at most once.
 We will show that for these pairs, if the relevant cohomological count is zero, 
then the spectral sequence \eqref{eq:ss} degenerates, albeit without a canonical splitting.

More precisely, we look at ``admissible" pairs $(M, \D)$, which are pairs for which the differentials on primitive cohomology classes $\alpha t^{\v_I}$ are tightly controlled--- given an admissible pair, a vector $\v_I$, and $\alpha t^{\v_I} \in H^*_{log}(M,\D)$, there is single possible non-vanishing differential, $d_{w(\v_I)}(\alpha t^{\v_I})$. We show that this differential can be encoded in Gromov-Witten type invariants (called ``obstruction classes" in \cite{GP1}) 
\begin{align} \label{eq:obstructionclass} GW_{\v_{I}}: H^*(\SIo) \to H^*(X) \end{align}

\begin{thm} \label{thm: degenerescence1}(Theorem \ref{thm: degenerescence}) Let  $(M,\D)$ be an admissible pair and assume $\mathbf{k}= \mathbb{Z}$. For any primitive vector $\v_I$, we have an equality 
    \begin{align} 
        d_{w(\v_I)}(\alpha t^{\v_I})=GW_{\v_{I}}(\alpha). 
    \end{align} 
\end{thm}

In particular, the above discussion shows that for admissible pairs, these invariants determine when the spectral sequence degenerates:

\begin{cor} \label{cor: degenerescencecor1}(Corollary \ref{cor: degenerescencecor}) 
 Assume $\mathbf{k}= \mathbb{Z}$. Given an admissible pair $(M,\D)$, suppose the maps \eqref{eq:obstructionclass} vanish for all $I \in \lbrace 1,\cdots k \rbrace$. Then the spectral sequence \eqref{eq:ss} degenerates at the $E_1$ page. \end{cor}

While Corollary \ref{cor: degenerescencecor1} is more technical to state than
Corollary \ref{cor:simplecor}, it is likely significantly more general.  For
example, \cite{diogothesis}*{\S 6.1.2} consider cases which fall outside
of the purview of Corollary \ref{cor:simplecor}; however the Gromov-Witten
counts vanish and the spectral sequence degenerates. Other related examples can
easily be constructed and it appears likely that many more cases could be
treated by slightly more elaborate Gromov-Witten theory computations. 
As a concrete question, we ask:

\begin{ques} 
    Given a smooth hypersurface $M \subset \mathbb{C}P^{n+1}$ and a collection
    of hyperplanes $\D = D_1\cup \cdots\cup D_k$, with $X = M \backslash \D$,
    is it the case that either $SH^*(X)=0$ or the spectral sequence
    \eqref{eq:ss} degenerates at the $E_1$ page?  
\end{ques}

We can also consider variants of Theorem \ref{thm: degenerescence1} where we
focus on the spectral sequence in a given degree. The most interesting example
of this is the following theorem.  

\begin{thm} 
    \label{thm: logCYintro} (Theorem \ref{thm: Fano})  Let $(M, \D)$ be a pair
    with $M$ a Fano manifold and $\D$ an anticanonical divisor. Assume that all
    strata $D_I$ are connected. Then the spectral sequence degenerates in
    degree zero. With respect to the standard filtration, we have an
    isomorphism 
    \begin{align} 
        gr_F SH^0(X) \cong \mathcal{SR}(M,\D) 
    \end{align}
    where $\mathcal{SR}(M,\D)$ is the Stanley-Reisner ring on the dual
    intersection complex of $\D$, defined in \eqref{eq:SRdef}
    (roughly, the subalgebra consisting of $e \in H^0(X)$ and $\alpha t^{\v}$
    where $\alpha \in H^0(\mathring{S}_I)$).   
\end{thm} 

  The setting of Theorem \ref{thm: logCYintro} (or more generally log Calabi-Yau pairs) is of special interest in mirror symmetry, where it is expected that under suitable circumstances, the mirror variety to $X$ is birational to the affine variety $\operatorname{Spec}(SH^0(X))$ \cite{MR3415066}. It is suggested there that there should exist a flat one parameter degeneration with general fiber $\operatorname{Spec}(SH^0(X))$ and central fiber $\operatorname{Spec}(\mathcal{SR}(M,\D))$. Theorem \ref{thm: logCYintro} validates this expectation in a large number of cases. It is worth noting that the resulting deformation is usually non-trivial and should be expressible in terms of log Gromov-Witten invariants of the pair $(M,\D)$ (see \cite{GS}). 

As suggested in the previous paragraph, it is likely that Theorem \ref{thm: logCYintro} holds without the assumptions that
$M$ is Fano and that the strata of $\D$ are all connected and we expect that the methods developed in this paper can be extended to treat the general case. As evidence for this, we use our methods to prove this when  $\operatorname{dim}_\mathbb{C}(M)=2$ in Theorem \ref{thm:
logCYsurf}, recovering a result of James Pascaleff \cite{Pascaleff}. Our
method of proof bears some similarities to Pascaleff's argument in that both
use knowledge of ``low energy" product operations in symplectic cohomology.  In
our case, these computations of the product are implied by basic considerations
in TQFT, whereas Pascaleff's argument is more geometric.

\subsection{Comparison with other methods for computing symplectic cohomology} \label{subsection: othermthds} 
Most methods of explicitly computing symplectic cohomology $SH^*(X)$ begin by
explicitly describing a (Weinstein) handle presentation of $X$ (possibly, but
not necessarily, one that comes from a Lefschetz fibration), or 
alternatively a Lagrangian skeleton for $X$, any of which we might collectively
call ``a presentation of the core'' of $X$. From such presentations, one
can attempt to extract a computation of associated ``open-string'' Floer-theoretic or
pseudo-holomorphic curve invariants, such as the Chekanov-Eliashberg
DGA, an object of symplectic field theory associated to the Legendrians in the
handle presentation, or the wrapped Fukaya category, associated to non-compact
Lagrangians in $X$. Finally, one takes the Hochschild homology and/or cohomology of
such a computation, which relates to $SH^*(X)$ via surgery formulae and/or
open-closed maps \cite{BEE, BEE2, Abouzaid:2010kx, ganatra1_arxiv}. 
While this strategy has been carried out in interesting examples (see e.g.
\cite{BEE, BEE2, EkholmNg,Etgu-Lekili}) and remains a remarkably effective tool
for identifying instances when $SH^*(X)$ vanishes, 
general computations can run into two difficult issues: computing the relevant
open-string invariant (a process that in some cases can be combinatorialized),
and computing its Hochschild invariants (which is known to be a difficult
algebraic problem except in special circumstances). Also, realizing an affine
variety as an explicit handlebody or calculating a skeleton may be challenging in
practice, though there is some work to systematize the process
\cite{CasalsMurphy}.

The results in this paper, which instead make use of a ``presentation at
infinity'' of $X$ (in terms of the compactification $(M,\D)$), give many
cases where calculations of $SH^*(X)$ can be performed purely topologically.  It
seems likely that any elaboration of these models to include non-trivial counts
of relative spheres will frequently produce models for symplectic cohomology with far
fewer generators, where the differentials may be approached using algebraic
geometry. Of course, applying these techniques for a given $X$ requires finding
an explicit compactification $(M,\D)$, which is also known to be a non-trivial problem \cite{Complexity}.  On the other hand, affine varieties arising in mirror symmetry are typically
presented as divisor complements, which provides a major impetus for developing
these methods. 

\begin{rem} The interplay between the ``compactification" and ``core"
    approaches to studying $SH^*(X)$ has been explored in \cite{Khoa}. It would
    likely be profitable to pursue this interplay further, in light of the fact
    that frequently one of the two presentations is much simpler than the other
    for the purpose of computing $SH^*(X)$. 
\end{rem}

\subsection{Overview of paper}
In \S \ref{section:SHtor} we review the symplectic geometry of normal
crossings compactifications, and construct a `normal-crossings adapted' model
of the action-filtered symplectic cohomology co-chain complex (this builds off
of a model constructed in \cite{GP1}) with its ring structure, inducing the 
multiplicative spectral sequence \eqref{eq:ss}.
In \S \ref{section: PSSreview}, we define the log cohomology ring
$\QH^*(M,\D)$, construct the low energy log PSS map \eqref{eq:sspage1}
between log cohomology and the
first page of the spectral sequence, and show that low energy log PSS is a ring map. In \S
\ref{sect: PSSiso2}, we prove that this map is an isomorphism, and complete the
proof of Theorem \ref{thm:main}. In \S \ref{sec:computations}, we establish the
various geometric criteria under which the spectral sequence \eqref{eq:ss}
degenerates and the ring structure can be computed, and apply these results to
give a number of
computations and qualitative results: in more detail, degeneration for
topological pairs and ring structure computations for multiplicatively
topological pairs 
are discussed in \S \ref{subsec:topologicalpairs}, trivializing ring
deformations in the presence of 2-pointed spheres are discussed in \S
\ref{subsec:trivializingdeformations}, degeneration criteria in the presence of
(0 and 1-pointed) spheres are discussed in \S \ref{sect: suitesspectrales}, and
the special case of log Calabi-Yau pairs is discussed in \S \ref{sect:logCY}.

\subsection*{Acknowledgements}

Our use of normal crossings-adapted symplectic and Liouville
structures and subsequent Morse-Bott analysis borrows extensively from work of
Mark McLean \cite{McLean2, McLean:2012ab}. We would like to thank him for
explanations of his work as well as several other conversations related to this
paper. We also benefited greatly from discussions with Strom Borman, Luis Diogo, Yakov Eliashberg, Eleny Ionel, Samuel Lisi, and Nick Sheridan at an early stage of this project; we thank all of them. Finally, we would like to thank an anonymous referee for suggestions that improved the exposition of this article.

\subsection*{Conventions}
Our grading conventions for symplectic cohomology follow
\cite{Abouzaid:2010kx}. We work over an arbitrary ground ring $\K$ (unless
otherwise stated).

\section{A divisorially-adapted model of filtered symplectic cohomology} \label{section:SHtor} 
In this section, we show that the symplectic cohomology complex $SC^*(X)$ of
any affine variety $X$ admits a particularly nice  ``tailored to the divisor
$\D$'' model in the presence of a projective compactification $(M,
\D)$, which
satisfies two key properties: first, it is generated by (small perturbations
of) Morse-Bott families (with corners) of orbits associated to strata of
$(M,\D)$ and secondly, it possesses an an {\em integral} refinement of its (a
priori real) action filtration, with integer weights measuring (a weighted sum
of) winding of orbits around divisors in $\D$ at infinity.  This filtration is
compatible with products and induces the desired multiplicative spectral
sequence converging to $SH^*(X)$.

Achieving both of the above properties (and in particular the second property)
entails a substantial refinement of our earlier work \cite{GP1}*{\S 3}: roughly
speaking, to achieve the above integral filtration, orbits contributing to
$SC^*(X)$ which wind a large number of times around $\D$ are made to to
occur arbitrarily close (depending on the amount of winding) to $\D$, in order to 
to compensate for action errors (between the usual symplectic action and the
relevant divisorial winding number) that would otherwise magnify with winding.  This
integral filtration can also be thought of as a limit of the action filtrations
on (symplectic cohomologies of) a family of Liouville domains which exhaust
$X$, see Remark \ref{rem:inverselimit}.

\subsection{Normal crossings symplectic geometry} \label{subsec:ncsymplectic}
\begin{defn}
    \label{def:logsmooth} 
    A {\em log-smooth compactification} of a smooth complex $n$-dimensional
    affine variety $X$ is a pair $(M,\mathbf{D})$ with $M$ a smooth, projective
    $n$-dimensional variety and $\mathbf{D} \subset M$ a divisor satisfying
    \begin{align}
     &X = M \backslash \mathbf{D};\\
     &\textrm{The divisor $\mathbf{D}$ is normal crossings in the strict sense, e.g.,}\\
     &\nonumber\mathbf{D} := D_1 \cup \cdots \cup D_i \cup \cdots \cup D_k\ \textrm{where $D_i$ are smooth components of $\mathbf{D}$; and}\\
     &\label{eq:kappai}\textrm{There is an ample line bundle $\mathcal{L}$ on $M$ together with a section $s \in H^0(\mathcal{L})$ whose }\\
     &\nonumber\textrm{divisor of zeroes is $\sum_i \kappa_i D_i$ with $\kappa_i \in \Z_{>0}$}.
    \end{align}
\end{defn}

Going forward, fix a log-smooth compactification $(M,\D = D_1 \cup \cdots \cup
D_k)$. There is a natural induced ``divisorial'' stratification of $M$, with
strata indexed by subsets of $\{1, \ldots, k \}$: for $I \subset \{1, \ldots, k \}$,
define
\begin{equation}
    D_I:= \cap_{i \in I} D_i.
\end{equation}    
We refer to the associated open parts of the stratification induced by
$\mathbf{D}$ as 
\begin{equation}
    \DIo = D_I \backslash \cup_{j \notin I} D_j,
\end{equation}
and allow $I = \emptyset$, with the convention 
\begin{align*}
    D_{\emptyset} &:= M\\
    \mathring{D}_{\emptyset} &= M \backslash \D = X.
\end{align*}
Denote by $\mathfrak{i}: X \hookrightarrow M$ the natural inclusion map. 

We equip $M$ with a symplectic form $\omega$ which is a K\"{a}hler form associated to some
positive Hermitian metric $|| \cdot ||$ on $\mathcal{L}$. On $X$, consider
the potential $h=- \operatorname{log}||s||$, where $s$ is the section given in
\eqref{eq:kappai}. Restricting to $X$, we have that $\omega:=-dd^c h$. Using $\theta = -d^c h$,
and we further equip the submanifold $X$ with the structure of a finite
type convex symplectic manifold (see e.g., \cite[\S
A]{McLean:2012ab} for a definition) which, up to deformation, is independent of
the compactification or the choice of ample line bundle $\mathcal{L}$, and
equivalent to the canonical structure obtained from an embedding $X
\hookrightarrow \C^N$ \cite{Seidel:2010fk}.

As is typically done, we find it convenient (for understanding geometry
at infinity) to further deform this finite type convex symplectic structure on
$X$ to one which is ``nice'' or ``adapted to $\D$'', meaning it admits a 
system of suitably compatible (symplectic, punctured) tubular
neighborhoods around each $D_i$ \cite{Seidel:2010fk, McLean:2012ab}, see
Theorems \ref{thm: MTZ} and \ref{thm:exactnicestructure} below. We first review
compatibility properties for such systems of symplectic tubular neighborhoods
which make them ``nice'', following the comprehensive notion of an {\em
$\omega$-regularization} developed in \cite{MTZ} (see also
\cite{McLeanMultiplicity} for a slightly different exposition).

We recall the local standard model for symplectic forms near a symplectic
submanifold $Z$ which is either of codimension-2 or an intersection of
codimension-2 submanifolds.  Let $\pi: L \to Z$ be a real-oriented rank-2 vector bundle
equipped with a {\em Hermitian structure} $(\rho, \nabla)$, meaning a pair such
that, with respect to the complex structure on $j:=j_{\rho}$ uniquely induced
by the associated Riemannian metric $\mathrm{Re}(\rho)$  (recall $L$ is rank-2
oriented and $SO(2) = U(1)$), 
$\rho$ is a Hermitian metric and 
$\nabla$ is a $\rho$-compatible connection. We use the notation $L_z :=
\pi^{-1}(z)$ for the fibers of $\pi$,
and (by a standard abuse of notation) we also use the abbreviation
$\rho(v):=\rho(v,v)$ to refer to the norm-squared function. If $L$ is
symplectic vector bundle with symplectic structure $\Omega$, we say a Hermitian
structure $(\rho, \nabla)$ is {\em compatible} with $\Omega$ if
$\Omega(-,j_{\rho}) = \mathrm{Re}(\rho))$ as usual.
The Hermitian structure induces a splitting of the
tangent bundle of $L$: $TL \cong T^{vert} L \oplus T^{horiz} L$ where $(T^{vert} L
L$ denotes the vertical sub-bundle of the tangent bundle, which has fibers
$\ker(d\pi_p) \cong L_{\pi(p)}$ (and in particular has a complex structure). 
Recall that there is a standard angular one-form on the vertical tangent space
$d \varphi \in \Gamma(T^{vert, *} (L \backslash Z))$ (which, for instance, admits the
global definition $d (\log \rho) \circ j$, where $j$ is the complex structure
on the vertical tangent space). Using the Hermitian structure,  
we lift $d \varphi$ to a connection 1-form $\theta_{e} \in \Omega^1(L\setminus Z)$
by the prescription that for each $p \in L \backslash Z$,
\begin{align} 
    (\theta_{e})|_{T^{vert}L_p}&= (d\varphi) \\ 
    (\theta_{e})|_{T^{horiz}L_p}&=0.
 \end{align} 
If we are further given a symplectic structure $\omega_Z$ on $Z$,
we can associate a 2-form on $L$
\begin{equation}\label{localsymplecticformcodimension2}
    \omega_{(\rho,\nabla)}:= \pi^* \omega_Z + \frac{1}{2} d (\rho \theta_{e}),
\end{equation} 
which is symplectic in a neighborhood of $Z$. Similarly, given
a collection of Hermitian line bundles $\{L_i = (L_i, \rho_i, \nabla_i)\}_{i
\in I}$, we have connection 1-forms $\{\theta_{e,i}\}_{i \in I}$, and we can
associate a 2-form on $(\oplus_{i \in I} L_i)$ which is symplectic in a neighborhood of zero:
\begin{equation} \label{localsymplecticformhighcodimension}
    \omega_{\{(\rho_i,\nabla_i)\}_{i \in I}}:= \pi^* \omega_Z + \frac{1}{2} \sum_{i \in I} \pi_{I,i}^*d (\rho_i \theta_{e,i})   
\end{equation} 
(above, $\pi_{I, j}: \oplus_{i \in I} L_i \to L_j$ is the canonical projection map).

Next we study systems of tubular neighborhoods in the smooth category.  For any
smooth submanifold $Z \subset M$, let 
\begin{equation}
    N_MZ
\end{equation}
(or $NZ$ if $M$ is implicit) denote its normal bundle.  
We implicitly identify $Z$ with its zero section in $NZ$. In
\cite[Def. 2.8]{MTZ}, a slight strengthening of the notion of a tubular
neighborhood is introduced: a {\em regularization} of $Z$ is a tubular
neighborhood $\psi: U \subset NZ \to M$ such that, under the identification
$N_{NZ} Z \cong N_M Z$,
\begin{equation}\label{eq:regularization}
    \textrm{The (normal component of) the derivative }(D\psi)|_Z: N_{NZ} Z = N_M Z \to  N_M Z\textrm{ is }id_{N_M Z}.
\end{equation}
We recall the normal component of the derivative can be defined as the following composition, where $\pi: NZ \to Z$ denotes the vector bundle structure:
\begin{equation}\label{eq:normalderivative}
    N_{NZ} Z = \pi^*(NZ)|_{Z} \stackrel{\simeq}{\longrightarrow} T^{vert}(NZ)|_{Z} \hookrightarrow T (NZ)|_{Z} \stackrel{D\psi|_{Z}}{\longrightarrow} TM|_{Z} \twoheadrightarrow N_M Z.
\end{equation}
Now, let $\{Z_i\}_{i \in S}$ be a collection of transversely intersecting
submanifolds, and denote by $Z_I := \cap_{i \in I} Z_i$ for $I\subseteq S$, with
$Z_{\emptyset} := M$.  Note that there are inclusions of normal bundles
$N_{Z_{I'}} Z_I \subset N_{Z_{I''}} Z_i$ for any $I'' \subset I' \subset I$,
and in particular the sub-bundles $N_{Z_{I'}} Z_I \subset N_M Z_I$ (``the local
model for $Z_{I'}$ near $Z_I$'') give a stratification
of $N_M Z_I$, also inducing a splitting 
\begin{equation}\label{splitting}
    N_M Z_I \cong \bigoplus_{i \in I} N_{Z_{I - \{i\}}} Z_I \cong \bigoplus_{i \in I} (N_M Z_i)|_{Z_I}
\end{equation}
One can ask for a
collection of regularizations for each $Z_I$, $\{\psi_I: U_I \to M\}_{I \subset
S}$ to {\em preserve the stratifications above them} in the sense that 
\begin{equation}\label{systempreservingstratifications}
    \textrm{For each $I' \subset I$, } \psi_I ( (N_{Z_{I'}} Z_I) \cap U_I)= Z_{I'} \cap \psi_I(U_I).
\end{equation}
(c.f., \cite[Def. 2.10]{MTZ}). In order to impose some compatibilities between the $\psi_I$'s, we
introduce some more notation: for $I' \subset I$, let
\begin{equation}
    \pi_{I; I'}: N_{Z_{I'}} Z_I  \to Z_I
\end{equation}
denote the vector bundle projection map. There are canonical identifications
\begin{align}
    \label{eq: quickident}
    N_M Z_I &= \pi_{I; I'}^* N_{Z_{I - I'}} Z_I = N_{N_M Z_I} (N_{Z_{I'}} Z_I) \\ 
    N_{Z_{I-I'}} Z_I &= N_M Z_{I'}|_{Z_I}
\end{align}
and using this, a canonical map
\begin{equation}\label{eq: i map}
    i: N_{N_M Z_I} (N_{Z_{I'}} Z_I) = \pi_{I; I'}^* N_{Z_{I - I'}} Z_I \hookrightarrow \pi_{I; I'}^* N_M Z_I = \pi_I^* (N_M Z_I)|_{N_{Z_{I'}} Z_I}  \hookrightarrow T (N_M Z_I)|_{N_{Z_{I'}} Z_I}
\end{equation}
where the last inclusion is by the fact that $\pi_I^*(N_M Z_I)$ is the vertical
tangent bundle of $N_M Z_I$.

Now for any collection of regularizations $\{\psi_I: U_I \to M\}_{I \subset S}$ satisfying \eqref{systempreservingstratifications}, and any $I' \subset I \subset S$,
$\psi_I$ induces a diffeomorphism from $N_{Z_{I'}} Z_I \cap U_I$ to its image $Z_{I'} \cap \psi_I(U_I)$, and hence $d\psi_I$ induces an isomorphism of vector bundles 
\begin{equation}
    T (N_M Z_I )|_{N_{Z_{I'}} Z_I \cap U_I} \stackrel{\cong}{\to} TM|_{Z_{I'} \cap \psi_I(U_I)}
\end{equation}
and preserves the sub-tangent bundles associated to the the submanifolds on
each side, hence there is a map of normal bundles (or normal derivative) given
by
\begin{equation} \label{stratifiednormalderivative}
    \begin{split}
        D\psi_{I;I'}: N_{N_M Z_I}(N_{Z_{I'}} Z_I)_{ N_{Z_{I'}} Z_I \cap U_I} = \pi_{I;I'}^* N_{Z_{I-I'}} Z_I |_{N_{Z_{I'}} Z_I \cap U_I} &\stackrel{i}{\hookrightarrow} T (N_M Z_I)_{N_{Z_{I'}} Z_I \cap U_I} \\
        &\stackrel{d \psi_I}{\longrightarrow}  TM|_{Z_{I'} \cap \psi_I(U_I)} \twoheadrightarrow N_M(Z_{I'})_{Z_{I'} \cap \psi_I(U_I)}.
\end{split} 
\end{equation}
where $i$ is as in \eqref{eq: i map} the last map is the canonical projection onto the normal bundle (compare also \eqref{eq:normalderivative}).

This map allows us to move between the domains of the different tubular
neighborhoods, and we can now ask that a collection or regularizations
satisfying \eqref{systempreservingstratifications} further be {\em compatible with
passing to substrata}, meaning that
\begin{equation} \label{eq:agreesubstrata}
    \begin{split}
    D\psi_{I;I'}(U_I)&= U_{I'}|_{Z_{I'} \cap \psi_I(U_I)}  \\
    \psi_I &= \psi_{I'} \circ D\psi_{I;I'}|_{U_{I}}.
\end{split}
\end{equation}
(compare \cite[Def. 2.11]{MTZ}).

Finally we come to the main definition of a ``nice system of tubular neighborhoods for symplectic divisors''.
For the definition that follows recall that, for a symplectic submanifold $Z \subset (M, \omega)$, the normal bundle $N_M Z$ 
inherits the structure of a symplectic vector bundle bundle (and in particular
is canonically oriented), with symplectic structure denoted $\omega_{N_M Z}$. 
\begin{defn}[\cite{MTZ}, Def. 2.12]
    Let $\{Z_i\}_{i \in S}$ be a collection of codimension-2 symplectic manifolds
    of $(M, \omega)$.  An {\em $\omega$-regularization} of $\{Z_i\}_{i \in S}$
    consists of
    a collection of 
    \begin{itemize}
        \item  tubular neighborhoods  $\{\psi_I: U_I \to M\}_{I \subset
            S}$ for each $\{Z_I = \cap_{i \in I} Z_i\}_{I \subset S}$;  and

        \item Hermitian structures $\{(\rho_{I,i}, \nabla_{I,i})\}_{i \in I}$
            on the normal bundles $\{(N_M Z_i)|_{Z_I} = N_{Z_{I - \{i\}}} Z_I\}_{i
            \in I}$ which are compatible with the canonically induced
            symplectic structures $(\omega_{N_M Z_i})|_{Z_I}$.
    \end{itemize}
    These structures satisfy the following conditions:
    \begin{enumerate}
        \item Each $\psi_I$ is a regularization in the sense of \eqref{eq:regularization};
        \item Each $\psi_I$ preserves the strata above them in the sense of \eqref{systempreservingstratifications};
        \item The collection $\{\psi_I: U_I \to M\}_{I \subset S}$ is compatible with passing to substrata, in the sense of \eqref{eq:agreesubstrata};
        \item With respect to the decomposition $N_M Z_I \cong \oplus_{i \in I}
            N_{Z_{I - i}} Z_I = \oplus_{i \in I} (N_M Z_i)|_{Z_I}$ and the
            Hermitian structures $\{(\rho_{I,i}, \nabla_{I,i})\}_{i \in I}$,
            the pullback along of $\omega$ along $\psi_I$ takes the standard
            local form from \eqref{localsymplecticformhighcodimension}:
            \begin{equation}
                (\psi_I)^*\omega = \omega_{\{(\rho_{I,i},\nabla_{I,i})\}_{i \in I}}.
            \end{equation}
        \item The maps $D\psi_{I;I'}$ are product Hermitian isomorphisms
            (meaning a vector bundle isomorphism respecting the direct sum
            decomposition on both sides and intertwining Hermitian structures
            on each summand), with respect to the natural splittings of the source and target
            \begin{align*}
                \pi_{I;I'}^* N_{Z_{I-I'}} Z_I |_{N_{Z_{I'}} Z_I \cap U_I} &\cong \bigoplus_{i \in I'} \pi_{I; I'}^* ((N_M Z_i)|_{Z_I})|_{N_{Z_{I'}} Z_I \cap U_I} \\
            N_M(Z_{I'})|_{Z_{I'} \cap \psi_I(U_I)} &\cong \bigoplus_{i \in I'} (N_M Z_i)|_{Z_{I'} \cap \psi_I(U_I)}.
        \end{align*}

    \end{enumerate}
\end{defn}

A key result ensures these structures can always be found on $\mathbf{D}$ in
our setting, possibly after deformation of symplectic form:
\begin{thm}[ \cite{McLean:2012ab} Lemma 5.4, 5.15,  \cite{MTZ} ] \label{thm: MTZ}
There is a deformation of symplectic structures $\{\omega_t\}_{t \in [0,1]}$, with
$\omega_0=\omega$, $\omega_1=\tilde{\omega}$, and $[\omega_t] = [\omega] \in
H^2(M)$, such that the collection of divisors
$\{D_i\}_{i=1}^k$ admits an $\tilde{\omega}$-regularization.  (Moreoever, this deformation can be taken to be 
supported in an arbitrarily small
neighborhood of the singular strata of $\mathbf{D}= \cup_{i=1}^k D_i$.)
\end{thm}
This deformation does not change the symplectomorphism type of the complement
$X = M \backslash \mathbf{D}$ \cite{McLean:2012ab}*{Lemma 5.15}.  For notational simplicity, we replace $\omega$
by $\tilde{\omega}$ going forward, i.e., assume that our divisors $\mathbf{D}$
admit (and come equipped with) an $\omega$-regularization $\{\psi_I: U_I \to M\}_{I \subset S}$, which we will henceforth refer to as simply a regularization.
Note that the existence of a regularization implies that the
divisors $\{D_i\}_{i=1}^k$ are symplectically orthogonal. In what follows, except when necessary
we will drop the parameterizations $\psi_i$ from our notation and identify the
source $U_i$ with its image in $M$.  We have projection maps 
\begin{equation}
            \pi_i: U_i \ra D_i
\end{equation}
such that on a $|I|$-fold intersection of tubular neighborhoods 
\begin{equation}
    U_I: = \cap_{i \in I} U_i = U_{i_1} \cap \cdots \cap U_{i_{|I|}},
\end{equation}
iterated projection $|I|$ times 
\begin{equation}\label{eq:piI}
    \pi_I := \pi_{i_1} \circ \pi_{i_2} \circ \cdots \circ \pi_{i_{|I|}}: U_I \ra D_I
\end{equation}
is a symplectic fibration with structure group $U(1)^{|I|}$ and with fibers symplectomorphic to a(n open subset of a) product of standard discs 
\begin{equation}\label{eq:fiberdisks}
    \pi^{-1}_I(p) \hookrightarrow \prod_{i \in I} \mathbf{D}_{\epsilon}.
\end{equation}
We also have the radial (norm-squared distance to $D_i$) functions
\begin{equation}\label{rhofunctions}
    \rho_i: U_i \to \R.
\end{equation}
See Figure \ref{normalcrossing} for a schematic picture of $(M, \D)$, the strata $D_I$, and the neighborhoods $U_I$ in a simple case.
\begin{figure}[h] 
    \caption{A schematic of a normal crossing divisor with two components (represented by lines meeting orthogonally, as the divisors can be arranged to meet symplectically orthogonally) and their tubular neighborhoods (represented by 1-dimensional disc bundles). \label{normalcrossing}}
    \centering
    \includegraphics[scale=1.5]{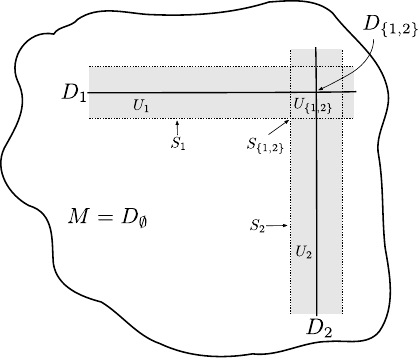}
\end{figure}

An elementary but important consequence of having such an $\omega$-regularization is:
\begin{lem}\label{nicehamiltonianvectorfields}
    \begin{enumerate}
        \item[(1)] The symplectic orthogonal to the tangent space of any fiber $\pi^{-1}_i(p)$
    is contained in a level set of the radial function $\rho_i: U_i \to \R$. 
\item[(2)] In particular, if $I = \{i_1, \ldots, i_s\}$, any smooth function
    $f$ of the corresponding radial functions $f(\rho_{i_1}, \ldots,
    \rho_{i_s}): U_I \to \R$, has Hamiltonian vector field $X_{f}$ tangent to
    the fibers of $\pi_{I}$. 
\item[(3)] For $f$ as in (2), if $F = \pi_{I}^{-1}(p)$ denotes any fiber with
    its standard (induced by \eqref{eq:fiberdisks}) symplectic form, 
    \begin{equation}\label{hamvectorfield}
        X_f|_{F} = X_{f |_{F}} = \sum_{i \in I} 2 \frac{\partial f}{\partial \rho_i}\partial_{\varphi_i}
    \end{equation}

\item[(4)] Let $I = \{i_1, \ldots, i_s\} \subset \{1, \ldots, k\}$ be any
    subset. In the associated neighborhood $U_I$ of $D_I$, any two functions of
    the radial coordinates $f(\rho_{i_1}, \ldots, \rho_{i_s})$, $g(\rho_{i_1},
    \ldots, \rho_{i_s})$ have commuting Hamiltonian vector fields: $\omega(X_f,
    X_g) = df(X_g) = -dg(X_f) = 0$.  In particular, $d\rho_{i_t}(X_f)$ = 0 for
    any $t$ and $f$ as above.

    \end{enumerate}
\end{lem}
\begin{proof}
    By the regularization property, we can check (1) using the model
    symplectic form for $D_i$ \eqref{localsymplecticformcodimension2}, where
    the result is immediate.  To see (2), note that 
    $df = \omega(X_{f}, -)$ must be zero on vectors tangent to the intersection of
    level sets of $\rho_{i_1}, \ldots, \rho_{i_s}$. Hence, $X_{f}$ is 
    orthogonal to the symplectic orthogonal to the tangent space to the fibers,
    i.e., tangent to the fibers. (3) is an immediate corollary of (2), and the
    specific form of $X_{f|_{F}}$ is a calculation in the fiber $F \subset \prod
    \mathbf{D}_{\epsilon}$. (4) can be deduced from \eqref{hamvectorfield} and
    the fact that $d\rho_j(\partial_{\varphi_i}) = 0$.
 \end{proof}

Finally, we recall what it means for the 1-form $\theta$ on $X = M \backslash
\D$ to be compatible with the regularization chosen above, in the sense of
\cite{McLean:2012ab}.
\begin{defn}\label{def:niceconvexstructure}
    Fix $(M, \D)$ be as above, as well as an $\omega$-regularization
    $\{\psi_I: U_I \to M\}_{I \subset S}$ for $\mathbf{D}$.  
    We say a convex symplectic symplectic structure $\theta$ on $X = M
    \backslash \D$ is {\em nice} (with respect to the fixed data) if on each
    $ \pi_I := U_I \ra D_I$ described in \eqref{eq:piI},
    $\theta$ restricted to each fiber $(\pi_I)^{-1}(pt)$ agrees with
   \[
       \sum_{i \in I} (\frac{1}{2}\rho_i - \frac{\kappa_i}{ 2\pi}) d\varphi_i.
   \]
   with respect to the identification of the fibers \eqref{eq:fiberdisks},
   where $\kappa_i \in \mathbb{Z}^{> 0}$ are as in \eqref{eq:kappai}.
\end{defn}

\begin{thm}[\cite{McLean:2012ab}*{Thm. 5.20}] 
    \label{thm:exactnicestructure}
    After possibly shrinking the neighborhoods $U_i$ (and $U_I$), there exists
    a deformation
    of the canonical convex symplectic structure $(X,\theta)$ to one $(X, \tilde{\theta})$ which is nice.
\end{thm}
Going forward, we replace $(X,\theta)$ by the corresponding nice structure, a
process which leaves symplectic cohomology unchanged (recall that symplectic
cohomology is unchanged under deformations such as those appearing in Theorem
\ref{thm:exactnicestructure}, compare \cite{McLean:2012ab}*{Lemma 4.11}).

\subsection{Normal crossings-adapted Liouville domains and Hamiltonians}
This section describes the exhaustive family of Liouville domains as well as the Hamiltonian functions which will be used to define Floer cohomology.  

Denote by $UD \subset M$ the union of the neighborhoods of $D_i$ in our
regularization (thought of as living in $M$) 
\begin{equation}
    UD:=\cup_iU_i.
\end{equation}
There is some $\epsilon$ such that $\rho_i^{-1}[0,\epsilon^2] \subset U_i$
(where now we think of $U_i \subset N_M D_i$ and $\rho_i: N_M D_i \to \R$) for
all $i$, i.e., each $U_i$ contains a tube of size $\epsilon$. 
For any $\rho_0 \leq \epsilon^2$ such that 
$\rho_0 \leq \operatorname{min}_i \frac{2\pi \epsilon^2} {\kappa_i}$, 
we define the subregion 
\begin{equation} \label{eq:randomUithing}
    U_{i,\rho_{0}}:= \{ \frac{\rho_i}{\kappa_i/2\pi} \leq \rho_0\} \subset U_i 
\end{equation}
and
\begin{equation} 
    UD_{\rho_{0}}:= \cup_i U_{i,\rho_{0}}.
\end{equation}

For any $\epsilon_1 \leq \sqrt{\rho_0}$ sufficiently small, we can associate a smoothing of the
hypersurface with corners, $\hatX:=\partial(M\setminus UD_{\epsilon_1^{2}})$
which depends on a vector of sufficiently real small numbers
$\vec{\epsilon}=(\epsilon_1,\epsilon_2,\epsilon_3)$. 
More precisely, let $q(s) = q_{\epsilon_1,\epsilon_2}(s): [0, \epsilon_1^2) \to \R$ be a non-negative function (implicitly depending on $\epsilon_1, \epsilon_2$) satisfying:
\begin{enumerate} 
    \item There is some $\epsilon_{2} \in (\epsilon_1-\epsilon_1^3/2,\epsilon_1)$ such that $q(s)=0$ iff $s \in [\epsilon_{2}^2,\epsilon_1^2)$.

    \item The derivative of $q(s)$ is strictly negative when $q(s) \neq 0$.

    \item $q(s)=1-s^2$ near $s=0$. 

    \item \label{itm: conv} There is a unique point $s=s_0$ with $q''(s_0)=0$ and $q(s_0) \neq 0$. 

\end{enumerate} 
(Compare \cite{McLean2}*{proof of Theorem 5.16} and \cite[\S 3.1]{GP1}) Define 
\begin{align}
    S_{\epsilon_1,\epsilon_2}(\rho_1,\cdots,\rho_k): UD_{\epsilon_1^{2}} &\to \mathbb{R}\\
    S_{\epsilon_1,\epsilon_2}(\rho_1, \cdots, \rho_k)&=\sum_i q(\frac{2 \pi \rho_i}{\kappa_i}) \label{Sfunction}.
\end{align} 
where we implicitly smoothly extend $q(\frac{2 \pi \rho_i}{\kappa_i})$ to be 0 outside of the region where $\rho_i$ is defined.
We let \begin{align} \label{eq: definitionsigma} \partial \bar{X}_{\vec{\epsilon}}= \Sigma_{\vec{\epsilon}}
:=S_{\epsilon_1, \epsilon_2}^{-1}(\epsilon_3) \end{align}  (Meaning $\bar{X}_{\vec{\epsilon}}$ is the region bounded by $\Sigma_{\vec{\epsilon}}$  for $\epsilon_3$ sufficiently small).
The fourth hypothesis implies that $q''(s)>0$ when $q$ is sufficiently small, which in turn implies that the function $S_{\epsilon_1,\epsilon_2}$ is convex near $S_{\epsilon_1,\epsilon_2}^{-1}(\epsilon_3)$; this will be useful in understanding the Reeb dynamics on the contact manifolds defined below.  Let $Z$ be the {\em Liouville vector field}, that is to say the
canonical vector field on $X$ determined by the equation: 
\begin{align}
    \iota_Z\omega = \theta 
\end{align} 
\begin{lem}
For $\epsilon_3$ sufficiently small the Liouville vector field $Z$ associated to
$\theta$ is strictly outward pointing along $\Sigma_{\vec{\epsilon}}$; in
particular $(\bar{X}_{\vec{\epsilon}}, \theta)$ is a Liouville domain.
\end{lem}
\begin{proof}
    This is the content of \cite[Lemma 3.7]{GP1} and uses the ``nice'' property
    of $\theta$ and the $\omega$-regularization of $\mathbf{D}$ (specifically
    Lemma \ref{nicehamiltonianvectorfields}); see also \cite[Theorem
    5.16]{McLean2} for a similar discussion in the case of concave boundaries. 
\end{proof}

\begin{rem} 
    \label{rem: C0close} By choosing $\epsilon_3$ and $|\epsilon_2 - \epsilon_1|$
    sufficiently small, we can ensure that the rounding $\partial
    \bar{X}_{\vec{\epsilon}}$ is arbitrarily $C^0$ close to the original
    cornered domain $\hatX$, a fact which will allow us to obtain explicit
    estimates on the actions of Hamiltonian orbits.
\end{rem}
It follows also that $\Sigma_{\vec{\epsilon}} = \partial
\bar{X}_{\vec{\epsilon}}$ admits a contact structure with contact form $\alpha
= \theta|_{\Sigma_{\vec{\epsilon}}}$. We recall that any contact manifold
equipped with a contact form $(Y, \alpha)$ possesses a canonical {\em Reeb
vector field} $X_{Reeb}$ determined by $\alpha(X_{Reeb}) = 1$,
$d\alpha(X_{Reeb}, -) = 0$. The {\em spectrum} of $(Y, \alpha)$, denoted
$\operatorname{Spec}(Y)$ if $\alpha$ is implicit, is the set of real numbers
which are lengths of closed Reeb orbits of $X_{Reeb}$, and is discrete if
$\alpha$ is sufficiently generic. 

Any such choice of $\bar{X}_{\vec{\epsilon}}$ induces a Liouville
coordinate defined as 
\begin{equation}\label{eq:inducedliouvillecoordinate}
    R^{\vec{\epsilon}}(x) = e^t,
\end{equation}
where $t$
is the time it takes to flow along $Z$ from the hypersurface $\partial
\bar{X}_{\vec{\epsilon}}$ to $x$. Flowing for some small negative time $t_0^{\vec{\epsilon}}$ defines a collar neighborhood of $\partial \bar{X}_{\vec{\epsilon}} = \Sigma_{\vec{\epsilon}}$,
\[
    C(\Sigma_{\vec{\epsilon}}) \subset \bar{X}_{\vec{\epsilon}}.
\] 
Letting $X_{\vec{\epsilon}}^{o}$ denote the complement of this collar in the domain $\bar{X}_{\vec{\epsilon}}$
\[
    X_{\vec{\epsilon}}^o := \bar{X}_{\vec{\epsilon}} \setminus C(\Sigma_{\vec{\epsilon}}),
\]
we see that $R^{\vec{\epsilon}}$ may be viewed as a function 
$R^{\vec{\epsilon}}: X\setminus X_{\vec{\epsilon}}^{o} \to \mathbb{R}.$
As shown in \cite{GP1} (see Lemma 3.10 and the preceeding discussion), 
the fact that our convex symplectic structure is ``nice'' with respect to our fixed
regularization implies that $R^{\vec{\epsilon}}$ is a function of
$\rho_1,\cdots,\rho_k$
\[
    R^{\vec{\epsilon}}=R^{\vec{\epsilon}}(\rho_1,\rho_2,\cdots,\rho_k)
\]
(meaning in each $U_I$ it is a function of $\rho_i$ for $i \in I$)
and moreoever that $R^{\vec{\epsilon}}$ extends smoothly across the divisors $\mathbf{D}$, hence can be viewed
as a function (which by abuse of notation we use the same name for):
\begin{equation}
    R^{\vec{\epsilon}}: M\setminus X_{\vec{\epsilon}}^{o} \to \R.
\end{equation}
We say a function $h(r): \R \to \R_{\geq 0}$ is {\em linear adapted to $R^{\vec{\epsilon}}$ of slope $\lambda$} if
\begin{enumerate}
    \item $h$ vanishes for $r \leq e^{-t^{\vec{\epsilon}}_0}$ (where $t^{\vec{\epsilon}}_0$ is as before). 
    \item $(h) ' \geq 0$; 
    \item $(h) '' \geq 0$; and crucially, 
    \item (linearity at $\infty$ of slope $\lambda$)
        for some $K_{\vec{\epsilon}}$ 
        much closer to $1$ than $\operatorname{min}_{\D} R^{\vec{\epsilon}}$, \footnote{Recall $\operatorname{min}_{\D} R^{\vec{\epsilon}} > 1$, because $R^{\vec{\epsilon}} =1$ along $\partial \bar{X}_{\vec{\epsilon}}$.}
        \begin{align} \label{eq:linearity}
            h(r)=\lambda(r-1) \quad  \forall r \geq K_{\vec{\epsilon}} 
        \end{align}  
\end{enumerate}
Note that because of condition (1), the composition $h(R^{\vec{\epsilon}})$ is linear outside the compact subset of $X$ given by $X^o_{\vec{\epsilon}} \cup (R^{\vec{\epsilon}})^{-1}(-\infty, K_{\vec{\epsilon}}])$, and extends
smoothly to a Hamiltonian on all of $M$, which we also call $h$: 
\begin{equation}
    h:= h(R^{\vec{\epsilon}}): M \to \mathbb{R}^{\geq 0}.
\end{equation}
We often fix the $K_{\vec{\epsilon}}$ for which \eqref{eq:linearity} occurs,
and call it the {\em linearity level} of $h$. Often we also fix an auxiliary
$\mu_{\vec{\epsilon}} \in (K_{\vec{\epsilon}}, \operatorname{min}_{\D}
R^{\vec{\epsilon}})$ in order to define open neighborhoods of $\D$
$V_{0,\vec{\epsilon}} = (R^{\vec{\epsilon}})^{-1}(\mu_{\vec{\epsilon}}, \infty)$,
$V_{\vec{\epsilon}} = (R^{\vec{\epsilon}})^{-1}(K_{\vec{\epsilon}}, \infty)$ and
a ``(contact) shell'' region
\begin{equation}\label{eq:shell}
    V_{\vec{\epsilon}} \backslash V_{0,\vec{\epsilon}}
\end{equation}
along which we require almost complex structures to have a specified
(contact-type) form. We often also perform $C^2$-small (time-dependent) perturbations of the
functions $h$ to functions $H$, taking care that on some shell of the form
\eqref{eq:shell}, the function is unperturbed. 
If $\mathcal{L} X$ denotes the free loop space of $X$, and $H: S^1 \times X \to \R$ a (possibly time-dependent) Hamiltonian, recall that the {\em ($H$-perturbed) action functional} $A_{H}: \mathcal{L}X \to \mathbb{R}$ is defined to be 
\begin{align} \label{eq:action}
    A_{H}(x_0)= -\int_{S^1} x_0^*(\theta) + \int_0^{1} H(t,x_0(t)) dt 
\end{align}

Now we consider a sequence of these Liouville domains indexed by $\ell \in
\mathbb{Z}^{>0}$. Choose a sequence of integers $w_\ell \in \mathbb{Z}^{\geq 0}$ with
$w_{\ell+1}>w_\ell$ and fix for each $\ell$, a collection of parameters $\vec{\epsilon}_\ell$ needed to
define $\Sigma_{\vec{\epsilon}_{\ell}}$. Let
\begin{align} \label{eq:defnsigmaell} \Sigma_{\ell}:= \partial \bar{X}_{\ell}:= \Sigma_{\vec{\epsilon}_{\ell}} \end{align}  be the corresponding hypersurface, let
$\bar{X}_\ell$ be the corresponding domain bounded by this hypersurface, and let
$\epsilon_\ell:=\epsilon_{1,\ell}$. We assume that our choices are made so that 
\begin{equation}
    \bar{X}_{\ell_{1}} \subseteq \bar{X}_{\ell_{2}}\textrm{ whenever }\ell_1<\ell_2.
\end{equation}
Let $R^{\ell}$ denote the (induced by $\bar{X}_{\ell}$) Liouville coordinate as in \eqref{eq:inducedliouvillecoordinate}, and $X_{\ell}^o$ the complement in $\bar{X}_{\ell}$ of the time $-t_0^{\ell}$ flow collar (with respect to $Z$) of $\partial\bar{X}_{\ell}$. As before,
$R^{\ell}$ is a function $X \setminus X_{\ell}^o \to \R$ extending smoothly across the divisors $\D$ to a function
\begin{equation}
\rlm: M\setminus X_\ell^{o} \to \R,
\end{equation}
and moreoever is a function of the radial coordinates $\rho_1, \ldots, \rho_k$.
We further choose for each $\ell \in \Z$ a real number $\lambda_\ell \notin
\operatorname{Spec}(\partial \bar{X}_{\ell})$, satisfying
\begin{equation}
w_\ell <\lambda_\ell<w_{\ell}+1.
\end{equation}

Fix a collection of tuples $\{(w_{\ell}, \lambda_\ell, \bar{X}_\ell)\}_{\ell \in \Z_{> 0}}$ satisfying the
conditions of the previous paragraph. For each $\ell$, we choose a function
$\hlm(r): \R \to \R_{\geq 0}$ which is linear adapted to $R^{\ell}$ of slope $\lambda_\ell$ with linearity level $K_\ell$, in the sense defined earlier.
As before, the composition $\hlm(R^{\ell})$ extends
smoothly to a Hamiltonian on all of $M$, which we also call $\hlm$: 
\begin{equation}
\hlm:= \hlm(\rlm): M \to \mathbb{R}^{\geq 0}.
\end{equation}
Also as before, we choose an auxiliary constant $\mu_{\ell} \in (K_{\ell}, \operatorname{min}_{\D} \rlm)$, in order to define
two open sets $V_{0,\ell} \subset V_\ell \subset UD $ containing $\D$ by
\begin{align}
    V_{0,\ell}&:=(\rlm)^{-1}( \mu_{\ell} ,\infty), \\
    V_{\ell}&:=(\rlm)^{-1}(K_\ell,\infty) 
\end{align}
and the associated contact shell $V_\ell \setminus V_{0,\ell}$.  It is
straightforward to ensure (and henceforth we assume) that 
\begin{equation}\label{eq:disjointshells}
    \textrm{for $\ell \neq \ell'$, the contact shells' $V_{\ell} \backslash V_{0, \ell}$ and $V_{\ell'} \backslash V_{0,\ell'}$ are disjoint.}
\end{equation}

Compared to Section 3 of \cite{GP1}, we will shortly take $\bar{X}_\ell$ to be
an exhaustive family of Liouville manifolds for $X$ (that is let $\epsilon_\ell
\to 0$). Note that, in particular, the $C^0$-norm of Hamiltonian $\hlm$ can be
assumed arbitrarily small (for a given slope $\lambda_\ell$) by making
$\epsilon_\ell$ sufficently small and $K_\ell$ sufficiently close to 1. 
In fact,
along the divisors, we have the following estimate for $\hlm$:
\begin{lem}\label{lem:hamiltonianestimateonD}
    If $\Sigma_\ell$ is sufficiently $C^0$ close to $\hatXl$ (as in Remark
    \ref{rem: C0close}), then on $\D$ the following estimate holds:
\begin{align}
    \label{eq:estimateD} \hlm \approx
    \lambda_{\ell}(\frac{1}{1-\frac{1}{2}\epsilon_{\ell}^2}-1).
\end{align}
\end{lem}
\begin{proof}
    Given \eqref{eq:linearity}, the estimate \eqref{eq:estimateD} is equivalent
    to the following estimate for $R$ at points of $\D$:
    \begin{equation}\label{eq:estimateRD}
        R \approx \frac{1}{1 - \frac{1}{2}\epsilon_{\ell}^2}.
    \end{equation}
    To see this, let $p \in \D$ be any point; it is contained $D_I \subset
    U_I$ for some $I = \{i_1, \ldots, i_s\} \subset \{1, \ldots, k\}$. By the
    discussion in \cite[p. 15]{GP1}
    the time it takes for $p$ to flow by $-Z$ to $\Sigma_\ell$ is the time
    it takes for $p$ to flow by $-Z_{vert}$ to $\Sigma_\ell \cap U_I$, where
    $Z_{vert}$, the vertical component of $Z$ with respect to the symplectic
    fibration $\pi_I: U_I \to D_I$ is given by:
    \begin{equation}\label{eq:zvert}
        Z_{vert} = \sum_{i \in I} (\rho_i - \frac{\kappa_i}{\pi} ) \partial_{\rho_i}.
    \end{equation}
    This time now only (approximately) depends on a computation in the
    coordinates $(\rho_{i_1}, \ldots, \rho_{i_s}) \in \R^{|I|}$ as we 
    note that for any $x \in \Sigma_\ell \cap U_I \approx \hatXl \cap U_I$,
    $(\rho_{i_1}(x), \ldots, \rho_{i_s}(x)) \approx (\epsilon_\ell^2
    \frac{\kappa_{i_1}}{2\pi}, \ldots, \epsilon_\ell^2
    \frac{\kappa_{i_s}}{2\pi})$ (and $(\rho_{i_1}(p), \ldots, \rho_{i_s}(p)) = \mathbf{0}$). 
    Thus it (approximately) suffices to compute the (exponential of the) time
    it takes to flow in $\R^{|I|}$ from $\mathbf{0}$ to $(\epsilon_\ell^2
    \frac{\kappa_{i_1}}{2\pi}, \ldots, \epsilon_\ell^2 \frac{\kappa_{i_s}}{2\pi})$
    by $-Z_{vert}$;
    this can be done one component at a time as everything is split.  For $j
    \in I$, let $\gamma(t)$ be an integral curve of $-Z_{vert}^{j} = -(\rho_{j} -
    \frac{\kappa_{j}}{\pi} ) \partial_{\rho_j} $ starting at $0$, i.e.,
    $\gamma(t)$ solves the ODE $\gamma'(t) = -\gamma(t) +
    \frac{\kappa_{j}}{\pi}$ with initial condition $\gamma(0) = 0$, i.e.,
    $\gamma(t) = -\frac{\kappa_{j}}{\pi} e^{-t} + \frac{\kappa_{j}}{\pi}$. We
    find $\gamma(t) = \epsilon_\ell^2 \frac{\kappa_{j}}{2\pi}$ precisely when $e^t
    = \frac{1}{1-\frac{1}{2}\epsilon_\ell^2}$. 
\end{proof}

Because the Hamiltonian flow of $\hlm$ preserves $\D$, the time-1 orbits
of this flow are either completely contained in $\D$ or completely
contained in $X$ (in fact $X\setminus V_{\ell}$, in light of the form
\eqref{eq:linearity} of $\hlm$ in $V_{\ell} \backslash \D$.). We will refer to
the set of (time-1) orbits contained in $\D$ as {\em divisorial orbits} and
denote them by  
\begin{equation}
\mathcal{X}(\D; \hlm)
\end{equation}
and all other (time-1) orbits by  
\begin{equation}
\mathcal{X}(X; \hlm). 
\end{equation}
We first describe the orbits of the Hamiltonian flow of $\hlm$ inside of
$M\setminus V_{\ell}$. Let 
\begin{equation}
    R_{0,\ell} := \max \{ \rlm\ |\ \hlm(\rlm) = 0\}.
\end{equation}
be the largest value of $\rlm$ for which $\hlm (\rlm)=0$. The time-1 orbits of
the Hamiltonian $\hlm$ lying in the region $M \setminus V_{\ell}$ come in two types of families. The
first is
the set of constant orbits which is the complement 
\begin{equation}\label{constantorbits}
    \mathcal{F}_{\mathbf{0}}:= M\setminus \lbrace \rlm \geq R_{0,\ell} \rbrace.
\end{equation}
This is a manifold with boundary. 

The second type of Hamiltonian orbit corresponds to non-constant orbits.
By Lemma \ref{nicehamiltonianvectorfields},
over each stratum $U_I$ the Hamiltonian flow is tangent to the
fibers of the projections $ \pi_I: U_I \to D_I$ 
and has the explicit form
\begin{align} \label{Xhlm}
    X_{\hlm}=\sum_{i \in I} 2 \frac{\partial \hlm}{\partial \rho_i}\partial_{\varphi_i}.
\end{align} 
 The
 orbits of \eqref{Xhlm} in $U_I$ correspond to (possibly multiply-covered) circles in any fiber of $\pi_I$ where
\begin{align} 
    X_{\hlm}=\sum_{i \in I} -2 \pi v_i\partial_{\varphi_i}
\end{align}  
for an integer vector $\v = (v_1, \ldots, v_k) \in \mathbb{Z}^k_{\geq
0}$ which has non-zero $i$th component precisely when
 $i \in I \subset \{1, \ldots, k\}$; in other words, wherever $[\frac{\partial
\hlm}{\partial \rho_i}]_{i=1}^k=-\pi\v$. These sets, denoted by
\begin{equation} \label{eq: orbitsetsF}
    \mathcal{F}_{\v}
\end{equation}
are connected in view of item \eqref{itm: conv} in the definition of ($q$ and
hence) $S_{\epsilon_1,\epsilon_2}$ and are homeomorphic to 
manifolds-with-corners. For proofs of these assertions, see Step 2 of the proof of Theorem 5.16 of \cite{McLean2}.    The {\em multiplicity vector associated to an orbit $x_0$},
\[
    \v(x_0),
\]
is the unique $\v$ so that $x_0 \in \mathcal{F}_{\v}$ (where $\v = \mathbf{0}$ if $x_0 \in \mathcal{F}_{\mathbf{0}}$).
We define the 
{\em weighted winding number} of a multiplicity vector $\v$ to be
\begin{equation}\label{eq:windingvec}
    w(\v) := \sum_{i \in \{1, \ldots, k\}} \kappa_i v_i = \sum_{i \in I = \mathrm{support}(\v)} \kappa_i v_i,
\end{equation}
and the weighted winding number of an orbit $x_0$ to be
\begin{equation}\label{eq:weightedwindingnumber}
    w(x_0):= w(\v(x_0))
\end{equation}
(note these are integers). The weight $w(\v)$ will appear geometrically in our
setup as the limiting action of (meaning integral of $\theta$ over) any
sufficiently small loop in $X$ winding $\v_i$ times around $D_i$ (compare
\cite[Lemma 2.11]{GP1} or Lemma \ref{lem:energypss} below). 

We can make all of the orbits nondegenerate by a $C^2$ small time-dependent
perturbation $\Hlm: M \to \mathbb{R}$. Describing careful choices of
$C^2$ small time-dependent perturbations for the divisorial orbits as well as
near orbits in $X$ would take us on a small detour and so we postpone this to
\S \ref{subsection:spectral}. For now we just state the two most important
properties of these $C^2$-small perturbations: 
\begin{itemize} 
    \item The perturbation is disjoint from $V_{\ell} \setminus V_{0,\ell}$ and inside of $M\setminus V_{\ell}$, it is supported in the disjoint union of small isolating neighborhoods $U_\v$ of the orbit sets $\mathcal{F}_{\v}$.  

    \item The Hamiltonian flow of $\Hlm$  preserves each divisor $D_i$.
\end{itemize}
As before we let 
\begin{equation} \label{eq:divisorially}
    \mathcal{X}(\D; \Hlm)
\end{equation}
denote the time-1 orbits of
$\Hlm$ contained in $\D$ and let
\begin{equation}
    \mathcal{X}(X; \Hlm).
\end{equation}
denote the remaining time-1 orbits of $\Hlm$, which are disjoint from $\D$ (and lie in $X \backslash V_\ell$).
Let $A_{\ell}:=A_{H^{\ell}}$ denote the ($H^{\ell}$-perturbed) action functional as in \eqref{eq:action}
 If we assume that $\Sigma_\ell$ is taken sufficiently $C^0$ close to $\hatXl$,
 we also obtain good estimates on the actions of our Hamiltonian orbits:

\begin{lem} \label{lem: sharpactions}
   By taking $\Sigma_\ell$ sufficiently $C^0$ close to $\hatXl$, $K_{\ell}$ sufficiently close to 1 and $t^{\vec{\epsilon}_\ell}_0$ sufficiently small, the action of each orbit set $\mathcal{F}_\v$ can be made arbitarily close to 
    \begin{align} \label{eq: action} 
        A_\ell (x_0) \approx -w(x_0)(1- \epsilon_\ell^2/2) 
    \end{align}  
\end{lem} 

\begin{proof} 
   The orbit set $\mathcal{F}_\v$ lies in the region where $R^{\ell} \leq K_{\ell}.$ By assuming $K_\ell$ is sufficiently close to 1, the Hamiltonian term in can be taken arbitrarily small and so we focus on the contact term. Because the orbits lie in the fibers, the action may be calculated by:
    \begin{align} 
        -\sum_i \int_{x_{0}} (\frac{1}{2}\rho_i-\frac{\kappa_i}{2\pi}) d\varphi_i  \\
        =-\sum_i (\frac{1}{2}\rho_i-\frac{\kappa_i}{2\pi}) (-2\pi v_i) 
    \end{align} 
   By taking  $t^{\vec{\epsilon}_\ell}_0$ and $K_\ell$ close to zero, we also have that $\rho_i$ can be taken arbitrarily close to $(\frac{\kappa_i}{2\pi})(\epsilon_\ell)^2$ thus completing the computation.  
 \end{proof} 

 \begin{defn}\label{defn:complexint} 
Define $\mathcal{J}(M,\D)$ to be the space of $\omega$-tamed almost complex structures $J$ which preserve $\D$  and such that 
\begin{enumerate} 
    \item for any $i \in \lbrace 1,\cdots, k \rbrace$, $p \in D_i$, and tangent vectors $\eta_1,\eta_2 \in T_pM$, the Nijenhuis tensor $N_J(\eta_1,\eta_2) \in T_p D_i$. 
    
\end{enumerate}
\end{defn}

\begin{defn} For any choice of Liouville domain and shells $\bar{X}_{\vec{\epsilon}}, V_{\vec{\epsilon}}, V_{0,\vec{\epsilon}}$,
    define $\mathcal{J}(\bar{X}_{\vec{\epsilon}}, V) \subset \mathcal{J}(M,\D)$ to be the space of
    $\omega$-compatible almost complex structures which are of contact type on
    the closure of $V_{\vec{\epsilon}} \backslash V_{0,\vec{\epsilon}}$, meaning on this region 
    \begin{equation}
        \theta \circ J = -dR^{\vec{\epsilon}}.
    \end{equation}
For the case when $\bar{X}_{\vec{\epsilon}}=\bar{X}_\ell, V_{\vec{\epsilon}}= V_\ell, V_{0,\vec{\epsilon}}=V_{0,\ell}$,
we simplify the notation by 
$$\mathcal{J}_\ell(V):= \mathcal{J}(\bar{X}_{\ell}, V).$$ 
\end{defn}

It follows by standard arguments that these spaces are non-empty (see e.g.,
\cite{Ionel:2011fk}*{\S A}) and contractible.

\subsection{Floer cohomology} \label{subsection:Floercohom}
Choose a Hamiltonian 
$\Hlm$ for each $\ell$, by perturbing the Hamiltonians $h^\ell$ as in the previous section. 
For each Hamiltonian orbit $x \in \mathcal{X}(X; \Hlm)$, a
choice of trivialization $\gamma$ of $x^*(TX)$ determines a 1-dimensional real vector
space 
\begin{equation}\label{eq:detline}
    \mathfrak{o}_x,
\end{equation}
the {\em determinant line} associated to a local Cauchy-Riemann operator
$D_\gamma$. Implicitly this depends on the choice of trivialization $\gamma$,
but for notational simplicity we remove this ambiguity (assuming $c_1(X) = 0$) by working in the $\Z$-graded context: 
we fix a holomorphic volume form $\Omega_{M, \mathbf{D}}$ on $M$ which is
non-vanishing on $X$, 
and to define \eqref{eq:detline} choose the unique trivialization $\gamma$ of
$x^*TX$ compatible with the trivialization of $\Lambda_{\C}^n T^*X$ induced by
$\Omega_{M, \D}$.
Recall also that the {\em $\K$-normalization} of any vector space $W$, denoted
\begin{equation}
    |W|,
\end{equation} 
is the free $\K$ module generated by the set of orientations of $W$, modulo the
relation that the sum of the orientations vanishes. We call $|\mathfrak{o}_x|$
the {\em orientation line} associated to $x$, and define 
\begin{align} \label{eq:Floerrelcoc}
    CF^*(X \subset M;\Hlm) := \bigoplus_{x \in \mathcal{X}(X; \Hlm)} |\mathfrak{o}_x|. 
\end{align} 

\begin{defn} \label{defn: S1acdef} Let $\mathcal{J}_{F}(M,\D)$ denote the space of $S^1$ dependent complex structures, $\mathcal{C}^\infty(S^1; \mathcal{J}(M,\D))$, and let $\mathcal{J}_{F,\ell}(V) \subset \mathcal{J}_{F}(M,\D)$ denote the space $\mathcal{C}^\infty(S^1; \mathcal{J}_{\ell}(V))$. 
\end{defn}

Choose a generic $S^1$-dependent almost complex structure $J_F \in \mathcal{J}_{F,\ell}(V)$. For pairs of orbits $x_0$, $x_1 \in  \mathcal{X}(X; \Hlm)$, let
$\widetilde{\mathcal{M}}(x_0,x_1)$ denote the moduli space of {\em Floer
trajectories} between $x_1$ and $x_0$, namely the space of solutions to the
following PDE with asymptotics:
\begin{align} \label{eq:FloerRinv}
\left\{
\begin{aligned}
 & u \colon \mathbb{R} \times S^1 \to X, \\
& \lim_{s \to -\infty} u(s, -) = x_0\\
& \lim_{s \to +\infty} u(s, -) = x_1 \\
 &  \partial_s u + J_{F}(\partial_tu-X_{\Hlm})=0.
\end{aligned}
\right.
\end{align}
For generic $J_F$, $\widetilde{\mathcal{M}}(x_0,x_1)$ is a manifold of dimension
\[ 
    \deg(x_0) - \deg(x_1) 
\]
where $\deg(x)$, the index of the local operator $D_x$ associated to $x$, is
equal to $n - CZ(x)$, where $CZ(x)$ is the Conley-Zehnder index of
$x$ 
\cite{Floer:1995fk}. 
There is an induced $\mathbb{R}$-action on the
moduli space $\widetilde{\mathcal{M}}(x_0,x_1)$ given by translation in the $s$-direction,
which is free for non-constant solutions. Whenever $\deg(x_0) - \deg(x_1) \geq 1$, 
for a generic $J_F$ the quotient space 
\begin{equation}\label{eq:modulispacetraj}
    \mathcal{M}(x_0, x_1) := \widetilde{\mathcal{M}}(x_0,x_1)/ \mathbb{R}
\end{equation}
is a manifold of dimension $\deg(x_0) - \deg(x_1) - 1$. 
Whenever $\deg(x_0) - \deg(x_1) = 1$ (and $J_F$ is generically chosen), standard Gromov compactness arguments show that
the moduli space \eqref{eq:modulispacetraj} is compact of dimension 0, provided
elements of $\widetilde{\mathcal{M}}(x_0, x_1)$ have image contained in some
compact subset $K \subset X$ only depending on $x_0$ and $x_1$. In turn this
latter a priori $C^0$ bound (away from $\D$) is a consequence of standard
maximum principle arguments, which prevent Floer solutions from crossing the
region where the Hamiltonian has the special form \eqref{eq:linearity} and
$J_F$ is also contact type; see e.g., \cite[Lem. 7.2]{Abouzaid:2010ly}.
Now orientation theory associates, to every rigid element 
$u \in \mathcal{M}(x_0,x_1)$
an isomorphism of orientation lines $\mu_u: \mathfrak{o}_{x_1}
\stackrel{\sim}{\ra} \mathfrak{o}_{x_0}$ and hence an induced map $\mu_u:
|\mathfrak{o}_{x_1}| \ra |\mathfrak{o}_{x_0}|$. 
Using this, one defines the $|\mathfrak{o}_{x_1}| - |\mathfrak{o}_{x_0}|$
component of the differential
\begin{equation} 
    (\partial_{CF})_{x_1, x_0}  = \sum_{u \in \mathcal{M}(x_0,x_1)} \mu_u
\end{equation} 
whenever $\deg(x_0) = \deg(x_1) + 1$. 
A standard analysis of the boundary of the (compactified by adding broken
trajectories) 1-dimensional components of \eqref{eq:modulispacetraj} implies
that $\partial_{CF}^2 = 0$.\footnote{Once more, this requires establishing that
relevant Floer trajectories are a priori bounded away from $\D$, which is a
consequence of the maximum principle mentioned earlier.}
We define $HF^*(X\subset M; \Hlm)$ to be the cohomology of the complex 
$(CF^*(X \subset M; \Hlm),\partial_{CF})$. 
\begin{lem}\label{indJshellchange}
   Fix a Liouville domain $\bar{X}_\ell$ as above: then $HF^*(X \subset M; \Hlm)$ is independent of \begin{itemize} \item the choice of $\Hlm$ satisfying \eqref{eq:linearity} with fixed slope $\lambda_\ell$ and $K_\ell$ as well as the choice of (generic) $J_F$
    in $\mathcal{J}_{F,\ell}(V)$ (with respect to a fixed contact shell). \item Moreover, it is independent of the choice of $K_\ell$ arising in the definition of $\Hlm$ and the choice of contact shell $V_{\ell} \setminus V_{\ell, 0}$ arising in the definition of $\mathcal{J}_{F, \ell}(V).$
\end{itemize}
\end{lem}
\begin{proof}
The first assertion is a consequence of standard arguments. The second is too,
but we give a brief discussion: first, the maximum principle implies there is a
bijection of chain complexes if we shrink the shell region along which
$J_{F,\ell}$ is contact type to $V_{\ell} \setminus \tilde{V}_{\ell, 0}$
(the point being that Floer cylinders for any such $J$ cannot even cross $\partial
V_{\ell}$). Next, given two different shells $ S = V_{\ell} \setminus V_{\ell, 0}$
and $T = V_{\ell}' \setminus V_{\ell, 0}'$, if $V_{\ell} = V_{\ell}'$, the
shrinking argument implies we are done; otherwise, we can shrink $V_{\ell, 0}$,
$V_{\ell, 0}'$ so that the two shell regions are disjoint, without loss of
generality say $V_{\ell, 0}' \subset V_{\ell}' \subset V_{\ell, 0} \subset
V_{\ell}$. Now, starting with a $J$ which is contact type for $S$, we can
simultaneously make it contact type for $T$ without changing the Floer complex
at all (as Floer cylinders don't cross $S$). Finally, we can turn off the
contact type condition on $S$, which could possibly change the Floer complex on the
chain level, but does not change the cohomology, by the first assertion in the Lemma.
\end{proof}

For any $w$, let $\FwCF$ denote the sub-$\K$-module generated by those orbits
$\mathcal{X}(X; \Hlm)_{\leq w}$ with $w(x_0) \leq w$: 
\begin{align}
    \label{filteredsubcomplex}
    \FwCF:= \bigoplus_{x \in \mathcal{X}(X; \Hlm)_{\leq w}} |\mathfrak{o}_x| 
\end{align} 
It follows from Equation \eqref{eq: action} and the well known fact that $\partial_{CF}$ strictly increases 
action that 
\begin{lem}\label{lem:filteredsubcomplex}
    If $\epsilon_\ell$ is sufficiently small, $\Sigma_\ell$ is sufficiently $C^0$ close to
    $\hatXl$, and if $C^2$ small perturbations are used when defining $\Hlm$, then the differential $\partial_{CF}$ preserves the submodule $\FwCF$.  In particular, \eqref{filteredsubcomplex} is filtration of chain complexes.\qed
\end{lem}
We let $\FwHF$ denote the filtration on $HF^*(X\subset M; \Hlm)$ induced by
the cochain level filtration $\FwCF$. Throughout the rest of the paper,
\begin{equation}
    \textrm{Choose $\epsilon_\ell$, $\Sigma_\ell$, and $\Hlm$ so that Lemma \ref{lem:filteredsubcomplex} holds for each $\ell \in \mathbb{N}$.}
\end{equation}
The $\epsilon_\ell$ guaranteed by the proof of Lemma \ref{lem:filteredsubcomplex} tend to zero as $\ell \to \infty$
and the $\bar{X}_\ell$ form an exhaustive family of domains.
We next turn to defining continuation maps: 
\begin{equation}\label{eq:continuation}
    \mathfrak{c}_{\ell_1,\ell_2}: HF^*(X \subset M; \Hm^{\ell_1})  \to HF^*(X
    \subset M ; \Hm^{\ell_2}) 
\end{equation}
(note we will also use $\mathfrak{c}_{\ell_1, \ell_2}$ to refer to the
chain-level maps). 
Typically, continuation maps between Floer complexes
associated to Hamiltonians $H_a$ and $H_b$ are given by counting solutions to
Floer's equation with respect to a domain dependent Hamiltonian (and complex
structure) varying between $H_a$ and $H_b$ (as well as the respective complex
structures). 
Of course, in this non-compact setting, some care must be taken to ensure that
solutions are $C^0$ bounded away from $\D$, and hence that the requisite Gromov
compactness results hold. The present situation requires slightly subtler arguments than usual (in which one notes that a maximum principle holds
provided the surface dependent Hamiltonian is monotone), in light
of the fact that even when the slope of $H_a$ is less than the slope of $H_b$,
it may not be possible to ensure monotonicity of the interpolation. To
constrast the situation above with the usual
construction of symplectic cohomology, or rather its variant given in
\cite{GP1}, recall that in {\em loc. cit.} we fixed once and for all a single
Liouville domain $\overline{X}_{\vec{\epsilon}_{0}} \subset X$ (which
determined a single function $R: M \backslash X^o_{\vec{\epsilon}_0} \to \R$),
and considered
\begin{itemize}
    \item functions $G^\ell$ which are $C^2$ small perturbations 
         of the functions $g^{\ell}$ which are linear adapted to $R$ of slope $\lambda_{\ell}$, and moreoever, say, equal to 
$\lambda_\ell h^{1}(\rm)$ for a fixed $h^1(\rm)$ linear adapted to $R$ of slope 1; and

    \item (sufficiently generic) $S^1$-dependent almost complex structures $J_t \in \mathcal{J}(M,\D)$ which are contact type with respect to the function $R$ on a fixed shell $V \backslash V_0$, where $V = R^{-1}(K, \infty)$ and $V_0 = R^{-1}(\mu, \infty)$ for some $\mu$.
\end{itemize}
This data defines a chain complex $CF^*(X \subset M; G^{\ell})$ as above (and note
that $R$ and $J$ do {\em not} depend on $\ell$).
Continuation maps for any $\ell_1 \leq \ell_2$ were constructed by counting
solutions to Floer's equation with respect to a {\em monotone} homotopy of
Hamiltonians $G_{s,t}$ between $G^{\ell_1}$ and $G^{\ell_2}$ (i.e. a family of functions satisfying
$\partial_sG_{s,t} \leq 0$) which take on the standard form above $R=K$, 
\begin{align}\label{eq: stndform} 
     G_{s,t}= \lambda_s(\rm-1),
\end{align} 
for $\lambda_s$ a monotone homotopy between $\lambda_{\ell_{1}}$ and
$\lambda_{\ell_{2}}$, as well as generic families of almost complex structures
satisfying the same conditions as above. 
In light of the standard form \eqref{eq: stndform} and the contact-type condition of the almost
complex structures chosen, and monotonicity, solutions of the continuation map equation satisfy a maximum
principle (see e.g., \cite[Lem. 7.2]{Abouzaid:2010ly}), implying by the usual
analysis that the counts of such solutions give a chain map.
Using this system of maps, we set
\[
SH^*(\overline{X}_{\vec{\epsilon}}):=\varinjlim_{\ell} HF^*(X \subset M; G^{\ell}).
\]

The situation we will need to consider in the present paper somewhat more
delicate because each $H^{\ell}$ is constructed using a different
$\bar{X}_{\ell}$ (and hence different radial function $R^{\ell}$) for a
sequence of exhaustive domains $\bar{X}_\ell$.
Hence the standard form \eqref{eq: stndform} for a single function $R$ (on a
region where $J$ is also contact type for the same $R$) may be impossible to
arrange; i.e., it may not be possible to construct strictly monotone homotopies
for pairs $\ell_1 < \ell_2$.  Nevertheless, we can use homotopies which are
``monotone up to a sufficiently small error", as we now describe. We will
forget for a moment about our family of $H^{\ell}$ chosen and give slightly
more general criteria under which a continuation map exists.

To do so, let $\bar{X}_{a}$ and $\bar{X}_{b}$ denote any pair of Liouville domains in $X$
constructed as in the previous section using parameters $\vec{\epsilon}_a$,
$\vec{\epsilon}_b$; and let $R^a$ and $R^b$ be respective induced Liouville
coordinates as in \eqref{eq:inducedliouvillecoordinate}. For simplicity, we assume that $\partial \bar{X}_a$ and $\partial \bar{X}_b$ are disjoint,
so that one domain strictly contains the other. We let $R^{\operatorname{out}}, K^{\operatorname{out}}$ denote the Liouville coordinate and  $K_{\vec{\epsilon}}$ constant corresponding to the bigger domain.

Pick Hamiltonians $h^a$, $h^b$ which are linear adapted to $R^a$
respectively $R^b$ of (generic) slope $\lambda^a$ respectively $\lambda^b$,
with linearity levels $K^a$ and $K^b$ respectively, and 
denote by $H^a$ and $H^b$ $C^2$-small perturbations of these respective
Hamiltonians so that all orbits are non-degenerate, and so that the perturbation is
trivial on certain contact shells $V_{a}\setminus V_{a,0}$ and $V_{b}\setminus
V_{b,0}$. The discussion so far gives, for generic time-dependent almost
complex structures $J_t^a, J_t^b \in \mathcal{J}(M, \D)$ (which are contact
type on the respective shell-regions) Floer complexes $CF^*(X \subset M, H^a)$
and $CF^*(X \subset M, H^b)$ respectively. Using the analysis in Lemma
\ref{indJshellchange}, it is safe to assume that \begin{itemize} \item the $K_{\vec{\epsilon}}$ (hence $V_{\vec{\epsilon}}$) and $V_{0,\vec{\epsilon}}$ on the inner domain have been shrunk so that the contact shell for the inner domain
is disjoint from the contact shell for the outer domain. \end{itemize} 

Now, let $\rho(s)$ be a non-negative, monotone non-increasing cutoff function such that 
\begin{equation} \label{eq:cuttoff}
    \rho(s) = \begin{cases} 0 & s \gg 0 \\ 1 & s \ll 0 \end{cases}
\end{equation} 
Set
\begin{align} 
    \label{eq: concretehomotopy} H_{s,t}= (1-\rho(s))\Hm^{a} +\rho(s)\Hm^{b} 
\end{align} 
Up to a small perturbation near $\D$, we claim that  
\begin{align}\label{eq:hcond1} 
    \theta(X_{H_{s,t}})&=H_{s,t}+\lambda_s, \\
    \label{eq:hcond2}  d\rm^{\out}(X_{H_{s,t}}) &= 0 
\end{align} 
whenever 
$\rm^{\out} \geq K^{\out}$, where 
\begin{equation}\label{lambdas}
    \lambda_s := (1-\rho(s)) \lambda_{a} + \rho(s) \lambda_{b}
\end{equation}
is monotone if $\lambda_b \geq \lambda_a$.
The first equation \eqref{eq:hcond1} is a basic consequence of linearity
\eqref{eq:linearity} of both Hamiltonians and the fact that for any radial
coordinate $R$ as above, $\theta(X_{H(R)}) = \omega(X_H, Z) = dH(Z) = RH'(R)$.
The second equation \eqref{eq:hcond2} follows from the fact that in each $U_I$ above $R^{\out} = K^{\out}$, $H_{s,t}$ only
depends on the radial coordiantes $\rho_i$ for $i \in I$, given that it is a linear function of $R^{a}$ and
$R^{b}$ in this region, hence $df(X_{H_{s,t}}) = 0$ is zero for any function
$f$ of $\{\rho_i\}_{i \in I}$
(see Lemma \ref{nicehamiltonianvectorfields} item (4)).

Let $J_{s,t}$ be a generic compatible $\R \times S^1$ dependent almost complex
structure which \begin{itemize} \item is of contact type on the outer contact shell for all $s,t.$ \item agrees at $s=\pm \infty$ with the choices of $J^a_t$ and $J^b_t$ defined earlier. \end{itemize}
If $x_1$ is an orbit of $\Hm^{a}$ and $x_2$ is an orbit of
$\Hm^{b}$, let $\mathcal{M}_s(x_2,x_1)$ denote the moduli space of maps
$u: \R \times S^1 \ra X$ satisfying Floer's equation for 
$H_{s,t}$ and $J_{s,t}$:
\[
    \partial_s u + J_{s,t} (\partial_t u -  X_{H_{s,t}}) = 0
\]
which in addition satisfy requisite asymptotics:
\begin{align} \label{eq: continuationmaps}
\left\{
\begin{aligned}
 & \lim_{s \to -\infty} u(s, -) = x_2\\
 & \lim_{s \to +\infty} u(s, -) = x_1
\end{aligned}
\right.
\end{align}  
As usual, one defines the $|\mathfrak{o}_{x_1}|\! -\! |\mathfrak{o}_{x_2}|$ component
of the continuation map \eqref{eq:continuation} by counting rigid elements $u
\in \mathcal{M}_s(x_2,x_1)$ (for a suitably generic $J_{s,t}$).

 To establish the necessary estimates, note that up to arbitrarily small error, 
we have that for any map $u: \mathbb{R}\times S^1 \to M$ and any closed $\bar{S} \subset
\mathbb{R}\times S^1$, 
\begin{align}\label{eq: errorterm} 
    \int_{\bar{S}} u^*(\partial_s H_{s,t}) dsdt \leq \operatorname{sup} (\Hm^{a}-\Hm^{b}) < \operatorname{sup} (\Hm^{a}).
\end{align}  We now establish the necessary Gromov compactness result for these continuation solutions:

 \begin{lem} 
     \label{lem: homotopycomp} 
     Let $H_{s,t}$ be as above, let $\deg(x_2)-\deg(x_1) \leq 1$, and suppose that either
     \begin{itemize}
         \item[(a)] (strict monotonicity at $\infty$) $\lambda_b \gg \lambda_a$ (meaning specifically that $\lambda_b > \lambda_a\frac{R^a-1}{R^b-1}$) on the region $R^{\out} \geq K^{\out}$ and the action of $x_1$ with respect to $A_{a} :=
             A_{H^a}$ (recall  \eqref{eq:action}) satisfies  \begin{align} -A_a(x_1) < \lambda_a. \end{align}  
 
         \item[(b)] (monotonicity of slopes with bounded error) $\lambda_b \geq
             \lambda_a$ and
             \begin{align} \label{actionx1estimate}
                 -A_{a}(x_1)+ \operatorname{sup} (\Hm^{a})<\lambda_{a}.
         \end{align} 
     \end{itemize}
 Let $ \overline{\mathcal{M}}_s(X; x_2,x_1)$ denote the Gromov compactification
 of $\mathcal{M}_s(x_2,x_1)$ in $M$. Then, \begin{itemize} \item if $\deg(x_2)-\deg(x_1)=0$, $\overline{\mathcal{M}}_s(X; x_2,x_1)=\mathcal{M}_s(X;x_2,x_1)$ i.e. the moduli space is compact. \item if  $\deg(x_2)-\deg(x_1)=1$, $\partial
 \overline{\mathcal{M}}_s(X;x_2,x_1)= \partial_1
 \overline{\mathcal{M}}_s(x_2,x_1) \cup \partial_2
 \overline{\mathcal{M}}_s(x_2,x_1)$ where 
 \begin{align} \partial_1 \overline{\mathcal{M}}_s(x_2,x_1)= \bigsqcup_{y, \deg(y)-\deg(x_1)=1}
     \mc{M}(x_2,y) \times \mc{M}_s(y, x_1) \\
     \partial_2\overline{\mathcal{M}}_s(x_2,x_1) = \bigsqcup_{y, \deg(y)-\deg(x_1)=0}
     \mc{M}(x_2,y) \times \mc{M}_s(y, x_1) 
 \end{align}  
\end{itemize}
 \end{lem} 

 \begin{proof}
Gromov compactness in $M$ implies that the Gromov-Floer compactification of
these moduli spaces as maps to $M$ are compact. One needs to show therefore that
this compactification only contains broken solutions in $X$ (not intersecting
$\D$), i.e., elements of these moduli spaces do not limit to broken solutions
intersecting $\D$  (After this, standard transversality and gluing arguments
imply the desired result). 

The argument, like several others in this paper, follows the pattern of
\cite[Lemma 4.13]{GP1}. Namely, we first make the key claim that 
\begin{itemize}
    \item[(i)]a broken Floer solution
cannot break along an orbit in $\mathcal{X}(\D,\Hlm)$. 
\end{itemize}
If this is true, it follows
that any broken Floer curve $\tilde{u}$ in $M$ in the limit of trajectories above has (all
asymptotics in $X$ and hence) a well-defined total topological intersection
number with $\D$, equal to $0$ (the original intersection number) and additive
over its components, and positive over all components not completely contained in $\D$ (which includes all Floer trajectories), see e.g. Lemma 4.13 of
 \emph{loc. cit.}. 
If there were holomorphic sphere bubbles in $\tilde{u}$, the disjoint union of
all of these bubbles necessarily give an $H_2$ class which has positive
symplectic area hence positive intersection with some $D_i$, implying 
the remaining broken Floer trajectory must have negative intersection number
with $D_i$, a contradiction. So,
\begin{itemize}
    \item[(ii)] No holomorphic sphere bubbles can form in the limit of broken Floer trajectories.
\end{itemize}
Finally, the remaining broken Floer trajectories satisfy positivity of
intersection with $\D$, hence (since the total intersection number is 0) do not
intersect $\D$ as desired.

 So it remains to show (i). For example suppose that there is 
 $y \in \mathcal{X}(\D, \Hm^{a})$
 together with a sequence in
 $\overline{\mathcal{M}}_s(X; x_2,x_1)$ limiting to a configuration $u_1 \in
 \mathcal{M}(y,x_1)$ and $u_2 \in \overline{\mathcal{M}}_s(x_2,y)$. Consider
 the piece of the curve 
 $\bar{S}:=u_2^{-1}((R^{\out})^{-1}[K^{\out},\infty))$ 
    which lives above the slice where $R^{\out} = K^{\out}$ (note in this region, we have both
    $R^{a} \geq K^{a}$ and $R^{b} \geq K^{b}$ and both of our functions are linear with respect to their respective coordinates).
    Then, if $\tilde{\bar{S}}$ denotes the union of $\bar{S}$ with the domain
    of $u_1$, the geometric energy of $\bar{S}$ (see \cite{GP1} or the discussion around
    \eqref{eq:geo} below 
    for a review of this concept) satisfies the following inequality
 \begin{align} 
      \label{eq:homeest} E_{geo}(u_2|_{\bar{S}}) &\leq E_{geo}(u|_{\tilde{\bar{S}}}) \\
     &\leq E_{top}(u|_{\tilde{\bar{S}}}) + \int_{\tilde{\bar{S}}} u^* (\partial_s H_{s,t}) ds dt \\ 
     \label{eq:stokes}& = -A_{a}(x_1) + \int_{\partial \bar{S}}u^*\theta- H_{s,t}dt + \int_{\bar{S}} u^* (\partial_s H_{s,t}) ds dt \\
     \label{eq:twocases}&\leq -A_{a}(x_1)+ \{0,\operatorname{sup} (\Hm^{a})\} +\int_{\partial \bar{S}}u^*\theta- H_{s,t}dt \\ 
     \label{eq:finalest216} &= -A_{a}(x_1)+ \{0,\operatorname{sup} (\Hm^{a})\} + \int_{\partial \bar{S}}u^*\theta-\theta(X_{H_{s,t}})dt + \int_{\partial \bar{S}}\lambda_s dt 
 \end{align} 
 where $E_{top}$ denotes the {\em topological energy} of a map, defined in
 \cite{GP1} or \eqref{eq:top} below, \eqref{eq:stokes} follows from Stokes' theorem\footnote{Let $\underline{S}:= u_2 \setminus \bar{S}$. By Stokes, $E_{top}(\underline{S})=A_a(x_2)-(\int_{\partial \bar{S}}u^*\theta- H_{s,t}dt)$. As the topological energy of $u_1 \cup u_2$ is $A_{a}(x_2)-A_{a}(x_1)$, the equation holds.}, and the terminology $\{0,\operatorname{sup}
 (\Hm^{a})\}$ in \eqref{eq:twocases} means one should use $0$ in case (a) of
 the Lemma (by strict monotonicity of $H_{s,t}$ in the region $R^{\out} \geq K^{\out}$ in
 this case) and $\operatorname{sup} \Hm^a$ in case (b) by \eqref{eq:
 errorterm}.  Going forward we will just assume that term is
 $\operatorname{sup} \Hm^a$, as that case is strictly more difficult.

By Stokes' theorem we have that 
\begin{align} 
    \int_{\partial \bar{S}}\lambda_s dt= -\int_{y}\lambda_sdt + \int_{\bar{S}} d(\lambda_s dt) \leq -\lambda_{a}  
\end{align} 
where the last inequality used $\lambda_a \leq \lambda_b$ and monotonicity of $\lambda_s$, as defined in \eqref{lambdas}.
Therefore it follows from \eqref{actionx1estimate} 
that 
\begin{align}\label{eq:keyinequality}  
    -A_{a}(x_1)+ \operatorname{sup} (\Hm^{a}) + \int_{\partial \bar{S}}u^*\theta-\theta(X_{H_{s,t}})dt + \int_{\partial \bar{S}}\lambda_s dt \leq \int_{\partial \bar{S}}u^*\theta-\theta(X_{H_{s,t}})dt 
\end{align} 
The rest proceeds as in the proof of Lemma 4.13 of \cite{GP1} or \cite[Lem.
7.2]{Abouzaid:2010ly}, where it is shown that, under the hypotheses of $J$ being contact type along $R^{\out} = K^{\out}$, 
this last expression is non-positive and so $u|_{\bar{S}}$ must have 0 energy
and hence be constant, a contradiction.  
\end{proof} 

Returning to our system of Hamiltonians (and Liouville domains, etc.)
$H^{\ell}$, \eqref{eq: action} shows that by taking $\epsilon_\ell$
sufficiently small (and $\Sigma_\ell$ taken sufficiently $C^0$ close to
$\hatXl$), we may also assume that 
\begin{align} 
    \mathcal{F}_\v \in
    \mathcal{X}(X;\hlm) \ \text{whenever} \ w(\v) \leq  w_\ell; \textrm{ and} \\
    \lambda_{\ell}(1-\epsilon_{\ell}^2)> w(\v_\ell)(1-\epsilon_\ell^2/2)^2.
\end{align} 
We do this throughout the rest of the paper. In view of the
estimate \eqref{eq:estimateD}, Lemma \ref{lem: homotopycomp}
and the above inequality then implies that \eqref{actionx1estimate} holds for
all orbits
$x_1$, 
hence implies the existence of continuation maps
$\mathfrak{c}_{\ell_1,\ell_2}$ between our Hamiltonians $H^{\ell_1}$ and
$H^{\ell_2}$ as desired.

A standard elaboration of the above argument then shows that continuation maps
compose as expected (homologically). Also, any such continuation map from a
chain complex to itself is homologically the identity. In particular
\begin{cor}
When $\epsilon_\ell$ is sufficiently small (and as usual
$\Sigma_\ell$ is sufficiently $C^0$ close to $\hatXl$),  $HF^*(X\subset M;
\Hlm)$ is independent of $\epsilon_\ell$. \qed
\end{cor}

Define 
\begin{align}\label{adaptedSHdefinition}
    SH^*(X):= \varinjlim_{\ell} HF^*(X \subset M; \Hlm) 
\end{align}

Using continuation maps whose existence is guaranteed by Lemma \ref{lem: homotopycomp},
we also deduce that
\begin{lem}
There is a natural isomorphism \begin{align}
SH^*(X) \to SH^*(\bar{X}_{\vec{\epsilon}}).
\end{align}
\end{lem}
\begin{proof}
At the expense of possibly increasing the slope, Lemma \ref{lem: homotopycomp} says one can construct a continuation map $HF^*(X \subset M;
H^{\ell})$ to some $HF^*(X \subset M; G^{\ell+N})$ for each $\ell$ using the homotopies from \eqref{eq: concretehomotopy} (which are monotone, in
the sense of satisfying the hypotheses of Lemma \ref{lem: homotopycomp}, for $N$ large), compatibly with maps in both systems, getting a map on direct limits. One can also go the other way by the same argument, and naturality properties of continuation maps imply the composition in either direction is the identity (on the direct limits).
\end{proof}

 \begin{cor}
     $SH^*(X)$ as defined in \eqref{adaptedSHdefinition} coincides 
      with the usual symplectic cohomology of $X$, as
     defined by taking $SH^*(\bar{X}_{\vec{\epsilon}})$ for any
     $\bar{X}_{\vec{\epsilon}}$ in the sense of \cite{Seidel:2010fk}. \qed
 \end{cor}
In view of Equation \eqref{eq: errorterm}, we have that after
 possibly shrinking $\epsilon_\ell$ further, the continuation maps can be made
 to induce maps of filtered subcomplexes (see \eqref{filteredsubcomplex})
 \begin{align} 
     \mathfrak{c}_{\ell_1,\ell_2}: F_wCF^{*}(X \subset
     M; \Hm^{\ell_1}) \to F_wCF^{*}(X \subset M; \Hm^{\ell_2}) 
 \end{align}
This enables us to equip $SH^*(X)$ with a filtration $F_wSH^*(X)$. 
\begin{rem}\label{rem:inverselimit}
As remarked earlier the filtration $F_wSH^*(X)$ is a limit of the natural
action filtrations on the various $SH^*(\bar{X}_{\ell})$ where $\bar{X}_{\ell}$
is our sequence of exhausting domains. A variant of our construction would be
to define 
$$ SH^*(X) := \varprojlim_{\ell} SH^*(\bar{X}_{\ell}):=\varprojlim_{\ell}
\varinjlim_{\lambda} HF^*(X \subset M; G^{\lambda,\ell}) $$ 
where $G^{\lambda,\ell}$ denote
Hamiltonians which agree with $\lambda(\rlm-1)$ when $\rlm \geq
K^{\ell}$ and the inverse limit is formed using monotone continuation maps in
\eqref{eq: concretehomotopy}.

Of course all of the maps in the inverse limit are isomorphisms. However the
natural action filtration in the inverse limit is given by $F_w$. The
equivalence between our definition and this one amounts to the fact that in the
present setting one may commute the limit and the colimit. 

A third natural possibility would be to take the inverse limit with respect to the system of Viterbo
functoriality maps $SH^*(\bar{X}_{\ell_2}) \to  SH^*(\bar{X}_{\ell_1})$ (for
$\ell_2 > \ell_1$) as in \cite[Eq. (7.2)]{Seidel:2010fk}. Doing this would
require checking that our $\PSSlog$ maps are compatible with Viterbo's
construction, and involve further technical detours.
\end{rem} 

Our final task is to put a product structure on $SH^*(X)$. We again describe
this first for the directed system $\Glm$. Namely we may define a product
operation on $SH^*(\overline{X}_{\vec{\epsilon}})$ by considering a variant of
Floer's equation defined over a pair of pants $\Sigma$ as in Figure \ref{pairofpants} equipped with standard
cylindrical ends $\epsilon_i$.  
\begin{figure}[h] 
    \caption{\label{pairofpants}}
    \centering
    \includegraphics[scale=0.7]{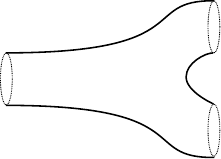}
\end{figure}
For any $\ell_1$ and $\ell_2$ we choose $\ell_3$
such that $$\lambda_{\ell_{3}} \geq \lambda_{\ell_{1}}+ \lambda_{\ell_{2}} $$
To each cylindrical end associate a time dependent Hamiltonian $H_i$. Let $\mathcal{K}
\in \Omega^1(\Sigma,C^{\infty}(X))$ be a 1-form on $\Sigma$ which along the
cylindrical ends, satisfies:  $$\epsilon_i^*(\mathcal{K})=(\Gm^{\ell_i}) \otimes dt$$
whenever $|s|$ is large. We also require that outside of a compact set in $X$
we have that 
$$K= h^{1} \otimes \beta+\mathcal{K}_{pert}$$ 
for some subclosed $\beta$
and $\mathcal{K}_{pert}$ is supported near the divisor $\D$ and along the cylindrical
ends.  To such a $\mathcal{K}$, we may associate a Hamiltonian one form $X_\mathcal{K} \in
\Omega^1(\Sigma,C^{\infty}(TX))$ which is characterized by the property that
for any tangent vector at a point $z \in \Sigma$, $\vec{r}_z$, we have that
$X_{\mathcal{K}}(\vec{r}_z)$ is the Hamiltonian vector-field of $\mathcal{K}(\vec{r}_z)$. Over
$\Sigma$, we consider a generalized form of  Floer's equation: 
\begin{align}
    \label{eq:generalFloer} \left\{ 
        \begin{aligned} & u \colon \Sigma \to X, \\
            & (d u - X_{\mathcal{K}})^{0,1} = 0.  
    \end{aligned} \right.
\end{align} 

To such data we can associate the {\em geometric energy}
\begin{equation}\label{eq:geo}
    \Egeo(u) :=  \frac{1}{2}\int_{\Sigma}|| du - X_{\mathcal{K}}|| 
 \end{equation} 
as well as the {\em topological energy}
\begin{equation} \label{eq:top}
    \Etop(u) = \int_{\Sigma} u^* \omega - d(u^*\mathcal{K}).
\end{equation}

We say that a perturbation $\mathcal{K}$ is monotonic if its curvature (Equation 8.12 of \cite{Seidel_PL}) is nonnegative (standard examples are perturbations of the form $H(R^{\ell})\otimes \beta$ with $H \geq 0$ and $\beta$ subclosed). When $\mathcal{K}$ is monotonic, we have an energy inequality 
\begin{equation} 
    \label{eq:basicenergyineq}  \Egeo(u) \leq \Etop(u).
\end{equation}

As usual, we assume that when $\rm \geq K$, $\mathcal{K}$ is monotonic and satisfies a
suitable variant of Equation \eqref{eq: stndform} near $R_H$. By counting
solutions to this equation we may define an associative and commutative product
\begin{equation}
    SH^*(\overline{X}_{\vec{\epsilon}})\otimes SH^*(\overline{X}_{\vec{\epsilon}}) \to SH^*(\overline{X}_{\vec{\epsilon}})
\end{equation}

To define appropriate Floer data $X_\mathcal{K}$ for our system $\Hlm$, we take
$\overline{X}_{\vec{\epsilon}}= \bar{X}_{\ell_3}$ and assume that
$\Gm^{\ell_3}=\Hm^{\ell_3}$. Along the incoming cylindrical ends we glue in the
homotopies used to defined continuation maps from $\Hm^{\ell_i}$ to
$\Gm^{\ell_i}$.  It follows 
that the product operations so defined will respect the filtration $F_w$ from \eqref{filteredsubcomplex}.

\subsection{Action spectral sequences} \label{sect: actionspec}

We define  the low energy Floer cohomology of weight $w$, $\HFw$, by the formula
\begin{align} \label{eq:lowlow}
    \HFw :=H^*(\frac{F_wCF^*(X \subset M; \Hlm)}{F_{w-1}CF^*(X
    \subset M; \Hlm)}) 
\end{align} 

We consider the corresponding descending filtrations:\footnote{We do this so that our conventions for cohomological spectral sequences match those found in standard textbooks such as \cite{McCleary}.} $$F^{p}CF^*(X \subset M; \Hlm):= F_{-p}CF^*(X \subset M; \Hlm).$$ By definition, the filtration $F^p$ on the cochain complex gives rise to a spectral
sequence 
\begin{align}  
    \label{eq:Spec1} \lbrace E_{\ell,r}^{p,q}, d_r \rbrace
    \implies \HFXM  
\end{align} 
where  the first page is by definition
identified with 
\begin{align} 
    \bigoplus_q E_{\ell,1}^{p,q}: = HF^*(X \subset M;
    \Hlm)_{w=-p}   
\end{align} 

 As is customary, we set $E_{\ell,1}=\bigoplus_{p,q} E_{\ell,1}^{p,q}.$ We have
 seen that the continuation map can be made to respect the filtration by
 $w(\v)$ and thus induce a filtration on $SH^*(X)$. For the purposes of
 constructing a spectral sequence, it is convenient to use a (co)chain-level
 direct limit construction. We define  
 \begin{align} SC^*(X):=
     \bigoplus_{\ell}CF^*(X \subset M; \Hlm)[q] 
 \end{align} 
 where $q$ is
 a formal variable of degree $-1$ such that $q^2=0$.  For $a+qb \in CF^*(X
 \subset M; \Hlm)[q]$, the differential on this complex is given by the
 formula 
\begin{align}
    \partial(a+qb)=(-1)^{\deg(a)}\partial(a) +(-1)^{\deg(b)}(q\partial(b)+\mathfrak{c}_{\ell,\ell+1}(b)-b)  
\end{align} 
The fact that $\partial^2=0$ relies on the fact that the differential on each
complex $CF^*(X \subset M; \Hlm)$ satisfies $\partial^2=0$ and the fact that
the $\mathfrak{c}_{\ell,\ell+1}$ are chain maps. It is an algebraic consequence
of the definition that $SH^*(X) \cong H^*(SC^*(X))$. The benefit of working
with $SC^*(X)$ is that the filtrations by $w(\v)$ gives it the structure of a
filtered complex. We again consider the corresponding descending filtration $F^p SC^*(X)$, which are bounded above and exhaustive. As a result (see e.g. \cite{McCleary} Theorem 3.2), the descending filtration $F^pSC^*(X)$ gives rise to a convergent cohomological spectral sequence: 
\begin{align} \label{eq:Spec2} 
    \lbrace E^{p,q}_{r},d_r \rbrace \Rightarrow SH^*(X) 
\end{align}

Recall that we have chosen our data in such a way that the continuation maps  $\mathfrak{c}_{\ell_1,\ell_2}$ respect the filtrations and thus give rise to induced continuation maps: 
 \begin{align} \label{eq:lowencont} \mathfrak{c}_{\ell_1,\ell_2}:
     HF^{*}(X \subset M; \Hm^{\ell_1})_{w} \to HF^{*}(X \subset M;
     \Hm^{\ell_2})_{w}. 
 \end{align}
By construction, the $E_0$ page is again a cochain level direct limit of low energy Floer complexes. It follows that the $E_1$ page may be concretely described as:
\begin{align} \label{eq:E1concrete} \bigoplus_q E_{1}^{p,q}: = \varinjlim_\ell HF^*(X \subset M;
    \Hlm)_{w=-p}   
\end{align} 
where the maps in the direct limit are the maps \eqref{eq:lowencont}. 
The product operation on $SH^*(X)$ can similarly be lifted to a map of filtered
co-chain complexes 
\begin{align} \label{eq: chainlevelprod} SC^*(X)\otimes
    SC^*(X) \to SC^*(X).
\end{align} 
Setting $E_r := \bigoplus_{p,q} E^{p,q}_{r}$, the theory of spectral sequences
shows that this induces a (bi-graded) product operation 
$$E_r \otimes E_r \to E_r$$ 
which satisfies the Leibnitz rule with respect to the differential $d_r$.
Because the map \eqref{eq: chainlevelprod} is well known 
to be associative up to filtered homotopies, the induced
multiplications are associative for $r=1$ and consequently for all $r$.

\section{The low energy log PSS map} \label{section: PSSreview}
The first goal of this Section is to define the low energy log PSS map, $\PSS_{log}$ from log cohomology to the $E_1$ page  \eqref{eq:E1concrete} of the action spectral sequence and to prove that it is a ring homomorphism. To prepare for this, in \S \ref{section: blowups}, we recall the notion of the real-oriented blow ups, which will be useful in providing an elegant model for describing the restrictions between strata involved in the definition of the product on log cohomology.  We then introduce log cohomology in \S \ref{subsec:logcoh} and a Morse model for the product structure (Morse theory is a convenient model for carrying out cochain level constructions, but one could use other versions of cochains such as singular cochains or various flavors of geometric (co)chains as well). 

In \S \ref{section:lowenergydef}, we describe the construction of the low energy log PSS map, \eqref{eq:Speciso2}. The main result of \S \ref{section:rings} is that this map is a ring homomorphism(Theorem \ref{lem:spectralrings}). Finally, \S \ref{sect: PSSiso1}, shows that after perturbing our symplectic form and Hamiltonians slightly, we may further refine the $\PSS_{log}$ map to a map \eqref{eq:realPSSloc}  between log cohomology classes with multiplicity vector $\v$ and a Floer cohomology group generated by orbits that wind around the divisors $\D_i$ with multiplicity $\v$.

\subsection{The real blowup} \label{section: blowups}

This sub-section gives the definition of the real-oriented blow up $M^{log}$ of $M$ along the divisor $\D$. Recall the notation from the previous section $D_I$, $\DIo$, for the
stratified components of $\mathbf{D}$ and their open parts. Let $S_I$ denote
the $T^{|I|}$ torus bundle associated to $ND_I$, 
\begin{align} 
    S_I= (ND_I\setminus \cup_i D_i)/ (\mathbb{R}^{+})^{I},
\end{align} 
and set 
\begin{equation}\label{opentorusbundle}
    \SIo: = (S_I)|_{\DIo} 
\end{equation}
to be the torus bundle restricted to the open stratum $\DIo$, with $S_{\emptyset} = M$ and $\mathring{S}_{\emptyset} = X$. 
Note that all of these manifolds can all be oriented. For $S_I$, this comes
from the exact sequence at each tangent space 
\begin{align} 
    TT^{I} \to T_pS_I \to T_{\pi(p)} D_I. 
\end{align}    
Our convention is that each circle in the torus is oriented clockwise (this
convention is the opposite of the usual one) and that $T^{I}$ is oriented
lexicographically. 

We now form the real oriented blow up of $M$ along the divisor $\D$, 
\begin{equation}
    M^{log},
\end{equation}
which canonically realizes the torus bundles \eqref{opentorusbundle} as strata
of a space.
There are several
possible constructions for this; the most expedient for us is in terms of local
coordinate charts (see \cite{Melrose} for another possible construction based on
the tubular neighborhood theorem).  We first consider the linear/local
situation. Let $\mathbb{V}:= \mathbb{C}^{n}$ with coordinates $y_1 \cdots, y_n$
and let $\mathbf{H}$ be the union of first $\mathbf{a}$ coordinate hyperplanes
$H_1,\cdots, H_\mathbf{a}$ for some $n \geq \mathbf{a} \geq 0$. Define the real
oriented blow-up of $\mathbb{V}$ along $\mathbf{H}$ to be given by
$\mathbb{V}^{log}:= (\mathbb{R}^{\geq 0} \times S^1)^\mathbf{a} \times
\mathbb{C}^{n-\mathbf{a}}$. There is a canonical morphism $\mathbb{V}^{log} \to
\mathbb{V}$ given by sending $$((r_1,\theta_1),\cdots,
(r_\mathbf{a},\theta_\mathbf{a}),x) \to (r_1\theta_1,\cdots r_\mathbf{a}
\theta_\mathbf{a}, x).$$

To globalize this blowup construction, we need to show that local
diffeomorphisms of $\mathbb{V}$ can be lifted uniquely to the blowup. Given a
diffeomorphism $F: \mathbb{V} \to \mathbb{V}$ and any $x \in
\mathbb{C}^{n-\mathbf{a}}$, we let $F_{\mathbf{a},x}$ denote the induced map
\[F_{\mathbf{a},x}: \mathbb{C}^{\mathbf{a}} \stackrel{z \mapsto (z,x)}{\to} \C^n \stackrel{F}{\to} \C^n \stackrel{(y,w) \mapsto y}{\to} \mathbb{C}^{\mathbf{a}}\] 
and we
 also let $\pi_\mathbf{a}$ denote the projection $\pi_\mathbf{a}:
 \mathbb{V}^{log} \to (S^1)^{\mathbf{a}}.$ The key computation is then the
 following: 

\begin{lem} 
Given a diffeomorphism $F: \mathbb{V} \to \mathbb{V}$ which preserves the hyperplanes $H_1,\cdots H_\mathbf{a}$, there is a unique diffeomorphism $\tilde{F}: \mathbb{V}^{log} \to \mathbb{V}^{log}$ lifting $F$, i.e. such that we have a commutative diagram:   \[
\xymatrix{
   \mathbb{V}^{log} \ar[d] \ar[r]^{\tilde{F}} &  \mathbb{V}^{log} \ar[d] \\
  \mathbb{V}  \ar[r]^{F}  & \mathbb{V}
}
\] 
Restricted to the preimage of the locus where $y_1=y_2=\cdots y_\mathbf{a}=0$,
we have the following explicit formula for $\pi_\mathbf{a} \circ \tilde{F}$,
\begin{align} \label{eq: diffofF}  
    \pi_\mathbf{a} \circ \tilde{F}:
    ((0,\theta_1), (0,\theta_2), \cdots, (0,\theta_\mathbf{a}),x) \to
    [D_0F_{\mathbf{a},x}(\theta_1,\cdots,\theta_\mathbf{a})],
\end{align} where
$[D_0F_{\mathbf{a},x}(\theta_1,\cdots,\theta_\mathbf{a})]$ denotes the
equivalence class of the differential of $F_{\mathbf{a},x}$ at $y_1=y_2=\cdots
y_\mathbf{a}=0$ applied to (any positive real multiple of) $(\theta_1, \ldots,
\theta_\mathbf{a})$, modulo rescaling by $(\mathbb{R}^{+})^{\mathbf{a}}$.
\end{lem}

\begin{proof} As the oriented blow-up is an isomorphism away from the coordinate hyperplanes, the lifts $\tilde{F}$ are unique if they exist. The existence and explicit formula for the extension follows from the case of a single smooth divisor by taking fibre products. The calculation in the case of a single smooth divisor is given in Lemma 2.8 of \cite{Arone}, see also of \cite[\S 2.5]{KM}.
\end{proof} 

Returning to the global situation, we may cover $M$ by charts $W \subset \mathbb{V} \to M$ on which $\D \cap W$ is the locus where at least one of $y_1,\cdots,y_\mathbf{a}=0$ ($\mathbf{a}$ here is equal to the depth of the stratum of $\D$ at which our chart is centered). By uniqueness of lifts, we may glue together the local models $W^{log}$ to form a manifold-with-corners $M^{log}$ which admits a canonically defined map $M^{log} \to M$.

For every non-empty stratum of $D_I$ of the normal cossings divisor $\D$,  the preimage of $D_I$ under the blow-up map defines a closed a closed stratum of $M^{log}$, which is canonically isomorphic (by equation \eqref{eq: diffofF}) to the blow up
\begin{equation}\label{eq:sIlog}
    S_I^{log}
\end{equation}
of $S_I$ along the preimages of the strata $D_j \cap D_I$ for $j \notin I$. 
In particular, the fiber of the map $M^{log} \to M$ over
any point in $m \in \mathring{D}_I$ is a rank $|I|$ torus, $T^{I}$. Away from $\D$, there is an open inclusion $X \to M^{log}$ which is easily seen to induce a canonical homotopy equivalence;
similarly, the open inclusions $\mathring{S}_I \to S_I^{log}$ are homotopy equivalences.

\begin{rem}\label{rem:kato} 
     Algebro-geometrically, the blow-up can be constructed as the closure of
     the graph $X \to \oplus_i S_i$ given by the defining sections $s_i$. More
     intrinsically, the real oriented blow-up may be viewed as a special case
     of the Kato-Nakayama construction in logarithmic geometry \cite{KN}. 
 \end{rem}

In what follows, we will identify $H^*(M^{log})$ and $H^*(X)$
without mention and similarly for the torus bundles over lower dimensional
strata.  Note that if $I \subset K$, then $S_K^{log} \subset S_I^{log}$,
inducing a restriction map on cohomology
\begin{equation}\label{stratumrestriction}
    r_{IK}^*: H^*(\mathring{S}_I) \ra H^*(\mathring{S}_K),
\end{equation}
which will factor into our definition of product structures on log cohomology below.
\begin{rem} \label{rem: diffeoblow}
    Up to non-canonical diffeomorphism, the real blowup $M^{log}$ is
    diffeomorphic to the complement $M \backslash \cup_{i} U_i$ of choices of
    disc-bundle tubular neighborhoods for each $D_i$ chosen previously. In such a model, the restriction maps
    \eqref{stratumrestriction} can be explained by observing first that 
    $\mathring{S}_I = \partial \mathring{U}_I \subset X$ for any $I$
    and more generally for $I \subset K$, there are inclusions $\mathring{S}_K
    \subset \mathring{S}_I$ living over the inclusion $\partial
    (\mathrm{nhood}(D_K \cap D_I)) \subset \mathring{D}_I$ (where these
    inclusions are implicitly using the tubular neighborhood identifications).
\end{rem}

\subsection{Logarithmic cohomology}\label{subsec:logcoh}

We now turn to recalling the {\em log(arithmic) cohomology ring} of $(M,\D)$, 
\begin{equation}\label{eq:logcoh}
    \QH^*(M,\mathbf{D}),
\end{equation}
which was defined additively in \cite{GP1}.
To define it, we use standard multi-index
notation, i.e., we have fixed formal variables $t_1, \ldots, t_k$, and
for any vector $\v = (v_1, \ldots, v_k) \in \Z_{\geq 0}^k$,
\[
t^{\v} := t_1^{v_1} \cdots t_k^{v_k}.
\]
Next, denote by $\v_I$ the multiplicity vector $(v_1, \ldots, v_k)$ whose components are non-zero precisely when
they are in $I$, in which case they are 1:
\begin{equation}\label{eq:primitivevI}
    \v_I := (v_1, \ldots, v_k)\textrm{ where } v_i := \begin{cases} 1 & i \in I \\ 0 & \textrm{otherwise}. \end{cases}
\end{equation}
In the case that $I$ consists of a single element $\lbrace i \rbrace$, we will
use the notation $\v_i:=\v_{\{i\}}$.  We refer to the vectors $\v_I$ as {\em
primitive vectors}. In terms of the primitive vectors $\v_I$, log cohomology
can be described as follows:
\begin{equation}
    \QH^*(M, \D) := \bigoplus_{I \subset \{1, \ldots, k\}}  t^{\v_I} H^*(\SIo)[t_i\ |\ i \in I]
\end{equation}
where $S_\emptyset = X$, and $S_I = \emptyset$ if the intersection $\cap_{i \in
I} D_i$ is empty.

$\QH^*(M,\D)$ is  generated as a $\K$-module by elements of the form
$\alpha t^{\v}$, where $\v \in \Z^k_{\geq 0}$ is a multiplicity vector, $I = \{i \in
\{1, \ldots, k\} | \v_i \neq 0\}$ denotes the indices of $\v$ which are
non-zero, and $\alpha$ is an element of the $\K$-module $H^*(\SIo)$.

These groups also come equipped with natural filtrations: the {\em logarithmic cohomology} of $(M,
\mathbf{D})$ {\em of slope $< w$}, denoted \begin{equation} \label{eq:filterlogchainys} F_w \QH^*(M,\D), \end{equation} is the sub
$\K$-module generated by those elements of the form $\alpha t^{\vec{\v}}$ for
some subset $I \subset \{1, \ldots, k\}$, such that 
\[
    w(\v) < w
\]
(recall the definition of $w(\v)$ in \eqref{eq:windingvec}).
By definition, whenever $w_{1} \leq w_{2}$ there is an inclusion
\[
i_{w_{1}, w_{2}}: F_{w_{1}} \QH^*(M,\D) \to F_{w_{2}} \QH^*(M,\D).
\]
and hence $F_w \QH^*(M,\D)$ defines a canonically split ascending filtration. Finally, we also
define the associated graded with respect to this filtration 
\begin{equation} \label{associatedgraded}
    \QH^*(M,\D)_{w}:=\frac{F_{w}\QH^*(M,\D)}{F_{w-1} \QH^*(M,\D)}
\end{equation}
and the multiplicity $\v$ submodule of $\QH^*(M,\D)$
\begin{equation} \label{eq:anothereqlogcohv}
    \QH^*(M,\D)_{\v}:= H^*(\SIo)t^{\v} \textrm{ where $I = \{i | v_i \neq 0\}$}.
\end{equation}
We say a multiplicity vector $\v$ is {\em supported on $I$} if $I = \{i | v_i
\neq 0\}$.

Given a choice of holomorphic volume form $\Omega_{M, \mathbf{D}}$ on $M$ which is non-vanishing
on $X$ with poles of order $a_i$ along $D_i$, i.e., a fixed isomorphism
\begin{align} 
 \wedge^n_{\C} T^*M & 
 \stackrel{\cong}{\to} \mathcal{O}(\sum_i -a_iD_i), 
    \label{eq:volform} 
\end{align}  
the vector space $H^*_{log}(M,\D)$ inherits a cohomological grading given by 
\begin{equation} 
     \label{eq:loggrading}
     \deg (\alpha t^{\v}) =\deg(\alpha)+2\sum_{i=1}^k (1-a_i) v_i.
\end{equation}

For  $\alpha \in H^*(\mathring{S}_I)$ and $ \beta \in H^*(\mathring{S}_J)$, let $K = I \cup J$ and define
\begin{align} 
    \label{eq:convol} \alpha \star \beta := r_{IK}^*\alpha \cup r_{JK}^*\beta \in H^*(\mathring{S}_K),
\end{align} 
where $r_{IK}$ and $r_{JK}$ are as in \eqref{stratumrestriction} (note that if
$I = J$, this is just the usual cup product in $H^*(\mathring{S}_I)$).  Using
\eqref{eq:convol}, we observe that there is a natural convolution product
on $H^*_{log}(M,\D)$:
\begin{defn} 
    \label{defn: ringstructure} 
    The ring structure on $H^*_{log}(M,\D)$ is by definition the unique (graded-)
    commutative ring structure additively extending the following product rule:
    \begin{align} 
        \alpha_1t^{\v_1} \cdot \alpha_2t^{\v_2}:= (\alpha_1 \star \alpha_2)t^{\v_1 + \v_2}.
    \end{align} 
    for any $\alpha_1 \in H^*(\SIo)$, $\alpha_2 \in H^*(\mathring{S}_J)$, and
    $\v_1, \v_2$ supported on $I$, $J$ respectively, where $\star$ is as in
    \eqref{eq:convol}.  
\end{defn}

With respect to this product there is a subalgebra 
\begin{align}\label{eq:SRdef} 
    \mathcal{SR}^*(M,\D)=  \bigoplus_{I \subset \{1, \ldots, k\}}  t^{\v_I} H^0(\SIo)[t_i\ |\ i \in I] \subset \QH^*(M,\D).
\end{align}
In cases where all of the strata $D_I$ are connected this is a graded version
of the {\em Stanley-Reisner ring} on the dual intersection complex of $\D$. The log
cohomology $H^*_{log}(M,\D)$ is generated by cohomology classes $\alpha
t^{\v_I}$ as a module over $\mathcal{SR}^*(M,\D)$; in particular, $H^*_{log}(M,\D)$ is
a finitely generated $\mathcal{SR}^*(M,\D)$-module. 

\begin{rem} 
    The Kato-Nakayama space $M^{log}$ (see Remark \ref{rem:kato}) arises
    naturally as the target of evaluation maps of punctured \emph{stable} log
    curves \cite{GS} decorated with trivializations of the unit normal bundles
    to each marked point (see Lemma \ref{lem: evalKN} for a special case). In
    particular, the standard ``push-pull" formalism allows one to construct the
    ring structure on
    $H^*_{log}(M,\D)$ using 3-pointed genus zero punctured curves in $(M,\D)$
    whose underlying stable map is constant. In view of this, it may be useful
    for algebraic geometers to view $H^*_{log}(M,\D)$ as a kind of orbifold
    cohomology of the log pair $(M,\D).$\footnote{We thank Alessio Corti and
    Nicolo Sibilla for suggesting this point of view.} 
\end{rem}

It will prove useful to have a Morse co-chain level model for $\QH^*(M, \D)$ 
with its product structure. To this end, fix Riemannian metrics $g_I$ on the compactified
strata $S_I^{log}$ as well as Morse functions $f_I: S_I^{log} \to \mathbb{R}$ whose gradient flows are ``outward pointing near the boundary of $S_I^{log}.$" To make this condition on $f_I$ more precise, let $\pi_I^{log}$ denote the projection $\pi_I^{log}: {S}_I^{log} \to
D_I$, 
and let $US_{I,\rho_0}^{log}$ denote the tubular neighborhood $(\pi_I^{log})^{-1}(\cup_{j \notin I} (U_{j,\rho_0} \cap D_I))$ for any sufficiently small $\rho_0$ (recall \eqref{eq:randomUithing}). Then we want that the gradient flow of $f_I$ points outwards (meaning into the relevant tubular neighborhood) along $\partial (S_I^{log} \setminus US_{I,\rho_0}^{log})$ for each $\rho_0$ (including along the boundary of $S_I^{log}$ itself).
Let $\phi^s_{f_{I}}$ denote the time $s$ flow of the negative
gradient $-\nabla(f_I)$. For any critical point $c$ of $f_I$, let 
\begin{align}
    W^u(f_I,c)&:= \lbrace x \in S_I^{log}, \ \operatorname{lim}_{s \to -\infty}
    \phi^s_{f_{I}}(x)=c \rbrace \\
    W^s(f_I,c)&:= \lbrace
    x \in S_I^{log}, \ \operatorname{lim}_{s \to +\infty} \phi^s_{f_{I}}(x)=c
    \rbrace  
\end{align}  
denote the unstable and stable manifolds of $f_I$ respectively, with respect to the metric $g_I$.
The regularising tubular neighborhoods $U_I$ are diffeomorphic to open subsets of $ND_I$ and hence inherit partial $(\R^{+})^I$ actions
which in turn lift to the blow-ups $U_I^{log}$. We assume that, inside of $U_I^{log}$, our metrics and functions are chosen so that the stable and
unstable manifolds are invariant under these partial actions.

We recall the definition of Morse cohomology in order to fix our conventions.
As all critical points and all flowlines between critical points lie in the
interior $\SIo$, we will often blur the distinction between doing Morse theory
on $\SIo$ and on $S_I^{log}$ (the notable exception being when we construct the
convolution product below). We denote the set of critical points of $f_I$ on
$\SIo$ by \begin{align} \label{eq:morsecritl} \chifI. \end{align} To each critical point $c \in \chifI$, we let the
orientation line, $|\mathfrak{o}_c |$, denote the $\mathbf{k}$-module generated
by choices of orientations of the manifolds $W^u(c)$ modulo the relation that
the sum of generators associated to opposite orientations vanishes.
Let 
\begin{align}
    CM^*(\mathring{S}_I,f_I):= \bigoplus_{c \in \chifI} |\mathfrak{o}_c |
\end{align} 
to the vector space generated by $\mathbf{k}$-orientation
lines associated to each critical point. For a critical point $c \in \chifI$, let
\[
    \operatorname{deg}(c)=\operatorname{dim} W^u(f_I,c).
\]
Let $c_0$ and $c_1$ be two critical points such that
$\operatorname{deg}(c_0)-\operatorname{deg}(c_1)=1$ and let
$\widetilde{\mathcal{M}}(c_0,c_1)$ denote the moduli space of flow lines of $-\nabla(f_I)$
\begin{equation}
    u: \R \to \mathring{S}_I, \textrm{ } \dot{u}(t) = -\nabla(f_I)
\end{equation}
with asymptotics given by $c_1$ and $c_0$ respectively at $\pm \infty$:
\begin{align} \operatorname{lim}_ {s\to \infty} u = c_1 \\
\operatorname{lim}_{s \to -\infty} u= c_0.
\end{align}
(this can equivalently be described as the intersection $W^u(f_I, c_0) \cap W^s(f_I, c_1)$).
 For a generic metric $g_I$ this moduli space is a compact 1-dimensional
 manifold with free $\mathbb{R}$ action, and quotienting by $\mathbb{R}$
 induces a compact 0-dimensional manifold, the moduli space of unparametrized
 flowlines 
 \begin{equation}\label{eq:unparametrizedmorseflowlines}
     \mathcal{M}(c_0,c_1):= \widetilde{\mathcal{M}}(c_0,c_1)/
 \mathbb{R}
 \end{equation}
Any rigid $u \in \mathcal{M}(c_0,c_1)$ induces an isomorphism of orientation
lines
\begin{align} 
     \mu_u: |\mathfrak{o}_{c_{1}}| \to |\mathfrak{o}_{c_{0}}|,
\end{align} 
and as usual counting such solutions (using the orientation isomorphisms above)
induces the Morse differential.  

Let $C^*_{log}(M,\D)$ denote the cochain complex generated by elements of the
form $\alpha_c t^{\v}$, where $\alpha_c$ is a (co)-chain in
$CM^{*}(\mathring{S}_I,f_I)$ for some subset $I \subset \{1, \ldots, k\}$, and
$\v = (v_1, \ldots, v_k)$ is a vector of
non-negative integer multiplicities {\em strictly supported} on $I$, meaning
$v_i > 0$ if and only if $i \in I$. 

The differential on  $C^*_{log}(M,\D)$ is induced from the differential on $CM^{*}(\SIo,f_I)$, and as before we can give an efficient description of $C_{log}^*(M,\D)$ as follows:
\begin{align}
    C_{log}^*(M,\D)&:= 
    \bigoplus_{I \subset \{1, \ldots, k\}}  t^{\v_I} CM^*(\SIo,f_I)[t_i\ |\ i \in I]\\
    \label{morselog}&= \bigoplus_{I \subset \{1, \ldots, k\}}\bigoplus_{\v\textrm{ supported on }I} CM^*(\SIo,f_I).
\end{align}
where $S_\emptyset = X$, and $S_I = \emptyset$ if the intersection $\cap_{i \in
I} D_i$ is empty, with differential induced by \eqref{morselog}. 
With respect to the grading defined earlier, the degrees of each 1-dimensional subspace associated to the $t^\v$ copy of a critical point $c \in \chifI$ is
\begin{equation}\label{logdegreemorse}
\deg(|\mathfrak{o}_c| t^{\v}) = \deg(c)  +  2\sum_{i=1}^k (1-a_i) v_i.
\end{equation}
Observe that indeed this gives a cochain model for logarithmic cohomology,
i.e.,
\begin{align} 
    \QH^*(M,\D)\cong H^*(C_{log}^*(M,\D)).
\end{align}

To implement a co-chain level product inducing the cohomological
product from Definition \ref{defn: ringstructure}, we need to recall the
construction of pullbacks in
Morse cohomology. If $I \subset K$, and given any $c_1 \in \chifI$, define
\begin{align} 
    W^s_{K}(f_I,c_1):= W^s(f_I,c_1) \cap S_K^{log} 
\end{align}

\begin{defn} For any $c_2 \in \mathcal{X}(\mathring{S}_K, f_K)$ with Morse index $\deg(c_2)=\deg(c_1)$, consider the moduli space
    \begin{align} 
        \label{eq: IKMorse} \mathcal{M}_{IK}(c_2,c_1) :=W^s_{K}(f_I,c_1) \cap W^u(f_K,c_2) 
    \end{align} 
\end{defn} 
Note that by our assumption that the Morse function $f_K$ has inward pointing negative gradient flow near the boundary, this intersection lies in $\mathring{S}_K$ and in fact outside of tubular neigborhood of the lower dimensional strata. It therefore follows from Appendix A.2. of \cite{AbSch2} that for generic metrics and Morse functions on $S_I^{log}$, and $S_K^{log}$, \eqref{eq: IKMorse} is a compact zero dimensional manifold (Note that condition (A.2.) of \emph{loc. cit.}  is vacuous in our case and (A.3.) can be achieved by a generic perturbation of the function $f_I$). To orient moduli spaces of rigid solutions, note that a choice of orientation of $W^u(f_I,c_1)$, induces an orientation of the normal bundle to $W^s_{K}(f_I,c_1)$ inside of $S_K^{log}$ via the canonical isomorphism 
\begin{align} 
    N_{S_{K}^{log}}W^s_{K}(f_I,c_1) \cong  N_{S_{I}^{log}}W^s(f_I,c_1)|_{W^s_K(f_I, c_1)}
\end{align} 
A choice of orientation on $W^u(f_K,c_2)$ then allows us to orient the
intersection $\mathcal{M}_{IK}(c_2,c_1)$ and we view this as giving a map on
the corresponding orientation lines. By counting the induced maps on
orientation lines associated to such solutions, we obtain a map  
\begin{align}
    r_{IK}^*: CM^*(\mathring{S}_I,f_I) \to CM^*(\mathring{S}_K, f_K)
\end{align}

Standard arguments in Morse theory prove that this defines a cochain map. Upon
identifying Morse  cohomology and singular cohomology, this operation
corresponds to the usual pull-back on singular cochains (see pages 1712-1713 of \cite{AbSch2} for a proof in the dual setting of homology). We next recall the
construction of the cup-product in Morse cohomology by counting Y-shaped
trajectories, a special case of more general operations on Morse complexes
induced by metrized ribbon trees.

For any tree $T$, let $E(T)$ denote the set of edges and $V(T)$ the set of
vertices. In the Morse context, edges are allowed to
come in two types; infinite (external) edges $E_{\operatorname{ext}}(T)$ and
finite (internal) edges $E_{\operatorname{int}}(T)$. Operations in Morse theory
are defined using rooted metrized ribbon (Stasheff) trees. For such graphs,
there is a distinguished outgoing external edge $\cev{e}_0$ and the remaining
external edges $\vec{e}_1, \cdots \vec{e}_{|E_{\operatorname{ext}}(T)|-1}$
(called incoming edges) acquire a linear ordering.  Furthermore, each edge $e
\in E(T)$ may be canonically oriented and 
comes with a canonical length and orientation preserving map $e \to
\mathbb{R}$.  We let $t_e$ denote the induced coordinate on each edge. 

 The simplest example of such a graph is  $T_{2,1}$, the unique trivalent
 Stasheff tree with three external edges, two of which are incoming and one
 which is outgoing. Each of the two incoming edges $\vec{e}_1, \vec{e}_2$ is
 identified with $[0,\infty)$ and the outgoing edge $\cev{e}_0$ is identified with
 $(-\infty,0]$. 

 Rather than equip each edge of $T_{2,1}$ with different Morse function, we
 will ensure transversality of gradient flow solutions for maps from $T_{2,1}$ by
 perturbing the gradient flow equation (for a single Morse function). We make
 use of the following special case of \cite[Def. 2.6]{Abouzaidtop}: 
 \begin{defn}
    A gradient flow perturbation datum on $T_{2,1}$ (on the stratum $S_K$) is a choice, for each edge
    $ e \in E(T_{2,1})$ of a smoothly varying family of vector fields, 
    $$X_e: e \to C^{\infty}(S_K^{log}) $$ 
    vanishes away from a bounded subset of $e$
    and which is invariant under the local action of $\prod _j \mathbb{R}^{+}$
    in each tubular neighborhood.  
\end{defn}  

Given a gradient flow perturbation datum as above, for each edge $e \in
E(T_{2,1})$ and any map $u: e \to S_{K}^{log}$, one can ask for $u$ to solve
the perturbed gradient flow equation (for $f_K$ with respect to $X_e$):
\begin{align}\label{eq:pertgradflow} 
    du_e(\partial_{t_{e}})= -\nabla f_K + X_e. 
\end{align}
   
\begin{defn} \label{defn:Ymoduli} 
    Let $f_{K}:S_K^{log} \to \mathbb{R}$, be a Morse function and fix gradient flow perturbation data
    perturbation data $\{X_e\}_{i = 0,1,2}$ as above. Suppose that $c_0, c_1, c_2$ lie in
    $\mathcal{X}(\mathring{S}_K, f_{K})$ and satisfy
    $\deg(c_0)=\deg(c_1)+\deg(c_2)$. With respect to this data, let 
    \begin{equation}\mathcal{M}_Y(c_0,c_1,c_2)\end{equation} 
    denote the moduli space of continuous maps $u: T_{2,1} \to S_K^{log}$ whose
    restriction to each edge is (smooth and) a solution to
    \eqref{eq:pertgradflow}, and which are asymptotic to $c_0$, $c_1$, and $c_2$ at the
    non-compact ends of $\cev{e}_0$, $\vec{e}_1$, and $\vec{e}_2$ respectively.
\end{defn}
See Figure \ref{morseproduct} for a schematic element of $\mathcal{M}_Y(c_0,c_1,c_2)$.
\begin{figure}[h] 
    \caption{\label{morseproduct}}
    \centering
    \includegraphics[scale=2.0]{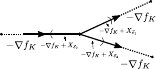}
\end{figure}

 Note near the vertex, the perturbation data can be arbitrary.
 It is not difficult to show that for a generic choice of
 perturbation data, our moduli spaces are cut out transversally.
 Somewhat
 informally, this corresponds to the fact that infinitesimally, solutions to
 the perturbed gradient flow equations correspond to intersections of the
 unstable and stable manifolds under perturbations by the diffeomorphisms
 $\phi_e$ given by integrating the
 vector fields $X_e$. As the vector fields $X_e$ can be chosen arbitrarily,
 these diffeomorphisms are essentially arbitrary (for a complete proof, see
 \cite[\S 7]{Abouzaidtop}). Futhermore, when our data is chosen generically,
 the zero dimensional components of the moduli spaces above induce maps between
 orientation lines as before, hence an operation on Morse complexes. This map
 can be composed with the restriction maps \eqref{eq: IKMorse}  to obtain a
 chain model for the $\star$ product map \eqref{eq:convol}. Equivalently, the
 constituent (chain-level) product map \eqref{eq:convol} for the ring structure
 defined in Definition \ref{defn: ringstructure}
 \[
     t^{\v_1} CM^*(\mathring{S}_I, f_I) \otimes t^{\v_2} CM^*(\mathring{S}_J, f_J) \to t^{\v_1 + \v_2} CM^*(\mathring{S}_K, f_K)
 \]
 can be defined by counting (in the usual signed sense) rigid elements of the fiber
 product of moduli spaces from \eqref{eq: IKMorse} and Definition
 \ref{defn:Ymoduli}:
 \[
     \coprod_{w,w' \in \mathcal{X}(\mathring{S}_K, f_K)} \mathcal{M}_Y(z, w, w') \times \mathcal{M}_{IK}(w, x) \times \mathcal{M}_{JK}(w', y),
 \]
 for each triple $x \in \mathcal{X}(\mathring{S}_I, f_I)$, $y \in
 \mathcal{X}(\mathring{S}_J, f_J)$, $z \in \mathcal{X}(\mathring{S}_K, f_K)$.
 See Figure \ref{starproduct} for a picture of an element of such a (broken) moduli space.
 \begin{figure}[h] 
    \caption{\label{starproduct}}
    \centering
    \includegraphics[scale=2.0]{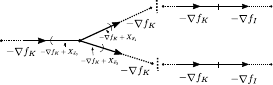}
\end{figure}

\subsection{Moduli spaces and operations} \label{section:lowenergydef}
Here we relate log cohomological structures to the Floer theoretic
constructions of \S \ref{section:SHtor}, by constructing the {\em low energy
log(arithmic) PSS map}:
\begin{equation}\label{psslowmodulispacesection}
    \PSSlog^{low} := (\bigoplus_{\v} \PSSlog^\v): \QH^*(M,\D)  \to \bigoplus_{p,q} E^{p,q}_{1},
\end{equation}
where the right hand side is the first page of the spectral sequence converging
to $SH^*(X)$ from \S \ref{sect: actionspec}.

Consider the domain 
\[S = \mathbb{C}P^1 \setminus \lbrace 0 \rbrace,\] 
thought of as a punctured Riemann sphere, with a distinguished marked point $z_0 =
\infty$ and a negative cylindrical end 
\begin{equation}
\epsilon: (-\infty, 0] \times S^1 \to S
\end{equation}
near $z=0$  given by 
\begin{equation}\label{cylend}
    \epsilon: (s,t) \mapsto e^{s+it} \end{equation} 
The map \eqref{cylend} can evidently be defined on all of $\R \times S^1$, and
correspondingly the coordinates $(s,t)$ extend to all of $S \setminus z_0$. 
Let $\rho(s)$ be a cutoff function as in \eqref{eq:cuttoff} and let 
\begin{equation}  
    \beta=\rho(s)dt \label{eq:sco} 
\end{equation} 
(implicitly smoothly extended across $z_0$).

Observe that $\beta$ restricts to $dt$ on the negative cylindrical end (at least for $s \ll 0$) and restricts to 0 in a neighborhood of $z_0$. At $z_0$, we also fix a
distinguished tangent vector which points towards the positive real axis. 
\begin{defn} \label{defn: sdcs} Let $\mathcal{J}_S(M,\D)$ denote the space of $S$-dependent
    families of complex structures $J_S = \{J_z\}_{z \in S}$ satisfying:
    \begin{itemize} 
    \item $J_z \in \mathcal{J}(M,\D)$ (see Definition \ref{defn:complexint}) for all $z \in S$;
    \item Near $z_0$, $J_S$ agrees with some fixed surface-independent almost
    complex structure $J_0 \in \mathcal{J}(M,\D)$; and
    \item along the negative cylindrical end, $J_S=J_F$ for some $J_F \in \mathcal{J}_F(M,\D)$ (recall Definition \ref{defn: S1acdef}).
\end{itemize}  
\end{defn} 

\begin{defn} 
    Let $\mathcal{J}_{S,\ell}(V) \subset \mathcal{J}_S(M,\D)$ be the space of
    almost complex structures such that along the negative cylindrical end,
    $J_S=J_F$ for some $J_F \in \mathcal{J}_{F,\ell}(V) \subset
    \mathcal{J}_F(M,\D)$ (again see Definition \ref{defn: S1acdef}).
\end{defn}

We first recall the definition of the classical PSS moduli space: 

\begin{defn} \label{def:classicalPSSmoduli}
Fix $\ell$ and $J_S \in \mathcal{J}_S(M,\D)$. For every orbit $x_0$ of $\mathcal{X}(X;\Hlm)$ or $\mathcal{X}(\D, \Hlm)$ define 
\begin{equation}\label{modulispace2}
\mc{M}(x_0)
\end{equation}
(which implicitly depends on $\Hlm$, $J_S$) to be the moduli space of maps
\[
u: S \ra M
\]
satisfying Floer's equation
\begin{equation} \label{eq:PSSeq}
    (du - X_{\Hlm} \otimes \beta)^{0,1} = 0
\end{equation}
(where $({0,1})$ is taken with respect to $J_S$) with asymptotics 
\begin{align}\label{eq:limitingcondition}
    \lim_{s \ra -\infty} u(\epsilon(s,t)) &= x_0
    \end{align}
\end{defn}
In \cite{GP1}, we introduced a relative version of these moduli spaces, where one imposes tangency conditions at the
marked points: 
\begin{defn} 
Fix $J_S \in \mathcal{J}_S(M,\D)$. For any orbit $x_0$ of $\mathcal{X}(X;\Hlm)$ and any multiplicity vector
$\v = (v_1, \ldots, v_k) \in \Z_{\geq 0}^k$, let 
\begin{equation}\label{modulispace1}
\mc{M}(\v, x_0)
\end{equation}
denote the moduli space of maps
\[
u: S \ra M
\]
satisfying Floer's equation \eqref{eq:PSSeq} (with respect to $\Hlm$, $J_S$), with
asymptotics \eqref{eq:limitingcondition}, and the following additional
tangency/intersection constraints:
\begin{align} \label{eq:incidencecondition}
u(z) &\notin \D\textrm{ for $z \neq z_0$}; \\
\label{eq:incidence} u(z_0) &\textrm{ intersects $D_i$ with multiplicity $\vi$ for $i \in I$}.
\end{align}
\end{defn} 
As explained in 
\cite[\S 3.5]{GP1}, when \eqref{eq:incidencecondition} and \eqref{eq:incidence} hold, the (real-oriented) projectivization of the $\vi$ normal jets of the map $u$ (with respect to a fixed real tangent ray in
$T_{z_0} C$) give an {\em enhanced evaluation map}
\begin{equation}\label{eq:enhancedevaluationpss}
    \Evzo:  \mc{M}(\v, x_0) \ra \mathring{S}_I
\end{equation}
where $\mathring{S}_I$ is defined as in \eqref{opentorusbundle}(compare \cite[\S 6]{CieliebakMohnke}, \cite[\S 3]{Ionel:2011fk}).
\begin{defn} \label{def: higher} 
For $c \in \chifI$ a critical point of $f_I$ on $\mathring{S}_I$, 
let $\mc{M}(\v,c, x_0)$  denote the moduli space 
    \begin{align} 
        \label{eq:eval} 
        \mc{M}(\v,c, x_0):= \mc{M}(\v, x_0)\times_{ \Evzo} W^s(f_I,c).
    \end{align}
\end{defn}  \vskip 5 pt
See Figure \ref{logpssfig} for a picture of the above definition.
  \begin{figure}[h] 
      \caption{The domain(s) of the moduli space $\mc{M}(\v,c, x_0)$ along with the incidence, matching, asymptotics, and PDE satisfied at various places.\label{logpssfig}}
    \centering
    \includegraphics[scale=1.5]{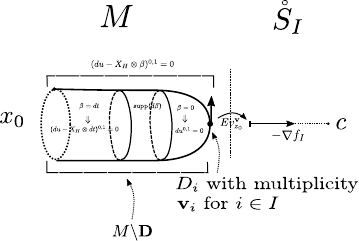}
\end{figure}
A straightforward index computation (see e.g., \cite[Lemma 4.11]{GP1}) shows
that the virtual dimension of $\mc{M}(\v,c, x_0)$ is given by
\begin{align}
    \label{vdimlogpss}
        \operatorname{vdim}(\mc{M}(\v,c, x_0))&= \deg(x_0)-  \deg(|\mathfrak{o}_c| t^{\v}), 
\end{align} 
where $\deg(|\mathfrak{o}_c| t^{\v})$ is as in \eqref{logdegreemorse}.
It follows from the same arguments explained in  \cite[Lemma 4.15]{GP1}
that for a generic choice of almost complex structure\footnote{In \cite{GP1},
we do not make any assumptions on the Nijenhuis tensor of our almost complex
structure, however the arguments explained there apply without modification.},
the moduli space \eqref{eq:eval} is a manifold of dimension equal to its virtual
dimension \eqref{vdimlogpss}.  In general, the Gromov compactification of this
moduli space necessarily contains holomorphic sphere bubbles, which could
obstruct counts of the (compactified) moduli space from (being well-defined or)
giving a chain map (see Lemma 4.21 of \cite{GP1}). 
The following Lemma gives us an estimate for the energy of the
PSS moduli space, which will be useful for further analysis: (Before stating it, we remind the reader of the definitions of the weighted winding numbers $w(\v)$ from \eqref{eq:windingvec} and $w(x_0)$ from \eqref{eq:weightedwindingnumber}.)
\begin{lem}\label{lem:energypss}
Fix a multiplicity vector $\v$  
as well as an orbit $x_0$.
If $\epsilon_\ell$ and $||\Hlm-\hlm||_{C^2}$  are each sufficiently
small and if $\Sigma_\ell$ (see \eqref{eq:defnsigmaell} and \eqref{eq: definitionsigma}) is sufficiently $C^0$-close to $\hatXl$ (recall Remark \ref{rem: C0close}), then the
topological energy of any $u
\in \mc{M}(\v, c, x_0)$ is approximately (i.e., becomes arbitrarily close to) 
\begin{align} 
    \label{eq:Etop} \Etop(u) \approx w(\v)-w(x_0)(1-\epsilon_{\ell}^2/2). 
\end{align} 
\end{lem}
\begin{proof}
    One obtains \eqref{eq:Etop} by taking the limit as $\delta \to 0$ of the
    topological energy of $u$ restricted to the complement $S_{\delta}:= S
    \backslash B_{\delta}(z_0)$. Now $u$ restricted to $S_{\delta}$ maps to $X$
    where $\omega$ is exact, so Stokes' theorem applies and this
    energy can be calculated as the integral of $\theta$ pulled back along
    $u|_{\partial B_{\delta}(z_{0})}$ minus the $H^{\ell}$ action of
    $x_0$, which we have seen in Equation \eqref{eq: action} is $\approx
    w(x_0)(1-\epsilon_{\ell}^2/2)$ under the above smallness hypothesis. 
    The term $w(\v)$ arises as limit of the former term, i.e., the limiting
    integral of $\theta$ over a small loop winding around each $D_i$ with
    multiplicity $v_i$, compare \cite[Lemma 2.11]{GP1}.  
\end{proof}

Using this, the next Lemma shows that the ``low energy'' PSS moduli spaces
\eqref{eq:eval} with $w(x_0)=w(\v)$ 
admit a nice compactification (i.e., without sphere bubbles):
\begin{lem} 
\label{lem:compactness} Suppose that $$w(\v) < \lambda_{\ell} $$ and that $\epsilon_\ell, ||\Hlm-\hlm||_{C^2}$  are both sufficiently small and $\Sigma_\ell$ is sufficiently $C^0$-close to $\hatXl$.
    Consider $|\mathfrak{o}_c| t^{\v} \in C_{log}^*(M,\D)$ and take $x_0 \in \mathcal{X}(X;\Hlm)$
    to be an orbit such that $w(x_0)=w(\v)$: 
\begin{itemize} \item If 
    $\deg(x_0)-\deg(|\mathfrak{o}_c| t^{\v})=0$, 
     then for generic $J_S \in
    \mathcal{J}_{S,\ell}(V)$, the moduli space ${\mc{M}}(\v, c, x_0)$ is compact. \vskip 5 pt
\item If  $\deg(x_0)-\deg(|\mathfrak{o}_c| t^{\v})=1$, then for generic  $J_S \in
    \mathcal{J}_{S,\ell}(V)$,  ${\mc{M}}(\v, c, x_0)$
    admits a compactification (in the sense of Gromov-Floer convergence)
    $\overline{\mc{M}}(\v,c, x_0)$, such that $\partial\overline{\mc{M}}(\v, c, x_0)= \partial_M \bigsqcup \partial_F$ where 
    \begin{align} 
        \partial_F:= \bigsqcup_{x',\deg(x_0)-\deg(x')=1}^{w(x')=w(\v)}  \mc{M}(x_0,x') \times \mc{M}(\v,c, x') \\
     \partial_M:= \bigsqcup_{c',\deg(c')-\deg(c)=1} \mc{M}(\v,c', x_0) \times \mc{M}(c',c)
     \end{align}
\end{itemize} 
  \end{lem}
  \begin{proof}  
      We consider the closure $\overline{\mc{M}}(\v, x_0) \subset
      \overline{\mc{M}}(x_0)$. The argument of \cite{GP1}*{Lemma 4.13}
      rules out Floer cylinder breaking
      along $\D$, so the only ``bad'' possibilities which prevent the
      compactification from being as stated are sphere bubbling 
      (and cylinders breaking in $X$ but touching $\D$, but this is excluded by positivity of intersection).
      Now under our assumptions, Lemma \ref{lem:energypss} implies that the
      (topological) energy of a PSS solution is approximately
      $w(\v)\epsilon_\ell^2/2$. If $\epsilon_\ell$ is sufficiently small,
      then this quantity can be made much smaller than the minimal energy of a
      non-constant $J$-holomorphic sphere in $M$. Hence, 
      assuming all of the parameters in the statement of the Lemma are
      small enough to make the estimate \eqref{eq:Etop} have sufficiently small
      error, it follows that the energy of a PSS solution will be smaller
      than the minimal energy of a non-constant $J$-holomorphic sphere. This
      implies that such sphere bubbling is impossible. 
      The same argument shows
      that breaking can only occur along orbits with $w(x')=w(\v).$ 
  \end{proof} 
  See Figure  \ref{logpsslocalgromovcompactness} for a schematic picture of the various ``bad breakings'' exclude for low energy log PSS moduli spaces by Lemma \ref{lem:compactness}.

\begin{figure}[h] 
    \caption{ Bad breakings of low energy log PSS moduli spaces excluded by Lemma \ref{lem:compactness}. First, Floer cylinder breakings along orbits in $\D$ are ruled out by \cite{GP1}*{Lemma 4.13}. Then, sphere bubbling at any point in the domain (either the marked point or another point) are ruled out as there is not enough energy to product a non-constant $J$-holomorphic sphere. \label{logpsslocalgromovcompactness}}
    \centering
    \includegraphics[scale=1.0]{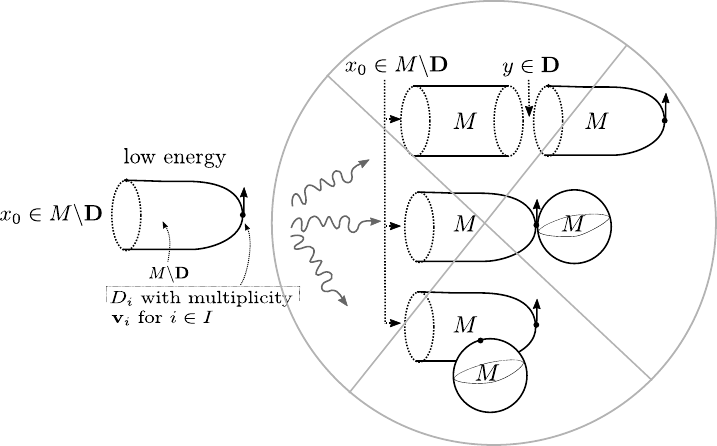}
\end{figure}

We next describe, for any $\v$ with $w(\v) \leq w_\ell$, a map: 
\begin{equation} \label{eq: PSSvmapeq}
\PSSlog^{\v,\ell}: H^*(\SIo)t^{\v} \to  HF^*(X \subset M ; \Hlm)_{w(\v)}
\end{equation} where $HF^*(X \subset M ; \Hlm)_{w(\v)}$ is defined in \eqref{eq:lowlow}. To do this, choose generic $J_S \in
\mathcal{J}_{S,\ell}(V)$ and apply the following prescription: as usual a rigid element $u \in
\mc{M}(\v,c,x_0)$ induces an isomorphism of orientation lines $\mu_u:
|\mathfrak{o}_c| \to |\mathfrak{o}_{x_0}|$. Then following a standard pattern,
we define $\PSSlog^{\v,\ell}$ by using such isomorphisms to give signed counts
of rigid elements (which is a well-defined finite count by Lemma
\ref{lem:compactness}): for $z \in |\mathfrak{o}_c|$, 
\begin{align} \label{eq:PSSdef}
 \PSSlog^{\v,\ell} (z t^{\v}) := \sum_{w(x_0) = w(\v), \operatorname{vdim}(\mc{M}(\v, c, x_0))=0} \sum_{u \in\mc{M}(\v, c, x_0)} \mu_u (z)
\end{align}
Under the smallness conditions of $\epsilon_\ell$, this is a well-defined
finite count and gives a chain map
\begin{equation}
    \label{eq:singlevlogpss}
    \PSSlog^{\v,\ell}: C^*_{log}(M,\D)_{\v} \to  CF^*(X \subset M ; \Hlm)_{w(\v)}
 \end{equation}
which induces  \eqref{eq: PSSvmapeq} on the level of cohomology.
 Taking a direct sum of these maps  \eqref{eq: PSSvmapeq} over all multiplicity vectors $\v$ with
 $w(\v)= w$, we obtain a map on cohomology 
 \begin{align}  
     \label{eq:locPSS} \PSSlog^{w,\ell}:
     \QH(M,\D)_w \to  HF^*(X \subset M ; \Hlm)_{w} 
 \end{align} 
 for any $w \leq w_\ell$. 
A straightforward variation of the above analysis (implemented in Lemma 4.18 of
\cite{GP1}) verifies that \eqref{eq:singlevlogpss}, hence
\eqref{eq:locPSS}, is compatible with the 
continuation maps \eqref{eq:lowencont}:
\begin{lem} \label{lem: lowcontok}
For $\ell_2 \geq \ell_1$, we have a commutative triangle: 
\[
\xymatrix{
   H^*(\SIo)t^{\v} \ar[d]^{\PSSlog^{\v,\ell_1}} \ar[dr]^{\PSSlog^{\v,\ell_2}} \\
  HF^{*}(X \subset M; \Hm^{\ell_1})_{w(\v)}  \ar[r]^{\mathfrak{c}_{\ell_1,\ell_2}}  & HF^{*}(X \subset M; \Hm^{\ell_2})_{w(\v)}
}
\] 
\end{lem}

In view of this, we drop the superscript $\ell$ from the notations $\PSSlog^{\v}$ and $\PSSlog^{w}.$ Summarizing the above considerations, we have the following: 
\begin{lem} 
    \label{lem:spectral} 
There are canonical maps: 
\begin{align} 
    \label{eq:Speciso1} \left(\bigoplus_{\v, w(\v) \leq w}\PSSlog^{\v}\right): F_w \QH^*(M,\D) \to \bigoplus_{p\geq -w,q} E_{\ell,1}^{p,q} \\  
    \label{eq:Speciso2}\PSSlog^{low} := \left(\bigoplus_{\v} \PSSlog^\v\right): \QH^*(M,\D)  \to \bigoplus_{p,q} E^{p,q}_{1}  
\end{align} 
where the spectral sequence appearing in  \eqref{eq:Speciso1} is defined in \eqref{eq:Spec1} and the spectral sequence appearing in \eqref{eq:Speciso2} is defined in \eqref{eq:Spec2}.
\end{lem} 
\begin{proof} The first statement is immediate from the above discussion and the second follows from that statement in view of the explicit description of the $E_1$ page from \eqref{eq:E1concrete} and Lemma \ref{lem: lowcontok}. \end{proof} 
We refer to the map defined in \eqref{eq:Speciso2} as the \emph{low energy} PSS map.

\subsection{Low energy log PSS is a ring map} \label{section:rings}

The aim of this subsection is to prove that
\begin{thm} 
    \label{lem:spectralrings} The low energy log PSS map \eqref{eq:Speciso2} is a ring homomorphism.  
\end{thm}
A proof of Theorem \ref{lem:spectralrings} appears at the bottom of this
subsection, using intermediate results we now describe.
For intuition, the reader could also simply consult Figure \ref{logpssringmap} which
gives an overview of the series of 1-parametric families of moduli spaces used
to construct chain homotopies (and the intermediate moduli spaces whose
operations they provide chain homotopies between) between the product of low
energy log PSS elements and the low energy log PSS of the product element. 
\begin{figure}[h] 
    \caption{The 1-parameter families of surfaces on the right column of this figure indicate the domains of interpolating moduli spaces of maps which produce chain homotopies between the operations coming from the surfaces on the left column. To turn all these domains into the desired low energy log PSS-type moduli spaces, one should study maps from (each component of) these domains satisfying the relevant PDE along with the following conditions/modifications (compare e.g., in Figure \ref{logpssfig}): (a) the complement of all marked points on each Riemann surface $\Sigma$ maps to $M \backslash \D$. (b) for any $\Sigma$ above with two marked points (cases (1)-(4)), the ``top'' point in the picture hits $\D$ with multiplicity $\v_1$ and the ``bottom'' point hits $\D$ with multiplicity $\v_2$. (c) for any $\Sigma$ with only one marked point (cases (5)-(7)), this point hits $\D$ with multiplicity $\v_1 + \v_2$ (d) For $T$ any rooted metrized ribbon tree pictorially attached to a marked point $p$ on any Riemann surface $\Sigma$, the associated space of Morse trajectories (in $S_K^{log}$, $S_I^{log}$, or $S_K^{log}$) should actually be attached to the {\em enhanced evaluation} at $p$ of the map from $\Sigma$ in the sense of \eqref{eq:enhancedevaluationpss} (e) Morse trajectories should be further perturbed (near 3-valent vertices) as spelled out in Figure \ref{morseproduct}.
    \label{logpssringmap}}
    \centering
    \includegraphics[scale=0.9]{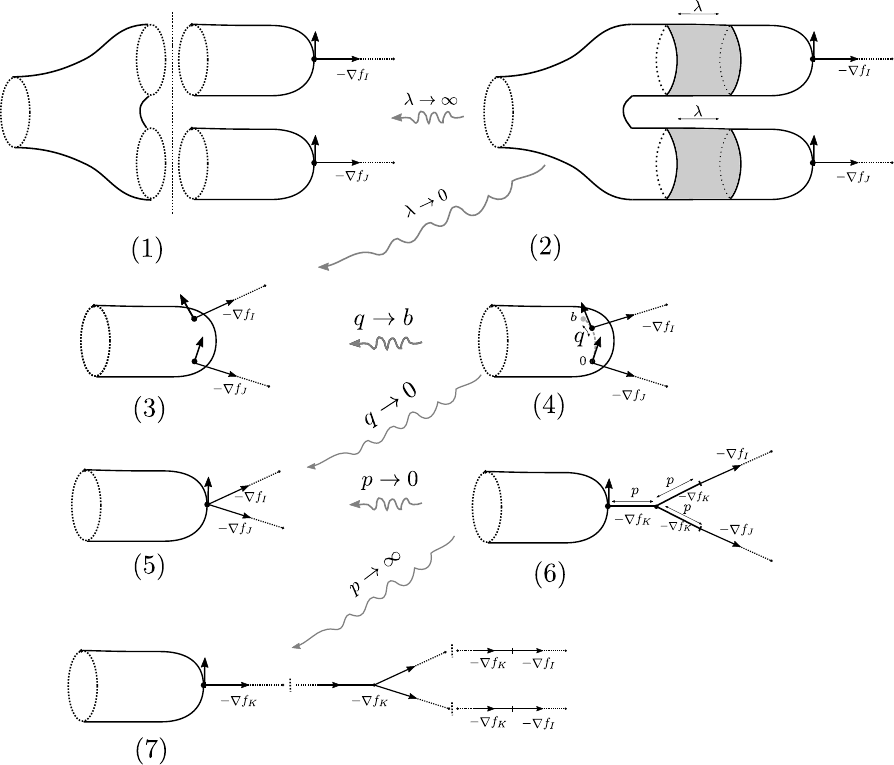}
\end{figure}

In order to simplify our proof, we begin by observing that the ring structure on
$H_{log}^*(M,\D)$ from Definition \ref{defn: ringstructure} is determined uniquely
by the value of the product on a much smaller collection of (pairs of)
elements. Given a stratum $\mathring{S}_J$ with connected components
$\mathring{S}_{J,m}$, $m \in \lbrace 1,\cdots,\pi_0(\mathring{S}_J) \rbrace$,
we denote the fundamental cycle (i.e., the unit) in $H^0(\mathring{S}_{J,m})$
by $[\mathring{S}_{J,m}]$. For what follows below, recall the definition
\eqref{eq:primitivevI} of the primitive vector $\v_I$ supported on a given $I
\subset \{1, \ldots, k\}$. 
\begin{lem}\label{lem:generatingrelationsfortheproduct}
The product on $H_{log}^*(M,\D)$ is the unique (graded-) commutative ring structure such that 
\begin{itemize} 
    \item[(i)]  For any $I$, $J$, $\alpha_1 \in H^*(\mathring{S}_I)$ and $\alpha_2 \in H^*(\mathring{S}_J)$,
        \begin{align} 
            \label{eq:drel2} 
            \alpha_1t^{\v_I} \cdot \alpha_2t^{\v_J}= (\alpha_1 \star \alpha_2)t^{\v_I + \v_J}.
        \end{align} 

    \item[(ii)] For any $I$, let $\v$ be an \emph{arbitrary} multiplicity vector supported on $I$ and $\alpha \in
        H^*(\mathring{S}_I)$ a cohomology class. For any $J \subseteq I$, we have that:
        \begin{align} 
            \label{eq:drel1} \alpha t^{\v} \cdot
            [\mathring{S}_{J,m}]t^{\v_J}=( \alpha \star [\mathring{S}_{J,m}])
            t^{\v+\v_J} 
        \end{align}
 \end{itemize} 
 \end{lem}
 \begin{proof}
It is a straightforward exercise to see that these relations imply the
multiplication rule given in Definition \ref{defn: ringstructure}.
\end{proof}
Thus Theorem \ref{lem:spectralrings} will follow from verifying the following
special cases:
\begin{itemize} 
    \item[(i)]  For any $I$, $J$,  $\alpha_1 \in H^*(\mathring{S}_I)$ and $\alpha_2 \in H^*(\mathring{S}_J)$,
    \begin{align} \label{eq:rel2} \PSSlog^{low}(\alpha_1t^{\v_I}) \cdot \PSSlog^{low}(\alpha_2t^{\v_J})= \PSSlog^{low}((\alpha_1 \star \alpha_2)t^{\v_I + \v_J}).
    \end{align}

\item[(ii)] For any $I$, $J \subseteq I$, $\alpha \in H^*(\mathring{S}_I)$ and an arbitrary multiplicity vector $\v$ supported on $I$,
    \begin{align} \label{eq:rel1} 
        \PSSlog^{low}(\alpha t^{\v}) \cdot  \PSSlog^{low}([\mathring{S}_{J,m}]t^{\v_J})=\PSSlog^{low}((\alpha \star [\mathring{S}_{J,m}]) t^{\v+\v_J}). 
    \end{align} 
 \end{itemize}

\begin{rem} The place in our proof of Theorem \ref{lem:spectralrings} where we restrict to these special cases is in Lemma \ref{lem: evalKN}. In fact, we expect that the continuity statement of Lemma \ref{lem: evalKN} is true for arbitrary pairs of multiplicity vectors $\v_1, \v_2$, however the analysis is somewhat simpler if we restrict to the cases above. \end{rem}

 The proof of Theorem \ref{lem:spectralrings} in these cases follows the standard
 pattern in TQFT of using interpolating families of moduli spaces to give a
 cobordism (relative ``chain homotopy terms'') between the operations either
 side of the equality. In our case, the argument bears a resemblance to
 Piunikhin-Salamon-Schwarz's original argument \cite{Piunikhin:1996aa} that the
 PSS map intertwines product structures, adapted in a non-trivial way to our
 (``relative $\D$'') setup.

 To begin, for some small $b > 0$ and any $q \in (0,b]$, let $\Sigma_{q}= \mathbb{C}P^1
 \setminus \lbrace 0 \rbrace$, with a negative cylindrical end \eqref{cylend}
 as before, along with two additional marked points at $z_1=-1/q$ and
 $z_2=\infty$. The family of $\{\Sigma_q\}_{q \in (0,b]}$ limits to a stable domain
 $\Sigma_0$ as $q \to 0$, and including this domain gives a family of stable curves over
 $q \in [0,b]$. 
 It will be convenient to work in coordinates centered about
 $\infty$, which we denote by  \begin{align} \label{eq:coordinfty} y = z^{-1}. \end{align}
 As usual, we will equip these domains with Floer data. To simplify
 the discussion a bit, note that our product operation from \S \ref{subsection:Floercohom} is defined by first continuing both inputs to a single Liouville domain and then defining the product on a fixed Liouville domain. As $\PSS_{log}$ is compatible with these continuation maps, we can and will assume that our Liouville domain $\bar{X}_\ell$ used to define the Hamiltonians remains fixed.
 The Floer data on domains $\Sigma_q, q \neq 0$ is
 determined by a sub-closed one form  $\beta$ satisfying
\begin{align}
    &\textrm{The form $\beta$ restricts to $2dt$ on the cylindrical end \eqref{cylend}.}\\
    &\textrm{$\beta=0$ in a neighborhood of the points $y(z_1)=-q$ and $y(z_2)=0$.}
\end{align}

In the next definition, we fix vectors $\v_1,\v_2$ as well as an orbit $x_0$ of $H^{\ell}$. We also fix a family of domain-dependent almost complex structures $J_{\Sigma_q} \in \mathcal{J}_S(\bar{X}_\ell, V)$, which varies smoothly in $q$ and so that $J_{\Sigma_{q}}$ is domain independent in a neighborhood of $y(z_1)=-q$ as well. We assume that as $q \to 0$, $J_{\Sigma_{q}}$ converges to a complex structure $J_S \in \mathcal{J}_S(\bar{X}_\ell, V)$ which we used to define low energy log $\PSS$. 

\begin{defn} \label{psstwomarkedpoints}
 Define $\mathcal{M}(\v_1, \v_2 ; x_0)$ to be the moduli space of pairs 
 \[
 \{(q,u) | q \in (0,b],\ u: \Sigma_q \to M\}
 \]
whose associated map $u$ satisfying Floer's equation
\begin{equation}
(du - X_{\Hlm} \otimes \beta)^{0,1} = 0
\end{equation}
(where $0,1$ component is taken with respect to $J_{\Sigma_q}$), with asymptotics given by $x_0$
\begin{equation}
\lim_{s \ra -\infty} u(\epsilon(s,t)) = x_0
\end{equation}
and tangency/intersection conditions along $\D$ at $(z_1, z_2)$ as specified by
$(\v_1, \v_2)$: 
\begin{align}
u(x) &\notin \D \textrm{ for $x \neq z_i$}; \\
 u(z_1) &\textrm{ intersects $\D$ with multiplicity $\v_1$}.\\
 u(z_2) &\textrm{ intersects $\D$ with multiplicity $\v_2$}.
\end{align}
We define $\mathcal{M}_s(\v_1, \v_2; x_0)$ to consist of the subspace of maps
$(q,u)$ as above with $q=s$, i.e., the space of maps $u: \Sigma_s \to M$ satisfying
the conditions above.  
\end{defn} 
The virtual dimension of these moduli spaces $\mathcal{M}_s(\v_1, \v_2; x_0)$ is \begin{align} \operatorname{vdim}(\mathcal{M}_s(\v_1, \v_2; x_0))=  \deg(x_0)- 2\sum_{i=1}^k (1-a_i) v_{1,i}- 2\sum_{i=1}^k (1-a_i) v_{2,i}. \end{align}  For generic choices of $J_{\Sigma_{q}}$, these are moduli spaces of the expected dimension by arguments similar to those in Lemma 4.14 of \cite{GP1}. 

As a degenerate case of this (when $q = 0$) we study maps from $\Sigma_0$. For what
follows, write $\Sigma_0 = \Sigma_{sphere} \cup_{z' = z_0} S$, where $\Sigma_{sphere}$ is the
sphere bubble containing $z_1$, $z_2$ and another (nodal) point $z'$ and
$S= (\mathbb{C}P^1 \backslash \{0\},z_0)$ is the usual domain for $\PSS_{log}$ (which is
connected to $z'$ at the point $z_0$).
\begin{defn}\label{psscollidedmarkedpointsghost}
    Define $\mathcal{M}_0(\v_1,
    \v_2; x_0)$ to be the moduli space of maps $u:
    \Sigma_0 \to M$ satisfying the following conditions: 
    \begin{align}
        &u|_{\Sigma_{sphere}}\textrm{ is a constant map};\\
        &u|_{S}\textrm{ is an element of the moduli space $\mathcal{M}(\v_1 + \v_2; x_0)$}.
    \end{align}
\end{defn}
Using this $(q=0)$ moduli space, we can state the key result about the Gromov compactification of $\mathcal{M}(\v_1, \v_2; x_0)$:
\begin{lem}\label{lem:nospheresinproducthomotopy}
    If $w(x_0) = w(\v_1) + w(\v_2)$ and the constants $\epsilon_{\ell}$, etc.
    are chosen to be small as in Lemma \ref{lem:energypss}, then the Gromov
    compactification of $\mathcal{M}(\v_1, \v_2; x_0)$ 
    \begin{equation}
        \overline{\mathcal{M}}(\v_1, \v_2 ; x_0)
    \end{equation}
    has codimension-1 boundary covered by the images of the natural inclusions:
    \begin{align}
        \overline{\mathcal{M}}_b(\v_1, \v_2 ; x_0) &\to \partial \overline{\mathcal{M}}(\v_1, \v_2 ; x_0)\  (q \to b)\\
        \overline{\mathcal{M}}_0(\v_1, \v_2 ; x_0) &\to \partial \overline{\mathcal{M}}(\v_1, \v_2 ; x_0)\  (q \to 0)\\
        \overline{\mathcal{M}}(\v_1, \v_2 ; x_1) \times \overline{\mathcal{M}}(x_1, x_2) &\to \partial \overline{\mathcal{M}}(\v_1, \v_2 ; x_0).
    \end{align}
\end{lem}
\begin{proof}
    The proof involves identical analysis as in Lemma \ref{lem:compactness}
    (except that there are new boundaries to consider when $q \to 0,b$ and no
    Morse flowlines). In particular, the condition that $w(x_0) = w(\v_1) +
    w(\v_2)$ implies, under the smallness hypotheses of Lemma
    \ref{lem:energypss} (and by a straightforward adaptation of the energy
    approximation in Lemma \ref{lem:energypss} to $\mathcal{M}(\v_1, \v_2 ;
    x_0)$) that the topological energy of any $u \in \mathcal{M}(\v_1, \v_2 ;
    x_0)$  is smaller than the minimal energy of a
    non-constant $J$-holomorphic sphere in $M$. Hence, there can be no sphere
    bubbles that arise, including when $q \to 0$ and the domain $\Sigma_q$
    degenerates into the stable domain $\Sigma_0$, meaning any limiting map from
    $\Sigma_0$ must be a constant map on the sphere component $\Sigma_{sphere}$, hence an
    element of $\mathcal{M}_0(\v_1, \v_2; x_0)$.  
\end{proof}

We further equip the marked points $z_1$ and $z_2$ with asymptotic markers, in the positive real direction at $z_1$ and in the negative real direction at $z_2$. As before, doing so gives rise to evaluation maps: 
\begin{align} 
    \operatorname{Ev}_{z_{1}}^{\v_1}: \mathcal{M}(\v_1, \v_2 ; x_0) \to \mathring{S}_I \\ 
    \operatorname{Ev}_{z_{2}}^{\v_2}: \mathcal{M}(\v_1, \v_2 ; x_0) \to \mathring{S}_J,  
\end{align} 
where $I$ denotes the support of $\v_1$ and $J$ denotes the support of $\v_2$.
There are natural embeddings $\mathring{S}_I, \mathring{S}_J \hookrightarrow M^{log}$ which allow us to view both of the above evaluation maps as landing in $M^{log}.$  Given a limiting map from the stable domain $\Sigma_0$ (which as we have seen above, intersects $\D$ at $z_0$ with multiplicity $\v_1+\v_2$), we set 
\begin{align} 
    \operatorname{Ev}^{\v_i} _{z_i}(u)=\operatorname{Ev}^{\v_1+\v_2}_{z_0}(u) 
\end{align} 
viewed as a point in $M^{log}$. Doing this allows us to extend the evaluation
map over points in $\overline{\mathcal{M}}(\v_1, \v_2 ; x_0)$  where $q=0$:
\begin{align} \label{eq:evz1} 
    \operatorname{Ev}_{z_{1}}^{\v_1}: \overline{\mathcal{M}}(\v_1, \v_2 ; x_0)
    \to M^{log} \\ 
    \label{eq:evz2} \operatorname{Ev}_{z_{2}}^{\v_2}:
    \overline{\mathcal{M}}(\v_1, \v_2 ; x_0) \to M^{log} 
\end{align} 

\begin{lem} \label{lem: evalKN} 
Assume that $\v_2=\v_J$ is primitive.
\begin{itemize} 
    \item[(i)]  If $\v_1=\v_I$ is primitive, then the evaluation maps \eqref{eq:evz1}, \eqref{eq:evz2} are continuous.  
    \item[(ii)] If $J \subseteq I$, then the evaluation \eqref{eq:evz1} is continuous. 
\end{itemize} 
\end{lem} 
\begin{proof}  
    For what follows, let $K = I \cup J$ denote the support of $\v_1 + \v_2$.
    Consider a convergent sequence of Floer curves $u_q: \Sigma_q \to M$, with $q
    \to 0$ with limit $u_0$.  We first discuss assertion (i), where
    $\v_1=\v_I$. To simplify our notation, we may assume that, after reordering
    the divisors, $K=\lbrace 1,\cdots, |K| \rbrace.$ Choose the data of a chart
    $W \subset \C^n$, $W \stackrel{\phi}{\hookrightarrow} M$ (as usual we suppress $\phi$)
    centered about $u_{0}(z_0)$, such that the divisors are sent to standard
    linear subspaces: $D_k \cap W =
    \{y_k = 0\} \cap W$ for any $k \in K$. We can
    assume for specificity that $W:= (B^2(0))^{|K|} \times \bar{W} \subset
    \mathbb{C}^{|K|} \times \mathbb{C}^{n-|K|}$ and that
    $W^{log}:=(B^2(0)^{log})^{|K|} \times \bar{W}$.

There exists an open set $U_S \subset S$ containing $y=-q$ and $y=0$ (recall that $y=z^{-1}$ from \eqref{eq:coordinfty})
with $u_q(U_S) \subset W$ for all $q$ sufficiently small. We will show that
$\operatorname{Ev}_{z_{0}}(u_0)$ is the limit as $q \to 0$ of the evaluations
$\operatorname{Ev}_{z_{i}}$ viewed as lying inside of $W^{log}$.
It suffices to prove this convergence in each of the
$(B^2(0)^{log})^{|K|}$ factors as the other factors are not modified by the blow-up and hence there is nothing to check. To this end, let $\pi_k$ denote the coordinate projections
$\pi_k: W \to \mathbb{C}$ and set $\bar{u}_{k,q}=\pi_k\circ u_q: U_S \to
\mathbb{C}$. We first consider the case when $k \in I \cap J$. 
\cite[Lemma 3.4]{Ionel:2003kx} allows us to determine the leading order terms of the
Taylor expansions of  $\bar{u}_{k,q}$: 
\begin{align} 
    \bar{u}_{k,0}= b_0y^2 +O(|y|^3) \\ 
    \label{eq:expansion} \bar{u}_{k,q}=a(q) y+O(|y|^{2}), \, q \neq 0 
\end{align} 
where $b_0,a(q) \neq 0$.  Because there is no bubbling, $C^{\infty}$
convergence together with the fact that $\bar{u}_{k,q} (-q)=0$ implies that 
\begin{align} 
    -a(q) q + b(q) q^2+O(q^3)=0 
\end{align} 
with $b(q) \to b_0$. In particular, we have used that the third order derivatives are
uniformly bounded in $q$ to conclude that the remainder term is $O(q^3)$. It
therefore follows that $\lim_{q\to 0}a(q)/q = b_0$. Performing a Taylor
expansion about $z_1$ instead and writing
$\bar{u}_{k,q}(y)=\tilde{a}(q)(y+q)+O(|y+q|^2)$, the same reasoning shows that
$\lim_{q \to 0} \tilde{a}(q)/q = -b_0$. In the case that $k \in J$ but not in
$I$ we have  that \eqref{eq:expansion} holds, this time when $q=0$ as well.
$C^{\infty}$ convergence of the maps $u_q$ imply that that $a(q) \to a_{0} \neq
0$ and $\lim_{q\to 0}\bar{u}_{k,q}(z_1) /q = a_0.$ 

Assertion (ii) of the Lemma follows similarly. Suppose that $k \notin J$, then
the result follows immediately by $C^{\infty}$ convergence, so we consider the
other case. Then, employing the notation from the previous paragraph and taking
$q \to 0$, we have 
\begin{align} \label{eq:expansion2} 
    \bar{u}_{k,q}=a(q)(y+q)^{\v_1} + b(q)(y+q)^{\v_1+1}+O(|y+q|^{\v_1+2}) 
\end{align}
As before we have that $a(q)/q \to -b_0$. 
\end{proof}

 We now incorporate Morse flow-lines into the picture.  To formulate our next
 definition, fix two semi-infinite incoming edges 
 \begin{equation}\label{eq:semiinfiniteedges}
     \vec{e}_1,\vec{e}_2 \cong [0,\infty).
 \end{equation}  
 When dealing with moduli spaces where  $I=J=K$, we fix
 perturbation data $X_{\vec{e}_{i}}: \vec{e}_i \to C^{\infty}(S_K^{log})$
 on each edge in order to assure that regularity holds.

\begin{defn} \label{defn:auxmod} 
    Fix two critical points $c_1 \in \mathcal{X}(\mathring{S}_I, f_I)$, $c_2 \in \mathcal{X}(\mathring{S}_J, f_J)$. 
    Let 
    \[ 
        \mathcal{M}(\v_1, \v_2 ; c_1, c_2 ; x_0)
    \] 
    denote the moduli of triples $((q, u),\phi_1,\phi_2)$, where 
    $(q,u) \in \mathcal{M}(\v_1, \v_2; x_0)$ is an element of the moduli space from Definition \ref{psstwomarkedpoints},  
    $\phi_1: \vec{e}_1 \to S_I^{log}$ and $\phi_2: \vec{e}_2 \to S_J^{log}$ solve the
    gradient flow equation for $f_I$ and $f_J$ respectively (perturbed by
    $X_{\vec{e}_{i}}$ in the case $I=J=K$), satisfying the following additional
    asymptotics and incidence conditions:
    \begin{align}
  & \lim_{t_{\vec{e}_{i}} \to \infty} \phi_i(t_{\vec{e}_{i}})=c_i; \\
 & \phi_1(0) \in \mathring{S}_I; \\
 & \phi_2(0) \in \mathring{S}_J;\\
& \operatorname{Ev}_{z_{i}}^{\v_i}(u)=\phi_i(0).
\end{align}
\end{defn}
See Figure \ref{logpssringmap} item (4) for a pictorial schematic of this parametric (in $q$) family of moduli spaces, along with items (3) and (5) for a schematic of its $q \to b$ and $q \to 0$ limits (and note the caveats in the caption about incorporating incidence and enhanced evaluation conditions).

Let us assume that $f_{J}: \mathring{S}_{J,m} \to \mathbb{R}$ is chosen to have a unique critical point of index 0, which we denote by $c_u$. In view of the defining relations given in \eqref{eq:rel2}, \eqref{eq:rel1}, it suffices to consider only those moduli spaces where $\v_2=\v_J$ is primitive supported on $J$ and furthermore that either 
\begin{enumerate} [label=A.\arabic*]
    \item \label{itm:caseprod1}  $\v_1=\v_I$ is primitive supported on $I$, or 
    \item \label{itm:caseprod2} $c_2=c_u$ and $J \subseteq I$, and $\v_1 = \v$ is an arbitrary multiplicity vector supported on $I$.
\end{enumerate}  
We will do this for the remainder of this section without further mention. As
before, we choose our data so that these spaces are cut out transversally. Next
consider the moduli space $\mathcal{M}_b( M, \v_1, \v_2 ; c_1, c_2 ; x_0)$
which is the restriction of the above moduli space to domains $\Sigma_{b}$ where $q=b$ (and hence the marked point $z_1$ is located at $-1/b$). See Figure \ref{logpssringmap} item (3) for a schematic picture of this moduli space. By
counting rigid elements $u \in \mathcal{M}_b(M, \v_1, \v_2 ; c_1, c_2 ; x_0)$,
we may define a map 
\begin{align} \label{eq:auxprod} 
    H^*(\mathring{S}_I) t^{\v_1} \otimes H^*(\mathring{S}_J)t^{\v_2} \to
    SH^*(X)  
\end{align}

\begin{lem} \label{lem: connectsumcob} 
    The operation defined in \eqref{eq:auxprod} agrees with $\operatorname{PSS}_{log}^{low}(\alpha_1t^{\v_1}) \cdot \operatorname{PSS}_{log}^{low}(\alpha_2t^{\v_2})$. 
\end{lem}
\begin{proof}  
This is a very standard gluing argument in TQFT and so we will only sketch
the main idea. Let $\Sigma$ denote the pair of pants, with three standard
cylindrical ends attached. We consider the broken 
domain $\Sigma_{split}$ of the form
$$S \cup_\epsilon \Sigma \cup_\epsilon S $$
where the negative cylindrical ends of $S$ are glued to the positive cylindrical
ends of $\Sigma$ and gradient flow lines are attached at the two marked points.  Maps from $\Sigma_{split} \to M$ (glued to appropriate gradient flow lines) are given by the fiber product of
moduli spaces: 
 \begin{align} 
     \label{eq:prodopss} 
     \coprod_{x_1,x_2} \mathcal{M}( \v_1, c_1, x_1) \times \mathcal{M}(\Sigma, x_0, x_1, x_2 )
     \times \mathcal{M}( \v_2, c_2, x_2) 
 \end{align} 

We construct a homotopy between this moduli space and $\mathcal{M}_b( M, \v_1,
\v_2; x_0)$ in two steps. First, we perform a finite connect sum along the
cylindrical ends, producing a domain with two distinguished marked points and
one negative cylindrical end. The precise complex modulus of this domain and
the Floer datum over it are determined by the gluing parameter. Then, we can
further homotopy the complex structure and Floer datum to the domain $\Sigma_{b}$
above.  We thus reach the desired conclusion.  

See Figure \ref{logpssringmap} item (2) for a pictorial schematic of
the above parametric family of moduli spaces with and items (1) and (3) for its
codimension 1 boundary (modulo differential/chain homotopy terms); Note the caveats
about incorporating incidence and enhanced evaluation conditions described in
the caption. Note also that the picture rather describes the finite connect sum
portion of this homotopy; we leave the further homotopy of complex structure
and Floer data implicit, or --- seeing as the pictures don't describe the
complex structure or Floer data --- we could imagine performing this homotopy
simultaneously with the connect sum.
\end{proof}

Let $T_{2,0}$ denote the graph consisting of two semi-infinite edges $\vec{e}_1$ and
$\vec{e}_2$ as \eqref{eq:semiinfiniteedges} which are glued along their vertices. Let 
$\mathcal{M}_V(c_1,c_2)$ denote the moduli space of ``$V$-shaped trajectories
on $S_I$ and $S_J$''; that is, continuous maps $\phi: T_{2,0} \to M^{log}$
where along each edge $\phi_i=\phi_{|\vec{e}_i}$ is a (possibly perturbed)
gradient trajectory of $f_I$ (if $i=1$) respectively $f_J$ (if $i=2$) in the
appropriate stratum ($S_I$ or $S_J$), asymptotic to $c_i$ at infinity, such
that $\phi_1(0)=\phi_2(0) \in \mathring{S}_K$.

\begin{defn} \label{defn:fibprod2} 
    Define the moduli space $\partial_V\overline{\mc{M}}(\v_1+\v_2,c_1, c_2,
    x_0)$ to be the space of pairs $(\phi,u)$, $\phi \in
    \mathcal{M}_V(c_1,c_2)$, $u \in \mc{M}(\v_1+\v_2, x_0)$ with
    $\operatorname{Ev}^{\v_1+\v_2}_{z_{0}}=\phi_i(0)$.   
\end{defn}
See Figure \ref{logpssringmap} item (5) for a pictorial schematic of
this moduli space (note the caveats about incorporating incidence and enhanced
evaluation conditions described in the caption).

\begin{lem} \label{lem:prodcompq0} 
Suppose that $\v_2=\v_J$ holds and furthermore either \ref{itm:caseprod1}
or \ref{itm:caseprod2} hold. For generic choices, when the moduli space
from Definition \ref{defn:auxmod} is 1-dimensional, it may be completed
over $q \to 0$ to a 1-dimensional manifold with boundary whose fiber
over $q=0$ is given by $\partial_V\overline{\mc{M}}(\v_1+\v_2,c_1, c_2,
x_0).$  
\end{lem} 
\begin{proof} 
Suppose first that we are in case \ref{itm:caseprod1} above with $\v_1=\v_I$.
It follows from Lemma \ref{lem: evalKN} that given a $u_0$, which is the limit
of curves $u_q \in \mathcal{M}(\v_1, \v_2 ; c_1, c_2 ; x_0)$, 
\begin{align}
    \label{eq:intersection} \operatorname{Ev}_{z_{0}} \in W_K^s(f_I, c_1) \cap
    W_K^s(f_J, c_2) 
\end{align} 
As in Definition \ref{defn:auxmod}, in the case where $I=J=K$, this should be
interpreted as stable manifolds of the perturbed gradient flow. To complete the
proof of the lemma (in the case \ref{itm:caseprod1}), requires that the
topology near a curve satisfying \eqref{eq:intersection} may be identified with
that near an endpoint of a closed interval. This is an elementary gluing result
(involving gluing in a constant sphere bubble) which follows the arguments of
Chapter 10 of \cite{McDuff:2004aa} quite closely and which we regard as
standard. 

It remains to consider the remaining cases when $c_2=c_u$, in which case the
stable manifold $W^s(f_J,c_2)$ has closure all of $S_{J,m}^{log}$ and
$S_{J,m}^{log}\setminus W^s(f_J,c_2)$ has codimension 1. In this case, by
choosing our data generically, we may assume that an element of the moduli
space $\mc{M}(\v_1+\v_2, W^s_{K}(f_I, c_1))$ is disjoint from
$S_{J,m}^{log}\setminus W^s(f_J,c_2)$. The rest proceeds as in the preceeding
paragraph.  
\end{proof}

\begin{lem} \label{lem:PSSofproduct}
    Suppose that $\v_2=\v_J$ holds and furthermore either \ref{itm:caseprod1} or \ref{itm:caseprod2} hold. The operation defined in \eqref{eq:auxprod} is equal to $\operatorname{PSS}_{log}^{low} (\alpha_1t^{\v_1} \cdot \alpha_2t^{\v_2})$ 
\end{lem} 
\begin{proof} 
This follows another homotopy argument. We can consider configurations
depending on a parameter $p \in [0,\infty]$. Let us first describe the case
where both $I,J \neq K$. For any such $p$, let $T_{2,0}^{(p)}$ denote the
metrized tree with two incoming external edges $\vec{e}_1$,$\vec{e}_2$, two
internal edges $\bar{e}_1$, $\bar{e}_{2} \cong [0,p]$ of length $p$ and one
outgoing edge $\bar{e}_{0}$ of length $p$. We glue these edges by setting
\begin{itemize} 
    \item $\vec{e}_1(0)=\bar{e}_1(0)$, $\vec{e}_2(0)=\bar{e}_2(0)$
        and 
    \item $\bar{e}_{1}(p)=\bar{e}_{2}(p)=\bar{e}_{0}(0)$. 
\end{itemize}
In the limit as $p=0$, this becomes $T_{2,0}$ above. For all $p$, we equip this
tree with perturbation data $e \to C^{\infty}(S_K^{log})$ on each of the finite
edges $\bar{e}_{1}, \bar{e}_{2}, \bar{e}_{0}$ which is supported in a
neighborhood of the vertex $\bar{e}_{0}(0)$. A $T_{2,0}^{(p)}$ shaped flow line
is a continuous map $\phi: T_{2,0}^{(p)} \to M^{log}$ where $\phi$ restricted
to each of the finite edges $\bar{e}_{1}, \bar{e}_2, \bar{e}_{0}$ are perturbed
gradient trajectories of $f_K$ and $\phi_1=\phi_{|\vec{e}_1}$ and
$\phi_2=\phi_{|\vec{e}_2}$ are gradient trajectories of $f_I$ and $f_J$
respectively. Note that this implies in particular that we have that
$\phi(e_i(0)) \in S_K^{log}$.

Let $\mathcal{M}_p(\v_1+\v_2, c_1,c_2, x_0)$ denote the moduli space of pairs
$(\phi,u)$ where $\phi$ is a $T_{2,0}^{(p)}$ shaped flow line and $u$ is a
solution in $\mathcal{M}(\v_1+\v_2, x_0)$ such that
$\operatorname{Ev}^{\v_1+\v_2}_{z_{0}}=\phi(\bar{e}_0(p))$. As $p \to \infty,$
solutions limit to broken curves which define the composition $\operatorname{PSS}_{log}^{low} (\alpha_1t^{\v_1} \cdot \alpha_2t^{\v_2})$.  When $p=0$, this moduli
space reduces to the moduli space of Definition $\ref{defn:fibprod2}$. This
gives rise to a cobordism (up to boundary components which give rise to chain homotopy terms) between these two moduli spaces.  
See Figure \ref{logpssringmap} item (6) for a pictorial schematic of
this parametric family (in $p$) of moduli spaces along with items (5) and (7)
for its $p = 0$ and $p \to \infty$ limits (with the caveats about incorporating
incidence and enhanced evaluation conditions described in the caption).
Lemma \ref{lem:prodcompq0} in turn gives rise to a cobordism (again up to boundary terms that contribute to a chain homotopy)
between the moduli space at $p=0$ and the moduli spaces which define \eqref{eq:auxprod}.  Thus the operation defined in \eqref{eq:auxprod} is equal to $\operatorname{PSS}_{log}^{low} (\alpha_1t^{\v_1} \cdot \alpha_2t^{\v_2})$ (on the level of cohomology).

The other cases where either $I$ or $J$ coincide with $K$ are an easy
modification of the above argument. To handle it, we modify the graph
$T_{2,0}^{(p)}$ above by keeping the edges $\vec{e}_1$,$\vec{e}_2$, $\bar{e}_0$
the same as above, but only introducing an edge $\bar{e}_1$ (respectively
$\bar{e}_{2}$) when $I$ (respectively $J$) is not equal to $K$. The rest
proceeds as before.  
\end{proof}

\begin{proof}[Proof of Theorem \ref{lem:spectralrings}]
    Combine Lemmas \ref{lem: connectsumcob} and \ref{lem:PSSofproduct} to verify
    the low energy log PSS map is compatible with ring structure for inputs of the
    form specified in (i) and (ii) of Lemma
    \ref{lem:generatingrelationsfortheproduct}; the aforementioned Lemma
     then implies compatibility for all inputs.  
\end{proof}

\subsection{Separating actions}
\label{sect: PSSiso1}

In this subsection, we construct a truly ``local" version of the $\PSS_{log}$ map, with source $H^*(\SIo)t^{\v}$ and target a certain Hamiltonian Floer group generated entirely by orbits with winding vector $\v$ (this local map is defined in \eqref{eq:realPSSloc}). As stated in Lemma \ref{lem:localizeisomorphism}, the map  $\eqref{eq:locPSS}$ is an isomorphism iff for every $\v$ with $w(\v) \leq w$, \eqref{eq:realPSSloc} is an isomorphism. The construction is accomplished in two steps, which we briefly summarize. The first step is to vary our symplectic form slightly on $M$ so that:
\begin{itemize}
\item the resulting weight of a cohomology class in $F_{w_{\ell}}\QH^*(M,\D)$ (recall \eqref{eq:filterlogchainys}) uniquely determines the
multiplicity vector $\v$.
\end{itemize}
Thus, the new weight filtrations on log cohomology will be finer than the ones considered in previous sections. One can define the low energy log PSS map with respect to this new symplectic form. The map is \emph{a priori} different, as modifying the symplectic form changes the Hamiltonian vector field (though in this case, the actual time-one orbits will not change for a sufficiently small perturbation). However, much of the work in this subsection goes into showing that (see Lemma \ref{lem: PSScoinc}): 

\begin{itemize}
\item the low energy log PSS map, defined with
respect to this perturbed symplectic form, agrees with the previously defined log PSS map.
\end{itemize}

The second step is to then perturb the Hamiltonian to a Hamiltonian $H_p^{\ell}$ (see the discussion following \eqref{eq:hpellthing}), so that the non-constant orbits of $H_p^{\ell}$ are even closer to the divisor. This is done in such a way that the action of a Hamiltonian orbit $x_0$ of $H_p^{\ell}$ is arbitarily close the perturbed weight of $\v(x_0)$. This will ensure that:

\begin{itemize}
  \item the (families of) Hamiltonian orbits (below a fixed slope) we consider below will all
have distinct actions.
\end{itemize}

Thus, the Hamiltonian Floer groups will acquire a refined filtration which matches the refined filtration on log cohomology. The map \eqref{eq:realPSSloc} is then constructed by passing to a suitable associated graded of the low energy log PSS map with respect to these refined filtrations.

We now turn to describing all of this more precisely.  Fixing an $\ell$ in our sequence, we will perturb the
divisorial weights $\kappa_i$ (as they appear in $[\omega]$
and the formula \eqref{eq:windingvec} for $w(\v)$) 
to rational numbers $\kappa_{i,\ell} \in \mathbb{Q}$ very close to $\kappa_i$,
which in turn leads to a
{\em perturbed weight} of a vector $\v$
\begin{equation}\label{perturbedweight}
    w_p(\v):= \sum \kappa_{i,\ell}v_i \in \mathbb{Q}
\end{equation}
(this depends on $\ell$, though we have suppressed that from the notation).
In terms of this perturbed weight, we require our perturbation $\{\kappa_{i,\ell}\}$ to separate
actions below level $w_\ell$ in the sense that: 
\begin{equation}\label{eq:separateactions}
    \textrm{Given any $\v_1 \neq \v_2$ with $w(\v_i) \leq w_\ell$,  $w_p(\v_1) \neq w_p(\v_2).$} 
\end{equation}
We will also want to assume that our perturbations $\kappa_{i,\ell}$ are sufficiently close to $\kappa_i$,
so that
\begin{align} 
    &\textrm{if $w(\v) \in [w_\ell+1,\infty)$, then $w_p(\v) \in [w_\ell+1-\delta_p, \infty)$ for some very small $\delta_p$}.\\
    &\textrm{If $w(\v) \in [0,w_\ell]$, then $|w(\v)-w_p(\v)|<\delta_p$}. 
\end{align} 
The first condition roughly says that we don't necessarily require the
perturbed weight $w_p(\v)$
to be that close to $w(\v)$ above the
critical action window $[0,w_{\ell}]$; in that range the perturbation simply needs to stay out of the
critical action window.

To define our perturbed symplectic form, choose two very small constants
$\underline{\epsilon}_\ell^{p},\bar{\epsilon}_\ell^{p}$ 
with
$\underline{\epsilon}_\ell^{p} <\bar{\epsilon}_\ell^{p} \ll \epsilon_\ell$ (recall $\epsilon_\ell$ as defined just below \eqref{eq:defnsigmaell} was fixed earlier). 
Fix a monotone increasing
function $g_{p}: [0, \epsilon_\ell] \to [-1,0]$ with 
\[
    g_p(x) = \begin{cases} 
                -1 & x \in [0,(\underline{\epsilon}_\ell^{p})^2]\\
                0 & x \in [(\bar{\epsilon}_\ell^{p})^2,\epsilon_\ell].
            \end{cases}
\]
Because the Hamiltonian flow of $X_{\Hlm}$ preserves the divisors, it follows
that on any Hamiltonian orbit (either divisorial or in $X$), either $\rho_i=0$
or $\rho_i$ is bounded away from zero, by say some $\tau_i > 0$. We can therefore
assume that by choosing our constants $\underline{\epsilon}_\ell^{p}
<\bar{\epsilon}_\ell^{p}$ suitably (i.e., so that $\bar{\epsilon}_\ell^{p} <
\min_i \tau_i$), 
\begin{equation}\label{supportdg}
    \begin{split}
    &\textrm{the functions $\frac{dg_p(\rho_i)}{d\rho_i}$ vanish in open neighborhoods of }\\
    &\textrm{the periodic orbits of $X_{\Hlm}$. }
    \end{split}
\end{equation}
We choose $\kappa_{i,p}^{\ell} \in \mathbb{Q}$ (so that $\kappail=\kappa_i+\kappaipert^{\ell}$ are also rational)
and set 
\begin{align} 
    \theta_\ell = \theta +\sum_i \frac{\kappaipert^{\ell}}{2\pi} g_{p}(\rho_i) \theta_{e,i} 
\end{align}  
Here we use the coordinates $\rho_i$ and forms $\theta_{e,i}$ which are part of the regularization data chosen in \S 2.1. 
Finally, we assume that $\kappaipert^{\ell}$ are sufficiently small so that  
\begin{align} 
    \omega_\ell:= d \theta_\ell 
\end{align} 
remains symplectic. The above condition on the support of
$\frac{dg_p(\rho_i)}{d\rho_i}$ implies that: 
\begin{equation}\label{eq:vanishingperturbation} \tag{$\star$} 
    \begin{split}
        &\textrm{the perturbation $\alpha=\omega_\ell-\omega$ vanishes in an open neighborhood}\\
        &\textrm{of all the time-1 orbits of the Hamiltonian vector field $X_{\Hlm}$}. 
    \end{split}
\end{equation}
This will be used in Lemma
\ref{lem:estvarsym} below. Moreover, it is easy to see that for
$\kappaipert^{\ell}$ sufficiently small, all time-1 Hamiltonian orbits of $\Hlm$ with respect to the two symplectic forms $\omega$ and $\omega_\ell$ coincide
 and the almost-complex structure $J_S$ used to define \eqref{eq:locPSS} is
tamed by both $\omega$ and $\omega_{\ell}$ (recall that being tamed is an open condition in the symplectic form).

 Let $X^{p}_{\Hlm}$ denote the Hamiltonian vector field of $\Hlm$ with respect
 to $\omega_\ell$.  Fix $J_S$ used to define the map \eqref{eq:locPSS}, and
 consider solutions to the equation 
 \begin{align} 
     (du-X^{p}_{\Hlm}\otimes \beta)^{0,1}=0. 
 \end{align} 
 Define the moduli space of PSS solutions with respect to $J_S$ $$\mc{M}(\v,c,
 X^{p}_{\Hlm},x_0)$$ as in Definition \ref{def: higher}.   Using the obvious modification of \eqref{eq:PSSdef}, we may define
 a map 
 \begin{align} \label{eq:anotherrandomPSSformula}
     \PSSlog^{\v,p}: H^*(\SIo)t^{\v} \to  HF^*(X \subset M ;
     \Hlm)_{w(\v)}, 
 \end{align} 
 where $HF^*(X \subset M ;
     \Hlm)_{w(\v)}$ is the low energy Floer cohomology from \eqref{eq:lowlow}. The map is well-defined because our complex structure remains of contact type near the boundary of our Liouville domain $\Sigma_{\ell}$ (recall \eqref{eq:defnsigmaell}) and hence the analysis needed to prove the compactness
 result of Lemma \ref{lem:compactness} goes through unchanged. Also implicit in \eqref{eq:anotherrandomPSSformula}
 is the fact that the (filtered and quotiented) Floer group of $\Hlm$ with respect to $\omega_\ell$ agrees with $HF^*(X
 \subset M ; \Hlm)_{w(\v)}$ canonically. This is because the symplectic forms $\omega$ and $\omega_\ell$ both agree on
 $\overline{X}_\ell$ and by the maximum principle, Floer trajectories do not escape this region. 
While the Floer groups coincide automatically it is not a priori clear that the
maps $\PSSlog^{\v,p}$ and $\PSSlog^{\v}$ coincide.  To prove this, we let
$\kappa(s) \geq 0$ be a nondecreasing function such that 
\begin{itemize} 
    \item $\kappa(s)=0$, for $s \leq -1/2$ 
    \item $\kappa(s)=1$ for $s \geq 1/2 $ 
\end{itemize} 
The symplectic forms $\omega_{\ell}$ are connected by a one parameter family of symplectic forms 
\begin{align} 
    \omega_s=\omega+\kappa(s)\alpha 
\end{align}  
The members of this family of symplectic forms all agree in a neighborhood of
the orbits, which as before enables us to identify Floer cohomologies $HF^*(X
\subset M ; \Hlm)_{w(\v)}$. For each $s$, let $X_{\Hlm}(s)$ denote the
Hamiltonian vector field taken with respect to the symplectic form $\omega_s$. 

For what follows, recall that $\rho(r)$ denotes a cutoff function as in
\eqref{eq:cuttoff}.
\begin{defn} 
    Fix a complex structure $J_S \in J_{S,\ell}(V)$, an orbit $x_0 \in U_\v$ of $X_{\Hlm}$, and a critical point of $c \in \chifI$. 
    For a parameter $q \in \mathbb{R}$, define the moduli space $\mathcal{M}_q(\v, c, x_0)$ to be the set of solutions 
    $u: S \to M$ to the following differential equation 
\begin{align} 
    \label{eq: PSSq} (du- \rho(s-q)X_{\Hlm}(s)\otimes dt)^{0,1}=0 
\end{align}  
with asymptotics \eqref{eq:limitingcondition}, incidence/tangency conditions \eqref{eq:incidence},  and 
enhanced evaluation constraint $\Evzo(u) \in W^s(f_I,c)$.
\end{defn} 
By choosing $J_S$ generically, one can ensure that the parameterized moduli space of pairs $(q,u)$, $u \in \mathcal{M}_q(\v, c, x_0)$ is cut out transversely and also that the same holds for $\mathcal{M}_q(\v, c, x_0)$ when $q \in \mathbb{R}$ is chosen generically. For $q\ll 0$, this equation coincides with the standard PSS solution moduli
space for the symplectic form $\omega$ (note that the cutoff function
$\rho(s-q)$ is only non-zero when $s \ll q \ll 0$, in which region
$X_{H^\ell}(s)$ is simply the usual $X_{H^\ell}$ computed with respect to
$\omega$). 

We define the geometric energy of such a solution $u \in \mathcal{M}_q(\v, c,
x_0)$ to be the usual (c.f. \eqref{eq:geo}) 
\begin{align} 
    E_{g}(u)=\int_S ||du- \rho(s-q)X_{\Hlm}(s) \otimes dt|| 
\end{align}
where the norm is taken with respect to the symplectic form $\omega_s$ and
the family of complex structures $J_S = \{J_z\}_{z \in S}$.
By compactness of $M$, the fact that our homotopy $\omega_s$ is compactly supported, and the fact that our space of complex structures $J_z$ depends only on $t$ along the cylindrical end, there
exists a smallest constant $|\alpha|$ such that for all $z \in S$ we have
\begin{align} \label{eq:baralphabound}
    |\alpha(v,w)| \leq |\alpha| |v|_z|w|_z 
\end{align} 
for any two
tangent vectors $v,w$ at any point $m \in M$. In this equation the norms $|v|_z$, $|w|_z$ are again the norms associated with $\omega_s$ and $J_z$.
Of course, the constant
$|\alpha|$ becomes smaller as the $\kappaipert^{\ell}$ all tend to zero.  We
may also make use of the following energy which is defined by analogy with the
topological energy (see  \eqref{eq:top}) 
\begin{align} 
    E_b(u) := \int_S \omega_s(\partial_s u,\partial_t u)-\int d(H\wedge \beta_q) 
\end{align}
where in this equation we have $\beta_q=\rho(s-q)dt$. As before, we have that \begin{align} \label{eq: gtopbound2} E_g(u) \leq E_b(u) \end{align} A slight complication is that this integral $E_b(u)$ is no longer topological, i.e. it depends on the curve $u$ and not just the multiplicity vector $\v$ and the output $x_0$.  Nevertheless, we still have:

\begin{lem} \label{lem:estvarsym} For any $u \in \mathcal{M}_q(\v, c, x_0)$, we have the following bounds 
    \begin{align} 
        \label{varsym:Egbound}        E_{g}(u)  \leq w(\v) + A_{\ell}(x_0) +C|\alpha|E_{g}(u) \\ 
        \label{varsym:Ebbound}        E_{b}(u)  \leq w(\v) + A_{\ell}(x_0) +C|\alpha|E_{b}(u)  
    \end{align} 
    for a constant $C$ which is independent of the curve $u$ or $q$. 
    Moreover the constant $C$ can be taken independent of the $\kappaipert^{\ell}$.  \end{lem} 
\begin{proof} 
These estimates are implicit in \cite{Zhang} (Proof of Theorem 1.6) and can be
treated in a very similar fashion to the proof of Proposition 1.5 of \emph{loc.
cit.} (see also \cite{Ritter2}).  For the reader's convenience and to
illustrate how the crucial assumption \eqref{eq:vanishingperturbation} on the
perturbation is used, we sketch how to modify the proof of Proposition 1.5 of
\emph{loc.cit}. 
First, note that
\begin{align} 
    E_b(u)=\int_S u^*(\omega)+\int_S\kappa(s)u^*\alpha - \int_{S}H_{s,t}dsdt.
\end{align} 
We have also seen that (as described in e.g., the proof of Lemma \ref{lem:energypss})
\begin{align} 
    \int_S u^*(\omega) - \int_{S}H_{s,t}dsdt = w(\v) + A_{\ell}(x_0).
\end{align} 
Thus, in view of \eqref{eq: gtopbound2}, the essential point is to
estimate $| \int_S \kappa(s) \alpha(\partial_su, \partial_t u)dsdt|$ for any
PSS solution $u$ in terms of the energy $E_g$ and $|\alpha|$. To begin, we have
the upper bound
\begin{align} \label{eq:middleofroad} 
    | \int_S \kappa(s) \alpha(\partial_su, \partial_t u)dsdt| \leq \int_S |\alpha(\partial_su, \partial_t u)|dsdt.
\end{align} 
Floer's equation and the triangle inequality allow us to further bound the right hand side of \eqref{eq:middleofroad} by
\begin{align} \label{eq:middleofroad2}
    \int_S|\alpha(\partial_su, \partial_t u)|dsdt \leq \int_S|\alpha(\partial_su, J\partial_s u)|dsdt + \int_S|\alpha(\partial_su, \rho(s-q)X_{\Hlm}(s)|dsdt. 
\end{align} 
Observe that the first term of the right hand side of \eqref{eq:middleofroad2}
is bounded by $|\alpha|E_g(u)$ (where $|\alpha|$ is the constant in
\eqref{eq:baralphabound}); hence the problem may be reduced to bounding the
second term of this equation. Let $s_*^0$ denote the minimum s for which
$\rho(s-q)=0$ and let $s_*^1$ denote the maximum $s$ for which $\rho(s-q)=1$.
Then
\begin{align}  \label{eq:middleofroad3}
    \int_S|\alpha(\partial_su, \rho(s-q)X_{\Hlm}(s)|dsdt = \int_{-\infty}^{s_0^*} \int_0^{1}|\alpha(\partial_su, \rho(s-q)X_{\Hlm}(s)|dsdt,
\end{align}  
seeing as the integrand vanishes for $s \geq s_0^*$.
By assumption \eqref{eq:vanishingperturbation}, there is an isolating set $\mathcal{N}_E$ about the
union of all period orbits where $\alpha$ vanishes. Let 
$$ \mathcal{I}:=\lbrace s \in (-\infty,s_1^*], u(s,-) \notin \mathcal{L}\mathcal{N}_E \rbrace. $$ 
Note that the integrand \eqref{eq:middleofroad3} is further $0$ when $s \notin \mathcal{I} \cup [s_1^*, s_0^*]$; hence
$$\int_{-\infty}^{s_0^*} \int_0^{1}|\alpha(\partial_su, \rho(s-q)X_{\Hlm}(s)|dsdt= \int_{\mathcal{I} \cup [s_1^{*}, s_0^{*}]} \int_0^{1}|\alpha(\partial_su, \rho(s-q)X_{\Hlm}(s)|dsdt. $$ 
Now 
$$|\alpha(\partial_su, \rho(s-q)X_{\Hlm}(s)| \leq C' |\alpha||\partial_su|$$ 
where $C'$ is the sup of the norm of the Hamiltonian vectorfields $X_{\Hlm}(s)$. Thus, 
$$ \int_{\mathcal{I} \cup [s_1^{*}, s_0^{*}]} \int_0^{1}|\alpha(\partial_su, \rho(s-q)X_{\Hlm}(s)|dsdt \leq C'|\alpha|\int_{\mathcal{I} \cup [s_1^{*}, s_0^{*}]} \int_0^{1}|\partial_su|dsdt $$ 
The crucial point is that, as shown by Zhang in \cite[Proof of Prop. 1.5]{Zhang}, the condition
\eqref{eq:vanishingperturbation} implies that the Lebesgue measure of
$\mathcal{I}$ is bounded by $E_g(u)/N$ for some constant $N$ which is
independent of the PSS solution.
Then by Cauchy Schwarz,
$$ (\int_{\mathcal{I} \cup [s_1^{*}, s_0^{*}]} \int_0^{1}|\partial_su|dsdt)^2 \leq (s_0^*-s_1^*+E_g(u)/N) E_g(u). $$ 
Putting all of this together, the first energy estimate \eqref{varsym:Egbound}
follows from the fact that either $E_g(u)$ is smaller than $w(\v) +
A_{\ell}(x_0)$ (in which case the estimate trivially holds) or there
will be some fixed constant for which $E_g \leq C''E_g^2$. The second
inequality \eqref{varsym:Ebbound} may be deduced similarly. Note also that since
$s_0^*-s_1^*$ is independent of $q$, so are the final estimates.
\end{proof} 

This provides the necessary energy bounds which are needed for Gromov compactness arguments to go through. Expanding on this, we have: 
\begin{lem} 
    Suppose that $|\alpha|$ is taken sufficiently small $\epsilon_\ell, ||\Hlm-\hlm||_{C^2}$  are both sufficiently small and $\Sigma_\ell$ is sufficiently $C^0$-close to $\hatXl$. Counting rigid index 0 solutions to
    \eqref{eq: PSSq} for a generic $J_S \in J_{S,\ell}(V)$ and a generic $q \in \R$ 
   defines a chain map, inducing the following cohomological map:  
    \begin{align}
        \label{eq:PSSqc} \PSSlog^{\v,q}: H^*(\SIo)t^{\v} \to  HF^*(X \subset M ; \Hlm)_{w(\v)}  
    \end{align}  
\end{lem}  
\begin{proof} 
    This again follows the pattern of Lemma 4.13 of \cite{GP1} 
(or Lemma
\ref{lem:compactness} of the present text) and to avoid being repetitive,
we only mention the two key points. The first is that can one use the
second equation of Lemma \ref{lem:estvarsym} to obtain a suitable
replacement for Equation (4.34) of \cite{GP1}. This allows one to handle
breaking along orbits in the divisor just as in \emph{loc. cit.} 

The second point concerns sphere bubbling: observe that this Lemma also implies
that when $|\alpha|$ is sufficiently small, $E_b(u)$ is still approximately
$w(\v)\epsilon_\ell^2/2$ and thus no sphere bubbling can occur at the point
$z_0$ as the minimal energy of any such non-constant sphere bubble is larger than
$E_b(u)$ (compare proof of Lemma \ref{lem:compactness}).  
It follows that given a sequence of curves $u_n$ limiting to a curve
$u_{\infty}$, the limiting configuration intersects $\D$ with multiplicity $\v$
at the point $z_0$ (here we identify the domain $S$ as one component of such a
conjectural limiting configuration). To rule out sphere bubbling elsewhere, 
note that the $\omega_s$-energy of the union of any putative sphere bubbles must
be positive and hence the collection of these sphere bubbles must intersect at
least one of the divisors $D_i$ positively. By positivity of intersection with
$\D$, any Floer cylinder or PSS solution must intersect $D_i$ with non-negative
multiplicity. This however contradicts the fact that the sum of all
intersections with $D_i$ away from $z_0$ must be zero.  It follows that no
sphere bubbling can occur. With these points in place, the rest of the proof
follows as in the proof of Lemma 4.13 of  \emph{loc. cit.}  
\end{proof}  

\begin{lem} \label{lem: continuationell} 
    Let $x_0,x_1$ be orbits in $U_\v$ such that $\deg(x_0)=\deg(x_1)$. Let $u$
    be a solution to \begin{align} \label{eq:continuationl}
        (du-X_{\Hlm}(s)\otimes dt)^{0,1}=0 \end{align} 
    For $|\alpha|$ sufficiently small, generic $J_t$, and
    $||H^{\ell}-h^{\ell}||$ sufficiently $C^2$ small, $u$ is s-independent.  
\end{lem}
\begin{proof} 
    Again assuming that we can show that solutions do not develop limits at
    orbits along $\D$, the usual compactness argument shows that all such
    solutions lie in $U_\v$ where the symplectic form is constant (and the
    result follows by noticing that, by applying $\R$-translations,
    $s$-dependent solutions come in non-rigid families). 
    In order to show solutions do not break along orbits in $\D$, 
    similarly to the construction of the map \eqref{eq:PSSqc}, we
    adapt the proof of Lemma \ref{lem: homotopycomp}. Note that this Lemma gets off the ground via the estimate in Equations \eqref{eq:homeest}-\eqref{eq:finalest216}. Using the notation of that Lemma, one can in the present situation use \cite[Proposition
    1.5]{Zhang} to obtain the following modified version of these equations: \begin{align} E_g(\bar{S}) \leq \frac {-A_{\ell}(x_1) + C|\alpha| A_{\ell}(x_0)}{1-C|\alpha|} + \int_{\partial \bar{S}}u^*\theta-\theta(X_{H_{s,t}})dt + \int_{\partial \bar{S}}\lambda^{\ell} dt \end{align} When $|\alpha|$ is sufficiently small, this is enough to carry out the rest of the argument of that Lemma to rule out breaking along $\D$. With this established, when $||H^{\ell}-h^{\ell}||$ sufficiently $C^2$ small,
    solutions to \eqref{eq:continuationl} are therefore just Floer trajectories
    and the only index 0 solutions are constant in $s$. 
\end{proof}

\begin{lem} \label{lem: PSScoinc} 
    The maps $\PSSlog^{\v}$ and $\PSSlog^{\v,p}$ (cohomologically) coincide.  
\end{lem} 
\begin{proof} 
    As observed above when $q\ll 0$, $\PSSlog^{\v,q}=\PSSlog^{\v}$. As $q \to
    +\infty$, the fact that Gromov compactness applies to our setting 
    shows that solutions tend to broken configurations consisting of PSS solutions 
    elements followed by cylinders solving \eqref{eq:continuationl}; appealing
    to Lemma \ref{lem: continuationell}, this broken moduli space can be
    expressed as:
    \begin{align} \label{brokenconstant}
        \mc{M}(\v,c, X^{p}_{\Hlm}, x_0) \times \widetilde{\mathcal{M}}(x_0,x_0). 
    \end{align} 
    The right hand side of the above product (which is the space of
    parametrized Floer trajectories prior to any quotient by $\R$) consists of
    a single trivial constant solution; in particular, the
    associated count gives the identity map, and the operation associated to
    counting elements of \eqref{brokenconstant} is simply $\PSSlog^{\v,p}$. It
    follows by standard methods and what we have seen previously that counts
    associated to the (rigid elements of the compactification of the)
    parametrized moduli space $\{ (q, u) | u \in \mathcal{M}_q(\v, c, x_0)\}$
    give, by looking at the boundary of the 1-dimensional components of the
    same moduli space, a chain homotopy between $\PSSlog^{\v}$ and $\PSSlog^{\v,p}$.
\end{proof} 

\begin{equation}\tag{$\maltese$}
    \begin{split} 
        &\textrm{{\em For the remainder of \S \ref{section: PSSreview} and all of \S \ref{sect: PSSiso2}, all Hamiltonian vector fields, Floer}} \\
        &\textrm{{\em cohomology groups, PSS maps, etc. will be defined using the symplectic form $\omega_\ell$}.}  
    \end{split}
\end{equation}

We now complete the process of separating out the Hamiltonian actions of our
orbits by perturbing our Hamiltonian to a new Hamiltonian $H_p^{\ell}$ so that
all time-1 Hamiltonian orbits with respect to $\omega_\ell$ are now supported
in the region where $g_{p}=-1$. 
Fix $\epsilon_\ell^{p}<\underline{\epsilon}_\ell^{p}$ and consider a new
Liouville boundary $\Sigma_{\ell}^{p}=\Sigma_{\vec{\epsilon}^{p}}$ which is
even closer
to the divisor by taking $\epsilon_1^{p}=\epsilon_\ell^{p}$ and rounding the
boundary as before. Next, fix functions \begin{align} \label{eq:hpellthing} h_p^{\ell}(R_p^{\ell}) \end{align} of the Liouville
coordinate $R_p^{\ell}$ (with respect to $\theta_\ell$) which are of slope
$\lambda_{\ell}$ and satisfy the same conditions as $h^{\ell}(\rlm)$ from Section
\ref{section:SHtor}. Finally, we perturb these Hamiltonians in small regions
about the Hamiltonian orbits which, in a slight abuse of notation we also denote by
$U_\v$, to obtain Hamiltonians $H_p^\ell$. 

For any $w_p \leq w_\ell+\delta_p$, we may form Floer cochain groups
$F_{w_{p}}CF^*(X \subset M; H_p^{\ell})$ and cohomology groups $\FpHF$ by
complete analogy with the unperturbed case. Continuation maps give rise to
isomorphisms: 
$$\FwHF \cong F_{w+\delta_{p}} HF^*(X \subset M; H_p^{\ell}) $$
These isomorphisms are compatible with the respective $\PSSlog^{\v}$ maps.
Shrinking $\epsilon_\ell^{p}$ as needed,  the complexes $F_{w_{p}}CF^*(X
\subset M; H_p^{\ell})$ admit a finer filtration
than $F_wCF^*(X \subset M; H_p^{\ell})$ because the function $w_p$ uniquely
determines the winding vector $\v$. Namely, for any vector $\v$ with $w(\v)
\leq w$, we may set 
\begin{align} \label{eq:CFv} 
    CF^*(X \subset M; H_p^{\ell})_{\v}=\frac{F_{w_{p}(\v)}CF^*(X \subset M; H_p^{\ell})}{F_{w_{p}(\v)-d_\v}CF^*(X \subset M; H_p^{\ell})} 
\end{align}
for a sufficiently small constant $d_\v>0$ so that this quotient is generated
by orbits in the single isolating set $U_\v$. We therefore have a ``local" log PSS
map 
\begin{align} \label{eq:realPSSloc} 
    \PSSlog^{\v}: H^*(\SIo)t^{\v} \to HF^*(X \subset M ; H_p^{\ell})_{\v} 
\end{align} 
where the right hand side
denotes the cohomology of the cochain complex defined in \eqref{eq:CFv}. The
following lemma follows immediately from the above considerations and allows us
to completely localize the task of proving that the map \eqref{eq:locPSS} is an
isomorphism:
\begin{lem}\label{lem:localizeisomorphism} 
    The map  $\eqref{eq:locPSS}$ is an isomorphism iff for every $\v$ with $w(\v) \leq w$, \eqref{eq:realPSSloc} is an isomorphism. \qed
\end{lem}

\section{Low energy log PSS is an isomorphism} \label{sect: PSSiso2}

The main goal of this section is to complete the proof of Theorem
\ref{thm:main} by showing that each ``local low energy log PSS map''
\eqref{eq:realPSSloc}, from a single multiplicity vector $\v$ summand of log
cohomology to the associated graded Floer cohomology of the (refined in the previous subsection) action
filtration is an isomorphism. To do so, we first
geometrically confine the (moduli spaces appearing in) each local low energy
log PSS map (and its source and target) to live in a small tubular neighborhood
of the divisor $D_I$ corresponding to $\v$, and hence in a neighborhood of the zero section in
the fibrewise complex projectivized normal bundle $PD_I$ to $D_I$. Then, once in
$PD_I$,
we define a one-sided inverse map by counting curves with incidence conditions
along the {\em $\infty$-divisors} of
the bundle.  (Unlike the classical case, there does not seem to be a way to
describe the inverse to the (local) $\PSS_{log}$ map using counts of curves in
$M$.) 
After a Morse-Bott type analysis 
(of the associated graded of Floer cohomology) which shows the source and
target of the map \eqref{eq:realPSSloc} are abstractly isomorphic finitely
generated $\K$-modules, we finally conclude \eqref{eq:realPSSloc} is an isomorphism.

In some more detail, here is an overview of what occurs in each sub-section:

\begin{itemize}
\item  In \S \ref{subsection:spectral}, we begin by carefully describing the perturbations $H_p^{\ell}$ of the Hamiltonians $h_p^{\ell}$ (from  \eqref{eq:hpellthing}). It follows from the construction that the generators of the associated graded Floer cohomology $CF^*(X \subset M ; H_p^{\ell})_{\v}$ (see \eqref{eq:CFv}) are all located in a small neighborhood $U_\v$ near the divisor stratum corresponding to $\v.$ We then show that after making suitable choices, the \emph{differentials} in this complex (given by counting Floer trajectories of low energy) can also be constrained to lie in $U_\v.$ It follows that the Floer cohomology group $HF^*(X \subset M ; H_p^{\ell})_{\v}$  is purely ``local" to $U_\v.$  

\item Having ensured that the target of the \emph{local log PSS} map \eqref{eq:realPSSloc} is local to a particular divisor stratum $D_I$ (the source of the map only depends on the normal bundles to $D_I$ and is hence ``local" by definition), we turn to discussing the low energy $\PSS_{log}$ moduli spaces which are used to define the map itself.  In \S \ref{subsection:controlPSS}, we show that by making additional careful choices, the low energy log PSS moduli spaces counted in \eqref{eq:realPSSloc} can be geometrically confined in $M$ to a small neighborhood of 
$D_I$. The key ingredient here is the well-known monotonicity lemma (Lemma \ref{lem:monotonicity}) for pseudo-holomorphic
curves.

\item In \S \ref{subsection:lowin}, we formally introduce the  $(\C P^1)^{|I|}$-bundles, $PD_I$, their symplectic forms and relevant divisors.

\item In \S \ref{sect: locHam}, we continue the analysis of the 
associated graded Hamiltonian Floer complex began in \S \ref{subsection:spectral}. Here, we adapt an argument of Hofer and Salamon \cite{Hofer-Salamon}  in order 
to argue that the source and target of \eqref{eq:realPSSloc} are abstractly isomorphic
finitely generated $\K$-modules (Corollary \ref{lem: Pozniack}). This can be viewed as a form of Morse-Bott analysis and proves the additive part of Theorem \ref{thm:main}. As explained in the proof of Corollary
\ref{cor:reallocPSSiso}, Corollary \ref{lem: Pozniack} implies that it suffices to construct a \emph{one-sided} inverse
to each map \eqref{eq:realPSSloc} in order to prove that it is an isomorphism.

\item In \S \ref{sect: locPSS}, we transfer the log PSS map to the projective bundles $PD_I.$ In these projective bundles, log PSS moduli spaces count
maps from a punctured sphere with incidence (and enhanced evaluation) conditions along
various 0-divisors (meaning the divisor where a fiber coordinate is equal 0). We show that for suitable choices of (domain-dependent) almost-complex structures, these log PSS moduli spaces are in bijection with the low energy log PSS moduli spaces from \S \ref{subsection:controlPSS} (this is again a sort of confinement argument). 

\item Finally, in \S \ref{section: inverseloc}, we exhibit the one-sided inverse to \eqref{eq:realPSSloc}, the ``local log SSP map.''
The relevant operation is built out of maps from a punctured sphere $S^{\vee}= \mathbb{C}P^1 \setminus \lbrace \infty \rbrace$ into
the various projective bundles $PD_I$ which intersect the $\infty$-divisors with multiplicity $\v$. The fact that this map is a one-sided inverse ultimately comes from the fact that a certain count of $J$-holomorphic spheres (given by gluing PSS and SSP moduli spaces) is 1, because it can be shown to coincide with a count of spheres which lie in the fiber of the bundle $PD_I \to D_I$. 
\end{itemize}

Combining everything we have done so far, 
the proof of Theorem \ref{thm:main} appears at the end of \S
\ref{sec:mainproof} as Theorem \ref{thm:spectral}.

\subsection{Local Floer cohomology} \label{subsection:spectral} 

In this subsection we describe careful perturbations of our Hamiltonian
$h_p^{\ell}$ (see \eqref{eq:hpellthing}) for a given $\ell \in \mathbb{N}$. Using these
carefully chosen perturbations, standard techniques
in local Floer cohomology (see Lemma \ref{lem: localFloersep}) will imply that (both the orbits and differential of) the associated graded Floer cochain complexes 
$CF^*(X \subset M ; H_p^{\ell})_{\v}$ (see \eqref{eq:CFv}) for each $\v$ localize to small
(disoint for different $\v$) neighborhoods $U_\v$ near the divisor stratum
corresponding to $\v$.

Recall that for nonzero values of $\v$, the orbit sets $\mathcal{F}_\v$ (defined in \eqref{eq: orbitsetsF})  occur at points in $U_I$ where $\rho_i= \rho_{i,\v}$ for some $\rho_{i,\v} \in \mathbb{R}$ and $i\in I$. The manifold also has boundary and
corner strata where $\rho_i=\rho^{c}_{\v,i}$ for some $\rho^{c}_{\v,i} \in \mathbb{R}$ and $i \notin I$. For later use, we now note the following: 

\begin{lem}\label{lem: SihomFv} For $\v \neq 0$, the manifolds-with-corners  $\mathcal{F}_\v$ are homeomorphic to the manifolds $S_I^{log}$ defined in \S \ref{section: blowups}. \end{lem} 
 \begin{proof} This follows from the above description of  $\mathcal{F}_\v$ and the discussion in Remark \ref{rem: diffeoblow}. \end{proof}

All of the $\mathcal{F}_\v$ are Morse-Bott in their interiors $\mathcal{F}_\v \setminus \partial \mathcal{F}_\v$ (for $\v=0$, this is obvious and we again refer to Step 2 of the proof of Theorem 5.16 of \cite{McLean2} for this result in the case $\v \neq 0$)
and thus Morse-Bott type perturbations are required to make
the orbits non-degenerate.  
Compared to genuine Morse-Bott situations, where the orbit sets form genuine closed manifolds, we must pay a bit of extra attention near
the boundaries of our orbit sets. 

We next choose the isolating
neighborhoods $U_\v$ (compare with the discussion just above \eqref{eq:divisorially})  for these critical sets for each critical set
$\mathcal{F}_{\v}$, that is to say neighborhoods of $\mathcal{F}_{\v}$, which
contain only the orbit set $\mathcal{F}_{\v}$ and no others. 
For $\mathcal{F}_{\mathbf{0}}$, 
choose our isolating neighborhood 
$U_{\mathbf{0}}$ to
be the complement of neighborhood where 
\[U_{\mathbf{0}}:= M \setminus \lbrace R_p^{\ell} \geq R_{0,\ell}+c_0 \rbrace \]
for a sufficiently small constant
$c_0$. For the other orbits, let
$D^{c_0}_{I}$ denote the open manifold $D^{c_0}_{I} := D_I \setminus \cup_{i
\notin I} U_{i,\rho^{c}_{\v,i}-c_0}$  and let $S^{c_0}_I$
denote the induced $T^{I}$ bundle over $D^{c_0}_I$ where
$\rho_i=\rho^{c}_{\v,i}$. After possibly shrinking $c_0$, the isolating sets
are then chosen to be the neighborhoods $U_{\v} \subset U_I$ such that
$\pi_I(x) \in D^{c_0}_{I}$  and $\rho_{i,\v}-c_0 <\rho_i <\rho_{i,\v}+c_0, \ i
\in I$. We choose $c_0$ sufficiently small so that these neighborhoods do not
pairwise intersect, e.g. $U_\v \cap U_{\v'}=\emptyset$ for $\v \neq \v'$. We
let $U_{\v}'$ to be slightly smaller subsets such that $U_{\v}' \subset U_{\v}$
which are of the same form (to construct them just take choose a constant
$c_0'$ which is slightly smaller than $c_0$). Choose a Morse function
$\hat{h}_I: S^{c_0}_{I}  \to \mathbb{R}$ such that near the corners the
function $\hat{h}_I$ are functions of the $\rho_i$ and point outwards along the
boundary.  Finally, we choose cutoff functions $\rho_\v$ such that
\begin{itemize} 
    \item $\rho_\v(x)=0, x\in M\setminus U_{\v}$  
    \item $\rho_\v(x)=1, x \in U_{\v}'$ 
\end{itemize} 

Next we recall the {\em spinning} construction in Morse-Bott
theory (see e.g. \cite[Proof of Prop. B.4]{KoertKwon}). For
non-constant orbits, observe that on all of the orbit sets
$\mathcal{F}_\v$, the Reeb flow generates an $S^1$-action on
$\mathcal{F}_\v$ which extends canonically to $S^{c_0}_{I}$ and $U_\v$. The
inverse circle action on $U_\v$ is a Hamiltonian flow with associated
Hamiltonian function $K: U_\v
\to \mathbb{R}$, where 
\begin{align} 
    K= \sum_i \pi v_{i,0}\rho_i 
\end{align}
We denote the associated time-$t$ flow of $K$ by $\Delta_t(x)$. If
$x(t)$ is a one periodic orbit of the Hamiltonian vector field, then
$\hat{x}(t)= \Delta_t \circ x(t)$ is a one periodic orbit corresponding to
the Hamiltonian 
\begin{align} \label{eq:shiftedham}
        \hat{h}_p^{\ell}=h_p^{\ell} + K(x) 
\end{align}
and similarly for Floer trajectories. Here we have used the fact
$h_p^{\ell}(t,\Delta^{-1}_t(x))=h_p^{\ell}$ by the local invariance of the
function $h_p^{\ell}$.  We thus obtain a new Hamiltonian system in which the
Hamiltonian $\hat{h}_p^{\ell}$ is constant on $\mathcal{F}_\v$ and hence has
constant orbits. Define $h_I$ to be the time dependent function
$h_I:=\widehat{h_I}(t,\Delta_t(x))$. For $\v \neq 0$, let $h_\v$ denote the
pull-back of $h_I$ to $U_\v$ under the projection map. We also set
$h_{\mathbf{0}}$ 
to be a function which near $R_p^{\ell}=R_{0,\ell}$ is a
function of $R_p^{\ell}$ with positive derivative.

We also recall how to choose perturbing data for the divisorial orbits (see
\cite[\S 4.1]{GP1} for more details on this perturbation), even though
the details of this will be less important for our applications. The divisorial
orbits come in strata which are also manifolds-with-corners. The orbits which
lie in $D_I$ correspond to the points where $\rho_i=0$ for all $i \in I$. 
The Hamiltonian vector field restricted to the fibers of $U_I$ at these points
is of the form 
$\sum_{i\in I}-\lambda_i \partial_{\varphi,i}$ 
with
$\lambda_i>0$ which
infinitesimally generates a non-trivial rotation of the fibers fixing the
points where $\rho_i=0$. 

It follows that these orbits are transversely nondegenerate over the open
parts. Similar to what we have seen above, choose an outward pointing Morse
function $\widehat{h}_\D$ on the disjoint union of these submanifolds and set
$h_\D=\widehat{h}_\D(t, \Delta_t(x))$ as before. Choose cutoff functions
$\rho_\D$ supported in a small neighborhood of the divisors (and in particular
inside $V_{0,\ell}$).  For sufficiently small constants $\delta_\v$ and
$\delta_\D$ define  
\begin{align} 
    \label{eq: Hpert2} H^{\ell}_{p}=  \sum_ \v \delta_\v \rho_\v h_\v + h_p^\ell +\delta_\D  \rho_\D h_\D 
\end{align} 
It follows from the analysis in \cite[\S 4.1]{GP1} 
that for suitable choices of $\rho_\D$ and $h_\D$ that the Hamiltonian flow of
this perturbed
function preserves the divisor $\D$. For the remainder of this section, we
focus on the orbits which lie in $M\setminus V_\ell$ In each open set $U_{\v}$
there are obvious time-1 orbits corresponding to critical points $x_i$ of the
function $\hat{h}_I$ on $\mathcal{F}_\v$. 

\begin{lem} 
    For $\delta_{\v}$ sufficiently small, all time-1 orbits of $H_p^{\ell}$
    are those created as critical points from the manifolds $\mathcal{F}_\v$ as
    critical points of $h_\v$.  
\end{lem}
\begin{proof} 
    For sufficiently small $\delta_\v$, it follows by compactness that such
    orbits must be contained in some $U_{\v}'$ where $\rho_\v=1$. There are
    also no fixed points near $\rho^c_{\v,i}$ because the derivative of
    $\hat{h}_I$ points outwards along the strata. Thus, the lemma reduces to
    the well-known case in Morse-Bott theory as covered in e.g.
    \cite{KoertKwon}.  
\end{proof}

We will also consider a very closely related version of Floer cohomology known
as {\em local Floer cohomology} (see e.g., \cite{McLeanlocalFloer} for a
treatment most directly applicable to the setting here), which is defined
whenever $\delta_{\v}$ above are taken sufficiently small. The abelian group
underlying the complex is the same as
$CF^*(X \subset M ; H_p^{\ell})$. However, the differential $\partial_{loc}$
counts local solutions to Floer's equation which lie in $W_\ell :=\cup_\v
U_\v$. Let $J_t \in \mathcal{J}(M,\D)$, be a (generic) time-dependent complex
structure, which we will always assume is $\omega_{\ell}$-compatible in an open
neighborhood of all Hamiltonian orbits $x_0 \in \mathcal{X}(X;
H_p^{\ell})$.\footnote{We do this to rely on standard gluing results in the
literature on (local) Floer cohomology. It is likely not necessary.} For any
two orbits $x_0, x_1$, let $\mathcal{M}(W_{\ell} ; x_0,x_1)$ denote the moduli
space of solutions to Floer's equation \eqref{eq:FloerRinv} for the Hamiltonian
$H_p^{\ell}$, modulo $\mathbb{R}$-translation, which additionally
satisfies the following ``locality'' and asymptotic conditions:
\begin{align} \label{eq:Floerloc}
\left\{
\begin{aligned}
 & u \colon \R \times S^1 \to \cup_\v U_\v, \\
 & \lim_{s \to -\infty} u(s, -) = x_0\\
& \lim_{s \to \infty}u(s, -) = x_1 \\
\end{aligned}
\right.
\end{align}
The following basic result ensures that such local Floer curves are confined
i.e., stay away from the boundary of $\cup_\v U_\v$: 
\begin{lem} 
    \label{lem: localcompactness}(\cite[Lemma 2.3]{McLeanlocalFloer}) For sufficiently small $\delta_{\v}$, those Floer curves contained in $\cup_\v U_\v$ are in fact contained in $\cup_\v U'_\v$.  \qed
\end{lem}
From here, Gromov compactness applies and the usual Floer-theoretic arguments
allow us to, for 
small $\delta_\v$ as above, define the
$|\mathfrak{o}_{x_1}|$-$|\mathfrak{o}_{x_0}|$ component of the differential as
\begin{equation} 
    (\partial_{loc})_{x_1, x_0}  = \sum_{u \in |\mathcal{M}(W_\ell ; x_0,x_1)|} \mu_u
\end{equation} 
where $\mu_u: \mathfrak{o}_{x_1} \stackrel{\sim}{\to} \mathfrak{o}_{x_1}$ is the induced
isomorphism of orientation lines for a rigid $u$ (which occurs whenever
$\deg(x_0)-\deg(x_1)=1$). Also, standard methods show
$\partial_{loc}^2=0.$ We denote the resulting cohomology
theory by $HF^{*}(W_\ell \subset M,H_p^{\ell})$. We similary form the group
$HF^*(U_\v \subset M,H_p^{\ell})$ which is the local Floer subcomplex
consisting of orbits just in $U_\v$.  There is an obvious decomposition 
$$
HF^*(W_\ell \subset M,H_p^{\ell}) \cong \bigoplus_{\mathcal{F}_\v \in \mathcal{X}(X; h_p^{\ell})} HF^*(U_\v \subset M, H_p^{\ell}).
$$ 

\begin{lem} 
    \label{lem: localFloersep} Fix a generic $J_t \in \mathcal{J}_\ell(\bar{X}_\ell^{p},V)$ and let $x_0, x_1 \in U_\v$. For
    $\epsilon_\ell^{p}$(see the discussion above \eqref{eq:hpellthing}), $\delta_\v$ sufficiently small and $\Sigma_\ell^{p}$
    sufficiently $C^0$-close to $\Sigma_{\vec{\epsilon}^{p}}$(again see just above \eqref{eq:hpellthing}), any two Floer
    trajectories in $X$ connecting $x_0,x_1$ lie in $U_\v$. As a consequence,
    we have canonical isomorphisms 
    \begin{align} 
        \label{eq: localsep} HF^*(X \subset M;H_p^{\ell})_{\v} \cong HF^*(U_\v \subset M,H_p^{\ell}) \\
        \bigoplus \limits_{w(\v) \leq w_\ell} HF^*(X \subset M; H_p^{\ell})_{\v} \cong HF^*(W_{\ell} \subset M,H_p^{\ell}) 
    \end{align}
\end{lem} 
\begin{proof} 
    Because we have chosen our $\kappail$ so that our orbits have distinct
    Hamiltonian actions, this follows from \cite[Lemma 2.8]{McLeanlocalFloer}. 
\end{proof}

\subsection{Controlling the low energy log PSS solutions} \label{subsection:controlPSS}

\begin{equation}\tag{$\diamondsuit$} \label{eq:fixv0}
    \begin{split}
        &\textrm{\emph{From this point until the end of \S \ref{section: inverseloc}, we work with a fixed non-zero vector $\v=\v_0$}},\\
        &\textrm{\emph{along with its associated divisor stratum index $I = I_0 = \{i \in  \{1, \ldots, k\} \ | \ (\v)_i \neq 0\}$}}.
    \end{split}
\end{equation}

Equation \eqref{eq: localsep} shows that the target of our $\PSSlog^{\v}$ map \eqref{eq:realPSSloc} is completely ``localized" near the divisor stratum $D_I$ corresponding to $\v$ (the source is of course by definition local to the same stratum). The main result of this subsection is Lemma \ref{lem: money}, which states that after choosing our data suitably, the low energy log $\PSS$ solutions themselves are constrained to
lie in a small neighborhood of the stratum corresponding to $\v$. 
The general idea for doing this is to apply the monotonicity lemma for $J$-holomorphic curves (Lemma \ref{lem:monotonicity}). In order to carry this out, we need to make the following specific choices: 
\begin{itemize} 
\item We need to work with complex structures that are (arbitrarily close to) ``split" complex
structures (c.f. Definition \ref{defn: split}) which have simple local models
(in our chosen tubular neighborhoods) near the various strata. In particular, for these complex structures, the maps $\pi_I: U_I \to D_I$ are $J$-holomorphic.
\item We need to consider the limit as the perturbing constants $\delta_\v$, $\v \neq \v_0$, and $\delta_\D$ from \eqref{eq: Hpert2} are taken to zero.  
\end{itemize} 
In order to show that we can work with ``split" complex structures, we need to show that our low energy log PSS map is well-defined for any  $J_S \in \mathcal{J}_S(M,\D).$ This requires a different argument for excluding breaking along $\D$ (given below in Lemma \ref{lem: anyJ}) from the one given in Lemma \ref{lem:compactness}, where we assumed $J_S \in \mathcal{J}_{S,\ell}(V)$.\footnote{On the other hand, it is important to note that the argument in Lemma \ref{lem:compactness} rules out breaking along divisorial orbits for log PSS moduli spaces that are not necessarily low energy. This will be used in Lemma \ref{lem:compactness2}.} ``Turning off" the parameters $\delta_\v$, $\v \neq \v_0$, and $\delta_\D$ also introduces some subtlety, as the limiting Hamiltonians $H^{\ell}_{\infty}$ are (mildly) degenerate. Nevertheless, we show in Lemma \ref{lem: MBcompactness} that the low energy log $\PSS$  moduli spaces  with the respect to $H^{\ell}_{\infty}$ have well-behaved compactifications.  Working with split complex structures/limiting Hamiltonians will allow us to project our $\PSSlog^{\v}$ solutions to genuine $J$-holomorphic curves, when we can then apply monotonicity to obtain confinement; see the proof of Lemma \ref{lem: money}. 

To begin, we introduce some further notation. Let $H^S_{2}(M)$ denote the image $\operatorname{im}(\pi_2(M)) \subset
 H_2(M)$. Since by assumption our symplectic form $\omega_{\ell}$ is rational,
 it follows there is a minimum quantity
 $\omega_{min}$ such for any class $A \in H^S_{2}(M)$,  $|\omega_\ell([A])|>\omega_{min}$ whenever
 $\omega_{\ell}([A]) \neq 0$; by convention we set $\omega_{min} = \infty$ if $\omega_{\ell}([A])=  0$ for all $A \in H^S_2(M)$.
 In particular, we also know that $\omega_\ell(A)<-\omega_{min}$ whenever
 $\omega_\ell(A)<0$. 

Now choose $J_S \in \mathcal{J}_S(M,\D)$ (as in \eqref{defn: sdcs}) which restricts on the cylindrical end to some
$J_t$ that we use to define local Floer cohomology. For any $x_0 \in U_\v$ and
$c \in \critfI$ (see \eqref{eq:morsecritl}), let $\mc{M}(\v,c, H_p^{\ell}, x_0)$ denote the moduli space of
PSS solutions with respect to the Hamiltonian $H_p^{\ell}$. The next Lemma
gives the basic compactness results for low energy log PSS moduli spaces, for
generic such $J_S$:
\begin{lem}\label{lem: anyJ} 
    Fix any critical point $c \in \critfI$ and an orbit  $x_0 \in U_\v$ of $H^{\ell}_p$ such
    that $\operatorname{vdim}(\mc{M}(\v, c, H_p^{\ell}, x_0))= 1$. Then
    for generic $J_S \in \mathcal{J}_S(M,D)$, $\epsilon_\ell^{p},
    ||H_p^{\ell}-h_p^{\ell}||_{C^2}$ sufficiently small, and $\Sigma_\ell^{p}$
    sufficiently $C^0$ close to $\hatXlp$, the Gromov-Floer compactification of
    $\mc{M}(\v,c, H_p^{\ell}, x_0)$ is a compact 1-manifold-with-boundary
    $\partial\overline{\mc{M}}(\v, c, H_p^{\ell}, x_0)= \partial_M \bigsqcup
    \partial_F$ where 
    \begin{align}
        \partial_F:= \bigsqcup_{x,|x_0|-|x'|=1} \mc{M}(\v,c, H_p^{\ell}, x') \times \mc{M}(x_0,x')  \\ 
        \partial_M:= \bigsqcup_{c', \deg(c') - \deg(c) = 1}
        \mc{M}(c',c)
        \times \mc{M}(\v,c',H_p^{\ell}, x_0)
    \end{align} 
 \end{lem}
 
\begin{proof}
  By Lemma \ref{lem:energypss}, the energy of a solution $u \in \mc{M}(\v, c, H_p^{\ell}, x_0)$  is arbitrarily close to $\frac{1}{2}w_p(\v)
    (\epsilon_{\ell}^{p})^2$ (recall \eqref{perturbedweight} for the definition of $w_p$).   
    Recall from Lemma \ref{lem:hamiltonianestimateonD}
    that by taking
$\Sigma_\ell^{p}$ sufficiently $C^0$ close to $\hatXlp$, we can assume that
$ H_p^{\ell} \approx
\lambda_{\ell}(\frac{1}{1-\frac{1}{2}(\epsilon_{\ell}^{p})^2}-1)$ along the
divisors. Assume that $\epsilon^{p}_{\ell}$ is chosen sufficiently small so
that 
\begin{equation}
    \lambda_{\ell}(\frac{1}{1-\frac{1}{2}(\epsilon_{\ell}^{p})^2}-1)\ll \omega_{min}.
\end{equation} 
The Lemma as usual follows from Gromov compactness provided we
exclude other possible ``bad'' limits from occuring in the compactification.
In light of our total energy being very small, sphere bubbling cannot occur, so
it suffices to rule out cylinder breaking along orbits in $\D$. So,
suppose $y$ is in $\mathcal{X}(\D; H_p^{\ell})$ (defined in \eqref{eq:divisorially}), 
and that the curve limits to a broken curve $(u_1,u_2)$ with
$u_1$ a PSS solution and $u_2 \in \overline{\mathcal{M}}(x_0,y)$. The class
of $u_1$ in relative homology must be the connect sum of 
\begin{itemize}
    \item a (multiply covered) fiber disc $-F$ (the canonical capping disc of $y$ in $\D$ which is a product of discs and constant discs in the fibers of the regularization), oriented so that the boundary is $y$ with the opposite orientation; with

    \item some absolute homology class $A \in H_2(M)$. 
\end{itemize}
 We have that \begin{align} \Etop(u_1)=\int_{u_1} u_1^*(\omega) + \lambda_{\ell}(\frac{1}{1-\frac{1}{2}(\epsilon^{p}_{\ell})^2}-1) \end{align}
Observe that $\int_F \omega_{\ell} \geq 0$ is small, at most $\frac{1}{2}
(w_{\ell}+\delta_p) (\epsilon^{p}_{\ell})^2$. In particular, it is not possible for $\omega_{\ell}(A)<0$
because otherwise it would be less than $-\omega_{min}$, implying the
topological energy of $u_1$ would be negative.
It follows that the topological energy of $u_1$ is at least
$\lambda_{\ell}(\frac{1}{1-\frac{1}{2}(\epsilon^{p}_{\ell})^2}-1)$, which is
bigger than the total energy $\frac{1}{2}w_p(\v) (\epsilon^{p}_{\ell})^2$. Thus the topological
energy of $u_2$ must be negative, a contradiction. 
\end{proof} 

Having shown that we can work with $J_S \in \mathcal{J}(M,\D)$ (and in particular ``split'' almost complex structures from Definition \ref{defn: split}), we recall (see the summary at the beginning of this sub-section) that in order to geometrically confine (or localize) PSS moduli spaces, we also need to take a limit where we turn off most of the
Hamiltonian perturbations supported anywhere except for our particular
$U_{\v_0}$. The resulting Hamiltonian is degenerate, but has the nice property
that for ``split'' almost complex structures (Definition \ref{defn: split}),
its Floer curves are genuinely holomorphic curves outside $U_{\v_0}$ (after
possibly projecting to whichever divisorial stratum it is near). In particular
an iterative application of the the usual monotonicity Lemma for
pseudo-holomorphic curves (over strata) gives the desired geometric control
over ``degenerate PSS'' solutions. By a version of Gromov compactness (Lemma
\ref{lem: generalHam}) this implies confinement for ``small'' perturbations.
This argument and the appeal to monotonicity appears in Lemma \ref{lem: money}.

To begin this process, for any $\v' \neq \v_0$ (where $\v_0$ is our fixed
vector \eqref{eq:fixv0}), we fix a sequence of $\delta_{\v',n} \to 0$ (recall these constants from \eqref{eq: Hpert2}). We also
fix a sequence of $\delta_{\D,n} \to 0$ and we let $H_n^{\ell}$ denote the
associated Hamiltonians constructed by Equation  \eqref{eq: Hpert2}. Let
$H_{\infty}^{\ell}$ denote the limiting Hamiltonian.  By construction, we have
that $H_{\infty}^{\ell}-h_p^{\ell}$ is supported in $U_{\v_{0}}$. For each $n$,
we fix $J_{S,n} \in \mathcal{J}_S(M,D)$ which converge to some $J_S \in
\mathcal{J}_S(M,D)$. Let $\mathcal{M}(\v_0, J_{S,n}, H_n^{\ell}, x_0)$ denote
the moduli space of PSS solutions for this sequence of Hamiltonians. We will
make use of the general fact about solutions to Floer's equation using a
possibly degenerate Hamiltonian (such as the limiting Hamiltonian $H^{\ell}_\infty$):
\begin{lem}[compare \cite{Oh} Proposition 18.4.10 or \cite{Salamon Conley Index} proof of Prop. 4.2] \label{lem: generalHam}
    Let $H:S^1 \times X \to \mathbb{R}$ be any Hamiltonian and $\Sigma$ be a
    domain decorated with suitable cylindrical ends and perturbation data $K$.
    Suppose that $u: \Sigma \to X$ be a finite energy solution to Floer's
    equation \eqref{eq:generalFloer}. Restrict $u$ to a cylindrical end
    $\epsilon_{+}: [0,\infty) \times S^1$ or $\epsilon_{-}:(-\infty, 0] \times
    S^1$. Then for any $s_m \to \pm \infty$ there is a subsequence $s_m,$ so
    that $u(s_m,-) \to \gamma(t)$ in $C^{\infty}(S^1, X)$,
    where $\gamma$ is some periodic orbit of the Hamiltonian vector
    field.  
\end{lem} 
When the Hamiltonian is degenerate there may exist sequences $s_m$ and $s_m'$
giving rise to different limits $\gamma$ and $\gamma'$ \footnote{It is still
true that $\gamma$ and $\gamma'$ must have the same Hamiltonian action.}.
Returning to our sequence of moduli spaces, the relevant version of Gromov
compactness for PSS solutions with respect to $H_{n}^{\ell}$ as $n \to \infty$ is the following: 
\begin{lem} 
    \label{lem: MBcompactness} 
    Fix $x_0 \in U_{\v_{0}}$. For $\epsilon_\ell^{p}$, $\delta_{\v_{0}}$
    sufficiently small and $\Sigma_\ell^{p}$ sufficiently $C^0$ close to
    $\hatXlp$ and given any sequence of solutions $u_n \in \mathcal{M}(\v_0,
    J_{S,n}, H_n^{\ell}, x_0)$, there is a subsequence which converges in the
    Gromov topology to some 
    \begin{align} 
        u_{\infty} \in \mathcal{M}(\v_0, J_{S},H_{\infty}^{\ell}, y_0) \times \mathcal{M}(y_1,y_0) \times \cdots
        \times \mathcal{M}(x_0,y_k) 
    \end{align} 
    for $y_i \in U_{\v_{0}}$.
\end{lem} 
\begin{proof} 
    Observe that by our assumptions no bubbling can occur by arguments already
    given (namely, that the energies of the $u_n$ are smaller than the minimal
    energy of $J$-holomorphic spheres) 
    It follows that 
    \begin{align} 
        |\nabla u_n|_{L_{\infty}}<\infty 
    \end{align}
    remains bounded and hence that we have that the solution converges in
    $C_{loc}^\infty$ to some solution 
    \begin{align} \mathring{u}_{\infty}: S
        \to M 
    \end{align} 
By Lemma \ref{lem: generalHam} it follows that there are sequences $s_m \to
-\infty$ and an orbit $y'$ such that \begin{align} \mathring{u}_{\infty}(s_m,t)
    \to y' \in C^\infty(S^1,M) \end{align} Suppose that $y'$ is a divisorial
orbit in $\D$. Then there would be a pair $(n,s_0)$ for which $u_n(s_0,t)$ is
arbitrarily close to $y' \in C^\infty$. The argument of Lemma \ref{lem: anyJ} shows
that this is not possible. 

 Suppose there is some orbit $y'$ in $X$ not in $U_{\v_{0}}$ together with a
 sequence of $s_m \to -\infty$ such that $\mathring{u}_{\infty}(s_m,t) \to y'
 \in C^\infty(S^1,M)$. Then there would again be a pair $(n,s_0)$ for which
 $u_n(s_0,t)$ is arbitrarily close to $y' \in C^\infty$. Let
 $\bar{u}_n=u_n^{-1}(-\infty,s_0)$. Then \begin{align} E(\bar{u}_n) \approx
 -A_{\ell}(y') -(1-\frac{1}{2}(\epsilon_{\ell}^p)^2) w(\v_0)  \end{align}
 (recall $A_{\ell}:=A_{H^{\ell}}$, where the latter is defined in \eqref{eq:action}) which by our assumptions would imply that either  
 \begin{itemize} 
     \item $E(\bar{u}_n) <0$; or
     \item $E(\bar{u}_n) \gg \frac{1}{2}w(\v_0)(\epsilon_{\ell}^p)^2$.
 \end{itemize} 
 Thus we conclude
 that $\mathring{u}_\infty$ limits uniquely to some orbit $y_0$ in
 $U_{\v_{0}}$. Likewise, the same arguments show that after rescaling by
 suitable $s_m \to -\infty$ we obtain Floer trajectories with asymptotes
 only in $U_{\v_{0}}$. 
 \end{proof} 

The version of monotonicity for pseudoholomorphic curves we will appeal to is:
\begin{lem}[Monotonicity lemma, compare \cite{BEHWZ} Lemma 5.2 or \cite{Sikorav} Proposition 4.3.1]\label{lem:monotonicity}
    Let $(W,J)$ be a compact almost complex manifold and suppose that $J$ is
    tamed by some $\omega$. Then there exists a positive constant $C_0$ having
    the following property: For any compact $J$-holomorphic curve $f: (S,j) \to (W,J)$,
    point in the domain $s_0 \in S\setminus \partial S$, and $r > 0$ smaller
    than the injectivity radius of $W$,  if the boundary $f(\partial S)$ is
    contained in the complement of the ball $B_r(s_0)$ about $s_0$ of radius
    $r$, then the area\footnote{Both the ball $B_r(s_0)$ and definition of area
    are with respect to the the metric induced by $\omega$ and $J$.} of the
    portion of $f$ mapping to $B_r(s_0)$ is bounded below by: 
    $$ A(f^{-1} (B_r(s_0))) \geq C_0r^2.$$ 
\end{lem} 

\begin{defn} \label{defn: split} 
    We say that $J_0 \in \mathcal{J}(M,\D)$ is {\em split} if over each $\pi_I: U_I \to D_I$
\begin{itemize} 
    \item the map $\pi_I$ is $J_0$-holomorphic;
    \item For every point p, the complex structure respects the decomposition $H_p \oplus F_p$;
    \item On each $F_p$, the complex structure is split with respect to the
            product decomposition on $D_{\epsilon}^{|I|}$. 
    \end{itemize}
\end{defn} 

For any $U_I$, we may consider the the horizontal piece of the symplectic form
$$ \omega_H=\omega_{\ell}-\omega_{vert} $$ This is symplectic on each
horizontal subspace $H$.  Observe that for a split $J_0$, the energy of a curve
which lies in $U_I$ can be split into horizontal and vertical pieces. $$
E(u)=E_{hor}+ E_{vert} $$  Let $g_{\omega_{H}}$ and $g_{\omega_{D_I}}$ denote
the corresponding metrics on $D_I$ and the horizontal subspsace of the tangent
space to $U_I$.  Since the metric $g_{\omega_H}$ can be made to extend to a
compactification $\bar{U}_I$ of $U_I$ (by shrinking $U_I$ if necessary), it
follows that 
there exists a constant $G_I>0$ such that 
\begin{equation} \label{eq:divisormetricbound}
    g_{\omega_{D_I}} < G_I \cdot g_{\omega_{H}} 
\end{equation}
as positive bilinear forms on each $H_p$.

We fix once and for all an almost-complex structure $J_0$ which is split
outside of a small neighborhood of $U_{\v_0}$ and compatible inside of
$U_{\v_{0}}$. All of our almost-complex structures in this section will be
arbitrarily close to this fixed complex structure. Let $C_{I}$ 
be the monotonicity constant associated to the complex structures over $D_I$. 

We are finally in a position to prove the promised confinement of 
low energy log $\PSS$-solutions: 

\begin{lem} \label{lem: money} 
    Fix $\v_0$ as above, let $I$ denote the support of $\v_0$, and fix some
    $\epsilon_{\ell}' > 0$ smaller than the size $\epsilon$ of all of the
    tubular neighborhoods $U_i$.
    For $\epsilon_\ell^{p}$, $\delta_\D$, $\delta_{\v'}$, $\delta_{\v_{0}}$
    sufficiently small (depending on $J_0$),
    $\Sigma_\ell^{p}$ sufficiently $C^0$ close to
$\hatXlp$ and $J_S$ sufficiently close to the split $J_0$ above,  any
    $\PSSlog^{\v_{0}}$ solution $u \in \mc{M}(\v_0, x_0)$ must lie in
    $U_{I,(\epsilon'_\ell)^2}:= \cap_{i \in I} U_{i,(\epsilon'_\ell)^2}$ (see \eqref{eq:randomUithing}). 
\end{lem}
\begin{proof} 
By the compactness statement of Lemma \ref{lem: MBcompactness}, it suffices
to prove that any limiting broken curve $u_\infty$ appearing in the statement
of Lemma \ref{lem: MBcompactness} with $J_S=J_0$ may not escape 
$U_{I,(\epsilon'_\ell)^2}$, for some sufficiently small $\epsilon_{\ell}^p$ 
(or equivalently large constant $C_e = \epsilon'_\ell / \epsilon_\ell^p$)
determined in the proof.
For sufficiently small $\delta_{\v_{0}}$ we
have seen that any Floer trajectory remains in $U_{\v_{0}}$ by Lemma
\ref{lem: localFloersep}. Hence it only remains to confine the PSS component $\mathring{u}_\infty$ of
the broken curve, which as we have seen has energy approximately
$\frac{1}{2}w(\v_{0})(\epsilon_{\ell}^p)^2$, which becomes arbitrarily small as
we decrease $\epsilon_{\ell}^p$.

For what follows, it suffices to replace the curve $\mathring{u}_\infty$ with
its intersection with $U_I$. Our argument will proceed by inductively
decreasing $\epsilon_{\ell}^p$ (and implicitly the other constants so that the
confinement of Floer trajectories continues to hold)
in order to confine this PSS curve away from various strata.

Note that $U_I \cap U_K=\emptyset$ for any subset $K$ with $D_I \cap
D_K=\emptyset$. Now assume by contradiction that this curve contains a point
$p$ in $U_I \setminus
U_{I,(\epsilon'_{\ell})^2}$. We let $K$ range over all subsets such that $D_K
\cap D_I \neq \emptyset$ and $I \not \subset K$ and consider the ordering where
$K_1<K_2$ if $\#|K_1|<\#|K_2|$ or if $\#|K_1|=\#|K_2|$ and the (lexicographically) first distinct digits
$k_1$ and $k_2$ 
between $K_1$ and $K_2$ satisfy $k_1<k_2$. Define
$\mathring{U}_K:= U_{K,(\epsilon'_\ell/2^{n-|K|-1})^2} \setminus \cup_{i \notin
K} U_{i,(\epsilon'_\ell/2^{n-|K|-1})^2}$ (as usual $U_\emptyset = U_{\emptyset,
(\epsilon'')^2} =M$ for any $\epsilon''$). These sets cover $U_I \setminus
U_{I,(\epsilon'_\ell)^2}$ and, throughout the rest of the proof we will, in an
abuse of notation, replace $U_{K,(\epsilon'_\ell/2^{n-|K|-1})^2}$ with its
intersection with $U_{I}\setminus U_{I,(\epsilon'_\ell)^2}$.

Our proof will inductively in $K$ (using the above ordering) show that, by
taking $\epsilon_{\ell}^p$ sufficiently small,
the curve cannot intersect $\mathring{U}_K$. First let us
consider the base case, which is to show that for $\epsilon_{\ell}^p$ small, the
curve cannot have a point $p$ in $U_\emptyset \setminus \cup_{i=1}^k U_{i,
(\epsilon_{\ell}' / 2^{n-1})^2}$ (implicitly intersected with
$U_{I}\setminus U_{I,(\epsilon'_\ell)^2}$).
This follows from monotoncity
(Lemma \ref{lem:monotonicity}) since 
for $\epsilon_{\ell}^p \ll \frac{1}{2^{n-1}} \epsilon_{\ell}'$ 
the curve is holomorphic in the larger
region $M\setminus \cup_i U_{i,(\epsilon_\ell^{p})^2}$,
and its intersection with this region has boundary
on the loci where one or more of the $\rho_i=\frac{\kappa_i}{2\pi}(\epsilon_\ell^{p})^2$. 
This loci is metrically bounded away from the smaller region $U_\emptyset \setminus \cup_{i=1}^k U_{i,
(\epsilon_{\ell}' / 2^{n-1})^2}$, and hence for some $r>0$ smaller than the injectivity radius,
any point in the smaller region has $r$-ball contained in the larger region,
a fact that continues to hold if $\epsilon_{\ell}^p$ is further shrunk.  
Appealing to monotonicity (Lemma
\ref{lem:monotonicity}), if there is a point $p \in U_\emptyset \setminus
\cup_{i=1}^k U_{i, (\epsilon_{\ell}' / 2^{n-1})^2}$ in the image of the curve,
the energy of the curve must be bounded below by $C_{\emptyset} r^2 > 0$ for
all $\epsilon_{\ell}^p$ sufficiently small, a quantity that becomes
greater than the energy of a PSS solution when $\epsilon_{\ell}^p$ is small, a
contradiction.

Now we turn to the inductive step in the proof. Suppose that there is a point
$p$ where the curve intersects $\mathring{U}_K$. Observe that $\mathring{U}_K
\subset U_{K,(\epsilon'_\ell/2^{n-|K|-1})^2}\setminus \cup_{i \notin K}
(U_{i,(\epsilon'_\ell/2^{n-|K|})^2})$. Consider
$u_K=\mathring{u}_{\infty}^{-1}(U_{K,(\epsilon'_\ell/2^{n-|K|-1})^2}\setminus
\cup_{i \notin K} (U_{i,(\epsilon'_\ell/2^{n-|K|})^2})$, the intersection of
$\mathring{u}_{\infty}$ with this larger set.  By the inductive hypothesis,
this curve can have no boundary along the loci where $\rho_i=\frac{\kappa_i}{2\pi}
(\epsilon'_\ell/2^{n-|K|-1})^2$ for any $i \in K$ (as such points are contained
in some $\mathring{U}_{K'}$ for $K' < K$).  Using the fact that the cylindrical
end of our curves lie in $U_\v$ and that $K$ does not contain $I$, we see that
the projection of this curve $\pi_K(u_K)$ is then a holomorphic curve in $D_K
\setminus \cup_{i \notin K} (U_{i,(\epsilon'_\ell/2^{n-|K|})^2})$ with boundary
along the boundary of this region.

The point $\pi_K(p)$ lies in $D_K \setminus \cup_{i \notin K}
(U_{i,(\epsilon'_\ell/2^{n-|K|-1})^2})$, a region which is bounded away from the boundary of the larger region $D_K
\setminus \cup_{i \notin K} (U_{i,(\epsilon'_\ell/2^{n-|K|})^2})$. 
It follows that we can find a ball $B_K(\pi_K(p))$ about $\pi_K(p)$ of some
non-zero radius that is independent of the particular point $p$ in the region 
(depending on the metric distance from $D_K \setminus \cup_{i \notin K}
(U_{i,(\epsilon'_\ell/2^{n-|K|-1})^2})$ to the boundary of the larger region)
which is disjoint from the boundary. In particular, monotonicity (Lemma
\ref{lem:monotonicity}) implies that, since $u$ has image containing this point
$p$, the energy of $\pi_K(u)$ must be bounded below by $C_K r_K^2$ for $r_K$
the lesser of the size of the ball $B_K$ and the injectivity radius of $D_K$.
On the other hand, 
 \begin{align} 
     \frac{1}{2}w(\v) (\epsilon_\ell^p)^2 \approx E(u_{\infty}) \geq E(u_K) \\ 
     \geq E_{horiz}(u_K) \\ 
     \geq \frac{E(\pi_K(u))}{G_K}
 \end{align} 
 where the last inequality uses \eqref{eq:divisormetricbound}. In particular,
 by taking $\epsilon_{\ell}^p$ sufficiently small, $E(\pi_K(u))$ can be made
 smaller than $C_K r_K^2$, implying such a $p$ cannot exist.
\end{proof}

\subsection{Projective bundle compactifications} \label{subsection:lowin}

Here we describe the projective bundle compactifications $PD_I$ (of neighborhoods of the strata $D_I$) and their symplectic structures.  Fix an $\epsilon_{\ell}'$ as in the statement of Lemma \ref{lem: money}.
Given a smooth component $D:=D_i \subset \D$, with associated constant $\kappa:= \kappa_i$ as in \eqref{eq:kappai}, consider the standard projective bundle $PD= P(ND
\oplus \mathcal{O}_D)$ over $D$, and let $\pi_P: PD \to D$ denote the
projection to $D$. There are two natural holomorphic sections $D_0$ and
$D_{\infty}$ and we may algebraically identify 
\begin{align} 
    PD \setminus D_\infty=ND=Tot(\mathcal{O}(D)) \\ PD \setminus D_0 \cong Tot(\mathcal{O}(-D)) 
\end{align} 
Observe also that 
$$PD \setminus (D_\infty \cup D_0) = ND \setminus D=Tot(\mathcal{O}(D)) \setminus D \cong SD \times \mathbb{R}$$ 
where the last isomorphism makes sense in the smooth category only. Turning to symplectic forms and letting $p$ denote the norm in the fiber with respect to the Hermitian metric on $ND$ fixed as part of the regularization chosen in \S \ref{subsec:ncsymplectic}, we equip $PD$ with the standard symplectic form 
\begin{align} 
    \omega_{PD}= \frac{\kappa(\epsilon'_\ell)^2}{2\pi} d(\frac{p^2}{1+p^2} \theta) +\pi_P^*(\omega_D) 
\end{align} 
After setting 
\begin{align} 
    \rho_{loc}= \frac{\kappa (\epsilon'_\ell)^2}{\pi} \frac{p^2}{1+p^2} 
\end{align}
we see that the complement of the divisor $D_\infty$ can be identified with a
standard (open) symplectic disc bundle $U_{D_{0}}$ of radius $\sqrt{\frac{\kappa}{\pi}}
\epsilon_\ell ' $. Symmetrically, we can identify a neighborhood of
$D_{\infty}$ inside of $PD$ with a disc bundle and the projective bundle as
arising from the gluing 
of these two disc bundles.  We may embed $U_{(\epsilon'_{\ell})^2} \subset
U_{D_{0}} \subset PD$. Denote its image by $U^{loc}_{\ell}$.    

Over lower dimensional strata, we let $PD_I$ be the $(\mathbb{C}P^1)^{|I|}$
bundle which is the fiber product over $D_I$ of the $PD_i$ for $i \in I$.  For
$i \notin I$, denote by $D_{i,I}$ the divisors $D_i \cap D_I$ (in a small abuse of notation, we will also use $D_{i,I}$ to denote the preimage of these divisors in $PD_I$), so
$\mathring{D}_I = D_I \setminus \cup_{i \notin I} D_{i,I}$.
Let $\mathring{N}D_I \to \mathring{D}_I$ denote the complement (in $PD_I$) of the
sections $D_{i,0}$, $D_{i,\infty}$ for $i \in I$,  restricted to the open
stratum of the base $\mathring{D}_I$.  See Figure \ref{projectivebundle} for a
sample schematic of some of these various strata in a special case.
This is a
$(\mathbb{C}^*)^I$ bundle over $\mathring{D}_I$. As $PD_I$ is a fiber product,
for any subset $J \subset I$ 
we also have maps $\pi_{(J)}: PD_I \to PD_I^{(J)}$, where $PD_I^{(J)}$ is the
$(\mathbb{C}P^1)^{|I|-|J|}$ bundle over $D_I$ given by taking fiber products of
$PD_i$ for $i \in I \setminus J$, and $\pi_{(J)}$ projects away from the fibers
corresponding to elements of $J$. 
  \begin{figure}[h] 
  \caption{A schematic identifying a neighborhood $U_1^{(\ell)}$ of a divisor $D_1$ in $M$ with a neighborhood $U_{1,(\ell)}^{loc}$ of the zero section in the projectivized normal bundle $PD_1$. Our schematic is draw in in toric moment map coordinates, assuming either we are in a situation in which such coordinates exist globally or working (and giving such a description) in a small neighborhood of $D_{\{1,2\}}$ where a local $T^2$ action exists).\label{projectivebundle}}
    \centering
    \includegraphics[scale=1.5]{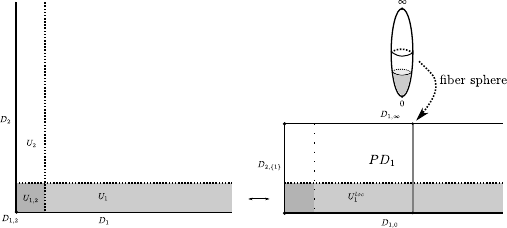}
\end{figure}

We equip $PD_I$ with the fiber product symplectic form where on each factor we
use Hermitian connections $\theta_{e,i}$. More precisely we set
\begin{equation}\label{hatkappa}
    \hatkappa=\kappail(\epsilon'_\ell)^2
\end{equation}
and let 
\begin{align} \label{omegaloc}
    \omega_{loc}:=\sum_{i \in I} \frac{\hatkappa}{2\pi} d(\frac{p_i^2}{1+p_i^2} \theta_{e,i}) +\pi_I^*(\omega_{D_{I}}) 
\end{align} 

As before, for any $i \in I$ we set 
\begin{equation}
    \rho_{i,loc}=\frac{\hatkappa}{\pi} \frac{p_i^2}{1+p_i^2}.
\end{equation}
For $i \notin I$, we set $\rho_{i,loc}:= \rho_i \circ
\pi_I$ where $\rho_i$ are the functions defined earlier in \S \ref{section:SHtor}. 
The cohomology class of the symplectic form \eqref{omegaloc} is Poincare dual to the
divisors 
\begin{align} \label{eq:cohomologyclasslocal}
    [\omega_{loc}] = \sum_{i \in I} \kappail D_{i,0}-(\kappail-\hatkappa) D_{i,\infty} +\sum_{i \notin I} \kappail
    D_{i,I} .
\end{align} 
We let 
\begin{equation}\label{eq:XIloc}
    X_I^{loc}
\end{equation}
be the symplectic manifold which is
the complement of the divisors $D_{i,0}$, $D_{i,\infty}$, $D_{i,I}$. Of course
this is diffeomorphic to $\mathring{N}D_I$ as smooth manifolds.  We denote by
$\Uiloc$ the regions of $PD_I$ where $\rho_{i,loc} \leq
\frac{\hatkappa}{2\pi}$, and set 
\begin{equation}\label{eq:UIlocdef}
    \UIloc:= \cap_{i \in I} \Uiloc.
\end{equation}
We may assume that $\theta_{e,i}$ agree with those fixed in the regularization from \S \ref{section:SHtor} and that $\rho_{i,loc}$ agrees with $\rho_i$.
Doing so gives rise to a symplectic identification:
\begin{align} 
    \label{eq: Uloc} \UIloc \cong U_{I,(\epsilon'_\ell)^2} 
\end{align}

\subsection{Local Floer cohomology in the projective bundle compactification} \label{sect: locHam}

The purpose of this sub-section is to define and calculate a version of the local Floer cohomology group from \S \ref{subsection:spectral} in the projective bundle compactification $PD_I$ of the tubular neighborhood about $D_I$ for suitable Hamiltonian orbits with winding $\v:= \v_0$, for each fixed $\v$ and $I$ as in \eqref{eq:fixv0}  (these are defined analogously to the groups in \S \ref{subsection:spectral}; see \eqref{eq:localchaincomplexPD}). The main result of this sub-section is Lemma \ref{lem: localmorsePDI}, which adapts an argument of Hofer-Salamon to calculate these local Floer cohomology groups in terms of the cohomology of (a neighborhood of) the relevant orbit set $\mathcal{F}^{loc}_\v$, which is homeomorphic to the unit torus bundle $S_I^{log}$ defined in \eqref{eq:sIlog}. 
The local Floer cohomology groups in $PD_I$ are canonically isomorphic to those in $M$ (because in both cases we only consider Floer differentials that lie in small isolating sets that are identified, along with choices of Hamiltonians and almost complex structures on these sets, under \eqref{eq: Uloc}). As a consequence, an essentially immediate Corollary of Lemma \ref{lem: localmorsePDI} is Corollary \ref{lem: Pozniack}, which gives the identification of the $E_1$ page of our spectral sequence with log cohomology.

We begin by transporting all of Hamiltonians (as well as isolating sets) to the projective bundle $PD_I$. To do this, consider the function $h_p^{\ell}(\rho_{1,loc},\cdots, \rho_{k,loc})$ where
$h_p^{\ell}$ are the perturbed versions of the functions $h^{\ell}$ defined in \S \ref{sect: PSSiso1} (see \eqref{eq:hpellthing}). This is a well-defined function on all of
$PD_I$ and following previous conventions we label its orbit sets by
\begin{equation}\label{localorbitset}
    \mathcal{F}^{loc}_\v 
\end{equation}
and fix isolating sets 
\begin{equation}\label{localisolatingset}
    U^{loc}_\v \supset \mathcal{F}^{loc}_\v.
\end{equation}
As in the earlier construction of the Hamiltonians $H_p^\ell$, we may perturb
$h_p^{\ell}$ to a function $H_{loc}^{\ell}: PD_I \to \mathbb{R}$ of the
following form (c.f. \eqref{eq: Hpert2}):
\begin{align} 
    \label{eq: Hpertloc} H^{\ell}_{loc}=  \sum_ \v \delta_\v \rho_\v h_\v + h_p^\ell +\delta_\D \rho_\D h_\D
\end{align} 
where each of the $h_\v$ are supported in $U^{loc}_\v$ and $h_\D$ is supported in a neighborhood of the $D_{i,0}$ and $D_{i,I}$ divisors. However, we no longer require these functions to be Morse except in $U^{loc}_{\v_0}$. The key properties of our perturbation are:
 \begin{itemize} 
    \item The Hamiltonian vector-field $X_{H^{\ell}_{loc}}$ should preserve our divisor $\D$. Moreover, within $\UIloc$, the Hamiltonian $H_{loc}^{\ell}$ should coincide with the Hamiltonian $H_p^{\ell}$ under the identification \eqref{eq: Uloc};

\item $H_{loc}^{\ell}=0$ over the locus where $\rho_{i,loc} \geq \frac{3\hatkappa}{4\pi}$ for all $i \in I$ and $\rho_j \geq  \frac{\kappa_j}{2\pi} (\epsilon_\ell^{p})^2$ for all $j \notin I$. 

    \end{itemize} 

All of the Hamiltonian orbits $x_0$
have canonical capping discs $F(x_0)$: In the case of constant orbits, we take these orbits to be constant capping discs
and in the case of orbits which wind non-trivially around the divisor, we take these orbits to be multiply covered fiber discs passing through the divisors $D_{i,0}$ oriented so that the boundary of $F(x_0)$ is
$x_0$.

\begin{rem} 
    While these Hamiltonians may have some degenerate orbits, only orbits inside of $U^{loc}_{\v_0}$ will be inputs or outputs of our Floer theoretic operations and we will be able to use Lemma \ref{lem: generalHam}, as in the proof of Lemma \ref{lem: MBcompactness}, to rule out undesirable breakings.
\end{rem}

We note for later use that several of our other constructions from \S 2 and \S \ref{sect: PSSiso1} have obvious analogues in the projective bundle, for example we let the hypersuface $\hat{\Sigma}_\ell^{loc}, \Sigma_{\ell}^{loc}$ denotes the local analogues of the hypersufaces $\hat{\Sigma}_{\epsilon_{\ell}^{p}},\Sigma_{\ell}^{p}$ in $X$. In the projective bundle $PD_I$, we work relative to the normal crossings
divisor 
\begin{equation}
    \D = \bigcup_{i \in I} (D_{i,0} \cup D_{i, \infty}) \cup \bigcup_{i \notin I} D_{i,I}.
\end{equation}
We will need
to adapt the Floer theoretic structures introduced in the previous sections to
this local setting.  
\begin{defn} 
Let $J_c(PD_I, \mathbf{D})$ denote the space of complex structures which are 
\begin{itemize} 
    \item split (in the sense of Definition \ref{defn: split}) outside of $\UIloc$ (recall the definition of this region in \eqref{eq:UIlocdef}); and
    \item split in some neighborhood 
        \begin{equation}\label{eq:vlocelldefn}
            V_{loc,\ell} 
        \end{equation}
        of the divisors which is disjoint from all Hamiltonian orbits. 
\end{itemize} 
\end{defn}

\begin{defn} \label{def:Floerindef}
    Let $J_{F}^{loc}(PD_I, \mathbf{D}) \subset C^\infty(S^1,J_c( PD_I,
    \mathbf{D}))$ denote the space of $S^1$-dependent almost complex structures in
    $J_c(PD_I, \mathbf{D})$ that are time independent outside of $\UIloc$.  
\end{defn}

We will assume our complex structures are $\omega_{loc}$ compatible in $U^{loc}_{\v_{0}}$. We can consider Floer's equation for these Hamiltonians: 
\begin{align} \label{eq:FloerRinvloc} 
    \left\{ 
        \begin{aligned}
             & u \colon \mathbb{R} \times S^1 \to PD_I, \\
             &  \partial_s u + J_{F}^{loc}(\partial_tu-X_{H_{loc}^{\ell}})=0.
    \end{aligned}
    \right.
\end{align}
subject to the usual asymptotic constraints. For any two orbits 
$x_0, x_1 \in U^{loc}_{\v_{0}}$, let $\tilde{\mathcal{M}}(X^{loc}_I; x_0,x_1)$ denote the
moduli space of these solutions which satisfy 
\begin{align} 
    u \cdot \D=0 
\end{align} 
and let $\mathcal{M}(X^{loc}_I; x_0,x_1)$ denote this moduli space modulo $\mathbb{R}$-translations. 

\begin{lem} 
    \label{lem: locdiffmorsebott} 
    Fix $\delta_{\v_{0}}$ sufficiently small. For any two orbits $ x_0 \in
    U^{loc}_{\v_{0}}$ and $ x_1 \in U^{loc}_{\v_{0}}$; we have that any $u \in
    \mathcal{M}(X^{loc}_I; x_0,x_1)$ lies in $U^{loc}_{\v_{0}}$. 
\end{lem}
\begin{proof} 
    The symplectic form is exact on $X^{loc}_I$ and so action considerations
    imply that the energy of such a solution must be very small. The argument
    then follows from the same Gromov compactness argument as Lemma  \ref{lem:
    localcompactness}.
\end{proof} 

For $\delta_{\v_{0}}$ sufficiently small, we therefore may define the Floer cochain groups 
\begin{align} \label{eq:localchaincomplexPD}
    CF^*(U^{loc}_{\v_{0}} \subset PD_I, H_{loc}^{\ell}) := \bigoplus_{x \in U_{\v_{0}}} |\mathfrak{o}_x|    
\end{align}

The differential is defined by counting solutions in (a compactification of)
$\mathcal{M}(X^{loc}_I;x_0,x_1)$ as we have seen previously. Denote the resulting
cohomology groups by $HF^*(U^{loc}_{\v_{0}} \subset PD_I, H_{loc}^{\ell})$.  In
fact, for $\delta_{\v_{0}}$ sufficiently small, 
we have a bijection between $\mathcal{M}(X^{loc}_I;x_0,x_1)$ and
$\mathcal{M}(x_0,x_1)$ where in the latter case the orbits and Floer
trajectories are thought of as lying in $X$.
Hence we have a canonical isomorphism: 
\begin{align} \label{eq:locglob} 
    \begin{aligned}  
        HF^*(U^{loc}_{\v_{0}} \subset PD_I, H_{loc}^{\ell}) \cong HF^*(U_{\v_{0}} \subset M,H_{p}^{\ell}) \\ \cong HF^*(X \subset M,H_{p}^{\ell})_{\v_{0}} 
    \end{aligned} 
\end{align}

We now turn to analyzing the cohomology groups $HF^*(U^{loc}_{\v_{0}} \subset
PD_I, H_{loc}^{\ell}).$   For this, we need to recall one more construction.
Consider the Hamiltonian on $PD_I$ 
\begin{align} \label{eq:shiftedhamII}  
    \hat{h}_p^{\ell}(\rho_{1,loc},\cdots, \rho_{\k,loc})=h_p^{\ell}(x) + K(x) 
\end{align}
where $K(x)$ is the Hamiltonian inducing the circle action considered in Equation
\eqref{eq:shiftedham}. For $0<\delta' \leq 1$ and $\delta_{\v_{0}}$ sufficiently
small (depending on $\delta'$) consider the family of functions 
\begin{align}
    \hat{H}_{\delta'}^{\ell}:= \delta' \hat{h}_p^{\ell}(\rho_{1,loc},\cdots,
    \rho_{\k,loc})+\delta_{\v_{0}} \rho_{\v_{0}} \hat{h}_I. 
\end{align}  
In $X^{loc}_I$, this defines a Morse function whose critical points lie entirely
in $U^{loc}_{\v_{0}}$. Set $\hat{H}_{loc}^{\ell}=\hat{H}_{1}^{\ell}$. As
discussed in Section \ref{subsection:spectral}, the spinning construction
induces a bijection of Hamiltonian orbits and Floer trajectories.  To prove
that it induces an isomorphism of local Floer cohomologies, we also need that
this bijection preserves orientation theories. The differences in determinant
lines is measured by a local system defined on the interior of
$\mathcal{F}^{loc}_{\v_{0}}$, where our Hamiltonians orbits are Morse-Bott.  In the
present situation, this local system is trivial
(see Lemma 8.7 of \cite{KoertKwon} or Section 8 of \cite{diogolisi}, 
for a careful discussion of closely related
situations) giving rise to an isomorphism 
\begin{align} 
    HF^*(U^{loc}_{\v_{0}} \subset PD_I , H_{loc}^{\ell}) \cong HF^*(U^{loc}_{\v_{0}} \subset
    PD_I,\hat{H}_{loc}^{\ell}) 
\end{align}  
The family of Hamiltonians $\hat{H}_{\delta'}^{\ell}$ form an isolated deformation i.e. in the neighborhood $U^{loc}_{\v_{0}}$, $\mathcal{F}^{loc}_{\v_{0}}$, are the only family of orbits for all $\delta'$. Local Floer cohomology is invariant under such isolated deformations \cite[Lemma 2.5]{McLeanlocalFloer} or \cite[(LF1) of Section 3.2]{Ginzburg} and so as a consequence we have that 
\begin{align} 
    HF^*(U^{loc}_{\v_{0}} \subset PD_I ,\hat{H}_{loc}^{\ell}) \cong HF^*(U^{loc}_{\v_{0}} \subset PD_I ,
    \hat{H}_{\delta'}^{\ell}).
\end{align}  
for any $\delta'$. For $\delta'$ sufficiently small $\hat{H}_{\delta'}^{\ell}$ is a $C^2$ small
Morse function and we claim that: 

\begin{lem} \label{lem: localmorsePDI} 
    We have an isomorphism 
    \begin{align} 
        HF^*(U^{loc}_{\v_{0}} \subset PD_I, \hat{H}_{\delta'}^{\ell}) \cong H^*(U^{loc}_{\v_{0}}) 
    \end{align} 
\end{lem} 
\begin{proof} 
    For any $q \geq 1$, we set  $\hat{H}_q= \frac{1}{q} \hat{H}_{\delta'}^{\ell}$. We
define $HF^*(X^{loc}_I,\hat{H}_q)$ to be the Floer cohomology of $\hat{H}_q$ in
$X^{loc}_I$ (defining this may require further shrinking $\delta_{\v_{0}}$). In order for this to be well-defined, we need the usual compactness
results to hold, which in our case reduces to showing that: \vskip 5 pt 

\emph{Subclaim:} Trajectories cannot break along Hamiltonian orbits in $\D$:  \vskip 5 pt

\emph{Proof of subclaim:} As in several other proofs below, it is useful to split the proof into two cases: \vskip 5 pt 

\emph{No breaking along $D_{i,\infty}$:} Because $\delta_{\v_{0}}$ is taken arbitrarily small, the energy of the Floer trajectories is 
arbitrarily small as well (independently of $q$). Then a simpler version of the monotonicity argument from Lemma  \ref{lem:
money} (note that the arguments there apply equally well to the space $PD_I$) implies that the trajectories cannot escape $U_{I,\ell}^{loc}$ and hence cannot approach $D_{i,\infty}.$ \vskip 5 pt 

\emph{No breaking along $D_{i,0}$ or $D_{i,I}$:} This follows by action considerations as
in Lemma \ref{lem: anyJ}. As in the previous case, by shrinking
$\delta_{\v_{0}}$ we can suppose the energy of the Floer trajectories is
arbitrarily small. Suppose that a trajectory from $x_1$
to $x_0$ breaks along some $y$ in $D_{i,0}$ or $D_{i,I}$. Using a fiber capping
for $x_1$, $-F(x_1)$, we can cap this trajectory to obtain a capping of $y$
which differs from the fiber capping $-F(y)$ by some element $A.$ As argued in the previous case, the trajectories themselves stay away from $D_{i,\infty}$, and hence we have that  $A \cdot
D_{i, \infty}=0$. So, by examining \eqref{eq:cohomologyclasslocal}, we see
that either $\omega_{loc}(A)=0$ or $|\omega_{loc}(A)| \geq \omega_{min}$ (as defined in \S \ref{subsection:controlPSS}). The rest
proceeds as in the proof of Lemma \ref{lem: anyJ}. 
 \emph{End of proof of subclaim.} \vskip 5 pt

It follows by definition that when $q=1$, 
\begin{align} 
    HF^*(U^{loc}_{\v_{0}} \subset PD_I, \hat{H}_{\delta'}^{\ell}) \cong HF^*(X^{loc}_I, \hat{H}_{\delta'}^{\ell}). 
\end{align} 

We aim to show that these coincide with the cohomologies $H^*(U^{loc}_{\v_{0}})
\cong H^*(X^{loc}_I)$ . To do this, we adapt a well-known argument of Hofer and
Salamon \cite[Lemma 7.1]{Hofer-Salamon} 
to show that for some $q \gg 0$ sufficiently large, all of the Floer
trajectories  are $t$-independent and coincide with Morse trajectories. We
explain briefly why their proof carries over to our non-compact setting. The
first step in their proof is a compactness argument. Given a sequence of curves
$u_q$ with $q \to \infty$ which are not $t$-independent, Hofer and Salamon
employ a compactness argument to show that
one may produce a sequence of curves $v_q$ converging to a possibly broken
$t$-independent solution $v_\infty$. This argument applies here for the same
reasons that Floer cohomology in $X^{loc}_I$ is well-defined i.e. by excluding undesired breakings along $\D$ as in the proof of the above.

In the next step of their proof, they observe that by comparing Fredholm
theories, rigid gradient flow lines are rigid as Floer trajectories and hence
the $v_q$ for large $q$ actually agree with $v_\infty$. From the construction
it follows that the $u_q$ are $t$-independent as well. In our situation, all
$t$-independent trajectories actually lie in $U^{loc}_{\v_{0}}$, where our
complex structure is taken $\omega$-compatible and $t$-independent solutions of
Floer's equation are (negative) gradient flow solutions. In this open subset,
the proof that the Fredholm theories coincide carries over without change. We
conclude that there is an isomorphism: 
\begin{align} 
    HF^*(X^{loc}_I,\hat{H}_q) \cong H^*(X^{loc}_I)   
\end{align} 
Finally, the usual continuation argument shows that the above cohomologies are
independent of $q$ and in particular that there is an isomorphism
\begin{align} 
    HF^*(X^{loc}_I, \hat{H}_{\delta'}^{\ell}) \cong HF^*(X^{loc}_I,\hat{H}_q) 
\end{align} 
\end{proof}

As a byproduct of our local analysis, we now complete the Morse-Bott
(additive) identification of the first page with log cohomology: 
\begin{cor}
    \label{lem: Pozniack} 
For $\delta_{\v_0}$ sufficiently small, there is an identification of cohomologies 
\begin{align} 
    \label{eq:Pzn} 
    HF^*(X \subset M,H_{p}^{\ell})_{\v_{0}} \cong  HF^*(U_{\v_{0}} \subset M,H_{p}^{\ell}) \cong \QH^*(M,\D)_{\v_{0}} 
\end{align} 
where the left-most group is the cohomology of \eqref{eq:CFv}, the middle group is defined just before Lemma \ref{lem: localFloersep} and the right-most group is defined in \eqref{eq:anothereqlogcohv}. 
\end{cor}

\begin{proof} 
    As an immediate consequence of Lemma \ref{lem: localmorsePDI} and
    \eqref{eq:locglob} we have that $$ HF^*(X \subset M,H_{p}^{\ell})_{\v_{0}}
    \cong HF^*(U_{\v_{0}} \subset M,H_{p}^{\ell}) \cong
    H^*(U^{loc}_{\v_{0}}).$$ Now $U^{loc}_{\v_{0}}$ deformation retracts onto
    $\mathcal{F}_{\v_0}^{loc}$, which by the analogue of Lemma
    \ref{lem: SihomFv}, is homeomorphic to $S_{I}^{log}$ (recall \eqref{eq:sIlog}), where $I$ is the support of $\v_0$. The cohomology of $S_{I}^{log}$ by
    definition agrees with $\QH^*(M,\D)_{\v_{0}}$.
\end{proof} 

\begin{rem} 
    A systematic description of the types of (smooth but possibly manifolds-with-corners)
    families of orbits which induce Morse-Bott spectral sequences, generalizing
    the situation considered in Corollary \ref{lem: Pozniack}, has been
    given by McLean \cite{McLeanInPreparation}.  
\end{rem}

\subsection{Local log PSS in the compactification}\label{sect: locPSS}
This sub-section describes versions of the local log $\PSS$ moduli spaces, $\mc{M}^{loc}_{PSS}(\v_0, c, x_0)$, in the projective bundle compactification (Definition \ref{def: higherloc}). By choosing our complex structures suitably, we have a confinement statement for these local log PSS solutions (Lemma \ref{lem:keymonoloc} which is analogous to Lemma \ref{lem: money}), allowing us to define a version of the local log PSS map, $\operatorname{PSS}_{loc,\v_0}$ (see \eqref{eq: PSSdeflocalv}). By construction, we have that under the canonical identification \eqref{eq:locglob}, $\operatorname{PSS}_{loc,\v_0} =\PSSlog^{\v_0}$ (c.f. Corollary \ref{cor:twoPSS}).

\begin{defn} 
Let $J_{S}^{loc}(PD_I, \mathbf{D}) \subset C^\infty(S,J_c( PD_I, \mathbf{D}))$ denote the space of almost complex structures which  
\begin{itemize} 
    \item are surface-independent in a neighborhood of $z=\infty$, as well as in $(X_I^{loc} \setminus \UIloc) \cup V_{loc,\ell}$ (recall the definition of $X_I^{loc}$ in \eqref{eq:XIloc}, $\UIloc$ in \eqref{eq:UIlocdef}, and $V_{loc,\ell}$ in \eqref{eq:vlocelldefn}); 
        and 
                
    \item agree with some $J_F^{loc}$ along the cylindrical end. 
\end{itemize}  
\end{defn} 

For what follows recall the domain $S = \C P^1 \setminus\{0\}$ with
distinguished marked point $z_0 = \infty$, equipped with a cylindrical end and
a distinguished tangent vector to $z_0$ described in \S
\ref{section:lowenergydef}. Recall also that $\v_0 \in \Z_{\geq 0}^k$ is a
fixed non-zero multiplicity vector and $I$ is its support.
\begin{defn} 
Fix $J_S^{loc} \in \mathcal{J}_S(PD_I,\mathbf{D})$. For every orbit $x_0$ of in $U^{loc}_{\v_{0}}$ define a moduli space 
\begin{equation}\label{localmodulispace2}
\mc{M}^{loc}_{PSS}(\v_0, x_0)
\end{equation}
to be the space of maps
\[
u: S \ra PD_I
\]
satisfying Floer's equation
\begin{equation} \label{eq:PSSeq2}
    (du - X_{H_{loc}^\ell} \otimes \beta)^{0,1} = 0
\end{equation}
with incidence and tangency conditions
\begin{align}
    &u(z)\notin \D\textrm{ for }z \neq z_{0};\\
    &u(z_{0})\textrm{ intersects $D_{i,0}$ with multiplicity $(\v_0)_i$ for $i \in I$};\\
    &u(z_0)\textrm{ does not intersect $D_{i,\infty}$ or $D_{j,I}$ for any $i \in I$ or $j\notin I$},
\end{align}
and asymptotics
\begin{align}\label{eq:limitingconditionlocalPSS}
    \lim_{s \ra -\infty} u(\epsilon(s,t)) &= x_0.
\end{align}

 \end{defn} 
As before the incidence and tangency conditions give us an enhanced evaluation
map \eqref{eq:enhancedevaluationpss} $\Evzo: \mc{M}^{loc}_{PSS}(\v_0, x_0) \to
\mathring{S}_I$.  
 \begin{defn} \label{def: higherloc} 
     Let $c \in \critfI$ be a critical point of $f_I$.
    The moduli space $\mc{M}^{loc}_{PSS}(\v_0, c, x_0)$ is defined to be the fiber product 
   \begin{align} 
        \label{eq:evalloc} \mc{M}^{loc}_{PSS}(\v_0, x_0) \times _{\Evzo} W^s(f_I,c).
    \end{align}
 \end{defn}  \vskip 5 pt
As we have seen previously, for generic $J_S^{loc}$ these moduli spaces are 
manifolds of dimension \eqref{vdimlogpss}. It is possible to prove a compactness result for these moduli spaces for 
generic $J_S^{loc}$ (provided the area of the fiber spheres are taken sufficiently large). However, for our purposes this is unnecessary: it 
will be sufficient to work with $J_S^{loc}$ which are close to a split (over all of $PD_I$) $J_0$, for which the following confinement statement holds.

\begin{lem} \label{lem:keymonoloc} 
    Fix $x_0 \in U^{loc}_{\v_{0}}$. 
    For $\Sigma_{\ell}^{loc}$ sufficiently $C^0$ close to $\hat{\Sigma}_{\ell}^{loc}$, $\epsilon_{\ell}^p$ and
    all $\delta_{\v}$, $\delta_\D$ sufficiently small, and $J_S^{loc}$ sufficiently close
    to a split $J_0$, 
    any $\operatorname{PSS}_{loc,\v_0}$ solution $u \in \mc{M}^{loc}_{PSS}(\v_0,c,x_0)$ must lie in 
    $U^{loc}_{I,(\epsilon'_\ell)^2}$. \qed
 \end{lem} 
\begin{proof}The Lemma is proven exactly as in Lemma \ref{lem: money} (\emph{mutatis mutandis}). \end{proof} 

  See Figure \ref{localPSSprojectivebundle} for a schematic picture of the (thimble component of) a local log PSS curve in a projective bundle.
  \begin{figure}[h] 
      \caption{A picture of the (Riemann surface part of) a local log PSS solution in a projective bundle about $D_1$ (using the same toric picture as in Figure \ref{projectivebundle}), which by Lemma \ref{lem:keymonoloc} must stay in the tubular neighborhood $U^{loc}_{1(,\epsilon'_{\ell})}$.
 \label{localPSSprojectivebundle}}
    \centering
    \includegraphics[scale=1.5]{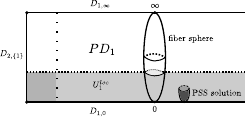}
\end{figure}

Choose a generic $J_S^{loc} \in J_{S}^{loc}(PD_I, \mathbf{D})$ which is sufficiently close to some split $J_0$ so that Lemma \ref{lem:keymonoloc} holds.
It follows that for any $\alpha_c \in |\mathfrak{o}_c| \subset CM^{*}(\SIo)$, we can 
define
$\operatorname{PSS}_{loc,\v_0} (\alpha_c) \in CF^{*}(U^{loc}_{\v_{0}} \subset PD_I, H_{loc}^{\ell})$ 
by
\begin{equation} 
 \operatorname{PSS}_{loc,\v_0} (\alpha_c) = \sum_{x_0, \operatorname{vdim}(\mc{M}^{loc}_{PSS}(\v_0,c,x_0))=0} \sum_{u \in \mc{M}^{loc}_{PSS}(\v_0, c,x_0)} \mu_u (\alpha_c)
\end{equation}
where $\mu_u: \mathfrak{o}_c \ra \mathfrak{o}_{x_0}$ is the induced isomorphism induced on orientation
lines we have seen previously. 
Lemma \ref{lem:keymonoloc} implies that (linearly extending the above
definition to all of $CM^{*}(\SIo)$) $\operatorname{PSS}_{loc,\v_0}$ defines a cochain map, giving
rise to a well-defined cohomological map (with the same name, by standard abuse
of notation):
\begin{align} \label{eq: PSSdeflocalv}
    \operatorname{PSS}_{loc,\v_0}: H^*(\SIo) t^{\v_0} \to  HF^{*}(U^{loc}_{\v_{0}} \subset PD_I, H_{loc}^{\ell}).
\end{align}

  In fact, it follows from the construction that:

\begin{cor} \label{cor:twoPSS} 
    Under the canonical identification \eqref{eq:locglob}, there is an equality of operations
    \begin{align} 
        \operatorname{PSS}_{loc,\v_0}=\PSSlog^{\v_0}.
    \end{align} where $\PSSlog^{\v_0}$ is defined in \eqref{eq:realPSSloc}. \qed 
\end{cor}

\subsection{Local log SSP} \label{section: inverseloc}

In this section, we finally construct a moduli space (and operations) which enable us to invert the
local log PSS maps (on one side). The moduli space is defined in Lemma \ref{def:SSPmodulispace} and its compactness properties are studied in Lemma \ref{lem: SSPloccompact}. These compactness properties enable us to use these moduli spaces to define a \emph{local SSP map}, $\operatorname{SSP}_{loc,\v_0}: HF^*(U^{loc}_{\v_{0}} \subset PD_I, H_{loc}^{\ell}) \to H^*(\SIo) t^{\v_0}$  (see   \eqref{eq:localSSPmapv2}). The fact that this map is a one-sided inverse to the local PSS map is proven in Lemma \ref{lem: degenerations}.

 We begin by choosing all of our data (including some generic $J_S^{loc}$ sufficiently close to some split $J_0$) so that Lemma \ref{lem:keymonoloc} holds.  For this section, it will be convenient to take our slope $\lambda_{\ell}$  very large in order to easily rule out certain breakings in Lemma \ref{lem: SSPloccompact}.  To this end, we assume that the slope has been chosen sufficiently large so that \begin{align} \label{eq:randomCeeq} \lambda_{\ell} > (2C_e)^2w_p(\v_0) \end{align}  (where $C_e = \frac{\epsilon_{\ell}'}{\epsilon_{\ell}^p}$ is as before).
 We may assume for simplicity that the set of constant orbits $\mathcal{F}_{\mathbf{0}}$  of $H_{loc}^{\ell}$ 
is a connected submanifold-with-corners in $PD_I$ (containing the locus where $\rho_{i,loc} \geq \frac{3\hatkappa}{4\pi}$ for all $i \in I$ and $\rho_j \geq  \frac{\kappa_j}{2\pi} (\epsilon_\ell^{p})^2$ for all $j \notin I$) and which intersects $D_{I,\infty}$ in a closed
submanifold-with-corners. Let
$\mathring{\mathcal{F}}_\mathbf{0}$ denote its interior and consider Floer
trajectories $\mathring{u}$ in $X_I^{loc}$
asymptotic as $s\to-\infty$ to an orbit $x_0 \in
\mathring{\mathcal{F}}_\mathbf{0}.$
As Floer curves are $J$-holomorphic curves in a neighborhood $x_0$ (for some time independent $J$, recall Definition \ref{def:Floerindef}),  Gromov's
removal of singularities theorem implies that by setting
\begin{align}
    u(0):=x_0
\end{align}
any such trajectory $\mathring{u}$ extends to a smooth map $u$ from the
``thimble domain" $S^{\vee} :=\mathbb{R} \times S^1 \cup \lbrace 0 \rbrace
\cong \C P^1 \setminus \{\infty\}$ \footnote{Here as in \S \ref{section:lowenergydef} we are viewing $\R \times S^1$ as embedded in $\C P^1$ via \eqref{cylend}.} which is $J$-holomorphic
in some neighborhood of the compactification point $0 \in \C P^1 \setminus \{\infty\}$. 
\begin{defn}  \label{defn: basicSSP}
For $x_1 \in U_{\v_{0}}$, we let $\widetilde{\mathcal{M}}^{loc}_{SSP}(
x_1)$ denote the moduli space of maps $u: S^{\vee} \to
PD_I$ with $u(0) \in \mathring{\mathcal{F}}_\mathbf{0}$ such that $u$ solves Floer's equation \eqref{eq:FloerRinvloc} on the complement of 
$z=0$ (and hence extends holomorphically to zero) with asymptotic condition
\begin{align} \label{eq: SSPprecond}
    &\lim_{s \to +\infty} u(s, -) = x_1. 
\end{align} 
We define $\mathcal{M}^{loc}_{SSP}(x_1)$ to be the quotient of this moduli space by $\mathbb{R}$ translations of the domain cylinder. 
\end{defn}

We again place incidence conditions on this moduli space.

\begin{defn} \label{def:SSPmodulispace}
    We let $\mathcal{M}^{loc}_{SSP}(\v_0, x_1)$ denote the moduli space of maps $u \in \mathcal{M}^{loc}_{SSP}(x_1)$
such that $u(0) \in \mathring{\mathcal{F}}_\mathbf{0} \cap D_{I,\infty}$ and such that
\begin{align}
    & u(z)\notin \D\textrm{ for }z \neq 0;\\
    &u(0)\textrm{ intersects $D_{i,\infty}$ with multiplicity $(\v_0)_i$ for $i \in I$};\\
    & u(0)\textrm{ does not intersect $D_{i,0}$ or $D_{j,I}$ for any $i \in I$ or $j\notin I$}.
\end{align} 
\end{defn}
Maps $u: S^{\vee} \to PD_I$ satisfying the conditions of Definition \ref{defn: basicSSP}
fit into a natural Fredholm theory and are cut out regularly for generic
choices of $J_t$.\footnote{By restricting to the Floer solutions in the complement of $z=0$, one can also place these maps into the more general framework of Morse-Bott Fredholm theory as explained in \S 6.1 of \cite{diogolisi}.} 
As in the case of $\PSS_{log}$ moduli spaces, the moduli spaces
$\mathcal{M}^{loc}_{SSP}(\v_0, x_1)$ are obtained by placing incidence
conditions with $\D$ and transversality for these moduli spaces can be handled
by similar methods as those explained in \cite[\S 4.4]{GP1}. It follows
that for generic choices of almost complex structure, $\mathcal{M}^{loc}_{SSP}(\v_0, x_1)$ is a manifold of
dimension
\begin{align}
    \operatorname{vdim}(\mathcal{M}^{loc}_{SSP}(\v_0, x_1))=2n-1 + 2\sum_i v_i(1-a_i) - \deg(x_1). 
\end{align}

Let  $i_{\v_{0}}: \mathbb{R}^{+} \mapsto (\mathbb{R}^{+})^{I}$ denote the
embedding $a \to a^{\v_{0}} := (a^{(\v_0)_i})_{i \in I}$ 
for any $a \in \mathbb{R}^{+}$.  The group $\mathbb{R}^{+}$ acts by scaling (by
different factors) in each fiber coordinate $\mathring{ND_I}$ 
via the embedding $i_{\v_{0}}$, and we denote by
$\mathring{ND_I}/\mathbb{R}^{+}$ the resulting smooth quotient.
Set $G:=(\mathbb{R}^{+})^{I}/\mathbb{R}^{+}$. The natural projection map
$\mathring{ND_I}/\mathbb{R}^{+} \to \SIo$ is a homotopy equivalence which we
use to identify the cohomology of the two spaces. More precisely, there is a
$G$-equivariant isomorphism
\begin{align} \label{NDIquotientsplit}
    \mathring{ND_I}/\mathbb{R}^{+} \cong G \times \SIo. 
\end{align} 
and in particular there are projection maps $\pi_G:
\mathring{ND_I}/\mathbb{R}^{+} \to G$ and $\pi_{S_I}:
\mathring{ND_I}/\mathbb{R}^{+} \to \SIo$.

Let $f_G: G \to \mathbb{R}$ be an outward pointing (at infinity) Morse function with a single
critical point at the identity element of $G$.  On
$\mathring{ND_I}/\mathbb{R}^{+}$, we may consider the ``split'' (with respect to \eqref{NDIquotientsplit}) Morse function $\hat{f}_I$
which given by 
\begin{align}
    \hat{f}_I := \pi_G^*(f_G) + \pi_{S_{I}}^*(f_I),
\end{align} 
(where $f_I$ is the Morse function we have previously fixed on $\mathring{S}_I$ in \S \ref{subsec:logcoh})
and similarly consider the product metric $\hat{g}_I = \pi_G^*(g_G) + \pi_{S_I}^* g_I$
where $g_I$ is the metric on $S_I$ fixed previously and $g_G$ is any metric
making $(f_G, g_G)$ a Morse-Smale pair.
We let $\hat{f}^{pert}_{I}$ be another (currently unspecified) Morse function
that is a small compactly supported perturbation (which is compactly supported
in fiber directions and supported outside of the neighborhoods $U_j \cap D_I$
for $j \notin I$) of $\hat{f}_I$. 

Denote the marked point at $z=0$ on $S^{\vee}$ by $z_1$. We equip this with an asymptotic marker that points in the positive real direction. As with the $\PSSlog$ map, we have an enhanced evaluation 
map defined at this marked point:
\begin{align}
    \operatorname{Ev}_{z_{1}}^{L}:\mc{M}^{loc}_{SSP}(\v_0, x_1) \to \mathring{ND}_{I,\infty}/\mathbb{R}^{+}.
\end{align}
which is a lift of the enhanced evaluation map defined previously in \eqref{eq:enhancedevaluationpss} given by only quotienting the tuple of all $\vi$ normal jets of the map by the diagonal $\mathbb{R}^{+}$ action (instead of real-oriented projectivizing). We can compose this with the natural diffeomorphism 
\begin{equation}\label{taumap}
    \tau: \mathring{ND}_{I,\infty}/\mathbb{R}^{+} \stackrel{\cong}{\to} \mathring{ND_{I}}/\mathbb{R}^{+}
\end{equation}
which in each fiber copy of $\mathbb{C}^*$ is given by the map $z \mapsto z^{-1}$ to obtain: 
\begin{align}\label{hatenhancedevaluation}
 \widehat{\operatorname{Ev}}_{z_{1}}:= \tau  \circ \operatorname{Ev}_{z_{1}}^{L}:\mc{M}^{loc}_{SSP}(\v_0, x_1) \to \mathring{ND_{I}}/\mathbb{R}^{+}.
\end{align}
Let $b$ be any critical point of the function $\hat{f}^{pert}_I$. By
construction we know that $b \in ND_I \setminus \cup_{j \notin I}
\pi_I^{-1}(U_j \cap D_I)$. Using $b$ and the map
\eqref{hatenhancedevaluation}, we 
form the moduli space:
\begin{align}
    \mc{M}^{loc}_{SSP}(\v_0, x_1,b) := \mathcal{M}^{loc}_{SSP}(\v_0,x_1) \times_ {\widehat{\operatorname{Ev}}_{z_{1}}} W^u(\hat{f}^{pert}_I,b),
\end{align}
where the unstable manifold $W^u$ is taken with respect to the metric
$\hat{g}_I$ defined above. See Figure \ref{sspschematic} a picture of the above definition. 
\begin{figure}[h] 
  \caption{The domain of the moduli space $\mc{M}^{loc}_{SSP}(\v_0, x_1,b)$ along with the incidence, matching, asymptotics, and PDE satisfied at various places.\label{sspschematic}}
    \centering
    \includegraphics[scale=1.0]{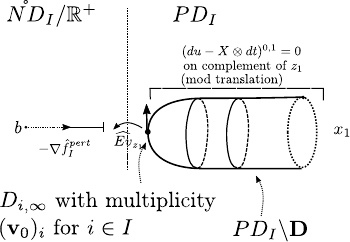}
\end{figure}
For generic choices of unstable manifold and complex
structure (making the constituents of the fiber product and the fiber product
itself transversely cut out), this is a manifold of dimension
\begin{align}
    \operatorname{vdim}(\mc{M}^{loc}_{SSP}(\v_0, x_1,b)) &=  \operatorname{vdim}(\mathcal{M}_{SSP}^{loc}(\v_0, x_1))-(2n-1-\deg(b))\\
    &=\deg(b) + 2\sum_i v_i(1-a_i) - \deg(x_1). 
\end{align} 
We observe that any point in $W^u(\hat{f}^{pert}_I,b)$ projects to a point inside of $D_{I,\infty} \setminus \cup_{j\notin I} U_j \cap D_{I,\infty}$ and in particular lies far away from $\partial \mathcal{F}_\mathbf{0}:= \mathcal{F}_\mathbf{0}\setminus \mathring{\mathcal{F}}_\mathbf{0}.$ We fix a small neighborhood $U_\infty \subset \mathring{\mathcal{F}}_\mathbf{0}$ which contains $D_{I,\infty} \setminus \cup_{j\notin I} U_j \cap D_{I,\infty}$.

In order to use these moduli spaces to define operations, we must as usual demonstrate that they have nice compactifications. To this end, we make a simple observation which will be used in our arguments below: 

\begin{lem} \label{lem:spherebubblefiber} Let $C_r$ be an arbitrary constant bigger than zero. Then if $\epsilon_\ell'$ is sufficiently small, and $J \in J_c(PD_I,\D)$ is an almost complex structure sufficiently close to a split $J_0$, any $J$-holomorphic sphere of area less than $C_r(\epsilon_\ell')^2$ represents a fiber homology class. \end{lem} 
\begin{proof}  Given a $J_0$-holomorphic sphere $u: \mathbb{C}P^1 \to PD_I$ of area less than $C_r(\epsilon_\ell')^2$, we can project it along the map $\pi_I: PD_I \to D_I$ to a holomorphic sphere in the base $D_I$. If $\epsilon_\ell'$ is sufficiently small, the symplectic area of this projection is smaller than $\omega_{min}$ (c.f. \S \ref{subsection:controlPSS}), and hence this projection is constant. Thus, $u$ lies in a fiber $(\mathbb{C}P^1)^{|I|} \to PD_I$.  The result now follows from Gromov compactness.  \end{proof}

The following is the key compactness result for $\operatorname{SSP}$ moduli spaces that will enable us to define the local log $\operatorname{SSP}$ map in \eqref{eq:localSSPmapv2}. 

\begin{lem}\label{lem: SSPloccompact} 
Fix a split complex structure $J_0 \in J_c(PD_I,\D)$. Suppose as usual that $\epsilon_{\ell}^{p}$ (and hence $\epsilon_\ell'$) and all parameters $\delta_\v$, $\delta_\D$ are sufficiently small and $\Sigma_\ell^{p}$ is sufficiently $C^0$-close to $\hat{\Sigma}_{\epsilon_\ell^{p}}$. 
Then there exists a generic $J_t \in J_F^{loc}(PD_I,\D)$ which is very close to $J_0$ such that: 
\begin{itemize} 
    \item If $\operatorname{vdim}(\mc{M}^{loc}_{SSP}(\v_0, x_1,b))$
        the moduli spaces $\mc{M}^{loc}_{SSP}(\v_0,b,x_1)$ are compact.
    
    \item If $\operatorname{vdim}(\mc{M}^{loc}_{SSP}(\v_0, x_1,b)) = 1$, then
        $\mc{M}^{loc}_{SSP}(\v_0,b,x_1)$
    admits a compactification (in the sense of Gromov-Floer convergence) 
       to a (compact 1-dimensional) manifold-with-boundary
       $\overline{\mc{M}}^{loc}_{SSP}(\v_0,b,x_1)$  such that
        $\partial \overline{\mc{M}}^{loc}_{SSP}(\v_0,b,x_1):= \partial_F \bigsqcup \partial_M$ where 
    \begin{align}
        \partial_F := \bigsqcup_{x' \in U^{loc}_{\v_{0}},\deg(x')-\deg(x_1)=1} \mc{M}(x',x_1) \times \mc{M}^{loc}_{SSP}(\v_0,b, x') \\ 
        \partial_M:= \bigsqcup_{b',\deg(b)-\deg(b')=1}\mc{M}^{loc}_{SSP}(\v_0,b', x' ) \times \mc{M}(b,b').
    \end{align}  
\end{itemize}
\end{lem} 
\begin{proof}  
    The energy of a solution $u \in \mc{M}^{loc}_{SSP}(\v_0, x_1,b)$ is arbitrarily close to
    $\frac{1}{2}w_p(\v_0)((2\epsilon'_\ell)^2- (\epsilon_{\ell}^p)^2)$. We will consider sequences $u_m \in \mc{M}^{loc}_{SSP}(\v_0,b',x_0)$ ($m \in \mathbb{N}$) which, as in standard proofs of Gromov compactness in the $\mathbb{R}$-invariant setting, we rescale so that $s=0$ is the first $s$ value which exits $U_\infty$. No sphere bubbles (or more precisely Floer breaking where Hamiltonian is zero) can occur in $D_{I,\infty}$ because such a
    sphere bubble would have too high energy in view of Lemma \ref{lem:spherebubblefiber}. 
    Because the evaluation $Ev_{z_{1}}^{L}$ is bounded away from the zero
    divisors of $ND_I$ (i.e., the leading order jets going into the definition
    of $Ev_{z_1}^{L}$ are constrained to lie along $W^u(\hat{f}_I^{pert},b)$ which
    is compact and away from zero or infinity),  it follows that the
    intersection multiplicity at $z_1$ with $D_{i,\infty}$ of any limit is still
    $(\v_{0})_i$. 

 We choose a sequence of parameters $\delta_{\v,n},\delta_{\D,n}$ from \eqref{eq: Hpertloc} depending on $n \in \mathbb{N}$ and which tend to 0 as $n \to \infty.$
 Let $H_{n}^{\ell}$ denote the corresponding functions which are perturbed according to the same equation (\eqref{eq: Hpertloc}) and which
 converge to $h_{loc}^{\ell}$ as $n \to \infty$.
 Also, choose a sequence of $J_{t,n}$ converging to $J_0$ with $J_{t,n}$ 
 generic (meaning achieves transversality) for all elements of
 $\mc{M}^{loc}_{SSP}(\v_0,x_0)$ and for all Floer trajectories (both defined
 with respect to the functions $H_n^{\ell}$) appearing in the statement of this
 Lemma. As before, by appealing to usual Gromov compactness arguments, the key
 claim that verifies the Lemma is the exclusion of various ``undesired''
 breakings from any limiting stable curve, 
 namely: \vskip 5 pt 

\emph{Subclaim:} 
For all $n$ sufficiently large there are no trajectories breaking along orbits of $H_n^{\ell}$ in
$D_{i,\infty}$, $D_{j,I}$, $D_{i,0}$. More precisely, fixing such $n$ sufficiently large, there are no sequences
$u_{m,n} \in \mc{M}^{loc}_{SSP}(\v_0,x_0)$ rescaled as above which converge (modulo sphere
bubbling at finitely many possible points) in $C^{\infty}_{loc}$ to solutions
$u_{\infty,n}$ which asymptotically limit to any orbit $y$ in $D_{i,\infty}$ $D_{i,0}$ or
$D_{j,I}$, in the generalized sense that there exists a sequence of $s_k \to
\infty$ with $u_{\infty,n}(s_k,-) \to y$.\footnote{Recall that these orbits $y$ may be degenerate; compare the discussion just below Lemma \ref{lem: generalHam}.} 

\emph{Proof of subclaim:} Our argument will work to exclude breakings along
each ``bad'' orbit set above one at a time for all $n$ sufficiently large;
since there are only finitely many such orbit sets, the desired result will
follow.  
To begin, 
suppose such a
sequence $u_{m,n}$ (as in the statement of
the subclaim) existed for every $n$, for an orbit $y$ in any of the divisors. Then by diagonalizing the sequence
$\{u_{m,n}\}$, we can find a collection of numbers $m_n$ and $s_n$ so that $u_{m_{n}, n}(s_n,-)$ lies in a
tubular neighborhood of of size $\epsilon'/n$ for $\epsilon'>0$ (with respect to the Riemannian metric determined by $J_0$ and $\omega_{loc}$) of
the relevant divisors. After passing to a subsequence, these converge in $C^{\infty}_{loc}$
(a priori modulo sphere bubbling at finitely many points) to an SSP solution
$u_0$ for the Hamiltonian $h_{loc}^{\ell}$ and with split complex structure
$J_0$. Let us now show that such a $u_0$ does not develop a limit (as before in the sense of Lemma \ref{lem: generalHam}) along any of the relevant divisors:

First, let us show such a $u_0$ stays away (i.e., does not develop a limit
along) $D_{j,I}$ for $j \notin I$.
We note by definition that $Ev_{z_1}^{L}(u_0)$ lies in the unstable locus
$W^u(\hat{f}^{pert}_I,b)$, hence projects to a point in $D_I\setminus \cup_j U_j
\cap D_I$. Meanwhile the function $h^{\ell}_{loc}$ is independent of $\rho_{j,loc}$
whenever $\rho_{j,loc} \geq \epsilon_\ell^{p}$, so it follows that (since the
projection map $PD_I \to D_I$ is $J_0$-holomorphic since $J_0$ is split),
the projection of $u_0$ is $J$-holomorphic
in $D_I \setminus \cup_j U_{j,\epsilon_\ell^{p}} \cap D_I$. If the projection
of $u_0$ is nonconstant, and did not stay away from $D_{j,I}$, we would obtain a holomorphic curve about $ev_{z_{1}}(u_0)$
with boundary along $\partial{U}_{j,\epsilon_{\ell}^{p}} \cap D_I$. 
Because the
$\omega_{D_{I}}$ energy of the projection is bounded by some $g_I \cdot
\frac{1}{2}w_p(\v_0)((2\epsilon'_\ell)^2- (\epsilon_{\ell}^p)^2)$, monotonicity (Lemma \ref{lem:monotonicity})
implies that such a curve would have to energy bounded below by a constant
which can be taken larger than the energy of $u$ and therefore $u_0$ (by
shrinking $\epsilon_{\ell}'$ as needed), so this is impossible. Thus, indeed $u_0$
stays away from $D_{j,I}$.

The above argument further implies that the projection of $u_0$ to $D_I$ is (after
applying removal of singularities) a closed holomorphic curve which, since it has energy less than
the minimal energy of a sphere in $D_I$ (by taking $\epsilon_{\ell}'$ small),
must be constant. Hence $u_0$ in fact lies in a fiber of $PD_I$.

Next, we observe that $u_0$ cannot have any other intersection points $z^*$ (apart from $z_1$) with any of the divisors $D_{i,0}$ or $D_{i,\infty}.$ This is because such an intersection would necessarily have some positive multiplicity $\operatorname{ord}(z^*)>0.$ The rescaling argument of (\cite[Lemma 4.9]{Tehrani},\cite[Proposition 7.3]{Ionel:2003kx}, see also \S \ref{subsec:topologicalpairs} below for a detailed review of these results) would then produce an attached configuration of sphere bubbles which intersects the same divisor with multiplicity $-\operatorname{ord}(z^*)>0$ which is impossible in a fiber of $PD_I$ (even if such a configuration were contained in the relevant divisor). 

This observation allows us to show that $u_0$ does not develop a limit along any $D_{i,\infty}$ (for $i \in I$). If it did, $u_0$ would wind positively around $D_{i,\infty}$ near $y$ because the function $h_\ell^{loc}$ is independent of the variables $\rho_{i,loc},\theta_{i,loc}$ (and thus locally projects holomorphically onto the $i$-th factor of $(\mathbb{C}P^1)^{|I|}$). For homology class reasons, this would imply that $u_0$ would need to have an  intersection point with $D_{i,0},$ which the previous paragraph showed to be impossible.

Now we argue that $u_0$ cannot develop a limit along $D_{i,0}$ for some $i \in I$; suppose it did along some $y$ in $D_{i,0}$. The
   class of $u_0$ must be the connect sum of the fiber capping disc $F(y)$ (described in \S \ref{sect: locHam}) and some other
   class $A(u_0) \in H_2((\mathbb{C}P^1)^{|I|})$.  The intersection multiplicities of the class $A(u_0)$ with $D_{i,\infty}$ 
imply that $\omega_{loc}(A(u_0))= w_p(\v_0)(\epsilon_\ell')^2$. As the energy of the fiber capping disc $F(y)$ itself (with its induced orientation) is nonpositive, we have that \begin{equation} \label{eq:intestSSP1} \int_{u_{0}} \omega_{loc} \leq w_p(\v_0)(\epsilon_\ell')^2. \end{equation} But along $D_{i,0}$,
we know from equations \eqref{eq:randomCeeq} and \eqref{eq:estimateD} (which
holds equally well in the projective bundle setting) that 
\begin{align}  \label{eq:needestSSP2} 
    H_{loc}^{\ell}>\frac{1}{2}\lambda_{\ell} (\epsilon_{\ell}^p)^2> 2w_p(\v_0)(\epsilon'_\ell)^2.
\end{align}  
Combining \eqref{eq:intestSSP1} and \eqref{eq:needestSSP2} implies that the energy of the solution $u_0$
\begin{align} \operatorname{E}_{top}(u_0) < w_p(\v_0)(\epsilon'_\ell)^2 -\frac{1}{2}\lambda_{\ell} (\epsilon_{\ell}^p)^2<0 \end{align}
would be negative and hence cannot
exist.
Given that $u_0$ cannot limit to any of the $D_{i,0}$, $D_{i,\infty}$ or
$D_{j,I}$, it follows that, because the function $h_{loc}^{\ell}$ is Morse-Bott
in the fibers, $u_0$ must converge (in fact exponentially
quickly to) to a unique orbit $z \in \mathcal{F}^{loc}_\v$. This implies two key facts.  The first is that because this $z$ lies in $\mathcal{F}^{loc}_\v$, 
we have by direct calculation that\footnote{In particular, no sphere bubbling can occur in the limiting process.} 
\begin{align} 
    E_{top}(u_0)=\operatorname{lim}_nE_{top}(u_{m_n,n}).
\end{align} 
 The second is that by convergence, there exists $s^*$ for which
$u_{0}$ does not escape $U_{\v_{0}}$ for $s\geq s^*$. Choose $n$ sufficiently
large so that $u_{m_{n},n}$ is close to $u_0$ for $s<s^*$ and so that
$u_{m_{n},n}$ lies in $U_{{\v_0}}$. Because our SSP solution $u_0$ does not
intersect $D_{i,0}$ and $U_{\v_{0}}$ lies a bounded distance
away from $D_{i,0}$, it follows that after possibly enlarging $n$ we can ensure that
$s_n>s^*$. In particular, $u_{m_{n},n}$ exits $U_{{\v_0}}$ at some parameter
$s'_n$ for $s'_n>s^*$. Then by considering the rescaling $u_{m_n,n}(s+s'_n)$, we
obtain a Floer trajectory that passes between an orbit $z'$ in
$\mathcal{F}^{loc}_\v$ and through a point on $\partial{U}_{{\v_0}}$. However, because $E_{top}(u_0)=\operatorname{lim}_nE_{top}(u_{m_n,n})$,  
this rescaled curve has
zero energy, hence must be constant in $s$, a contradiction.  \emph{End proof of subclaim.}

With this established, we see that for $n$ sufficiently
large, SSP solutions with respect to $H^{\ell}_{n}$ cannot develop any other intersections with the divisors $D_{i,0}$, $D_{i,\infty}$ or $D_{j,I}$ by positivity of intersection. This also implies that they cannot develop limits along
other orbit sets in $X_{I}^{loc}\setminus U_{\v_{0}}$ by the following argument: suppose that a broken SSP trajectory limits to some $y$ in $X_{I}^{loc}\setminus U_{\v_{0}}$. Then the energy of such a solution is up to small error  $((\epsilon_\ell')^2-1)w_p(\v_0) - \mathcal{A}_\ell(y)$. 
As our actions are
separated (see \S \ref{sect: PSSiso1}), this energy can be taken larger than the initial energy of our SSP solutions, hence such a broken trajectory does not exist. 
This concludes the proof of the Lemma.
\end{proof}
See Figure \ref{localSSPprojectivebundle} for a schematic picture of the (thimble component of) a local log SSP curve in a projective bundle.
  \begin{figure}[h] 
      \caption{A picture of the (Riemann surface part of) a local log SSP solution in a projective bundle about $D_1$ (using the same toric picture as in Figure \ref{projectivebundle}).
 \label{localSSPprojectivebundle}}
    \centering
    \includegraphics[scale=1.5]{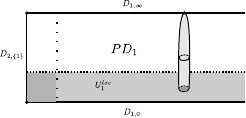}
\end{figure}

Choosing a $J_t$ as in the Lemma, for an orbit $x_1 \in U_{\v_{0}}$ and element
$\alpha_c \in |\mathfrak{o}_{x_1}|$, we define 
\begin{align}
    \operatorname{SSP}_{loc,\v_0}(\alpha_c):= \sum_{b,\operatorname{vdim}(\mc{M}^{loc}_{SSP}(\v_0,x_1, b))=0} \sum_{u \in \mc{M}^{loc}_{SSP}(\v_0,x_1, b)} \mu_u(\alpha_c)  t^{\v_{0}},
\end{align}
where as usual $\mu_u: \mathfrak{o}_{x_1} \to \mathfrak{o}_b$ is the
isomorphism of orientation lines induced by a rigid element of the moduli space
$u$. The Lemma implies that this prescription gives rise to a (finite, well-defined) count and moreover
a cochain map: 
\begin{align}
    CF^{*}(U^{loc}_{\v_{0}} \subset PD_I, H_{loc}^{\ell}) \to CM^*(\mathring{ND}_I/\mathbb{R},\hat{f}^{pert}_I) t^{\v_0}.
\end{align}
This induces a well-defined cohomological map, which we call the {\em local log SSP
map} with multiplicity $\v_0$ (we may sometimes refer to the cochain map
as local log SSP too):
\begin{align} \label{eq:localSSPmapv2}
    \operatorname{SSP}_{loc,\v_0}:  HF^*(U^{loc}_{\v_{0}} \subset PD_I, H_{loc}^{\ell})  \to H^*(\SIo) t^{\v_0} 
\end{align} 

We will need one final moduli space which is constructed again from the domain
$S = \mathbb{C} P^1 \backslash \{0\}$ used in the definition of local log PSS in \S \ref{sect: locPSS}, 
equipped with the
Floer data from that section.  
For the moduli space defined in Definition \ref{defn: interpolateme}, we assume that the complex structure $J_S^{loc}$ agrees with some split (surface-independent) $J_0$ near $z=\infty$ as well as in a neighborhood of $D_{i,\infty}$. We further assume that $J_S^{loc}$ is sufficiently close to $J_0$ so that Lemma \ref{lem:keymonoloc} holds and that it agrees on the cylindrical end with a $J_t$ so that Lemma \ref{lem: SSPloccompact} holds.
As above, any PSS
solution $\mathring{u}:S \to PD_I$ which is asymptotic to an orbit $x_0$ for some $x_0 \in
\mathring{\mathcal{F}}_\mathbf{0} \cap D_{I,\infty}$ can be compactified to a
map $u:S^2 \to PD_I$ which is $J$-holomorphic near $0$. In particular, we can
form the following moduli space:
\begin{defn} \label{defn: interpolateme}
    We let $\mathcal{M}_{0,2}(\v_0)$ denote the moduli space of maps $u:S^2 \to PD_I$ such that 
    \begin{align}
        \mathring{u}: \mathbb{C}P^1 \setminus \lbrace 0 \rbrace \to PD_I
    \end{align} solves
  \begin{equation} \label{eq:PSSeqblah}
    (d\mathring{u} - X_{H_{loc}^{\ell}} \otimes \beta)^{0,1} = 0
\end{equation}
with $u(z) \notin \D$ for $z \neq 0,\infty$ and with the following incidence conditions:
\begin{align}
\label{eq:starti}    &u(0)\textrm{ intersects $D_{i,\infty}$ with multiplicity $(\v_0)_i$ for $i \in I$};\\
    &u(0)\textrm{ does not intersect $D_{i,0}$ or $D_{j,I}$ for any $i \in I$ or $j\notin I$}.
\end{align}
\begin{align}
    &u(\infty)\textrm{ intersects $D_{i,0}$ with multiplicity $(\v_0)_i$ for $i \in I$};\\
  \label{eq:finali}  &u(\infty)\textrm{ does not intersect $D_{i,\infty}$ or $D_{j,I}$ for any $i \in I$ or $j\notin I$}.
\end{align} 
\end{defn}

The techniques used to regularize $\PSS_{log}$ moduli spaces again allow us to show that for generic choices of $J_S$ this is a manifold of dimension $2n$. The next lemma checks that  $\operatorname{SSP}_{loc,\v_0}$ is a one-sided inverse to $\operatorname{PSS}_{loc,\v_0}:$
 
\begin{lem} \label{lem: degenerations} For any class $\alpha t^{\v_0} \in H^*_{log}(M,D)$, we have that 
\begin{align}\label{eq: keyequality} 
    \operatorname{SSP}_{loc,\v_0} \circ \operatorname{PSS}_{loc,\v_0}(\alpha t^{\v_0})=\alpha t^{\v_0} 
\end{align}
\end{lem} 
\begin{proof} 
We will consider the moduli spaces 
\begin{align}
    \mathcal{M}_{0,2}(\v_0, c , b):= W^s(f_I, c) \times_{Ev_{z_{0}}} \mathcal{M}_{0,2}(\v_0) \times_{\widehat{\operatorname{Ev}}_{z_{1}}} W^u(\hat{f}^{pert}_I,b) 
\end{align} 
For generic choices of $J_S^{loc}$ and perturbed Morse function $\hat{f}_I^{pert}$,
this is a manifold of the correct dimension. We will be interested in cases
when $\mathcal{M}_{0,2}(\v_0, c , b)$ has dimension 1. For $J_S^{loc}$ taken as above (and as usual all perturbing parameters sufficiently small), the Gromov
compactification of  $\mathcal{M}_{0,2}(\v_0, c , b)$ has two strata
$\partial_B \overline{\mathcal{M}}_{0,2}$ and $\partial_S
\overline{\mathcal{M}}_{0,2}$ (as well as other intermediate strata where Morse
trajectories of $\hat{f}^{pert}_I$ or $f_I$ break off, giving chain homotopy
terms whose descriptions are standard, hence omitted) 
corresponding  respectively to breaking along
Hamiltonian orbits and sphere bubbling with respect to the split, surface-independent
complex structure for which $J_S=J_0$ near $z=\infty$. We have that 
$$ \partial_B \overline{\mathcal{M}}_{0,2}:= \bigsqcup_{x_0,\operatorname{vdim}(\mathcal{M}_{PSS}(\v_0,c,x_0)=0)} \mathcal{M}^{loc}_{PSS}(\v_0,c,x_0) \times \mathcal{M}^{loc}_{SSP}(\v_0,x_0,b).  $$ 
To analyze the other boundary $\partial_S\overline{\mathcal{M}}_{0,2}$, first
note that all $J_0$-holomorphic sphere bubbles must lie in the fiber of the projection $\pi_I: PD_I \to D_I$ by the argument of Lemma \ref{lem:spherebubblefiber}. As such a sphere bubble does not lie in $D_{i,0}$, it follows that it must have intersection number with $D_{i,0}$ exactly equal to $(\v_0)_i$ at the point $z_0$ (and cannot intersect $D_{i,0}$ at any other points otherwise the energy would exceed the energy of our initial PSS solution). The sphere bubble therefore has the same energy as an element of $\mathcal{M}_{0,2}(\v_0)$, and thus the PSS component of an element of $\partial_S \overline{\mathcal{M}}_{0,2}$ must be constant and lie entirely in $D_{I,\infty}$. 

Now let $z_1$ denote the marked point on the sphere bubble that attaches to the PSS solution. By $C_{loc}^{\infty}$ convergence, for smooth PSS solutions sufficiently close to this bubble configuration an arbitrarily large part of the domain $S$ (and in particular all of the region where the 1-form $\beta$ is supported) can be made to map into a neighborhood about $D_{I_{\infty}}$ where $H^{loc}_\ell=0$ and $J_S^{loc}=J_0.$  This implies that such nearby PSS solutions are in fact $J_0$-holomorphic curves which converge modulo $\mathbb{R}$-rescaling to the limiting bubble. It follows that the limiting bubble configurations equipped with the two marked points $(z_0, z_1)$ satisfy the incidence conditions \eqref{eq:starti}-\eqref{eq:finali} and moreover that the limit of the enhanced evaluations of an element of $\mathcal{M}_{0,2}(\v_0)$ at $(z_0, z_1)$ coincides with the enhanced evaluation of this sphere bubble at $(z_0, z_1)$.
We can therefore identify the moduli space $\partial_S
\overline{\mathcal{M}}_{0,2}$ with the elements of the moduli space of $J_0$-holomorphic
spheres $u$ whose domain has two marked points $(z_0, z_1)$ satisfying
\eqref{eq:starti}-\eqref{eq:finali} modulo $\mathbb{R}$-translation in the
domain which further satisfy $Ev_{z_{0}}(u) \in W^s(f_I,c)$ and
$\widehat{\operatorname{Ev}}_{z_{1}}(u) \in W^u(\hat{f}_I^{pert},b)$. 

The operation associated to counting rigid elements of $\partial_B
\overline{\mathcal{M}}_{0,2}$ gives rise to (on cohomology) the left-hand side
of \eqref{eq: keyequality}. We also claim that the operation associated to counting rigid elements of $\partial_S \overline{\mathcal{M}}_{0,2}$ gives rise on cohomology to the right hand side of \eqref{eq: keyequality}, i.e., the identity map. To see this let  $\mathcal{M}_{0,2}^{f}$ denote the moduli
space of (smooth) $J_0$-holomorphic spheres with two marked points satisfying
\eqref{eq:starti}-\eqref{eq:finali} modulo $\mathbb{R}$-translation which lie
in the fiber of the projection. Regularity for such spheres follows from Proposition 6.3.B of of \cite{MR1844078} (see also the final paragraph of the proof of Lemma 5.23 of \cite{Ionel:2011fk}). As before, for any map $u \in
\mathcal{M}_{0,2}^{f}$ let $Ev_{z_{0}}^{L}(u)$ and $Ev_{z_{1}}^{L}(u)$ denote
the lifted higher evaluation to $\mathring{ND}_I/\mathbb{R}^{+}$ and $\mathring{ND}_{I,\infty}/\mathbb{R}^{+}$
respectively, given by taking the tuple of all leading order jets
modulo (only) the diagonal $\mathbb{R}^{+}$ action. The
higher evaluation maps $Ev_{z_{i}}^{L}$ define diffeomorphisms 
\begin{align}
    Ev_{z_{0}}^{L}: \mathcal{M}_{0,2}^{f} &\stackrel{\cong}{\to} \mathring{ND}_I/\mathbb{R}^+\\
    Ev_{z_{1}}^{L}: \mathcal{M}_{0,2}^{f} &\stackrel{\cong}{\to} \mathring{ND}_{I,\infty}/\mathbb{R}^+.
\end{align}
This is because we can assume that in any fiber, up to reparameterization, our
maps take the form $z \to (a_i z^{(\v_{0})_{i}})$  where $a_i$ ranges over
$(\mathbb{C}^*)^{I}/\mathbb{R}^+$. 
The composition $Ev_{z_{1}}^{L}\circ (Ev_{z_{0}}^{L})^{-1}$ is the map $\tau$
above in \eqref{taumap}.   

Next, recall from of \cite[Appendix A.2]{AbSch2}, that the signed count of
points in (compare with the moduli space in \eqref{eq: IKMorse}) 
\begin{align}
    \label{eq: AScontin} W^s(\hat{f}_I,(id,c)) \cap W^u(\hat{f}_I^{pert},b) 
\end{align} 
(where $(id,c)$ denotes the critical point of $\hat{f}_I$ induced by the critical point $id \in G$ of $f_G$ and  $c \in \mathring{S}_I$ of $f_I$, thought of as a point of $\mathring{ND}_I/\R^+$ via \eqref{NDIquotientsplit})
induces a map 
$$CM^*(\mathring{ND_I}/\mathbb{R}^{+}, \hat{f}_I) \to CM^*(\mathring{ND_I}/\mathbb{R}^{+},\hat{f}_I^{pert}).$$
By Remark A.2. of \emph{loc. cit.}, this map induces the identity on cohomology once both sides are identified with singular cohomology. 
Next note that having $Ev_{z_{0}}(u) \in W^s(f_I,c)$, is equivalent to requiring that $Ev_{z_{0}}^{L}(u) \in W^s(\hat{f}_I,(id,c))$, as we have made choices so that $W^s(\hat{f}_I, (id,c)) \cong G \times W^s(f_I, c)$.
This gives rise to an identification between points in Equation \eqref{eq: AScontin} and $\partial_S \overline{\mathcal{M}}_{0,2}$. It is straightforward to check that this correspondence preserves orientations and the proof of the Lemma is concluded.   

See Figure \ref{ssppssschematic} for an image of the moduli space $\mathcal{M}_{0,2}(\v_0, c , b)$ and its two degenerations (along with a pictorial description of the analysis above).
\begin{figure}[h] 
    \caption{The domain of the moduli space $\mathcal{M}_{0,2}(\v_0, c , b)$ along with the incidence, matching, asymptotics, and PDE satisfied at various places, along with a description of this moduli space's two codimension 1 boundary components (excluding Morse breaking which serves to produce chain homotopy terms). The first codimension 1 boundary component evidently induces the operation $\operatorname{SSP}_{loc,\v_0} \circ \operatorname{PSS}_{loc,\v_0}(\alpha t^{\v_0})$ and the second is equated (in the Figure) with the identity map, using the existence of a unique fiber spheres at any point in $PD_I$ with any given enhanced evaluation. \label{ssppssschematic}}
    \centering
    \includegraphics[scale=1.0]{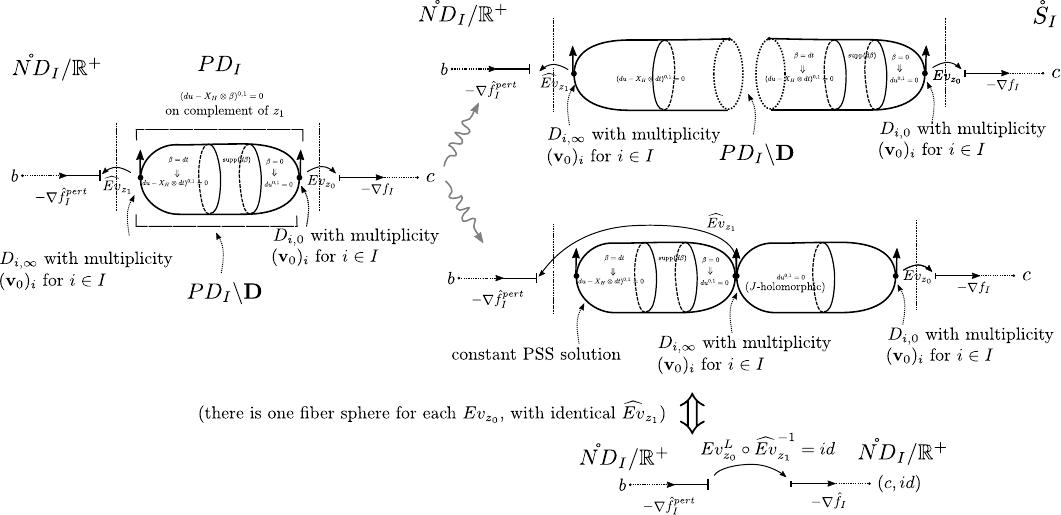}
\end{figure}

\end{proof}

\begin{cor} \label{cor:reallocPSSiso}  For any $\v$, 
    the map \eqref{eq:realPSSloc} induces an isomorphism: $$ \PSSlog^{\v}: H^*(\SIo)t^{\v} \cong  HF^*(X \subset M ; H_p^{\ell})_{\v}. $$  \end{cor} 
\begin{proof} From Corollary \ref{lem: Pozniack} and Lemma \ref{lem:
    degenerations} we deduce that the map $\operatorname{SSP}_{loc,\v}$ is a
    surjective map between isomorphic finitely-generated $\mathbf{k}$-modules 
    It follows from \cite{Vasconcelos} that it is an isomorphism. Applying
    Lemma \ref{lem: degenerations} again shows that
    $\operatorname{PSS}_{loc,\v}$ is an isomorphism and the result follows
    from Corollary \ref{cor:twoPSS}.  
    (Strictly speaking our description of the one-sided inverse
    $\operatorname{SSP}_{loc,\v}$ used $\v \neq \mathbf{0}$ in order to work in
    a projective bundle; however if $\v = \mathbf{0}$, $\PSSlog^{\mathbf{0}}$
    is the usual PSS map $H^*(X) \to HF^*(X \subset M;
    H_p^{\ell})_{\mathbf{0}}$, which is well known to also be an isomorphism of
    $\K$-modules by e.g. \cite[\S 15.2]{Ritter}).  
\end{proof} 

We now complete the proof of the main result of this section:

\begin{thm} \label{thm:locPSSisom} 
    The map \eqref{eq:locPSS} is an isomorphism: 
    $$\PSSlog^{w}: \QH(M,\D)_w \cong  HF^*(X \subset M ; H^{\ell})_{w}.$$  
\end{thm} 
\begin{proof} 
    Using Corollary \ref{cor:reallocPSSiso} it follows from Lemma
    \ref{lem:localizeisomorphism} that \eqref{eq:locPSS} induces an
    isomorphism.
\end{proof}

\subsection{Proof of Theorem \ref{thm:main}}\label{sec:mainproof}
Collecting all of the results proven so far, we prove our first main theorem:

\begin{thm}[Theorem \ref{thm:main}] \label{thm:spectral} There is a multiplicative spectral sequence converging to the symplectic cohomology ring
\begin{align} 
    \label{eq:sstext} \lbrace E_r^{p,q},d_r \rbrace \Rightarrow SH^*(X). 
\end{align}
whose first page is isomorphic as rings to the logarithmic cohomology of $(M,\D)$:
\begin{equation}\label{eq:sspage1text}
 H^*_{log}(M,\D)  \stackrel{\cong}{\rightarrow} \bigoplus_{p,q} E_1^{p,q}.
\end{equation} \end{thm} 
\begin{proof} 
The spectral sequence was constructed in \eqref{eq:Spec2} of \S \ref{sect:
actionspec}. In Equation \eqref{eq:Speciso2}, we constructed a map
    $$\PSSlog^{low} : \QH^*(M,\D)  \to \bigoplus_{p,q} E^{p,q}_{1}. $$  
 The fact that this map respects ring structures was proven in Theorem
 \ref{lem:spectralrings}. It therefore remains to show that these maps
 \eqref{eq:Speciso2} are isomorphisms. By taking limits, showing that this is
 an isomorphism reduces to showing that \eqref{eq:Speciso1} is an isomorphism,
 which in turn immediately reduces to proving that \eqref{eq:locPSS} is an
 isomorphism. As this is the main result of Theorem \ref{thm:locPSSisom}, the
 proof of Theorem \ref{thm:main} is complete. 
\end{proof} 
\vskip 5 pt

\textbf{Comparing $\mathbb{Z}$-gradings}: Under the isomorphism \eqref{eq:sspage1text}, a class $\alpha t^{\v}$ lies in bidegree 
\begin{align}
    \label{eq:SSbigradings} (p,q)=(-w(\v),\deg(\alpha t^{\v})+w(\v)) 
\end{align} \vskip 3 pt

\textbf{Comparing $H_1(X)$-gradings}: It is useful to observe that
symplectic cohomology (and Hamiltonian Floer cohomology) admits an optional
second grading by $H_1(X)$,  which assigns to any orbit $x$ (rather its
orientation line) the associated homology class $[x] \in H_1(X)$. The
differential preserves this grading and the
multiplication is additive; in particular, symplectic cohomology additively
splits as a direct sum over $H_1(X)$ classes. An analogous $H_1$ grading can be
associated to $H_{log}(M,\D)$. On a generator of the form $\alpha t^{\v}$, it is described as follows: 
Let $[y_1]$, ... $[y_k]$ be the $H_1(X)$ classes
of small loops around each $D_i$ (for instance, $y_i$ = the boundary of a disc
fiber in $U_i$). Then to a class of multiplicity $\v$, associate the $H_1(X)$ class
$\sum \v_i [y_i]$. (this is also the $H_1(X)$ class of a small loop that winds
$\v_i$ times around each $D_i$). As this $H_1(X)$ grading depends
(additively) only on the vector $\v$, we immediately see that differential on
log cochains preserves this grading and the product of two elements of grading
$[x_1]$ and $[x_2]$ has grading $[x_1]+[x_2]$. Finally, it is straightforward
to see that the low
energy log PSS map is compatible with the two $H_1(X)$ gradings, and in
particular the spectral sequence and identification of its first page from
Theorem 1.1 split as a direct sum over $H_1(X)$ classes (in a manner
multiplicatively compatible with adding $H_1(X)$ classes).
\vskip 3 pt

\begin{rem}[The $\Z/2\Z$-graded case]\label{rem: Z2Zgr}
To simplify the exposition of the numerous moduli spaces and operations
appearing here, we have described dimensions of moduli spaces, their associated
operations, and the relevant cohomology groups in the $\Z$-graded setting (which
applies when $c_1(X) = 0$ using a choice $\Omega_{M,\D}$ of
a holomorphic volume form on $M$ which is non-vanishing on $X$). 

However, our proof remains valid in $\Z/2\Z$-graded\footnote{or $\Z/2k\Z$-graded, or
fractionally-graded when $c_1(X)$ is torsion, etc. though we leave these
details to the reader.} settings as well, with the following adaptations
to the definitions of complexes, moduli spaces, and operations. First, one defines the Floer
cohomology for any Hamiltonian as in \S \ref{subsection:Floercohom} by
associating to an orbit $x$ the determinant line associated to {\em some}
trivialization $\gamma$ of $x^*TM$. There is an ambiguity in such a choice, but
any two choices of $\gamma$ induce canonically isomorphic vector spaces
$\mathfrak{o}_x$ (compare \cite[Prop. 1.4.10]{Abouzaid:2015ad}). The degrees
associated to different choices of $\gamma$ only coincide mod 2, so the complex
inherits a well-defined $\Z/2\Z$ grading. The space of Floer trajectories
between $x$ and $x'$ now contains components of varying dimension (coinciding
mod 2 with the difference of gradings) depending on the underlying homotopy
class of cylinders from $x$ to $x'$, and we only count the 0-dimensional
components when defining operations. Finally, note that any homotopy class of
cylinder induces, for any choice of trivialization $\gamma$ for $x^*TM$, a
canonical induced trivialization $\gamma'$ for $(x')^* TM$; the pair $(\gamma,
\gamma')$ can be used to comute the dimension of this component of the moduli
space, and gluing theory associates, for rigid Floer trajectories in this
homotopy class, isomorphisms between the determinant lines of $\gamma$ and
$\gamma'$, which is the necessary input to getting signed counts.

Next we $\Z/2\Z$ grade log cohomology by defining $\deg(\alpha t^\v)= \deg(\alpha)$ mod 2 
(compare this to the mod 2 reduction of \eqref{eq:loggrading}). Again, the 
low energy log PSS moduli spaces $\mc{M}(\v, x_0)$ contain components of varying dimension
depending on the  underlying homology class $[u]$ in $H_2(M, M\setminus \D)$ of the map;
these are dealt with as in the case of Floer trajectories by noting that any
such homotopy class defines a trivialization $\gamma$ of $x_0^* TM$ and thus a
choice of orientation line to map to (in the rigid case). The remainder of the
moduli spaces and operations (local log PSS, local Floer homology, and local log SSP)
work in a similar fashion. 

In particular, the proof of Theorem \ref{thm:main} continues to hold, seeing as
nothing about the proof of isomorphism used gradings: once the maps are defined
as above, the same appeal to energy and action considerations to confine curves
and/or argue e.g., that PSS is a ring map or SSP and PSS compose in the desired
fashion go through (with suitable definitional changes as above to all of the
intermediate chain homotopies).  
\end{rem}

\section{Computations of symplectic cohomology}\label{sec:computations}
\subsection{Topological and multiplicatively topological pairs}\label{subsec:topologicalpairs}
We recall definitions and give examples of topological pairs (a slight
generalization of the notion used in \cite[\S 3]{GP1}) and multiplicatively
topological pairs, introduced in the introduction.  In the terminology of \S
\ref{sec:intropointedspheres},  a pair $(M,\D)$ is topological (respectively
multiplicatively topological) if there is some $J_0 \in \mathcal{J}(M,\D)$ (recall Definition \ref{defn:complexint}) such
that $(M,\D)$ has no 0 or 1-pointed (respectively no 0, 1, or 2-pointed)
relative $J_0$-holomorphic spheres. To spell this out:

\begin{defn} \label{def:topologicalpair}
    We say that a pair $(M,\D = D_1 \cup \cdots \cup D_k)$ is {\em topological} if there exists a $J_0 \in \mathcal{J}(M,\D)$ such that
    for any subset $I \subset \{1, \ldots, k\}$, there are no non-constant
    $J_0$-holomorphic curves $u: \mathbb{C}P^1 \to D_I$ which intersect
    $\cup_{j \notin I} (D_j \cap D_I)$ in 1 or fewer distinct points.
\end{defn}
Note that as per our convention, we include the case $I = \emptyset$ above,
with $D_{\emptyset} = M$.

\begin{example} 
    \label{example:toppairs} To illustrate that this is a reasonably broad class of pairs, we list some examples: 
\begin{enumerate}
        \item If $\pi_2(M)=0$ or more generally $\omega(\pi_2(M)) = 0$, any
            pair $(M,\mathbf{D})$ will be topological, as there are no
            $J$-holomorphic spheres in $M$ at all for any $J$.

        \item \label{ex:manyampledivisors}
            If each smooth component $D_i$ of $\mathbf{D}$ corresponds to some
            power of the same line bundle and the number of components $k$ of
            $\mathbf{D}$ satisfies $k \geq \dim_{\C} M+1$, then $(M,
            \mathbf{D})$ is topological. To see this, note that if $[u] \in
            H_2(M)$ is the class of a $J$-holomorphic sphere in $M =
            D_{\emptyset}$ which isn't contained in any $D_j$, then since
            $\omega([u]) > 0$, $[u] \cdot D_j > 0$ for every $j$. Thus, $u$
            must intersect $\D$, and moreoever cannot intersect $\D$ at only
            one point, because then that point would be contained in
            $\cap_{i=1}^{n+1} D_i = \emptyset$. A similar argument applies to
            any $[u] \in H_2(D_I)$ the class of a $J$-holomorphic sphere which
            is not contained in any $D_j \cap D_I$.\\

            \noindent As a specific case, note that $(\P^n, \mathbf{D} = \{\geq
            n+1\textrm{ generic planes}\})$ is a topological pair.
 \item In the $\Z/2\Z$ graded setting (as in Remark \ref{rem: Z2Zgr}),
             another general class of
            topological pairs $(M, \mathbf{D})$ with $\mathbf{D} = D$ a smooth
            divisor can be constructed as follows: Let $M$ be a hypersurface of
            degree at least $2n+1$  in $\mathbb{C}P^n$, $n \geq 2$ and let $D$
            be any smooth hyperplane section. The topological pair condition
            holds since for any $A \in H_2(M,\mathbb{Z})$ with $\omega(A)>0$, the virtual
            dimension of the moduli space of $J$-holomorphic spheres in
            homology class $A$ is negative (and the same applies to curves that
            lie entirely in $D$). As any nonconstant curve factors through a somewhere
            injective curve, the moduli spaces are generically empty. 
        \end{enumerate} 
        As (\ref{ex:manyampledivisors}) shows, $M$ (and each stratum
        $D_I$) could contain many $J$-holomorphic spheres even if $(M,\D)$ is a
        topological pair.  
\end{example}

To give a slightly more elaborate example, let $M_0=\mathbb{P}^2$ and let $D_0$ be a (possibly non-generic) hyperplane arrangement such that for every component $D_{0,i}$ of $D_0$, there are at least two distinct points $D_{0,i} \cap D_{0,j}$ and $D_{0,i} \cap D_{0,j'}$ for $j,j' \neq i$. 

\begin{lem} \label{lem:hyperplanes} Let $M \to M_0$ denote the blowup of $M_0$ at each of the points where $\geq 3$ components of $D_0$ meet. Let $\D$ denote the union of the proper transform of the divisors of $D_0$ and the exceptional divisors. Then the pair $(M,\D)$ just constructed is a topological pair. \end{lem} 
\begin{proof} 
    Every sphere which lies in $\D$ must intersect at least two of the other
    components in distinct points, so consider spheres $u$ which meet $\D$
    transversely. Either it meets more than one exceptional divisor, in which
    case we are done, or it meets at most one exceptional divisor. If it meets
    no exceptional divisors, then it must intersect all components of $\D$ with
    the same intersection multiplicity. If it meets one exceptional divisor,
    there is at least one other component of $\D$ disjoint from that
    exceptional divisor and it must meet this divisor too.  
\end{proof}

\begin{rem} 
    Similar examples can likely be constructed for higher dimensional
    hyperplane arrangements, but this requires tracing through the resolutions
    of such arrangements, which are necessarily more elaborate than in the
    2-dimensional case.  
\end{rem} 

There is a strengthening of the topological condition which is useful for comparing product structures. 

\begin{defn} \label{def:multiplicativelytopological}
    We say that a pair $(M,\D)$ is {\em multiplicatively topological} if there exists a $J_0 \in \mathcal{J}(M,\D)$ such that for any $I \subset\{1, \ldots, k\}$, there are no non-constant holomorphic curves 
    $u: \mathbb{C}P^1 \to D_I$ which intersect $\cup_{j \notin I} (D_j \cap
    D_I)$ in 2 or fewer distinct points. \end{defn}

\begin{example} \label{example:multoppairs} Some examples of multiplicatively topological pairs include: 
\begin{enumerate}
        \item If $\pi_2(M)=0$ or $\omega(\pi_2(M)) = 0$, any pair
            $(M,\mathbf{D})$ will be multiplicatively topological because there
            are no spheres in $M$ at all.

        \item 
            Whenever each smooth component $D_i$ of $\mathbf{D}$ corresponds to
            powers of the same line bundle and the number of components of
            $\mathbf{D}$, $k$, satisfies $k \geq 2 \dim_{\C} M+1$, then $(M,
            \mathbf{D})$ is multiplicatively topological, by the same analysis
            as (\ref{ex:manyampledivisors}) of Example \ref{example:toppairs}.
            For example $(\P^n, \mathbf{D} = \{\geq 2n+1\textrm{ generic
            planes}\})$ is multiplicatively topological. 

           \item 
            Let $M_0=\mathbb{P}^2$ and let $D_0$ be a hyperplane arrangement
            such that every component $D_{0,i}$ meets other components of $D_0$
            at at least three distinct points. Let $(M,\D)$ be the resolution
            of this hyperplane arrangement constructed in Lemma
            \ref{lem:hyperplanes}. Then $(M,\D)$ is multiplicatively
            topological, by an analogous argument to Lemma
            \ref{lem:hyperplanes}.

        \item In the $\Z/2\Z$-graded setting, 
            we can let $M$ be a hypersurface of
            degree at least $2n+1$  in $\mathbb{C}P^n$, $n \geq 2$ and let $D$
            be any smooth hyperplane section as above. This is once more
            multiplicatively topological because (as seen in Example
            \ref{example:toppairs}) there are no spheres in $M$ at all, for
            generic $J$.  \end{enumerate}
\end{example}

We now turn to the proofs of  Theorem \ref{thm: toppair} (Theorem \ref{thm: toppair2}) and Theorem \ref{thm: toppairring} (Theorem \ref{thm: ringstructures}), which involves introducing some basic definitions arising in log Gromov-Witten theory. To fix notation, for $T$ a tree with  $|E_{\operatorname{ext}}(T)|=2$, we let
 $\mathcal{M}_{0,2}(T,M)$ to be the moduli space of $J_0$-holomorphic sphere
 bubbles modelled on $T$. Denote the two external edges by $\vec{e}$ and
 $\cev{e}$.  
 \begin{defn} 
     We say that a $J_0$-holomorphic sphere $u: S^2 \to D_{I}$ has {\em depth $I$} if 
     $u(S^2) \subset D_I$ and 
     $u(S^2) \not \subset D_j$ for $j \notin I$. 
 \end{defn}

 For any $u \in \mathcal{M}_{0,2}(T,M)$, and any vertex $\nu \in V(T)$, let
 $I_{\nu} \subset \{1, \ldots, k\}$ denote the depth of $u_\nu$; 
 over all $\nu \in V(T)$, one has a corresponding depth function associated to
 $u$ 
 \[
     I_{(-)}: V(T) \to \mathcal{P}(\{1, \ldots, k\}),
 \]
(as usual $\mathcal{P}(-)$ denotes powerset).

Assume that for each $\nu \in V(T)$, $u_\nu$ is enhanced with the additional
data of meromorphic sections $\psi_i \in \Gamma_m(S^2, u_{\nu}^*(ND_i))/
\mathbb{C}^*$ for every $i \in {I_{\nu}}$ 
For any $z \in S^2$ and $i \in I_{\nu}$ we set $\operatorname{ord}_{\nu,i}(z)$
to be the order of any zero or pole of $\psi_i$ at $z$. This function is
non-vanishing at at most finitely many points $z \in S^2$.  We may extend the
definition of $\operatorname{ord}_{\nu,j}(z)$ to any
 $j \notin I_{\nu}$ by  recording the intersection
multiplicity $m_i(z)$ of $u_\nu$ with $D_i$ at $z$ (which is again
non-vanishing at at most finitely many points); note that positivity of
intersection implies that this number is strictly positive if $u_\nu$ intersects
$D_i$. Putting these constructions together gives rise to a function 
\begin{align*} 
    \operatorname{ord}_\nu: S^2 \to \mathbb{Z}^k \\ z \to \lbrace ord_{\nu,i}(z)\rbrace
\end{align*}
(which implicitly depends on the choice of extra data $\{\psi_i \in \Gamma_m(S^2, u_{\nu}^*(ND_i))/
\mathbb{C}^*\}_{i \in I_{\nu}}$).
\begin{defn} 
    A {\em pre-logarithmic enhancement} of a stable curve $u \in \mathcal{M}_{0,2}(T,M)$
    consists of, for each $\nu \in V(T)$, a collection of meromorphic sections
    $\{\psi_i \in \Gamma_m(S^2, u_{\nu}^*(ND_i))/ \mathbb{C}^*\}_{i \in
    I_{\nu}}$ such that the associated functions
    $\{\operatorname{ord}_{\nu}\}_{\nu \in V(T)}$ satisfy:
    \begin{itemize} 
        \item[(i)] $\operatorname{ord}_\nu$ is non-vanishing only at the marked points corresponding to edges $e \in E(T).$ 
        \item[(ii)] For any internal edge $e \in E(T)$ bounding two vertices $\nu^{+}$ and $\nu^{-}$, we have 
        \begin{align} \label{eq:ordersmatch}
            \operatorname{ord}_{\nu^{+}}(z_e^{+})=-\operatorname{ord}_{\nu^{-}}(z_e^{-}),
        \end{align} 
        where $z_e^{+}$ and $z_{e}^{-}$ are the marked points corresponding to
        $e$ on $u_{\nu^{+}}$ and $u_{\nu^{-}}.$ 
    \end{itemize} 
\end{defn} 
\begin{rem}
    We note that unlike the more sophisticated notion of log map from
    Definition 3.8 of \cite{Tehrani}, the data and conditions constituting a
    pre-logarithmic enhancement are simply a fiber product of the data and conditions
    defined individually for (a pre-logarithmic enhancement relative) each
    smooth component of $\D$.
\end{rem}
If $z_{\infty}$ is the marked point on $S^2$ corresponding to the edge
$\cev{e}$, we have an evaluation $$ ev_{z_{\infty}}: \mathcal{M}_{0,2}(T,M) \to
M $$ Let  $\mathcal{M}(T, x_0)$ be the moduli space of  
$$ \mathcal{M}(T, x_0):=\mathcal{M}_{0,2}(T,M) \times_{ev_{z_{\infty}}} \mc{M}(x_0) $$ 
(recall the definition of $\mc{M}(x_0)$ in Definition \ref{def:classicalPSSmoduli}).

The next Lemma shows that for topological pairs, the log PSS moduli spaces from \eqref{modulispace1} (not
just in low energy) have a suitable compactification provided that $\ell$ and
hence $\lambda_{\ell}$ is taken sufficiently large and provided we choose
generic $J_S \in \mathcal{J}_S(V)$ such that the complex structure $J_0$ at $z_0$ 
(recall Definition \ref{defn: sdcs}) agrees with the one from Definition
\ref{def:topologicalpair}. Before stating the Lemma, we remind the reader that $w(\v)$ is the weighted winding number from 
\eqref{eq:windingvec} and $\lambda_\ell$ is the slope of the Hamiltonian $H^{\ell}.$

\begin{lem} 
\label{lem:compactness2} Let $(M,\D)$ be a topological pair, $\v$ a multiplicity vector such that $$w(\v) < \lambda_{\ell}, $$ 
and let $c$ be a critical point of the Morse function $f_I: \SIo \to \mathbb{R}$ where $I$ is the support of $\v$. Assume that $||H^{\ell}-h^{\ell}||$ is chosen sufficiently $C^2$ small and let $x_0$ be a Hamiltonian orbit in $\mathcal{X}(X; H^{\ell})$.
\begin{itemize} 
    \item If $\deg(x_0)-\deg(|\mathfrak{o}_c| t^{\v})=0$, then for generic $J_S \in
    \mathcal{J}_{S,\ell}(V)$ with $J_{z_{0}}=J_0$, the moduli space ${\mc{M}}(\v, c, x_0)$ is compact. \vskip 5 pt
    \item  If $\deg(x_0)-\deg(|\mathfrak{o}_c| t^{\v})=1$, then for generic $J_S \in
    \mathcal{J}_{S,\ell}(V)$ with $J_{z_{0}}=J_0$, ${\mc{M}}(\v, c, x_0)$ admits a compactification (in the sense of Gromov-Floer convergence)
    $\overline{\mc{M}}(\v,c, x_0)$, such that $\partial\overline{\mc{M}}(\v, c, x_0)= \partial_M \bigsqcup \partial_F$ where 
\begin{align} 
    \label{pf}\partial_F:= \bigsqcup_{x',\deg(x_0)-\deg(x')=1}  \mc{M}(x_0,x') \times \mc{M}(\v,c, x') \\
    \label{pm}\partial_M:= \bigsqcup_{c',\deg(c')-\deg(c)=1} \mc{M}(\v,c', x_0) \times \mc{M}(c',c)
 \end{align}
\end{itemize} 
\end{lem}
\begin{proof}   
  We consider the closure $\overline{\mc{M}}(\v, x_0) \subset
  \overline{\mc{M}}(x_0)$. The argument of \cite[Lemma 4.13]{GP1} rules
  out cylinders breaking along orbits in $\D$, so we are only concerned with preventing sphere
  bubbling (after which, positivity of intersection implies as usual that
  broken cylinders stay away from $\D$). For simplicity, we first discuss
  sphere bubbling at the distinguished point $z_0 \in S$, temporarily
  ignoring the possibility of sphere bubbles forming at other marked points
  along $S$ or along Floer cylinders. Consider a subsequence $u_n$
  converging to some limit 
  \[ u_{\infty} \in \prod_{x_1,\cdots,x_r} \mc{M}(x_0,x_r) \times \cdots \times \mathcal{M}(x_2,x_{1})  \times \mathcal{M}(T,x_1).\] 
  As our PSS solution is $J_0$-holomorphic near $z_0$ and in view of the
  condition on the Nijenhuis tensor for complex structures in
  $\mathcal{J}(M,\D)$, we may apply the rescaling analysis of \cite[Lemma
  4.9]{Tehrani}  to conclude that the corresponding $u_{\infty,T} \in
  \mathcal{M}_{0,2}(T,M)$ admits a pre-logarithmic enhancement such that if
  $\nu_f$ is the vertex bounding $\cev{e}$, 
  \begin{align}
     \operatorname{ord}_{\nu_{f}}(z_\infty)=-m(z_0) 
 \end{align} 
 where $-m(z_0)$ is equal to the intersection multiplicity of the PSS solution
 at $z_0$.\footnote{The argument given in \cite[Lemma 4.9]{Tehrani} begins
  by noting that it suffices to verify \eqref{eq:ordersmatch} for the
  orders relative each smooth component $D_i$ of $\D$ individually, immediately
  reducing to the case of a single smooth divisor. Once in the smooth case, the
  result appears in various places; see e.g., \cite[Proposition 7.3]{Ionel:2003kx} for a
  separate treatment.}
  By definition $u_{\nu_f}$ lies in the stratum $D_{I_{\nu_f}}$, and  because
 $\operatorname{ord}_{\nu_{f}}(z_\infty) \in (\mathbb{Z}^{\leq 0})^k$ 
 $z_\infty$ cannot be an intersection point of $u_{\nu_f}$ with $\cup_{j \notin
 I_{\nu_f}} (D_j \cap D_{I_{\nu_f}})$ by positivity of intersection. The
 topological pair condition therefore implies $u_{\nu_f}$ has at least 2
 other marked points where it intersects $\cup_{j \notin
 I_{\nu_f}} (D_j \cap D_{I_{\nu_f}})$, i.e., there are
 at least 3 marked points including $z_\infty$.  The topological pair
 condition similarly implies that any other component $u_\nu$ must have at
 least 2 marked points, but now the stable curve has too many external marked
 points to arise as a bubble tree at $z_0 \in S$.

Now we rule out the possibility of sphere bubbling occuring at some other point
in $S$ or along a Floer cylinder.  We have seen that no sphere bubbling can
occur at at $z_0 \in S$ and thus $u_{\infty}$ intersects $\D$ with multiplicity
$\v$ at this point. If $u_{\nu,\infty}$ denote the collection of non-trivial
sphere bubbles, by positivity of symplectic area, there must exist a divisor
$D_i$ for which $\sum_\nu u_{\nu,\infty} \cdot D_i >0$. By positivity of
intersection with $\D$, any Floer cylinder or PSS solution must intersect $D_i$
with non-negative multiplicity. This however contradicts the fact that the sum
of all intersections with $D_i$ away from $z_0$ must be zero.  

\end{proof}
See Figure \ref{1pointedsphere} for a schematic image of how the presence of a 1-pointed sphere could lead to additional breakings appearing in the codimension 1 boundary of the compactification $\overline{\mc{M}}(\v,c, x_0)$.

\begin{figure}[h] 
    \caption{A 1-pointed sphere appearing (in a degenerate configuration) as a limit of elements of $\mc{M}(\v,c, x_0)$. The existence of such a 1-pointed sphere (for non-topological pairs) a priori obstructs compactness of the union of the 1-dimensional moduli space $\mc{M}(\v,c, x_0)$ with \eqref{pf}-\eqref{pm}, and hence (a priori) obstructs the map \eqref{eq: fullPSSdef} from being a cochain map. \label{1pointedsphere}}
    \centering
    \includegraphics[scale=1.5]{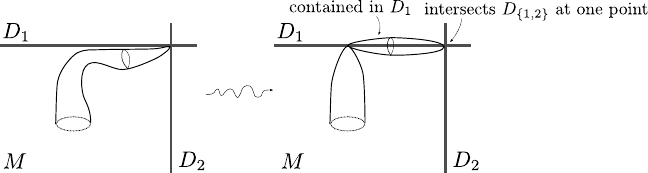}
\end{figure}

 It follows from Lemma \ref{lem:compactness2} that, after as usual choosing generic $J_S$, we can define: \begin{align} \label{eq: fullPSSdef}
 \PSSlog^{\ell,\v} (z t^{\v}) :=  \sum_{\operatorname{vdim}(\mc{M}(\v, c, x_0))=0} \sum_{u \in\mc{M}(\v, c, x_0)} \mu_u (z)
\end{align} for any $z \in |\mathfrak{o}_c|$ (we are using the same notation here as in \eqref{eq:PSSdef}). Extending by $\mathbf{k}$-linearity induces a cochain map 
$$ \PSSlog^{\ell}: F_{w_{\ell}} C_{log}^*(M,\D) \to CF^*(X \subset M;H^{\ell}), $$
where the left-hand side is (the chain-level version of) \eqref{eq:filterlogchainys} and the right-hand side is the Floer cochains from \eqref{eq:Floerrelcoc}. This induces a well-defined cohomological map 
\begin{align} \label{eq: PSSl} 
    \PSSlog^{\ell}: F_{w_{\ell}} H_{log}^*(M,\D) \to HF^*(X \subset M;H^{\ell}).
\end{align}

The argument from
\cite[Lemma 4.13]{GP1} 
can also be similarly adapted to show that for $\ell_2 \geq \ell_1$, the
continuation map commutes with the PSS maps :

\begin{lem}  We have a commutative triangle

 \[
\xymatrix{
   F_{w_{\ell_{1}}} H_{log}^*(M,\D) \ar[d]^{\PSSlog^{\ell_1}} \ar[dr]^{\PSSlog^{\ell_2}} \\
  HF^{*}(X \subset M; H^{\ell_1})  \ar[r]^{\mathfrak{c}_{\ell_1,\ell_2}}  & HF^{*}(X \subset M; H^{\ell_2})
}
\]

\end{lem}

Passing to the limit, we therefore obtain a map 
\begin{align}\label{eq: PSSlim} 
    \PSSlog: H_{log}^*(M,D) \to SH^*(X) 
\end{align}

For topological pairs $(M, \D)$, the spectral sequence \eqref{eq:Spec2}
degenerates. Moreover, the map \eqref{eq: PSSlim} gives a canonical splitting
of this spectral sequence:
\begin{thm}  \label{thm: toppair2} Suppose that $(M, \D)$ is a topological pair. The map \eqref{eq: PSSlim} is an isomorphism. \end{thm} 
\begin{proof} Recall that we have inclusion maps \[
i_{w_{\ell}, w_{\ell+1}}: F_{w_{\ell}} C^*_{log}(M,\D) \to F_{w_{\ell+1}} C^*_{log}(M,\D)
\]  Set \begin{align} \widehat{C}_{log}(M,D)= (\bigoplus_{\ell}F_{w_{\ell}} C^*_{log}(M,\D)[q],\partial) \end{align} to be the cochain level direct limit from Section \ref{sect: actionspec}. One may lift \eqref{eq: PSSlim} to a cochain map \begin{align} \underline{\operatorname{PSS}}_{log}: \widehat{C}_{log}(M,D) \to SC^*(X) \end{align}
using the chain homotopy constructed in the proof of \cite[Lemma 4.18]{GP1}. This cochain level lifting is filtered and chasing through the definitions, it is easy to see that the induced map on spectral sequences is given by  \eqref{eq:Speciso2}. It follows that the map \eqref{eq: PSSlim} is an isomorphism.   \end{proof} 

\begin{cor} For each of the Examples in Example \ref{example:toppairs} and Lemma \ref{lem:hyperplanes}, we have additive isomorphisms $H^*_{log}(M,\D) \cong SH^*(X)$. \end{cor}

We next show how this result may be strengthened for multiplicative pairs to
give a description of the ring structure on $SH^*(X)$:  

\begin{thm} \label{thm: ringstructures} 
    For multiplicatively topological pairs, the map \eqref{eq: PSSlim} is an isomorphism of rings.  
\end{thm}
\begin{proof}  
    We now have a global $\PSSlog$ map and we run the same analysis as in the proof of Theorem
    \ref{lem:spectralrings} but we consider the moduli spaces $\mathcal{M}(\v_1, \v_2 ; x_0)$ for all $x_0$ (not just those where $w(x_0)=w(\v_1)+w(\v_2)$). We assume that our complex structures near the marked points coincide with the $J_0$ which appears in the definition of multiplicatively topological pairs. The multiplicatively topological hypothesis then ensures, by an analysis identical to the proof of Theorem \ref{lem:compactness2},\footnote{The only difference is that, given that the points $z_1$ and $z_2$ collide, any possible stable sphere bubble will have at most 3 (rather than 2) external marked points. The same analysis thus requires the extra ``multiplicatively'' topological hypothesis to rule such bubbles out.}  that even though we are looking at all $x_0$, the Gromov compactification of this moduli space continues to behave as given in Lemma \ref{lem:nospheresinproducthomotopy}; in particular, 
    only constant spheres arise in the Gromov compactification $\overline{\mathcal{M}}(\v_1, \v_2 ; x_0)$. As in the proof of Theorem \ref{lem:spectralrings} it suffices to verify compatibility with ring structures for inputs of the form specified in cases (i) and (ii) of 
   Lemma \ref{lem:generatingrelationsfortheproduct}.
   
   The fact that only constant sphere bubbles arise now suffices to prove as before that we have continuous evaluations to the real blow-ups as in Lemma \ref{lem: evalKN}. From here, we may construct the moduli spaces $\mathcal{M}(\v_1, \v_2 ; c_1, c_2 ; x_0)$ of Definition \ref{defn:auxmod} as well as their Gromov compactifications. 
The remaining arguments in \S \ref{section:rings} needed to show that these compactified moduli spaces define a cobordism (as usual up to intermediate boundary components which define chain homotopies) between $\operatorname{PSS}_{log}(\alpha_1t^{\v_1}) \cdot \operatorname{PSS}_{log}(\alpha_2t^{\v_2})$ and $\operatorname{PSS}_{log}(\alpha_1t^{\v_1} \cdot \alpha_2t^{\v_2})$ carry through without change. 
\end{proof}

See Figure \ref{2pointedsphere} for a schematic image of how the presence of a 2-pointed sphere could lead to additional breakings appearing in the codimension 1 boundary of the compactification $\overline{\mc{M}}(\v_1, \v_2, x_0)$.

\begin{figure}[h] 
    \caption{A 2-pointed sphere appearing (in a degenerate configuration) as a limit of elements of $\mc{M}(\v_1, \v_2, x_0)$. The existence of such a 2-pointed sphere (when $(M, \D)$ is not multiplicatively topological) a priori obstructs the compactness result stated in Lemma \ref{lem:nospheresinproducthomotopy} from holding as stated, and hence (a priori)   
        obstructs the map \eqref{eq: fullPSSdef} from being a ring map. \label{2pointedsphere}}
    \centering
    \includegraphics[scale=1.5]{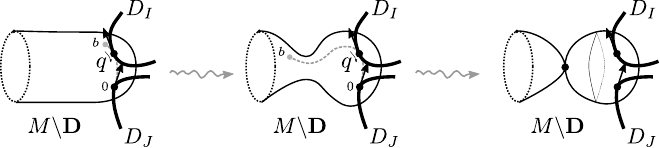}
\end{figure}

\begin{cor} For each of the Examples in Example \ref{example:multoppairs}, we have ring isomorphisms $H^*_{log}(M,\D) \cong SH^*(X)$. \end{cor}

\begin{rem} 
    More generally, one can define,  for each $r \geq 0$  the notion of a pair $(M,\D)$ being ``$r$-topological":
      there should exist a $J_0 \in \mathcal{J}(M,\D)$ such that $(M,\D)$ has no $m$-pointed relative spheres for all $m \leq r$. 
      The topological condition corresponds to $r=1$ and multiplicatively
      topological to $r=2$. It seems reasonable to expect that for such pairs,
      genus zero topological field theoretic operations of ``arity less than or
      equal $r$" on symplectic cohomology can be described in terms of the
      operations on log cohomology (or rather co-chains) of $(M,\D)$.
  \end{rem}

\subsection{Proving deformations are trivial using GW-invariants}\label{subsec:trivializingdeformations}

 For topological pairs, it follows from Lemma \ref{lem:spectralrings}
 and Theorem \ref{thm: toppair2} that the product on $SH^*(X)$ is a deformation
 of the product on $\QH^*(M,\D)$ which respects the filtration. Although this
 deformation is often non-trivial, it can frequently be trivial, even for pairs
 that are not multiplicatively topological.  To this end, we formulate a
 criterion for topological pairs which implies that the $\PSSlog$ map becomes a
 map of rings.  Our criterion will be valid for pairs $(M,\D)$ satsifying the
 following additional condition:
 \vskip 5 pt

\textbf{Condition A:} All divisors $D_i$, $i \in \lbrace 1,\cdots, k \rbrace$ for $k>dim(X)$, are in the same linear system and all strata $D_I$ are connected.  \vskip 5 pt

\begin{defn} 
 For any pair satisfying Condition A, choose an almost complex structure $J_0 \in \mathcal{J}(M,D)$.  For any partition $\lbrace 1,\cdots, k \rbrace= I \cup K$ with $I \cap K=\emptyset$ let
    $\mathcal{M}_{0,3}(M,\mathbf{D},\v_{I},\v_{K})$ denote the
    space of $D_{1}$ and $D_{J}$-regular maps $ u: (S^2,z_0,z_1,z_2) \to (M,\D)$
    such that 
    \begin{align*}
        u^{-1}(D_i)&=  z_0  \textrm{ {for} } i \in I.  \\
        u^{-1}(D_j)&=  z_1 \textrm{ {for} } j \in K.
    \end{align*}
\end{defn}

As the classes represented by curves are primitive, all such curves $u$ are automatically somewhere
injective and thus we may achieve transversality for such maps. Consider $$ev_{z_{2}}^{-1}(X)=\mathcal{M}_{0,3}(M,\mathbf{D},\v_{I},\v_{K})^{o}.$$
Then the evaluation map $$ev_{z_{2}}: \overline{\mathcal{M}}_{0,3}(M,\mathbf{D},\v_{I},\v_{K})^{o} \to X$$
is a proper map. Fix a pair of Liouville domains $\bar{X}_{\vec{\epsilon'}}, \bar{X}_{\vec{\epsilon}}$ such that $\bar{X}_{\vec{\epsilon'}} \subset \bar{X}_{\vec{\epsilon}}.$ After choosing $J_0 \in \mathcal{J}(\bar{X}_{\vec{\epsilon'}},V)$ generically, we may define a relative pseudocycle \begin{equation}
    \underline{GW(\v_I,\v_K)} \in H^*(\bar{X}_{\vec{\epsilon}})
\end{equation}
by intersecting with $\bar{X}_{\vec{\epsilon}}$.

\begin{lem}\label{lem:GWProduct}
    Let $(M,\D)$ be a topological pair which additionally satisfies Condition A. Let $\lbrace 1,\cdots, k \rbrace= I \cup K$ be a partition as above. Then there is a (cohomological) equality:
    \begin{align} 
        \PSSlog(Id t^{\v_I}) \cdot \PSSlog(Id t^{\v_K}) = \PSS(\underline{GW(\v_I,\v_K)}) 
    \end{align}  
\end{lem}
\begin{proof} As this again follows the same pattern as Theorem \ref{thm: ringstructures}, we will only point out the new points that arise here (the reader can also see Lemma 6.10 of \cite{GP1} for an almost identical argument). For simplicity, we assume that $J_z \in \mathcal{J}(\bar{X}_{\vec{\epsilon'}}, V)$ for all $z \in S$ when running this argument. The first point to note is that, as we are not putting any normal bundle constraints on the jets, we need not equip our marked points with projectivized tangent vectors or consider enhanced evaluations into the oriented blowup (hence the argument of Lemma   \ref{lem: evalKN} is not needed). The second is that there is a new stratum in the Gromov compactification compared to Lemma \ref{lem:nospheresinproducthomotopy} which consists of configurations in the fiber product 
\begin{equation}\label{eq:fiberproductGW}
\overline{\mathcal{M}}_{0,3}(M,\mathbf{D},\v_{I},\v_{K})^o \times
_{ev_{\infty}} \mc{M}(\vec{0}, x_0). 
\end{equation}
The operation associated to counting rigid configurations of
\eqref{eq:fiberproductGW} for varying $x_0$ is by definition the composition 
\begin{align}
    \label{eq: fiberprod} \operatorname{PSS}(GW(\v_I, \v_K))
\end{align}  
The proof now follows from the same cobordism arguments as before, where this
extra stratum gives rise to the right-hand term.   
\end{proof} 
We can now state our criterion for triviality of the deformation: 

\begin{thm} \label{eq:topproduct} Let $(M,\D)$ be a topological pair satisfying Condition A and suppose that \begin{itemize}  \item all invariants  $\underline{GW(\v_I,\v_K)}$ vanish. \item the restriction maps $H^*(X) \to H^*(\SIo)$ are surjective for all $I$. \end{itemize}  Then \eqref{eq: PSSlim} is an isomorphism of rings.  \end{thm}
\begin{proof}
    We first observe that in view of the previous Lemma and the first bulleted assumption,  
    $$\PSSlog: \mathcal{SR}^*(M,\D) \to SH^*(X)$$ 
    is a ring map, where $\mathcal{SR}^*(M,\D) \subset \QH^*(M,\D)$ is the subalgebra defined in \eqref{eq:SRdef}. 
    Indeed, consider the product of two elements 
    $$\PSSlog (\operatorname{Id} t^{\v_1}) \cdot \PSSlog(\operatorname{Id} t^{\v_2}).$$ 
    There are two cases: 
    \begin{itemize}
        \item[(i)]
    If the stratum on which $\v_1+\v_2$ is supported is nonempty, then the
    argument of Theorem \ref{thm: ringstructures} (using Condition A to exclude the relevant possible spheres) 
    shows that this product agrees with $\PSSlog (\operatorname{Id} t^{\v_1} \cdot \operatorname{Id}
    t^{\v_2})$.
    In more detail: let $A \subset \{1, \ldots, k\}$ denote the support of $\v_1$ and $B
    \subset \{1, \ldots, k\}$ the support of $\v_2$. Let us suppose that, in the process of the proof of Theorem
    \ref{thm: ringstructures} for the product of two such elements, a limiting stable curve containing
    non-constant $J_0$-holomorphic sphere bubbles arose. Denote by $u$ the
    union of all such bubbles.  Note first that Condition A (specifically that
    all divisors are in the same linear system) implies that $u$ has (positive symplectic area hence) positive
    homological intersection number with $[D_i]$ for each $i \in \{1, \ldots,
    k\}$. On the other hand, (by conservation of homological intersection
    number) the total stable curve, like the original relevant PSS moduli
    space, has positive homological intersection with $D_j$ for $j \in A \cup
    B$ and zero intersection with $D_i$ for $j \notin A \cup B$.  Since $D_A
    \cap D_B \neq \emptyset$, it must be the case that $\#|A \cup B| \leq
    \dim(X)$, so Condition A again (specifically $k > \dim(X)$) implies that
    such a $D_i$ with $i \notin A \cup B$ exists. By additivity of homological
    intersection numbers over the components of the broken curve, this implies
    that the PSS moduli space component of this stable curve has negative
    intersection with $D_i$, a contradiction to positivity of intersection
    (seeing as the PSS moduli spaces are not contained in $D_i$).

\item[(ii)] If the stratum on which $\v_1+\v_2$ is supported is empty, we aim to show that this product is zero. After
    factoring common supports out of the right hand term, we may assume
    that the the supports of $\v_1$ and $\v_2$ give a partition of $\lbrace
    1,\cdots, k \rbrace$. By further factoring, we can assume that $\v_1=\v_I$
    and $\v_2=\v_K$. The claim now follows from the hypothesis that all of the invariants $\underline{GW(\v_I,\v_K)}$ vanish and Lemma \ref{lem:GWProduct}.
\end{itemize}

 In view of the second hypothesis, we may next write any class in $H_{log}^*(M,\D)$ as a product of two elements $\alpha t^{\v_1}=Id t^{\v_1} \cdot \alpha' $ for $\alpha' \in H^*(X)$. Now again by Theorem \ref{thm: ringstructures} we have that $$ \PSSlog(\alpha t^{\v_1}) = \PSSlog(\operatorname{Id} t^{\v_1}) \cdot \PSS(\alpha') $$
For a general pair of elements $\alpha_1 t^{\v_1}$, $\alpha_2 t^{\v_2}$ factoring them each as $\PSSlog(\operatorname{Id} t^{\v_i}) \cdot \PSS(\alpha_i')$, we have 
\begin{align}
    \PSSlog(\alpha_1 t^{\v_1}) \cdot \PSSlog(\alpha_2 t^{\v_2})&=\PSSlog(\operatorname{Id} t^{\v_1}) \cdot \PSS(\alpha_1') \cdot \PSSlog(\operatorname{Id} t^{\v_2}) \cdot \PSS(\alpha_2') \\  
 &=\PSSlog(\operatorname{Id} t^{\v_1} \cdot \operatorname{Id} t^{\v_2}) \cdot \PSS(\alpha_1' \cup \alpha_2') \\ 
 &=\PSSlog(\alpha_1 t^{\v_1} \cdot \alpha_2 t^{\v_2}) 
 \end{align}  
 \end{proof}

In the situation of Theorem \ref{eq:topproduct}, we obtain a presentation for the symplectic cohomology given by 
\begin{align} 
    SH^*(X) \cong \frac{\mathcal{SR}(M,D) \otimes H^*(X)}{(t^{I}\operatorname{ker} (r_{\emptyset,I}^*)=0)} 
\end{align}

where $r_{\emptyset,I}^*$ is defined in \eqref{stratumrestriction} and $\mathcal{SR}(M,D)$ is defined in \eqref{eq:SRdef}. We now give a prototypical case where our criterion applies, the case where
$M=\mathbb{C}P^n$ and $\D$ is the union of $ k \geq n+2$ generic hyperplanes in
$\mathbb{C}P^n$.  For $k \geq 2n+1$, these examples are multiplicatively
topological and so $H^*_{log}(M,\D) \cong SH^*(X)$ as rings by
Theorem \ref{thm: ringstructures}. We now explain how Theorem
\ref{eq:topproduct} can be used to verify the same ring isomorphism holds in
the remaining cases (and give a different proof in the case $k \geq 2n+1$). 

To prepare for this, we recall the following expression for the cohomology of
$X$. For any arrangement of affine hyperplanes $\mathcal{A}=\lbrace H_i
\rbrace$ in $\mathbb{C}^{n+1}$ defined by linear equations $\phi_i$, Orlik and
Solomon \cite{OrlikSolomon} have introduced the following presentation the
cohomology of $M(\mathcal{A})=\mathbb{C}^{n+1} \setminus \cup_i H_i$. For any
$i$, consider the logarithmic differential form $d\phi_i/\phi_i$, which defines
a cohomology class $\hat{\beta_i}$ in $H^1(M(\mathcal{A}))$. Let
$E_\mathcal{A}$ be the vector space generated by these forms and
$\Lambda^*(E_\mathcal{A})$ the exterior algebra. Equip this with a derivation
$\delta$ which is determined by the fact that $\delta(\hat{\beta_i})=1$.
Consider the ideal $OS(\mathcal{A})$ generated by 
\begin{itemize} 
    \item $\hat{\beta}^{I}=\hat{\beta}_{i_{1}} \wedge \hat{\beta}_{i_{2}} \cdots \wedge \hat{\beta}_{i_{I}}$ for any stratum $H_I=\emptyset$. 

    \item $\delta(\hat{\beta}^{I})$ for any $H_I$ such that $\operatorname{codim}(H_I)<|I|$.  
\end{itemize} 
Then we have an
isomorphism 
\begin{align} 
    H^*(M(\mathcal{A})) \cong \Lambda^*(E_\mathcal{A})/ OS(\mathcal{A}) 
\end{align} 

In the present situation, we take $\mathcal{A}$ to be the affine cone over our arrangement of divisors. We have an identification $M(\mathcal{A}) \cong \mathbb{C}^* \times X$ and  we may identify the cohomology of $X$ with the algebra \begin{align} H^*(X) \cong\operatorname{ker}(\delta) \subset H^*(M(\mathcal{A})) \end{align} 

Of course, we may also view our hyperplane complement $X$ as the complement of an affine hyperplane arrangement $\bar{\mathcal{A}} \subset \mathbb{C}^n$ by removing one of the divisors $H_j$. After choosing this $H_j$, letting $\beta_i=\hat{\beta_i}-\hat{\beta_j}$ gives an identification of between our description of the cohomology of $H^*(X)$ and the Orlik Solomon algebra associated to $\bar{\mathcal{A}}$. The description using affine cones is, naturally, more symmetric. We now turn to the computation of $SH^*(X)$.    

\begin{lem} $\QH^*(\mathbb{C}P^n,\D) \cong SH^*(X)$ as rings. \end{lem} 
\begin{proof} It suffices to check that the criterion of Theorem \ref{eq:topproduct} are satisfied. The vanishing of the obstructions $\underline{GW(\v_I,\v_K)}$ follow easily from the fact that $H^*(\mathbb{C}P^n) \to H^*(X)$ vanishes except in degree 0.  So it suffices to check the surjectivity of the restriction maps $r_{\emptyset,I}^*$. The higher dimensional strata $S_I^{log}$ are trivial torus bundles over the real oriented blowups of lower dimensional hyperplane complements. 

Choose any $j\notin I$ and view $X$ as an affine hyperplane complement by removing $D_j$. Then the classes $\beta_i$ for $i \in I$ are represented by forms $\frac{1}{2\pi i} dz_i/z_i$ along the restriction maps $\mathbb{C}^n\setminus \cup_{i \in I} H_i$. Thus these classes $\beta_i$, $i\in I$ surject onto the cohomology of the fibers of the torus bundles. From the Orlik-Solomon presentation, we can see that the remaining classes $\beta_s$, $s \notin I$, generate the pull-back of classes on $\DIo$. The surjectivity of the map $H^*(X) \to H^*(\SIo)$ now follows from this. \end{proof} 

We can also describe the kernel $\operatorname{ker} (r_{\emptyset,I}^*)$. As above, choose any $j$ not in $I$. Then $$\operatorname{ker} (r_{\emptyset,I}^*)=\operatorname{span}\langle\beta^J\rangle$$ such that $J \cap I=\emptyset$.

\begin{rem} 
    It would be interesting to know whether this argument could be adapted to
    the non-generic hyperplane arrangements of Lemma \ref{lem:hyperplanes}.
\end{rem}

It is interesting to see how this calculation fits into the context of Kontsevich's homological mirror symmetry conjecture. More precisely, we let $\mathcal{P} \subset \mathbb{C}P^{n+1}$ denote the hyperplane cut out by the equations $\sum_i z_i=0$ and set $X = \mathcal{P} \setminus \cup_i \lbrace z_i=0 \rbrace$ (observe that in this example the number of generic hyperplanes removed from $\mathcal{P} \cong \mathbb{C}P^n$ is $k=n+2$). The affine variety $X$ is called the generalized pair of pants \cite{Mikhalkin} and is known to play an important role in the study of homological mirror symmetry for hypersurfaces in projective space \cite{SheridanPants, Sheridan}. The mirror to $X$ is the Landau-Ginzburg model $$(\mathbb{A}^{n+2}, W=z_1z_2 \cdots z_{n+2}).$$ Let us try to identify $SH^*(X)$ (with $\mathbb{C}$ coefficients) with the Hochschild cohomology of the matrix factorization category 
$\operatorname{MF}(\mathbb{A}^{n+2}, W=z_1z_2 \cdots z_{n+2})$. The first thing
to observe is that the Jacobian ring is given by $\mathcal{J}ac(W) \cong
\mathbb{C}[z_1,\cdots,z_{n+2}]/( z_1z_2 \cdots  \widehat{z_i}\cdots z_{n+2})$
and that the map sending $\alpha_i=\operatorname{Id}t^{\v_i} \to z_i$ induces
an isomorphism: $$\mathcal{SR}^*(M,\D) \cong \mathcal{J}ac(W)$$

Let $\beta_i$ denote a standard set of generators for $H^1(X)$ determined by
viewing $X$ as an affine hyperplane complement by removing the divisor
$z_{n+2}=0$. On the mirror side, we have corresponding cohomology classes
$$z_i\partial_{z_i}-z_{n+2}\partial_{z_{n+2}} \in H^*(T^{poly}_{\mathbb{A}^{n+2}},[W,-]) $$ 
where $i \in \lbrace 1, \cdots, n+1 \rbrace$. These classes generate
$H^1(T^{poly}_{\mathbb{A}^{n+2}},[W,-])$ as a module over $\mathcal{J}ac(W)$.
With these observations in place, one can easily check for low $n$ using a
computer algebra package that the map 
$$ \QH^*(\mathbb{C}P^n,\D) \to H^*(T^{poly}_{\mathbb{A}^{n+2}},[W,-])$$ 
defined by sending $\alpha_i \to z_i$ and $\beta_i \to
z_i\partial_{z_i}-z_{n+2}\partial_{z_{n+2}} $ is an isomorphism of rings.

\subsection{Proving degeneration at $E_1$ page using GW-invariants}\label{sect: suitesspectrales}

Note that the ring structure on $H^*_{log}(M,\D)$ is generated by primitive classes of the form $\alpha t^{\v_{I}}$. Because our spectral sequence \eqref{eq:Spec2} is multiplicative, it follows that it degenerates at the $E_1$ page iff all differentials $d_r(\alpha t^{\v_{I}})=0$ for all primitive classes and all $r \geq 1$. The purpose of this section is to record some consequences of this observation.  

\begin{cor}\label{thm: easycor}  Suppose that there is a divisor $H$ such that for
    each smooth component $D_i \subset \D$, $\mathcal{O}(D_i)\cong
    \mathcal{O}(n_i H)$ for $n_i \in \mathbb{Z}^{>0}$.  If any of the $n_i >
    1$, then the spectral sequence \eqref{eq:ss} degenerates at the $E_1$ page. 
\end{cor}
	\begin{proof} 
        We equip this pair with the Kahler form in class $[\sum_{i=1}^k D_i$] (meaning we take all of the $\kappa_i$ in Definition \ref{def:logsmooth} to be 1). For
        any $A \in H_2(M)_\omega$ (meaning a class with positive symplectic
        area), the hypotheses imply that $\sum_{i=1}^k A \cdot D_i > k$, where $k$ is the number of components of $\D$ (see Definition \ref{def:logsmooth}).
        Let $\underline{y}, \underline{y}'$ be a pair of Hamiltonian orbits (for one of the Hamiltonians $\Hlm$) whose
        weighted winding numbers (see \eqref{eq:weightedwindingnumber}) satisfy
        \begin{equation} \label{eq:smallwinding}
            w(\underline{y}') < w(\underline{y}) \leq k.
        \end{equation}
       and suppose that $u$ is a Floer trajectory with inputs $\underline{y}$ and output
        $\underline{y}'$. Then $u$ can be capped off by attaching $-F(\underline{y})$, the fiberwise capping disc of $\underline{y}$ with opposite orientation (described in Lemma \ref{lem: anyJ}), and $F(\underline{y}')$, the fiberwise capping of $\underline{y}'$, to produce a sphere of homology class $A.$
       The symplectic area of $A$ is given by $$\omega(A)=w(\underline{y})-w(\underline{y}')>0.$$ 
       On the other hand, \eqref{eq:smallwinding}
        implies that such an $A$ has $\sum_i A \cdot D_i \leq k$; so
        it follows that no such Floer trajectories exist. In particular, for orbits $\underline{y}$ with
        $w(\underline{y}) \leq k$, $d_{CF}(|\mathfrak{o}_{\underline{y}}|)$ only has terms corresponding to (meaning in sum of orientation lines of) orbits $\underline{y}'$ of the same weight $w(\underline{y}')=w(\underline{y})$. In other words, the only non-trivial differentials are those which appear in the low energy Floer cohomology from \eqref{eq:lowlow}.

        Let $[\underline{\beta}]$ be an element of the $E_1$ page of the spectral sequence \eqref{eq:E1concrete} of some weight $w([\underline{\beta}]) \leq k$, the number of components of $\D$. Such an element is represented by a cocycle $\underline{\beta} \in  \frac{F_wCF^*(X \subset M; \Hlm)}{F_{w-1}CF^*(X
    \subset M; \Hlm)}$, which can be represented by a generator which is a sum of elements of orientation lines corresponding to certain Hamiltonian orbits $\underline{y}$ with winding number $w([\underline{\beta}])$. By the above analysis, this same generator defines a cocycle in $F_wCF^*(X \subset M; \Hlm)$ (and thus a class in $SH^*(X)$). This means that all such elements $[\underline{\beta}]$ of weight $w([\underline{\beta}]) \leq k$ survive to the $E_\infty$ page, i.e. that $d_r([\underline{\beta}]) = 0$ for all $r \geq 1$.

    Observe that for any primitive vector $\v_I$, $w(\v_I)= |I| \leq k$ (to see that $w(\v_I)= |I|$, recall the formula \eqref{eq:windingvec} for the weights and that we have taken $\kappa_i=1$). 
     As the identification of $E_1$ with $H^*_{log}(M,\D)$ is weight-preserving, we therefore have that $d_r(\alpha t^{\v_{I}})=0$ for all primitive classes and all $r \geq 1$. From multiplicativity (see the discussion above the Corollary), it now follows that the spectral sequence \eqref{eq:ss} degenerates at $E_1$ as desired. 
\end{proof}

Simple examples of this include the cases where $M=\mathbb{C}P^n$ and $D$ is a smooth Calabi-Yau hypersurface. Let us consider the case where $n=2$ and choose $\mathbf{k}$ to be a field of characteristic zero. We begin by calculating the cohomology $H_{log}^*(M,D)$ above. We have that $H^*(X, \mathbf{k})$ is \begin{align}H^*(X) \cong \mathbf{k}\oplus \mathbf{k}^{2}[2] \end{align}  where $[a]$ means shift the grading by $a$.  We have \begin{align} H^0(SD) \cong H^3(SD) \cong \mathbf{k} \\ \QH^0(M,D) \cong \QH^3(M,D) \cong \mathbf{k}[t]  \end{align} It is also straightforward to calculate using the Gysin exact sequence that \begin{align} H^1(SD) \cong H^2(SD) \cong \mathbf{k}^2 \end{align} Let $e_1$ and $e_2$ denote standard generators of $H^1$ of the torus and call the generators of $H^2(X)$, $x,y$. Lastly, observe that the restriction map  $H^2(X) \to H^2(SD)$ is an isomorphism. We therefore have that as $\mathbf{k}[t]$ modules \begin{align}  \QH^1(M,D) \cong \mathbf{k}[t]\cdot e_1\oplus  \mathbf{k}[t]\cdot e_2 \\ \QH^2(M,D) \cong \mathbf{k}[t]\cdot x \oplus \mathbf{k}[t]\cdot y \end{align} 

\begin{lem} When $M=\mathbb{C}P^2$ and $D$ is a smooth elliptic curve, $SH^*(X)$ is a free module over $\mathbf{k}[t]$ of the same rank as $\QH^*(M,D)$. \end{lem}

\begin{proof} Because the spectral sequence degenerates, we have that there is a filtration on symplectic cohomology such that $gr_F(SH^*(X))=\QH^*(M,D)$. A simple result on associated graded modules (see e.g., Chapter 3, Section II of \cite{Bourbaki Chapter III}) says that generators of $\QH^*(M,D)$ lift to generators of $SH^*(X)$, giving rise to a surjection from a free module $s: N \to SH^*(X).$ Applying the same argument to the kernel $K$ of this surjection shows that $s$ is an isomorphism (here we are using the fact that each $\QH^*(M,D)$ page consists of free modules over $\mathbf{k}[t]$). \end{proof}

It is again interesting to compare the result with what happens in mirror symmetry in the case where $\mathbf{k}=\mathbb{C}$. In this setting, mirror symmetry predicts (in somewhat simplified terms) that there is a mirror partner to $X$, $Y$, which admits a proper algebraic map (see Lemma \ref{lem: smlCY} for another result consistent with this) \begin{align} f: Y \to \mathbb{A}_{\mathbf{k}}^1 \end{align} In the present situation, the mirror geometry is well-known and is studied in \cite{AKO}. The mirror is a noncompact Calabi-Yau $Y$ whch admits an elliptic fibration $f:Y \to \mathbb{A}^1$ with exactly three Lefschetz critical points. By properness, we have that $$\mathbb{H}^0(\mathcal{O}_Y) \cong \mathbb{H}^1(\Lambda^2T_Y) \cong \mathbf{k}[f]$$  An elementary sheaf theory computation shows that we have $$\mathbb{H}^1(\mathcal{O}_Y) \oplus \mathbb{H}^0(T_Y) \cong \mathbb{H}^1(T_Y) \oplus \mathbb{H}^0(\Lambda^2 T_Y) $$ are both free modules of rank two over $ \mathbf{k}[f]$. \vskip 5 pt

We next consider a slightly more general case where the differentials in the spectral sequence can be calculated in terms of Gromov-Witten theory. 

 \begin{defn} 
     Given a class $A \in H_2(M,\mathbb{Z})$, we will denote by $A \cdot
     \mathbf{D}$ the vector of multiplicities $[A\cdot D_i]_{i=1}^k \in
     \mathbb{Z}^k$.  
 \end{defn}

\begin{defn} Fix a pair $(M,\D)$ and let $H_2(M)_\omega$ denote the set of
    integral classes $A$ for which $\omega(A)>0$. We say that a primitive
    vector $\v_I$ is {\em admissible} if for all $A \in H_2(M)_\omega$, $$
    \sum_i \kappa_i (A \cdot D_i) \geq  w(\v_I). $$ \end{defn}

Given a primitive cohomology class $\alpha t^{\v_I}$, corresponding to an
admissible vector, 
there is exactly one possible non-vanishing differential on $\alpha t^{\v_I}$ in
the spectral sequence, namely
$$d_{w(\v_I)}(\alpha t^{\v_I}) \in H^*(X).$$

\begin{defn}\label{defn: modsphere} 
    Consider $\mathbb{C}P^1$ with two fixed marked points $z_0$, $z_1$ at 
    $0, \infty$ respectively, and let $\v_I$ be a primitive vector. For any $A \in H_2(M,\mathbb{Z})$ satisfying 
    \begin{align} 
        \label{eq: intersection} A \cdot \mathbf{D}=\v_I 
    \end{align} set
    $\widetilde{\mathcal{M} }_{0,2}(M,\mathbf{D},\v_I, A)$ to be the
    moduli space of maps $ u: (\mathbb{C}P^1,z_0,z_1) \to (M,\D)$ with 
    \begin{align}
        u_*([\mathbb{C}P^1])=A \\ u^{-1}(D_I)= z_0 
    \end{align}
  \end{defn}

 \begin{defn} From the moduli spaces in Definition \ref{defn: modsphere} form: 
     \begin{align} 
         \widetilde{\mathcal{M}}_{0,2}(M,\mathbf{D},\v_I) :=\bigsqcup_{A, A \cdot \D=\v_I} \widetilde{\mathcal{M}}_{0,2}(M,\mathbf{D},\v_I, A) \\ 
         \mathcal{M}_{0,2}^{S^1}(M,\mathbf{D},\v_I):=\widetilde{\mathcal{M}}_{0,2}(M,\mathbf{D},\v_I)/\mathbb{R} 
     \end{align} 
     where the quotient in the latter equation is by $\mathbb{R}$-translations. 
  \end{defn}

These moduli spaces are canonically oriented. We have an evaluation map 
at $z_1 = \infty$
$$ ev_{\infty}: \mathcal{M}_{0,2}^{S^1}(M,\mathbf{D},\v_I) \to  M. $$

Let $\mathcal{M}_{0,2}^{S^1}(M,\mathbf{D},\v_I)^o$ denote the preimage $ev_{\infty}^{-1}(X).$ By the arguments of 
\cite[Lemma 3.26]{GP1}, 
for an admissible vector $\v_I$,  the moduli spaces
$\mathcal{M}_{0,2}^{S^{1}}(M,\mathbf{D},\v_I)^o$ map properly to $X$. Moreover,
after equipping the marked point at $z_0$ with the additional data of real
positive ray in $T_{z_{0}}\mathbb{C}P^1$, we have as in
\eqref{eq:enhancedevaluationpss} an enhanced evaluation map 
$$\Evo: \mathcal{M}_{0,2}^{S^{1}}(M,\mathbf{D},\v_I)^o \to  \SIo. $$  
It follows that, for any class $\alpha$ in $H^*(\mathring{S}_I)$, we may define
a Borel-Moore homology class in $H^{BM}_*(X)$ via 
\begin{align} \label{eq: GWinvariantsnormal} 
    GW_{\v_{I}}(\alpha)=[ev_{\infty,*}(\operatorname{Ev}_0^{\v_I,*}(\alpha))] \in H^{BM}_{*}(X). 
\end{align}
(compare \cite[eq. (3.35)]{GP1}).

\begin{thm} \label{thm: degenerescence} 
    Assume $\mathbf{k}= \mathbb{Z}$. For any admissible vector $\v_I$, with
    respect to the identification of the first page of \eqref{eq:ss} with log
    cohomology by Theorem \ref{thm:main}, there is an equality 
    \begin{align} 
        d_{w(\v_I)}(\alpha t^{\v_I})=GW_{\v_{I}}(\alpha) \in H^*(X) \subset \QH^*(M,\D).
    \end{align}
\end{thm} 
\begin{proof} 
    Given an admissible vector $\v_I$, and a cohomology class $\alpha
    t^{\v_I}$, choose $\underline{\alpha} t^{\v_I} \in C^*_{log}(M,\D)$ be a
    cocycle which represents this class. We choose generators of the
    orientation lines associated to each critical point and denote them by $c$
    in a slight abuse of notation. For every $c$, let $\bar{W}^s(c)$ denote the
    partial compactification of $W^s(c)$ to a manifold with boundary given by
    adding in simply broken flow lines.  
    Recall from \cite[\S 4.1]{Schwartz} that, given a cocycle of the form
    $$ \underline{\alpha}= \sum_{c} \alpha_c,\ \alpha_c \in |\mathfrak{o}_c| $$ 
    we may construct a corresponding pseudocycle
    $Z(\underline{\alpha})$ by gluing together ($|\alpha_c|$ copies of) the stable manifolds
    $\bar{W}^s(c)$ (oriented by whichever generator $\alpha_c$ is a positive multiple of) 
    along cancelling boundary components.
    Given a primitive admissible class and a generic complex structure,
    restricting the fiber product 
\begin{align}
    \mathcal{M}_{0,2}(M,\D,\vec{\v}_{I})^o
    \times_{\operatorname{Ev}_0} Z(\underline{\alpha}) 
\end{align} 
to $\bar{X}_\ell \subset X$ defines a pseudocycle (rel boundary) $\GWvc$ whose corresponding relative homology class is 
$\GWv$. Consider also the variant of the PSS map defined using pseudocycles: 
\begin{align}  \label{eq:PSSpseudo}
    \PSS_{\vec{\v}_{\emptyset}}:H_{2n-*}(\bar{X}_\ell,\partial{\bar{X}}_\ell) \to
    F_0 HF^{*}(X \subset M; H^\ell),
\end{align} 
which is defined by representing elements of 
$H_{2n-*}(\bar{X}_{\ell},\partial{\bar{X}}_\ell)$ by relative psuedocycles $P$
such that $\partial P \subset \partial{\bar{X}}_\ell$. On homology this map agrees with
the classical PSS map defined using Morse co-chains.

In \cite[Lemma 4.21]{GP1}, 
we proved that 
\begin{align} \label{eq:randomeqGP1}
    \partial_{CF}(\PSSlog(\underline{\alpha} t^{\v_I}))= \PSS_{\vec{\v}_{\emptyset}}(\GWvc).  
\end{align}
The result follows from the definition of the pages $(E_r^{p,q},d_r)$ together with the fact that on the $E_1$ page of the
spectral sequence, $[\PSSlog(\underline{\alpha} t^{\v_I})]=\PSSlog^{low}(\alpha
t^{\v_I})$ and $[\PSS(\GWvc)]=\PSS(\GWv).$ 
\end{proof} 

\begin{rem} 
    The assumption that $\mathbf{k}= \mathbb{Z}$ in Theorem \ref{thm:
    degenerescence} was only for simplicity and can likely be removed.
\end{rem} 

\begin{cor} \label{cor: degenerescencecor}
 Assume $\mathbf{k}= \mathbb{Z}$. Given an admissible pair $(M,\D)$, suppose
 the maps \eqref{eq:obstructionclass} vanish for all $I \in \lbrace 1,\cdots k
 \rbrace$. Then the spectral sequence \eqref{eq:ss} degenerates at the $E_1$
 page. 
 \end{cor}

\begin{rem} 
    As mentioned in the introduction, Section 6.1.2 of \cite{diogothesis}
    considered the case where $M \cong \mathbb{P}^1\times \mathbb{P}^1$ is a
    quadric in $\mathbb{P}^3$ and $D$ is a smooth hyperplane. Here, the moduli
    spaces $\mathcal{M}_{0,2}(M,\mathbf{D},\v_I)$ are nonempty; however the
    invariants $GW_{\v_{I}}(\alpha)=0$. In this case, the relevant class lives
    in $H^2(X)$ and it is easy to compute the moduli spaces directly as is done
    in \emph{loc. cit}. A more abstract way to see that it vanishes is to
    observe $GW_{\v_{I}}(\alpha)$ is invariant under the action which switches
    the two factors of $\mathbb{P}^1$ and hence vanishes when restricted to the
    complement. Similar reasoning involving group actions applies e.g. in the
    case when $\D=D_1 \cup D_2$ is the union of two hyperplanes. It would be
    interesting to see if such reasoning applies in other cases.  
\end{rem}

Our partial description of the differentials in the spectral sequence may be
used to prove the following strong finiteness result:

\begin{thm} \label{thm: finitely} 
    Assume that all $\kappa_i=1$ and all of the strata $D_I$ are connected.
    Then $SH^*(X)$ is finitely generated as a graded algebra over $\mathbf{k}$.
\end{thm}
\begin{proof} 
    Because all of the $\kappa_i=1$, all of the vectors $\v_i$  are admissible
    for $i \in \lbrace 1,\cdots, k \rbrace$. Choose over each $\mathring{S}_i$
    a Morse function with a unique critical point $c$ of degree zero and again
    choose a generator for the corresponding orientation line which we denote
    by $c$ as well. Because the cycle  $\mathcal{M}_{0,2}(M,\D,\vec{\v}_{I})^o
    \times_{\operatorname{Ev}_0} Z(c)$ factors through a lower dimensional
    manifold, we have that the right hand side of \eqref{eq:randomeqGP1} vanishes (compare with \cite[Lem. 4.23 and Rem. 4.25]{GP1} for the same claim phrased in the language of singular cohomology). Equation  \eqref{eq:randomeqGP1} thus becomes: 
  
    \begin{align} \label{eq: BormanSheridan}  \partial_{CF}{\PSS_{log}(c \, t^{\v_i})}=0 \end{align}
    i.e., we obtain a class $\PSSlog(\alpha_i t^{\v_i})$ in
    symplectic cohomology, where $\alpha_i \in H^0(\mathring{S}_i)$ is the
    fundamental class of the stratum $\mathring{S}_i$. Let $\mathcal{A}$ denote
    the subalgebra of $SH^*(X)$ generated by $\PSSlog(\alpha_i t^{\v_i})$.
    Connectedness of all
    of the stata implies that the associated graded of $\mathcal{A}$ is a
    quotient algebra of $\mathcal{SR}(M,\D)$ and that the first page of the
    spectral sequence is a finitely generated module over $\mathcal{SR}(M,\D)$.
    As  the spectral sequence is multiplicative this means that the cohomology
    groups of all subsequent pages are also finitely generated modules over
    $\mathcal{SR}(M,\D)$. Because our filtrations on $\mathcal{A}$ and
    $SH^*(X)$ are increasing and in nonnegative degree (and in particular
    induce discrete topologies), Chapter 3, Section II of \cite{Bourbaki
    Chapter III} again implies that $SH^*(X)$ is finitely generated over
    $\mathcal{A}$ as well. 
\end{proof}

\subsection{Log Calabi-Yau pairs} \label{sect:logCY}

One may also consider how the spectral sequence behaves in specific
cohomological degrees. A natural situation to do this in is where $\D$ is an
anticanonical divisor, i.e. $M$ admits a meromorphic volume form with all
$a_i=1$. In this case, we say that $(M,\D)$ is a {\em log Calabi-Yau
pair}.
\footnote{This notion of log Calabi-Yau pair is stricter than the usual
usage in birational geometry \cite{bicluster}.}
Here we wish to use our methods to derive information about $SH^0(X)$. For such
pairs, we have $$\QH^*(M,\D) =0, \quad *<0$$ and to understand the degree zero
piece of the $E_\infty$ page, it suffices to analyze the differential in degree
zero.

Let $(M, \D)$ be a log Calabi-Yau pair such that all strata $D_I$ are
connected. Then the ring $\QH^0(M,\D)$ is generated by classes $\alpha_i
t^{\v_i}$  for $i \in \lbrace 1,\cdots, k \rbrace$ where $\alpha_i \in
H^0(\mathring{S}_i)$ is the fundamental class. Again, to prove that $d_r$
vanishes in degree zero, it suffices to prove that $d_r(\alpha_i t^{\v_i})=0$
for all $r \geq 1$.   

 \begin{thm} \label{thm: Fano} 
     Let $(M, \D)$ be a pair with $M$ a Fano manifold and $\D$ an anticanonical
     divisor. Assume that all strata $D_I$ are connected. Then the spectral
     sequence degenerates in degree zero. With respect to the standard
     filtration, we have an isomorphism 
     \begin{align} 
         gr_F SH^0(X) \cong \mathcal{SR}(M,\D),
     \end{align} 
     where $\mathcal{SR}(M,\D)$ is the Stanley-Reisner ring on the dual
     intersection complex of $\D$ (see \eqref{eq:SRdef}).  
 \end{thm} 
\begin{proof} 
Under the assumption that $M$ is Fano, all of the vectors $\v_i$ are admissible
for $i \in \lbrace 1,\cdots, k \rbrace$. Choose over each $\mathring{S}_i$ a
Morse function with a unique critical point $c$ of degree zero and again choose
a generator for the corresponding orientation line which we denote by $c$ as
well. As we saw in \eqref{eq: BormanSheridan}, $$ \partial_{CF}{\PSS_{log}(c \, t^{\v_i})}=0$$
i.e. this defines a class $\PSSlog(\alpha_i t^{\v_i})$ in symplectic
cohomology. As their image in the $E_1$ page generates $\QH^0(M,\D)$ as a ring,
the result follows. 
\end{proof}

We next show how to adapt our methods to recover a result of Pascaleff
\cite{Pascaleff} that for any log Calabi-Yau surface, there is an isomorphism
$gr_F SH^0(X) \cong \QH^0(M,\D)$. To prepare for this, we need to introduce a
little bit more terminology. If $(T,A)$ is a tree with
$|E_{\operatorname{ext}}(T)|=2$ equipped with a labelling function $$A: V(T)
\to  H_2(M,\mathbb{Z}),$$ we let  $\mathcal{M}_{0,2}(T,A, M)$ denote the moduli
space of sphere bubbles modelled on this labelled tree. By definition, we have
that $$ \mathcal{M}_{0,2}(T,M):=\bigsqcup_{A} \mathcal{M}_{0,2}(T,A, M).$$ Let
$z_1 = \infty$ and $z_0=0$ be the marked points corresponding to the external
vertices $\cev{e}$ and $\vec{e}$ respectively. We let $\nu_f$ and $\nu_{i}$
denote the vertices which are attached to each of these points. 

\begin{defn}  
    For any $\v',\v \in (\mathbb{Z}^{\geq 0})^k,$ we let 
$$ \mathcal{M}_{0,2}(T,A,\v', \v) \subset  \mathcal{M}_{0,2}(T,A, M) $$
denote the subset of maps admitting a pre-logarithmic enhancement with 
\begin{align}  
    \operatorname{ord}_{\nu_{i}}(z_0)=\v'  \\
    \operatorname{ord}_{\nu_{f}}(z_\infty)=-\v 
\end{align} 
\end{defn}

Notice that 
\begin{equation}\label{obs:emptymodulispaces}
    \textrm{$\mathcal{M}_{0,2}(T,A,\v', \v)$ is empty unless $A \cdot \D=\v'-\v$.}
\end{equation}

\begin{defn}\label{defn: logCY} 
Let $(M,\D)$ be a log Calabi-Yau pair and let $\v_I$ denote a primitive vector.
We say that the vector $\v_I$ is {\em degree zero admissible} if the moduli spaces 
$\mathcal{M}_{0,2}(T,A,\v_I, \v)= \emptyset $ unless  
\begin{itemize} 
    \item $\v=\mathbf{0}$; and 
    \item there is exactly one $\nu \in V(T)$ for which $A(\nu) \neq 0.$ 
    \end{itemize} 
\end{defn} 

In the case $T$ has a single vertex, we write $ \mathcal{M}_{0,2}(T,A,\v_I,
\mathbf{0}):= \mathcal{M}_{0,2}(A,\v_I)$  and let
$\mathcal{M}_{0,2}(A,\v_I)^{o}:= ev_{\infty}^{-1}(X)$. The moduli space
$\mathcal{M}_{0,2}(A,\v_I)$ has expected dimension $2n-2$ and, because the vector
$\v_I$ is primitive, $\mathcal{M}_{0,2}(A,\v_I)^{o}$ is a manifold of this
expected dimension for generic $J$.\footnote{This follows because each curve $u
    \in \mathcal{M}_{0,2}(A,\v_I)^{o}$ is necessarily somewhere injective.} 

When $\v_I$ is degree zero admissible, $|V(T)| \leq 2$ because configurations 
either consist of a single component with two marked points (in which case
we are in the situation of the previous paragraph) or a single constant
component in $\D$ with three marked points glued to a non-constant
component with one marked point. We can describe the configurations with two
components very explicitly: the constant sphere bubble comes equipped with an
$|I|-tuple$ of meromorphic functions (well-defined up to a rescaling action)
with two marked points corresponding to unique zeros and poles of order $\v_I$.
The zero corresponds to the point $z_0$ and the nonconstant component is glued
in along the pole. The marked point $z_\infty$ corresponds to the third marked
point on the constant sphere, at which the meromorphic functions has neither a
zero or pole. In the case $|V(T)|=2$, we therefore see that $ev_{\infty} \in
\D.$  

Corollary \ref{lem:logCYsur} below shows that in dimension 2, all $\v_I$ are
degree zero admissible. It is an immediate consequence of the following
stronger lemma:
\begin{lem} \label{lem:logCYlem} 
    Let $(M,\D)$ be a log Calabi-Yau pair with $\operatorname{dim}_\mathbb{C}
    M=2$.  Let $\v',\v$ be any pair and let $u \in \mathcal{M}_{0,2}(T,A,\v',
    \v)$. Then any component $u_\nu$ of $u$ which is contained in $\D$ must be
    constant. 
\end{lem}
\begin{proof} 
    Suppose that the statement of the Lemma is false and let $u_\nu$ be a
    non-constant component in $\D$ whose corresponding node is furthest away
    from $\vec{e}$ (in terms of number of nodes in between that node and
    $\vec{e}$). 
    When $\operatorname{dim}_\mathbb{C} M=2$, $\D$ is either a
    cycle of $\mathbb{P}^1$'s or an elliptic curve, and in particular we know
    that
    every non-constant component $u_{\nu}$ in $\D$ has at least two marked
    points $z$ and $z'$ at which $\operatorname{ord}_{\nu,i}(z)>0$
    (respectively $\operatorname{ord}_{\nu,i'}(z')>0$)  for some divisor $D_i$
    (respectively $D_{i'}).$  Without loss of generality, assume that $z$
    corresponds to an edge $e$ which lies in a different connected component of
    $T \setminus \nu$ from $\vec{e}$. In $T \setminus \nu$, the connected
    component containing $e$ defines a tree $T_e$ beginning at $e$. Then
    because $\vec{e}\notin E(T_e)$, we must have that $\sum_{\nu \in T_e}
    A(\nu) \cdot  D_i\leq -\operatorname{ord}_{\nu,i}(z)<0$. But this is
    impossible if none of the components $u_\nu$, $\nu \in T_e$ lies in $\D$,
    contradicting the fact that $u_\nu$ was a component the furthest away from
    $\vec{e}$. 
\end{proof} 

\begin{cor} \label{lem:logCYsur} 
    Let $(M,\D)$ be a log Calabi-Yau pair with $\operatorname{dim}_\mathbb{C}
    M=2$. Then every primitive vector $\v_I$ is degree zero admissible.
\end{cor} 
\begin{proof} 
    Suppose some $\mathcal{M}_{0,2}(T,A,\v_I, \v)\neq \emptyset$.
    By \eqref{obs:emptymodulispaces}, observe that $A \cdot \D \leq \v_I$. 
    In view of Lemma \ref{lem:logCYlem}, every nonconstant curve $u$ in this
    moduli space intersects $\D$ transversely with positive intersection
    multiplicites. Let us consider the
    component $u_\nu$ containing the marked point $z_0$. If it is non-constant,
    then $u_\nu$ must
    pass through $D_I$ and by primitivity the intersection must be at least
    $\v_I$. It follows that it equals $\v_I$ and that there can be no
    other intersections with the divisor or any other non-constant components.
    If $u_\nu$ is constant, then it is connected by a chain of constant curves
    to a non-constant curve $u_{\nu'}$ for which we again have $u_{\nu'} \cdot
    \D= \v_I$. In particular, $\v$ must equal $\mathbf{0}$, and there is
    exactly one non-constant component as desired.
\end{proof}

We now turn to describing how to modify the proof of Lemma
\ref{lem:compactness2}, using the fact that sphere bubbling is relatively
controlled in this setting:

\begin{lem} \label{lem: SH0class} 
    Let $(M,\D)$ be a log Calabi-Yau pair and let $\v_I$ be a degree zero
    admissible primitive vector. Choose a Morse function on $\SIo$ with a
    unique degree zero critical point on each connected component. If $c$ is
    one of these critical points, then for generic choices of complex
    structure, the count of elements in \eqref{eq: fullPSSdef} defines a class
    in $SH^*(X)$.  
\end{lem}  
\begin{proof} 
    As in the proof of Lemma \ref{lem:compactness2}, for any $x_0$ such that
    $\operatorname{vdim}(\mc{M}(\v_I, x_0)) \leq 1$, consider the closure
    $\overline{\mc{M}}(\v_I, x_0) \subset \overline{\mc{M}}(x_0)$. As before,
    we temporarily ignore the possibility of sphere bubbles forming at other
    marked points other than $z_0$ along $S$ or along Floer cylinders. Consider
    a subsequence $u_n$ converging to some limit 
    $$ u_{\infty} \in \prod_{x_1,\cdots,x_r}  \mc{M}(x_0,x_r) \times \cdots \times
    \mathcal{M}(x_2,x_{1})  \times  \mathcal{M}(T,x_1).$$ 
    We again conclude using the rescaling argument of \cite[Lemma 4.9]{Tehrani} 
    that the
    corresponding $u_{\infty,T} \in \mathcal{M}_{0,2}(T,M)$ admits a
    pre-logarithmic enhancement such that if $\nu_f$ is the vertex bounding
    $\cev{e}$, 
    \begin{align} 
        \operatorname{ord}_{\nu_{f}}(z_\infty)= -\v
    \end{align} 
    where $\v$ is equal to the intersection multiplicity of the PSS
    solution at $z_0$. Because $\v_I$ is degree zero admissible it follows that
    $\v=\mathbf{0}$ and that the PSS solution lies in $X$. Thus
    $ev_\infty(u_{\infty,T}) \in X$ and $u_{\infty,T} \in
    \mathcal{M}_{0,2}(A,\v_I)^{o}$. 

    Meanwhile, whether such bubbling occurs or not at $z_0$, we see that by
    conservation of intersection with $\D$ that there can be no further sphere
    bubbles at points other than $z_0$ (again c.f. the proof of Lemma
    \ref{lem:compactness2}). To conclude, it suffices to observe that for
    generic complex structures, $\operatorname{dim}(\mc{M}(\mathbf{0}, x_0))
    \leq 1$ while  $\operatorname{dim}(\mathcal{M}_{0,2}(A,\v_I)^{o}) =2n-2$
    and so such bubbling configurations do not exist generically. 
\end{proof} 

As in the discussion following \eqref{eq: fullPSSdef}, if $\alpha \in
H^0(\SIo)$ is the Morse cohomology class corresponding to $c$, we again denote
the class constructed in Lemma \ref{lem: SH0class} by $\PSS_{log}(\alpha
t^{\v_I}).$
\begin{thm}[\cite{Pascaleff}] \label{thm: logCYsurf} 
Assume that $(M,\D)$ is a log Calabi-Yau surface.  Then the spectral
sequence \eqref{eq:Spec2} degenerates in degree zero. With respect to the
standard filtration, we have an isomorphism 
    \begin{align} 
        gr_F SH^0(X) \cong \QH^0(M,\D) 
    \end{align} 
\end{thm}
\begin{proof} 
    This follows essentially as in Theorem \ref{thm: degenerescence}. We have
    that  $\QH^0(M,\D)$ is generated by the classes $\alpha t^{\v_I}$ for $i
    \in \lbrace 1, \cdots k \rbrace$. By Corollary \ref{lem:logCYsur} and Lemma
    \ref{lem: SH0class}, we have classes  $\PSS_{log}(\alpha t^{\v_I})$ for
    every $I$, whose image generates the first page of the spectral sequence
    \eqref{eq:Spec2} multiplicatively. It follows that the spectral sequence
    degenerates.  
\end{proof}  

We conclude the paper with an extended remark concerning  Theorems \ref{thm: Fano} and \ref{thm: logCYsurf}. As we have seen in Theorem $\ref{thm: toppair2}$, in the topological case, the $\PSSlog$ map defines a canonical splitting of the spectral sequence \eqref{eq:Spec2}.  It is very likely that in the setting of the above two Theorems, one may use the ``full" PSS moduli spaces to define a similar splitting of the spectral sequence in degree zero. For a simple example to illustrate this idea, suppose that $\D=D$ is a smooth anticanonical divisor.

\begin{lem} \label{lem: smlCY} There are canonically defined elements $s_\v$ together with an isomorphism 
    \begin{align}\label{eq:isolCY} 
        \operatorname{PSS}_{log}: \bigoplus_{\v \in \mathbb{N}^{\geq 0} } \mathbf{k} \cdot s_{\v} \cong SH^0(X, \mathbf{k}) 
    \end{align}  
    Moreover as a ring we have that 
    \begin{align} 
        \label{eq: multiso} \mathbf{k}[s_1] \cong SH^0(X,\mathbf{k}). 
    \end{align} 
\end{lem}

\begin{proof} 
    For any $\v \geq 1$, choose a Morse function on $SD$ with a unique critical
    point $c$. Then again as before
    the elements $\operatorname{PSS}_{log}(c\, t^{\v})$ define a class in symplectic
    cohomology $SH^*(X)$. The difference between this situation and Lemma
    \ref{lem:compactness2} is that in the present situation, sphere bubbling can
    arise near the point $z_0$. However, in this case, after passing to the
    somewhere injective images of curves, this sphere bubbling occurs in
    codimension 2 (this works with either the standard Deligne-Mumford
    compactification or logarithmic/SFT enhancements). 
    As a result, we obtain a
    map 
    \begin{align} \label{eq: PSSmorph} 
        \operatorname{PSS}_{log}: H_{log}^0(M,D) \to SH^0(X).
    \end{align} 
    The arguments of \S \ref{sect: PSSiso1} and \S \ref{sect: PSSiso2} 
    apply without change when restricted to the degree zero pieces to show that
    this an isomorphism.  
    Letting $\alpha_{\v}$ denote a copy of the fundamental class on either $X$
    or $SD$ for each $\v$ and setting $\alpha_\v t^{\v}$ by $s_{\v}$ therefore
    proves the first part. 

    The fact there is an isomorphism of rings $k[s_1] \cong SH^0(X)$ follows
    from (a very special case of) Theorem \ref{thm: Fano} because $gr_F(SH^0(X))$ is
    isomorphic to the ring for which $s_{\v_{1}} \cdot
    s_{\v_{2}}=s_{\v_1+\v_2}$. As this is a polynomial ring in $s_1$, $k[s_1]$,
    it follows that $SH^0(X)$ is as well because a polynomial ring has no
    commutative deformations. 
\end{proof}

 \begin{rem} 
     Unlike the case of multiplicatively topological pairs, the map
     $\operatorname{PSS}_{log}$ above is {\em not} compatible with the
     topological product on log cohomology defined in Definition \ref{defn:
     ringstructure}, even though there is an isomorphism abstractly \eqref{eq:
     multiso}.  For example, it is not difficult to see by a modification of
     Theorem \ref{thm: ringstructures} that, in the case where $M=\mathbb{P}^2$
     and $E$ is a smooth elliptic curve, 
     $\operatorname{PSS}_{log}(s_{1})^3=\operatorname{PSS}_{log}(s_3)+6$ (the
     number 6 arises here as the degree of the dual elliptic curve from
     classical algebraic geometry).

    In the general case, the isomorphisms \eqref{eq:isolCY} and \eqref{eq: multiso}
    contain rich enumerative geometry, worthy of further exploration. More
    precisely, the elements $s_{\v}$ correspond to canonically defined degree $\v$
    polynomials whose coefficients are defined in terms of certain (relative)
    Gromov-Witten invariants. We leave this and generalizations of Lemma \ref{lem: smlCY} 
    to the normal crossings setting to future work.  
    \end{rem}

%%fakesection: starting appendix
\appendix

%fakesection:bibliography
\bibliography{shbib}
\bibliographystyle{apalike} 
%\printbibliography

\end{document}